\documentclass[reqno]{amsart}

\usepackage{amsmath}
\usepackage{amsthm}
\usepackage{amsfonts}
\usepackage{amssymb}
\usepackage{enumerate}
\usepackage{graphicx}
\usepackage{todonotes}
\usepackage{stmaryrd}
\usepackage[nocompress]{cite}
\usepackage{mathabx}
\usepackage{comment}

\usepackage{algorithm}
\usepackage{algpseudocode}

\usepackage{dsfont}
\usepackage{textcomp}

\usepackage{hyperref}
\hypersetup{
    colorlinks,%
    citecolor=black,%
    filecolor=black,%
    linkcolor=black,%
    urlcolor=black
}

\DeclareMathOperator{\tr}{tr}

\DeclareMathOperator{\diag}{diag}
\DeclareMathOperator{\dist}{dist}
\DeclareMathOperator{\rank}{rank}

\DeclareMathOperator{\Span}{Span}

\newcommand{\Prob}{\mathbb{P}}
\newcommand{\E}{\mathbb{E}}

\renewcommand\Re{\operatorname{Re}}
\renewcommand\Im{\operatorname{Im}}
\newcommand{\eps}{\varepsilon}
\newcommand{\T}{\mathrm{T}}

\theoremstyle{plain}
  \newtheorem{theorem}{Theorem}

  \newtheorem{proposition}[theorem]{Proposition}
  
  \newtheorem{lemma}[theorem]{Lemma}
  \newtheorem{corollary}[theorem]{Corollary}

\theoremstyle{definition}
  \newtheorem{definition}[theorem]{Definition}
  
  \newtheorem{remark}[theorem]{Remark}

\begin{document}

\title{Analysis of singular subspaces under random perturbations}

\author{Ke Wang}
\thanks{K. Wang is partially supported by Hong Kong RGC grant GRF 16304222, GRF 16308219 and ECS 26304920.}
\address{Department of Mathematics, Hong Kong University of Science and Technology, Hong Kong}
\email{kewang@ust.hk}

\begin{abstract}
We present a comprehensive analysis of singular vector and singular subspace perturbations in the signal-plus-noise matrix model with random Gaussian noise. Assuming a low-rank signal matrix, we extend the Davis-Kahan-Wedin theorem in a fully generalized manner, applicable to any unitarily invariant matrix norm, building on previous results by O'Rourke, Vu, and the author. Our analysis provides fine-grained insights, including $\ell_\infty$ bounds for singular vectors,  $\ell_{2, \infty}$ bounds for singular subspaces,  and results for linear and bilinear functions of singular vectors. Additionally, we derive $\ell_{2,\infty}$ bounds on perturbed singular vectors, taking into account the weighting by their corresponding singular values. Finally, we explore practical implications of these results in the Gaussian mixture model and the submatrix localization problem. 
\end{abstract}




\maketitle

\section{Introduction}\label{sec:intro}
Matrix perturbation theory has emerged as a central and foundational subject within various disciplines, including probability, statistics, machine learning, and applied mathematics. Perturbation bounds, which quantify the influence of small noise on the spectral parameters of a matrix, are of paramount importance in numerous applications such as matrix completion \cite{MR2565240,Cands2010MatrixCW,KMO09}, principal component analysis (PCA) \cite{JL09}, and community detection \cite{VLBB08, Vu18}, to mention a few. This paper aims to present a comprehensive analysis establishing perturbation bounds for the singular vectors and singular subspaces of a low-rank signal matrix perturbed by additive random Gaussian noise.

Consider an unknown $N \times n$ data matrix $A$.  Suppose we cannot observe $A$ directly but instead have access to a corrupted version $\widetilde{A}$ given by
\begin{equation} \label{def:tildeA}
	\widetilde{A} := A + E, 
\end{equation} 
where $E$ represents the noise matrix.  In this paper, we focus on real matrices, and the extension to complex matrices is straightforward. 

Assume that the $N \times n$ data matrix $A$ has rank $r \geq 1$.  The singular value decomposition (SVD) of $A$ takes the form $A=U D V^T,$ where $D= \diag(\sigma_1,\ldots, \sigma_r)$ is a diagonal matrix containing the non-zero singular values $\sigma_1\ge \sigma_2 \ge \cdots \ge \sigma_r>0$ of $A$; the columns of the matrices $U=(u_1,\ldots,u_r)$ and $V=(v_1,\ldots,v_r)$ are the orthonormal left and right singular vectors of $A$, respectively.  In other words, $u_i$ and $v_i$ are the left and right singular vectors corresponding to $\sigma_i$.  It follows that $U^T U = V^T V = I_r$, where $I_r$ is the $r \times r$ identity matrix.  For convenience we will take $\sigma_{r+i} = 0$ for all $i \geq 1$.  
 Denote the SVD of $\widetilde A$ given in \eqref{def:tildeA} similarly by $\widetilde A = {\widetilde U} {\widetilde D} {\widetilde V}^\T$, where the diagonal entries of ${\widetilde D}$ are the singular values $\widetilde \sigma_1 \ge \widetilde \sigma_2 \ge \cdots \ge \widetilde \sigma_{\min\{N,n\}} \geq 0$, and the columns of $\widetilde U$ and $\widetilde V$ are the orthonormal left and right singular vectors, denoted by $\widetilde u_i$ and $\widetilde v_i$, respectively. 

The primary focus of this paper is the singular subspaces that are spanned by the leading singular vectors. For $1\le k \le r$, let us denote
\begin{equation*}
\begin{aligned}
&U_k := \Span\{u_1,\ldots, u_k \}, \quad V_k := \Span\{v_1,\ldots, v_k \},\\
&\widetilde U_k := \Span\{\widetilde u_1,\ldots, \widetilde u_k \}, \quad \widetilde V_k := \Span\{\widetilde v_1,\ldots, \widetilde v_k \}.
\end{aligned}
\end{equation*}
With a slight abuse of notation, we also use $U_k = (u_1,\ldots, u_k)$ to represent the singular vector matrix. We employ the notation $V_k, \widetilde U_k, \widetilde V_k$ in a similar manner. Let $P_{U_k}=U_k U_k^\T$ (resp. $P_{V_k}=V_k V_k^\T$) be the orthogonal projection on the  subspace $U_k$ (resp. $V_k$). Denote the orthogonal complement of a subspace $W$ as $W^\perp$.

The classical perturbation bounds related to the changes in singular values and singular vectors are detailed below. The matrix norm $\vvvert \cdot \vvvert $ on $\mathbb R^{N\times n}$ is said to be unitarily invariant if $\vvvert A\vvvert  = \vvvert UAV\vvvert $ for all orthogonal matrices $U \in \mathbb R^{N\times N}$ and $V \in \mathbb R^{n\times n}$. In addition, we always consider the norm $\vvvert \cdot \vvvert $ to be \emph{normalized}.  This means that the norm always satisfies $\vvvert A \vvvert = 1$ if $A$ has its $(1,1)$ entry equal to 1 and all other entries equal to zero. A more thorough exploration of the properties of unitarily invariant matrix norms can be found in the supplementary material \cite{WangSupp24}. 

Denote
$\diag(\sigma_i - \widetilde\sigma_i)=\diag(\sigma_1 - \widetilde\sigma_2, \cdots, \sigma_{\min\{N,n\}} - \widetilde{\sigma}_{\min\{N,n\}}).$ This represents the difference in singular values between $A$ and $A+E$. The perturbations or changes in the singular values of $A$ and $A+E$ are provided by Mirsky's theorem (see Theorem 4.11 in Chapter IV from \cite{MR1061154}). 
\begin{theorem}[Mirsky]
Let $\widetilde{A} = A + E$ as in \eqref{def:tildeA}. Then for any unitarily invariant norm $\vvvert \cdot \vvvert$,
$$\vvvert \diag(\sigma_i - \widetilde\sigma_i) \vvvert \le \vvvert E \vvvert.$$
\end{theorem}
When applied to the operator norm and eigenvalues of Hermitian matrices, the inequality stated can be recognized as the Weyl's inequality (see \cite[Corollary III.2.6]{Bhatia}).

The differences between subspaces $U_k$ and $\widetilde U_k$ of $A$ and $A+E$ can be quantified by calculating the separation between $U_k$ and $\widetilde U_k$. This is achieved using $k$ principal angles, defined as $0\le \theta_1\le \cdots \le \theta_k\le {\pi}/{2}$. These angles measure the distance between the two subspaces. A detailed definition is given in Section \ref{sec:wdk} below. Denote  
$$\sin\angle(U_k, \widetilde U_k) := \diag(\sin\theta_1, \cdots , \sin\theta_k).$$ Define $\sin\angle(V_k, \widetilde V_k)$ analogously. The classical perturbation bound, which concerns the variations in the eigenspaces for symmetric matrices $A$ and $A+E$, was initially investigated by Davis and Kahan \cite{DK}. Further generalizations to singular subspaces of rectangular matrices are encapsulated in Wedin's theorem (Eq. (3.11) from \cite{Wedin}).
\begin{theorem}[Wedin \cite{Wedin}] \label{thm:wedin}  Let $\widetilde{A} = A + E$ as in \eqref{def:tildeA}. If $\hat\delta_k:=\sigma_k- \widetilde\sigma_{k+1}>0$, then for any unitarily invariant norm $\vvvert \cdot \vvvert$,
\begin{equation}\label{eq:wedin}
\vvvert \sin\angle(U_k, \widetilde U_k) \vvvert  \le \frac{\max\{\vvvert P_{{U}_k^\perp} E P_{\widetilde{V}_k} \vvvert, \vvvert  P_{{V}_k^\perp} E^\T P_{ \widetilde{U}_k} \vvvert \}}{\hat\delta_k}.
\end{equation}
The same result also holds for $\vvvert \sin\angle(V_k, \widetilde V_k) \vvvert$.
\end{theorem}

In the context of a unitarily invariant norm $\vvvert \cdot \vvvert$, there exist several well-established methods to quantify the separation between $U_k$ and $\widetilde U_k$. These include using $$\vvvert \sin\angle(U_k, \widetilde U_k) \vvvert, \, \vvvert P_{U_k} - P_{\widetilde U_k}\vvvert \,\text{ and } \min_{O\in \mathbb{O}^{k\times k}} \vvvert U_kO- \widetilde U_k\vvvert.$$ In  the supplementary material \cite{WangSupp24}, we provide a detailed discussion about the equivalence or relationships among these various methods. 

The traditional bounds previously mentioned offer precise estimates, catering to worst-case scenarios. However, modern applications often operate under the premise that the data matrix $A$ satisfies specific structural assumptions. A typical case is when $A$ has a low rank $r$, where $r$ remains constant or experiences slow growth relative to $N$ and $n$. Moreover, the noise matrix $E$ is generally assumed to be random.

In this paper, we aim to develop a stochastic variant of Wedin's theorem under these additional assumptions that the signal matrix $A$ is low-rank and the noise $E$ consists of i.i.d.\ Gaussian entries. This work builds on the recent developments in \cite{OVW2, OVW22}, and offers several substantial improvements. We extend the classical Davis-Kahan-Wedin theorem for unitarily invariant norms, provide sharper bounds with improved dependence on the signal rank $r$, and relax several technical assumptions used in earlier analyses. A more detailed comparison with prior works appears in Section~\ref{sec:ourcontribution}.

There is currently a surging interest in $\ell_\infty$ analysis (also known as entrywise analysis) of eigenvectors and singular vectors. This dynamic research area focuses on deriving rigorous bounds, such as $\ell_\infty$ bounds \cite{FWZ17,ZB18, EBW18, CFMW19,AFWZ20,BV23} for eigenvectors or singular vectors, and $\ell_{2,\infty}$ bounds for eigenspaces or singular subspaces \cite{CTP19, Lei19, CLCPC21,AFW22, ALP22}, in relation to perturbed matrix models. These analyses have significant impact across statistics and machine learning applications.

Inspired by recent advancements, we have derived  precise $\ell_\infty$ bounds for the perturbed singular vectors and the $\ell_{2,\infty}$ bounds for the perturbed singular subspaces of $A+E$. We have also established results for linear and bilinear forms of the perturbed singular vectors and subspaces, and investigated the $\ell_{2,\infty}$ bounds on the perturbed singular vectors weighted by their singular values. These new results are presented in Section \ref{sec:entrywise}. Throughout, we assume that the noise $E$ has i.i.d.\  Gaussian entries; extensions to sub-Gaussian noise are discussed in Remark~\ref{rem:gaussian-var}. Section~\ref{sec:dis} summarizes our main contributions and discusses the optimality of our bounds, with a focus on rank dependence. A high-level overview of the proof strategy appears in Section~\ref{sec:overview}, following notation in Section~\ref{sec:LA}. Full proofs are deferred to the supplementary material \cite{WangSupp24}.

In Section \ref{sec:app}, we demonstrate the practical applications of our theoretical findings within two statistical models: the Gaussian mixture model and the submatrix localization problem. Our main goal is to use these results to examine how well spectral algorithms work and provide clear, straightforward proofs of their performance.

\medskip

\noindent\textbf{Organization:} 
Section \ref{sec:main} presents our new matrix perturbation results, with Section \ref{sec:wdk} extending Wedin's $\sin\Theta$ theorem to stochastic versions for arbitrary unitarily invariant norms. Section \ref{sec:entrywise} provides results on $\ell_{\infty}$ and $\ell_{2,\infty}$ norms of singular vectors and subspaces, while Section \ref{sec:literature} surveys related literature. Section \ref{sec:dis} discusses our main contributions and their optimality. Applications to Gaussian mixture model and submatrix localization are presented in Sections \ref{sec:gmm} and \ref{sec:submatrix}. Section \ref{sec:basic} introduces basic tools and proof strategies, with detailed proofs and extensive numerical simulations provided in the supplementary material \cite{WangSupp24}.

\medskip

\noindent\textbf{Notation:} For a vector $v=(v_1,\cdots,v_n)\in \mathbb{R}^n$, the following norms are frequently used: $\|v\| = \sqrt{\sum_{i=1}^n v_i^2}$ and $\|v\|_\infty = \max_{i} |v_i|$. Also, $\|v\|_0$ is the number of non-zero elements in $v$. For a real matrix $M$, $\|M\|$ denotes its operator norm, while $\|M\|_F$ represents its Frobenius norm. The term $\|M\|_{\max}$ refers to the largest absolute value among its entries, and $\|M\|_{2,\infty}$ indicates the maximum Euclidean norm of its rows. For a set $S$, let $\mathbf{1}_S$ be the indicator function of this set. For two functions $f(n),g(n)>0$, we use the asymptotic notations $f(n)\gg g(n)$ and $g(n)=o(f(n))$ if $f(n)/g(n)\to \infty$ as $n\to \infty$. The notation $f(n)=O(g(n))$ and $f(n)\lesssim g(n)$ are used when there exists some constant $C > 0$ such that $f(n) \leq Cg(n)$ for sufficiently large $n$. If $f(n)=O(g(n))$ and $g(n)=O(f(n))$, we denote $f(n) \asymp g(n)$. The set of $n\times n$ orthogonal matrices is denoted by $\mathbb O^{n\times n}$. We denote by $
A \oplus B := \begin{pmatrix}
A & 0 \\
0 & B
\end{pmatrix}
$ the direct sum (block diagonal concatenation) of two matrices \( A \) and \( B \).

\section{New results on the matrix perturbation bounds}\label{sec:main}

\subsection{Stochastic Wedin's $\sin\Theta$ theorem}\label{sec:wdk}
We first generalize the previous results in \cite{OVW2, OVW22} to an arbitrary unitarily invariant norm $\vvvert \cdot \vvvert$.  We start with the concept of principal angles. Let $H$ and $W$ are two subspaces each of dimension $k$. The principal angles $0\le \theta_1\le \cdots \le \theta_k\le {\pi}/{2}$  between $H$ and $W$ are defined recursively as follows:
\begin{equation*}
\cos (\theta_i) = \max_{h\in H, w\in W} h^T w =h_i^T w_i, \qquad \|h\|=\|w\|=1
\end{equation*}
subject to the constraint
$$ h_i^T h_l=0, \quad w_i^T w_l =0\quad \text{for } \,l=1,\ldots,i-1.$$
Denote  $\angle (H,W) := \diag(\theta_1, \cdots , \theta_k)$ and $\sin\angle (H,W)  := \diag(\sin\theta_1, \cdots , \sin\theta_k)$. 

For any $1\le k \le s \leq r$, denote
\[ U_{k,s} := \Span\{u_k,\ldots, u_s \}, \quad \widetilde U_{k,s} := \Span\{\widetilde u_k,\ldots, \widetilde u_s \}, \]
$P_{U_{k,s}}$ the orthogonal projection onto $U_{k,s}$, and analogously for $V_{k,s}$,  $\widetilde V_{k,s}$ and $P_{V_{k,s}}$. Denote $$D_{k,s}=\diag(\sigma_{k},\cdots, \sigma_s)$$ and analogously for $\widetilde D_{k,s}$. If $k=1$, we simply use $D_s, \widetilde D_s, U_s, \widetilde U_{s}, P_{U_{s}}$ and $V_s, \widetilde V_{s}, P_{V_{s}}$.

The spectral gap (or separation) $$\delta_k:=\sigma_k-\sigma_{k+1},$$ which refers to the difference between consecutive singular values of a matrix, will play a key role in the following results. 

\begin{theorem}[Unitarily invariant norms: simplified asymptotic version] \label{thm:subspace-simple}
Let $A$ and $E$ be $N \times n$ real matrices, where $A$ is deterministic with rank $r \geq 1$ and the entries of $E$ are i.i.d.  $\mathcal{N}(0,\tau^2)$ random variables. Let $\vvvert \cdot \vvvert$ be any normalized, unitarily invariant norm.  Consider $1\le k \le r$ such that $\delta_{k}\gtrsim  \tau r\sqrt{r+\log(N+n)}$. Denote $k_0=\min\{k, r-k\}$. Then with probability $1 - (N+n)^{-C}$ for some $C>0$,
\begin{align}\label{eq:subspacebd-simple}
\vvvert \sin \angle ( U_{k}, \widetilde{ U}_{k}) \vvvert \lesssim  \tau  \sqrt{k k_0} \frac{\sqrt{r+\log(N+n)}}{\delta_{k}} +  \frac{\vvvert P_{U^\perp} E P_{\widetilde V_{k}}\vvvert + \vvvert P_{V^\perp} E^\T P_{\widetilde U_{k}} \vvvert }{{\sigma}_{k}}.
\end{align}
Specifically, for the operator norm, we have with probability $1 - (N+n)^{-C}$,
\begin{align}\label{eq:subspace-op}
&\| \sin \angle ( U_{k}, \widetilde{ U}_{k}) \| \lesssim  \tau \sqrt{k}\frac{\sqrt{r+\log(N+n)}}{\delta_{k}} \mathbf{1}_{\{k\neq r \}} + \frac{\|E\|}{{\sigma}_{k}}
\end{align}  
The same conclusion also holds for $\sin \angle (V_{k}, \widetilde{V}_{k})$. 
\end{theorem}
This bound serves as a comprehensive generalization of the classical Wedin's bound in Theorem \ref{thm:wedin} when applied to the context of random noise. When $k=r$, the first term on the right-hand side of \eqref{eq:subspacebd-simple} vanishes, then \eqref{eq:subspacebd-simple} is essentially consistent with the Wedin's bound in Theorem \ref{thm:wedin}. When $k<r$, it is worth noting that  $P_{U_{k}^\perp} = P_{U_{k+1,r}} + P_{U^\perp}$ and $P_{V_{k}^\perp} = P_{V_{k+1,r}} + P_{V^\perp}$. Using Wedin's bound \eqref{eq:wedin}, one can deduce that
\begin{align}\label{eq:comparew}
&\vvvert \sin \angle ( U_{k}, \widetilde{ U}_{k}) \vvvert \nonumber\\
& \le \frac{\vvvert P_{U_{k+1,r}}E P_{\widetilde V_{k}}\vvvert + \vvvert P_{V_{k+1,r}}E^\T P_{\widetilde U_{k}}\vvvert}{\hat \delta_{k}}+\frac{\vvvert P_{U^\perp}E P_{\widetilde V_{k}}\vvvert+\vvvert P_{V^\perp}E^\T P_{\widetilde V_{k}}\vvvert}{\hat \delta_{k}}.
\end{align}
In the setting of a low-rank signal matrix $A$ and random noise $E$, our result \eqref{eq:subspacebd-simple} improves the second term on the right-hand side of \eqref{eq:comparew} by replacing the denominator $\hat \delta_{k} = \sigma_{k}- \widetilde \sigma_{k+1}$ with a usually much larger quantity $\sigma_{k}$. Additionally, we demonstrate that the first term on the right-hand side of \eqref{eq:comparew} is essentially $C(r)/\delta_{k}$, where $C(r)\lesssim r^{3/2}$.  This improvement is particularly important in statistical applications where $\delta_k \ll \sigma_k$, which frequently occurs in settings such as principal component analysis, matrix completion, and spiked covariance models. In these regimes, the Wedin's bounds become loose or even vacuous because they require the spectral gap to be comparable to the noise level. By contrast, our bounds remain effective by leveraging the signal strength $\sigma_k$, providing meaningful and tighter control even when $\delta_k \ll \sigma_k$.

In practice, computing the second term on the right-hand side of \eqref{eq:subspacebd-simple} precisely is challenging due to the dependence among $E, P_{\widetilde U_{k}}, P_{\widetilde V_{k}}$. Therefore, for practical applications, a simplified bound below offers convenience.
\begin{corollary}Under the assumptions of Theorem \ref{thm:subspace-simple}, the following holds with probability $1-(N+n)^{-C}$,
\begin{align}\label{eq:simplerbd}
\vvvert \sin \angle ( U_{k}, \widetilde{ U}_{k})\vvvert \lesssim  \tau \sqrt{k k_0} \frac{\sqrt{r+\log(N+n)}}{\delta_{k}} + k\frac{\|E\|}{\sigma_{k}}.
\end{align}
\end{corollary}

Theorem \ref{thm:subspace-simple} follows immediately from the next general and non-asymptotic result. For notation, let $\sigma_0:=\infty$, $\delta_0:=\infty$, and working with $\tau^{-1}(A+E)$, we assume $E$ has $\mathcal N(0,1)$ entries. We define $$\chi(b):= 1+ \frac{1}{4b(b-1)}  \text{ for } b\ge 2.$$
\begin{theorem}[Unitarily invariant norms: Gaussian noise]\label{thm:subspace}
Let $A$ and $E$ be $N \times n$ real matrices, where $A$ is deterministic and the entries of $E$ are i.i.d. standard Gaussian random variables. Let $\vvvert \cdot \vvvert$ be any normalized, unitarily invariant norm. Assume $A$ has rank $r \geq 1$. Let $K>0$ and $b\ge 2$. Denote $\eta:= \frac{11b^2}{(b-1)^2} \sqrt{2(\log 9) r+(K+7) \log (N+n)}$. Assume $(\sqrt N + \sqrt n)^2 \ge 32(K+7)\log(N+n) + 64(\log 9) r$. Consider $1\le r_0 \le r$ such that $ \sigma_{r_0} \ge 2b(\sqrt N+ \sqrt{n}) + 80b\eta r$ and $\delta_{r_0}\geq 75\chi(b)\eta r$. For any $1\le k \le s \leq r_0$, if  $\min\{\delta_{k-1}, \delta_s\}\geq 75\chi(b)\eta r$,  then
\begin{align}\label{eq:subspacebd}
\vvvert \sin \angle ( U_{k,s}, \widetilde{ U}_{k,s})\vvvert \le& 6\sqrt{2} \frac{(b+1)^2}{(b-1)^2}\sqrt{\min\{s-k+1,r-s+k-1\}} \frac{ \eta\sqrt{s-k+1}}{\min\{\delta_{k-1}, \delta_s\}} \nonumber\\
& + 2 \frac{\vvvert P_{U^\perp} E P_{\widetilde V_{k,s}} \oplus P_{V^\perp} E^\T P_{\widetilde U_{k,s}} \vvvert}{{\sigma}_{s}}
\end{align}
with probability at least $1 - 20(N+n)^{-K}$. 

Specifically, for the operator norm, we have with probability at least $1 - 20(N+n)^{-K}$
\begin{align*}
&\|\sin \angle ( U_{k,s}, \widetilde{ U}_{k,s})\| \le 3\sqrt{2} \frac{(b+1)^2}{(b-1)^2}\mathbf{1}_{\{s-k+1\neq r \}} \frac{ \eta\sqrt{s-k+1}}{\min\{\delta_{k-1}, \delta_s\}} + 2 \frac{\|E\|}{{\sigma}_{s}}.
\end{align*}
The same conclusion also holds for $\sin \angle ( V_{k,s}, \widetilde{ V}_{k,s})$.
\end{theorem}

\begin{remark}
The parameters $K$ and $b$ are user-specified positive constants. The parameter $K$ controls the probability level: larger $K$ yields a smaller failure probability $(N+n)^{-K}$ at the cost of slightly larger constants. The parameter $b$ governs the trade-off between the required singular value separation (through $\chi(b)$) and the constants in the bound. A larger $b$ corresponds to a stronger signal regime: it reduces $\chi(b)$, thereby relaxing the spectral gap assumption, and can also lead to sharper constants in the final estimates. 
\end{remark} 

\begin{remark}Throughout the proofs, we work on the event that $\|E\| \le 2(\sqrt N+ \sqrt{n}).$ Lemma \ref{lemma:norm} below guarantees this event holds with very high probability.   In Theorem \ref{thm:subspace}, the parameter $b\ge 2$ represents the signal-to-noise ratio, and in the proof, we ensure that $\sigma_{r_0}/\|E\| \ge b$. The parameter $b$, which could depend on $N$ and $n$, could account for a particularly strong signal. We have selected certain constants and expressions such as 80, $75\chi(b)$, and $ \frac{(b+1)^2}{(b-1)^2}$ for the sake of convenience in our computations while our primary objective was not to optimize these constants within the proof. It is also feasible to conduct work on the event that $\|E\| \le (1+\epsilon_1)(\sqrt N+ \sqrt{n})$, and assume $b \ge 1+\epsilon_2$ for $\epsilon_1, \epsilon_2>0$. By following the same proof, one can arrive at refined constants and bounds.
\end{remark}

To go beyond the i.i.d. Gaussian noise matrix, we record the following results on the perturbation of singular values and singular subspaces that is obtained using a similar approach as in the previous work by O'Rourke, Vu and the author \cite{OVW2}. In particular, these results remain valid for random noise of any specific structure, as long as the noise has a negligible effect on the singular subspaces of matrix $A$.
\begin{theorem}[Singular value bounds: general noise]\label{thm:general-sv}
Assume $A$ has rank $r$ and $E$ is random. Let $1\le k \le r$. Consider any $\varepsilon\in (0,1)$.
\begin{itemize}
\item If there exists $t>0$ such that  $\|U_k^T E V_k\| \le t$ with probability at least $1-\varepsilon$, then we have, with probability at least $1-\varepsilon$,
\begin{align}\label{eq:lowerbd}
\widetilde \sigma_k \ge \sigma_k - t.
\end{align}
\item If there exist $L, B>0$ such that $\|U^T E V\| \le L$ and $\|E\| \le B$ with probability at least $1-\varepsilon$, then we have, with probability at least $1-\varepsilon$,
\begin{align}\label{eq:upperbd}
\widetilde \sigma_k \le \sigma_k + 2\sqrt{k} \frac{B^2}{\widetilde \sigma_k} + k \frac{B^3}{\widetilde \sigma_k^2} + L.
\end{align}
\end{itemize}
\end{theorem}
\begin{theorem}[Singular subspace bounds: general noise]\label{thm:generalnoise} Assume $A$ has rank $r$ and $E$ is random. Let $1\le k \le r$. For  $\varepsilon>0$, assume there exist $L,B>0$ such that $\|U^\T E V\| \le L$ and $\|E\| \le B$ with probability at least $1-\varepsilon$. Furthermore, assume $\delta_k=\sigma_k-\sigma_{k+1}\ge 2L$. Then for any normalized, unitarily invariant norm $\vvvert \cdot \vvvert$, the following holds with probability at least $1-\varepsilon$,
\begin{align*}
\vvvert \sin\angle (U_k, \widetilde U_k) \vvvert \le 2 \sqrt{k \min\{k,r-k \}} \left(\frac{L }{\delta_k} + 2\frac{B^2}{\delta_k \sigma_k} \right) +2 k\frac{ B }{ \sigma_k}.
\end{align*}
More specifically, for the operator norm, 
\begin{align*}
\| \sin\angle (U_k, \widetilde U_k) \| \le 2 \sqrt{k} \left(\frac{L }{\delta_k} + 2\frac{B^2}{\delta_k \sigma_k} \right)\mathbf{1}_{\{ k< r\}} + 2 \frac{B}{\sigma_k}.
\end{align*}
The same result also holds for $\sin\angle (V_k, \widetilde V_k)$.
\end{theorem}

To apply Theorems \ref{thm:general-sv} and \ref{thm:generalnoise}, it is necessary to obtain effective bounds on $\|E\|$ and $\|U^\T E V \|$. In general, matrix concentration inequalities (refer to \cite{Tropp} for example) can provide good upper bounds on $\|E\|$ for random noise $E$ with heteroskedastic entries and even complex correlations among the entries. On the other hand, $U^\T E V$ is an $r\times r$ matrix, and $\|U^\T E V\|$ typically depends on $r$. Bounds on $\|U^\T E V\|$ can be obtained by applying concentration inequalities. We note that in the symmetric setting, Eldridge et al.~\cite[Theorem 6]{EBW18} established related eigenvalue upper bounds. Both approaches begin with the min-max characterization, but differ in methodology: we use a direct bilinear form argument, while they leverage signal-aligned subspaces. For random noise matrices, our bounds can provide better control when $k$ is not too large, while their framework offers additional precision through explicit exploitation of spectral gaps when present.

\subsection{$\ell_\infty$ and $\ell_{2, \infty}$ analysis}\label{sec:entrywise}
Next, we present a result regarding the estimation of the singular vectors on an entrywise basis. In this context, a parameter known as the incoherence parameter of the singular vector matrices $U$ and $V$, denoted as $\| U\|_{2,\infty}$ and $\| V\|_{2,\infty}$, is of central importance. Smaller values of  $\| U \|_{2,\infty}$ and $\| V \|_{2,\infty}$ suggest that the information contained in the signal matrix $A$ is less concentrated in just a few rows or columns. 

In this section, we use $U_{k,s}=(u_k,\cdots,u_s)$ to denote the singular vector matrix for $1\le k \le s \le r$. We abbreviate $U_{k,s}$ to $U_s$ when $k=1$. Note that $P_{U_{k,s}} = U_{k,s} U_{k,s}^\T$. These notations also apply to $\widetilde U_{k,s}$ and $\widetilde U_{s}$. For simplicity, we only state the results for the left singular vectors $U_{k,s}$. The corresponding results for the right singular vectors can be derived by applying these results to the transposes of matrices $A^\T$ and $A^\T+E^\T$.

Let us make a temporary assumption that $\sigma_1 \le n^2$. This assumption is reasonable because if $\sigma_k > n^2$, it indicates a highly significant signal, and the impact of noise becomes negligible in such cases.  Theorem \ref{thm:infinitybd} later will provide a precise treatment of these strong singular values (i.e., without the assumption $\sigma_1\le n^2$). Denote $\sigma_0:=\infty$ and $\delta_0:=\infty$. 
\begin{theorem}[$\ell_\infty$ and $\ell_{2, \infty}$ bounds: simplified asymptotic version]\label{thm:entrywise} Let $A$ and $E$ be $N \times n$ real matrices, where $A$ is deterministic with rank $r \geq 1$ and the entries of $E$ are i.i.d. $\mathcal{N}(0,\tau^2)$ random variables. Let $1\le k \le r$ and $ \sigma_{k} \ge 2\|E\|$.  Assume that $\sigma_1 \le n^2$.
\begin{itemize}
\item If $\min\{\delta_{k-1},\delta_{k}\} \gtrsim  \tau r\sqrt{r+\log(N+n)}$, then with probability $1-(N+n)^{-C}$,
\begin{align*}
\| \widetilde u_{k} - (\widetilde u_{k}^\T u_{k}) u_{k}\|_\infty \lesssim   \tau \frac{\sqrt{r+\log(N+n)}}{\min\{\delta_{k-1},\delta_{k}\}}  \|U\|_{2,\infty} + \tau \frac{\sqrt{r\log(N+n)}}{\sigma_k} .
\end{align*}
\item If $\delta_{k} \gtrsim  \tau r\sqrt{r+\log(N+n)}$, then with probability $1-(N+n)^{-C}$,
\begin{align*}
\| \widetilde{U}_{k} - P_{U_{k}} \widetilde{U}_{k} \|_{2,\infty} \lesssim    \tau \sqrt{k}\frac{\sqrt{r+\log(N+n)}}{\delta_{k}}  \|U\|_{2,\infty} +  \tau \sqrt{k}\frac{\sqrt{r\log(N+n)}}{\sigma_k}.
\end{align*}
\end{itemize}
\end{theorem}

In many applications, the primary interest lies in comparing  $\widetilde U_k$ with $U_k O$, accounting for the non-uniqueness of singular vectors via an orthogonal matrix $O$. A suitable choice of $O$ that aligns $U_k$ with $\widetilde U_k$ effectively can be determined by examining the SVD of $U_k^\T \widetilde U_k$. By Proposition A.5 in \cite{WangSupp24}, the SVD of $U_k^\T \widetilde U_k$ is $U_k^\T \widetilde U_k =O_1 \cos\angle(U_k, \widetilde U_k) O_2^\T$ and we choose $O=O_1 O_2^\T$. Note that the discrepancy between $U_k^\T \widetilde U_k$ and $O$ can be measured by the principal angles between the subspaces $U_k$ and $\widetilde U_k$. In Proposition A.6 in \cite{WangSupp24}, we establish
\begin{align}\label{eq:propbd}
 \| \widetilde{U}_{k} - U_k O \|_{2,\infty}\le \| \widetilde{U}_{k} - P_{U_{k}} \widetilde{U}_{k} \|_{2,\infty} + \left\|U_k \right\|_{2,\infty} \|\sin\angle(U_k, \widetilde{U}_{k}) \|^2.
\end{align}
Therefore, by combining Theorem \ref{thm:subspace} with Theorem \ref{thm:entrywise}, we obtain the next result.
\begin{corollary}\label{cor:matrix}Under the same assumption as Theorem \ref{thm:entrywise}, the following holds:
\begin{itemize}
\item If $\min\{\delta_{k-1},\delta_{k}\} \gtrsim  \tau r\sqrt{r+\log(N+n)}$, then with probability $1-(N+n)^{-C}$,
\begin{align*}
\min_{\mathsf s\in\{\pm1\}}\|u_k  - \mathsf s \widetilde u_k\|_{\infty} \lesssim&   \tau \frac{\sqrt{r+\log(N+n)}}{\min\{\delta_{k-1},\delta_{k}\}}  \|U\|_{2,\infty} +  \tau \frac{\sqrt{r\log(N+n)}}{\sigma_k} + \frac{\|E\|^2}{\sigma_k^2} \|u_k\|_\infty . 
\end{align*}
\item If $\delta_{k} \gtrsim  \tau r\sqrt{r+\log(N+n)}$, then with probability $1-(N+n)^{-C}$,
\begin{align*}
\min_{O\in \mathbb O^{k\times k}} \| \widetilde{U}_{k} - U_k O \|_{2,\infty} \lesssim &  \tau \sqrt{k}\frac{\sqrt{r+\log(N+n)}}{\delta_{k}}  \|U\|_{2,\infty}\\
& +  \tau \sqrt{k}\frac{\sqrt{r\log(N+n)}}{\sigma_k}  +  \frac{\|E\|^2}{\sigma_k^2} \|U_k\|_{2,\infty}.
\end{align*}
\end{itemize}
\end{corollary}

Theorems \ref{thm:entrywise} follows as a direct consequence of the next general and non-asymptotic result, which we will prove in \cite{WangSupp24}.  By a scaling, it suffices to assume $E$ has $\mathcal N(0,1)$ entries.

\begin{theorem}\label{thm:infinitybd} Let $A$ and $E$ be $N \times n$ real matrices, where $A$ is deterministic and the entries of $E$ are i.i.d. standard Gaussian random variables. Assume $A$ has rank $r \geq 1$. Let $K>0$ and $b\ge 2$. Denote $\eta:= \frac{11b^2}{(b-1)^2} \sqrt{2(\log 9) r+(K+7) \log (N+n)}$. Assume $(\sqrt N + \sqrt n)^2 \ge 32(K+7)\log(N+n) + 64(\log 9) r$. Consider $1\le r_0 \le r$ such that $ \sigma_{r_0} \ge 2b(\sqrt N+ \sqrt{n}) + 80 b\eta r$ and $\delta_{r_0}\geq 75\chi(b)\eta r$. For any $1\le k \le s \leq r_0$, if  $\min\{\delta_{k-1}, \delta_s\}\geq 75\chi(b)\eta r$,  then with probability at least $1 - 40(N+n)^{-K}$,
\begin{align}\label{infnorm}
\| \widetilde{U}_{k,s} - P_{U_{k,s}} \widetilde{U}_{k,s} \|_{2,\infty} &\le 3\sqrt{2}\frac{(b+1)^2}{(b-1)^2} \| U \|_{2,\infty}  \frac{\eta\sqrt{s-k+1}}{\min\{\delta_{k-1}, \delta_s\}} \mathbf{1}_{\{s-k+1\neq r \}} \nonumber\\
&  + \frac{4\sqrt{2} b^2}{(b-1)^2}\sqrt{\sum_{i\in \llbracket k, s\rrbracket, \sigma_i\le n^2} \frac{\gamma^2}{\sigma_i^2} + \sum_{i\in \llbracket k, s\rrbracket, \sigma_i> n^2} \frac{16 n}{\sigma_i^2}},
\end{align}
where $\gamma:= \frac{9 b^2}{(b-1)^2}\sqrt{r(K+7)\log(N+n)} $. 
\end{theorem}
It should be noted that, as per the aforementioned result, the term $\sum_{k\le i\le s, \sigma_i> n^2} \frac{16 n}{\sigma_i^2} < \frac{16}{n^2}$ can always be considered negligible in comparison to the other terms.  Indeed, when the signal is extremely strong, i.e., $\sigma_i>n^2 \gg \|E\|=\Theta(\sqrt N +\sqrt n)$, the impact of noise becomes minimal. 

More generally, we can establish the following result, which provides bounds for the singular subspaces in any arbitrary direction. The complete result follows a format similar to Theorem \ref{thm:infinitybd}. The proofs are provided in \cite{WangSupp24}.
\begin{theorem}[Bounds on linear and bilinear forms] \label{thm:linearform}Under the assumptions of Theorem \ref{thm:infinitybd}, for any unit vectors $x \in \mathbb R^{N}$ and $y=(y_k,\cdots,y_s)^\T \in \mathbb R^{s-k+1}$, the following holds with probability at least $1 - 40(N+n)^{-K}$:
\begin{align*}
\|x^\T (\widetilde U_{k,s}-P_{U_{k,s}}\widetilde U_{k,s}) \|  \le & 3\sqrt{2}\frac{(b+1)^2}{(b-1)^2}  \|x^\T U\| \frac{\eta\sqrt{s-k+1}}{\min\{\delta_{k-1}, \delta_s\}}\mathbf{1}_{\{s-k+1\neq r \}} \\
& + \frac{4\sqrt{2} b^2}{(b-1)^2}   \sqrt{\sum_{i\in \llbracket k, s\rrbracket, \sigma_i\le n^2} \frac{\gamma^2}{\sigma_i^2} +  \sum_{i\in \llbracket k, s\rrbracket, \sigma_i> n^2} \frac{16 n}{\sigma_i^2} } 
\end{align*}
and
\begin{align}\label{eq:bilinearbd}
\left| x^\T (\widetilde U_{k,s}-P_{U_{k,s}}\widetilde U_{k,s}) y \right|\le & 3\sqrt{2} \frac{(b+1)^2}{(b-1)^2}  \| x^\T U\| \frac{\eta \sqrt{\|y\|_0}}{\min\{\delta_{k-1}, \delta_s\}}\mathbf{1}_{\{s-k+1\neq r \}}\nonumber\\
& + \frac{4\sqrt{2} b^2}{(b-1)^2} \left[\sum_{i\in \llbracket k, s\rrbracket, \sigma_i\le n^2}  {\gamma } \frac{|y_i|}{\sigma_i} +  \sum_{i\in \llbracket k, s\rrbracket, \sigma_i> n^2} \frac{4\sqrt{n}}{\sigma_i} \right],
\end{align}
where $\gamma= \frac{9 b^2}{(b-1)^2}\sqrt{r(K+7)\log(N+n)}$. 
\end{theorem}
\begin{remark}When the focus is on comparing the linear (or bilinear) forms of $U_{k,s}$ and $\widetilde U_{k,s}$, in a manner analogous to Corollary \ref{cor:matrix}, one can leverage the fact provided in Proposition A.6 of \cite{WangSupp24}:
\begin{align*}
&\|x^\T (\widetilde U_{k,s}-U_{k,s}O ) \| \le \|x^\T (\widetilde U_{k,s}-P_{U_{k,s}}\widetilde U_{k,s}) \| + \|x^\T U_{k,s}\| \|\sin\angle(U_{k,s}, \widetilde U_{k,s})\|^2
\end{align*}
and combine Theorems  \ref{thm:subspace} and \ref{thm:linearform}. 

A natural extension is to consider entrywise control of the perturbation. By applying \eqref{eq:bilinearbd} with the canonical vectors and using Proposition A.6 of \cite{WangSupp24}, we obtain the following bound under the matrix max-norm (which measures the largest absolute entry):  with probability $1-(N+n)^{-C}$,
\begin{align}\label{eq:maxbd}
&\|\widetilde U_{k,s} - U_{k,s} O\|_{\max} \lesssim \frac{\sqrt{r+\log(N+n)}}{\min\{\delta_{k-1}, \delta_s\}}\|U\|_{2,\infty} + \frac{\sqrt{r\log(N+n)}}{\sigma_s}+ \frac{\|E\|^2}{\sigma_s^2}\|U_{k,s}\|_{2,\infty}
\end{align}
for some orthogonal matrix $O$. 
\end{remark}

Building upon the proof of Theorem 12 and incorporating minor modifications, we obtain the subsequent bounds. These describe the extent to which the dominant singular vectors of the perturbed matrix, when weighted by their singular values, deviate from the original subspace. The proof can be found in  the supplementary material \cite{WangSupp24}.
\begin{theorem}[Bounds on singular value-adjusted projection perturbation]\label{thm:weighted}Under the assumptions of Theorem \ref{thm:infinitybd}, the following holds with probability at least $1 - 40(N+n)^{-K}$:
\begin{align*}
\| \widetilde{U}_{k,s}\widetilde{D}_{k,s} - P_{U_{k,s}} \widetilde{U}_{k,s} \widetilde{D}_{k,s}\|_{2,\infty} &\le 3\sqrt{2}\frac{(b+1)^2}{(b-1)^2} \| U \|_{2,\infty}  \frac{\eta\sigma_k\sqrt{s-k+1}}{\min\{\delta_{k-1}, \delta_s\}} \mathbf{1}_{\{s-k+1\neq r \}} \nonumber\\
&  + \frac{4\sqrt{2} b^2}{(b-1)^2}  \sqrt{\gamma^2(s-k+1) +16},
\end{align*}
where $\gamma= \frac{9 b^2}{(b-1)^2}\sqrt{r(K+7)\log(N+n)} $. 
\end{theorem}

Comparing  $\widetilde{U}_{k,s}\widetilde{D}_{k,s}$ with ${U}_{k,s}O\widetilde{D}_{k,s}$, with respect to the choice of an orthogonal matrix $O$, requires more analysis than the unweighted case in Corollary \ref{cor:matrix}. Consider the top $k$-singular subspaces for any $1\le k \le r$. We establish that for some $O\in\mathbb{O}^{k\times k}$:
 \begin{align}\label{eq:weightedbdnew}
 \| \widetilde{U}_{k}\widetilde{D}_{k} - {U}_kO\widetilde{D}_{k}\|_{2,\infty}\le &\| \widetilde{U}_{k}\widetilde{D}_{k} - P_{U_k} \widetilde{U}_{k} \widetilde{D}_{k}\|_{2,\infty}\nonumber\\
 &+ \|U_k\|_{2,\infty} \|\sin\angle(U_k, \widetilde U_{k})\| \left(\|E\| +\sigma_{k+1}\|\sin\angle(V_k, \widetilde V_{k})\|\right) .
 \end{align}
The proof of \eqref{eq:weightedbdnew} is provided in \cite{WangSupp24}. Consequently, the weighted case bounds follow from combining Theorem \ref{thm:weighted} with Theorem \ref{thm:subspace}. Of particular interest is the special case $k=r$, which is crucial for our statistical applications. Here, \eqref{eq:weightedbdnew} simplifies to:
\begin{align}\label{eq:weightedwhole}
\| \widetilde{U}_{r}\widetilde{D}_{r} - {U}O\widetilde{D}_{r}\|_{2,\infty}\le &\| \widetilde{U}_{r}\widetilde{D}_{r} - P_{U} \widetilde{U}_{r} \widetilde{D}_{r}\|_{2,\infty}+ \|U\|_{2,\infty} \|\sin\angle(U, \widetilde U_{r})\| \|E\|.
\end{align}
Prior works (e.g., Proposition 3 in \cite{YW24}) have used an alternative formulation $\|\widetilde{U}_r \widetilde{D}_r O^\T- U D \|_{2,\infty}$. In \cite{WangSupp24} (after the proof of \eqref{eq:weightedbdnew}), we analyze their connections. 

Combining \eqref{eq:weightedwhole} with Theorems \ref{thm:weighted} and \ref{thm:subspace} yields the following key result for our applications:
\begin{corollary}\label{cor:weighted}
Assume the setting of Theorem \ref{thm:infinitybd}, where we set $r_0=s=r$ and $k=1$, the following holds with probability at least $1 - 40(N+n)^{-K}$:
\begin{align*}
\min_{O\in \mathbb{O}^{r\times r}}\| \widetilde{U}_{r}\widetilde{D}_{r} - U O \widetilde{D}_{r}\|_{2,\infty}\le &\frac{72 b^4}{(b-1)^4} r\sqrt{(K+7)\log(N+n)} + 2\|U\|_{2,\infty}  \frac{ \|E\|^2}{\sigma_r}.
\end{align*}
\end{corollary}
The proof technique of Theorem \ref{thm:weighted} could be extended to bound the bilinear form $x^\T(\widetilde{U}_{k,s}\widetilde{D}_{k,s} - P_{U_{k,s}} \widetilde{U}_{k,s} \widetilde{D}_{k,s})y$ for arbitrary unit vectors $x$ and $y$, similar to the analysis in Theorem \ref{thm:linearform}. However, we do not pursue this generalization in the present work.

\begin{remark}\label{rem:gaussian-var}
Our theoretical framework, while presented for noise matrices $E$ with i.i.d. centered Gaussian entries, extends naturally to more general settings. The underlying methodology accommodates sub-Gaussian noise matrices through established random matrix theory techniques, particularly isotropic local laws and concentration inequalities. The detailed technical treatment of these extensions, including sub-Gaussian cases and more general noise assumptions, will be addressed in forthcoming work.
\end{remark}

\section{Related Work}\label{sec:literature}

This section reviews recent developments in singular vector perturbation theory and their connections to high-dimensional statistics and random matrix theory. 

\paragraph*{Classical and \texorpdfstring{$\ell_2$}{l2}-Type Subspace Perturbation Results}
Foundational results such as the Davis--Kahan and Wedin $\sin\Theta$ theorems \cite{DK,Wedin} provide classical $\ell_2$-type perturbation bounds for eigenvectors and singular subspaces. Several deterministic extensions refine or generalize these results. Yu, Wang, and Samworth \cite{YWS15} replace empirical gaps with population gaps in the Frobenius norm setting. Vu and Lei \cite{VL13} offer a variational form applicable to unaligned subspaces. Cai and Zhang \cite{CZ18} and Luo, Han, and Zhang \cite{LHZ21} further generalize to unbalanced dimensions and Schatten-$q$ norm bounds, respectively. Zhang and Zhou \cite{ZZ22} give Frobenius norm bounds for closely related matrix pairs, useful for spectral clustering.

In the stochastic regime, Wang \cite{MR3310977} analyzes the non-asymptotic distribution of singular vectors under Gaussian noise. Allez and Bouchaud \cite{MR3256861} study eigenvector dynamics in symmetric matrices with Brownian motion noise. Benaych-Georges, Enriquez, and  Micha\"il  \cite{MR4260218} analyze a perturbative expansion for eigenvector coordinates in symmetric matrices with stochastic noise. Zhong \cite{Zhong17} develops a Rayleigh--Schr\"odinger-type expansion to study angular perturbations of leading eigenvectors in symmetric low-rank plus noise models.

\paragraph*{\texorpdfstring{$\ell_{\infty}$}{linf} and \texorpdfstring{$\ell_{2,\infty}$}{l2inf}-Type Subspace Perturbation Results}
Recent works have increasingly focused on more refined, entrywise and row-wise perturbation bounds, which offer localized guarantees particularly useful in applications such as submatrix localization and community detection. A foundational contribution in this direction is due to Fan, Wang, and Zhong \cite{FWZ17}, who derive $\ell_\infty$ perturbation bounds under incoherence conditions on the low-rank signal matrix. Building on this idea, a series of works \cite{AFWZ20, AFW22, CLCPC21, CFMW19, BV23, YCF21, YW24} develop increasingly sharp $\ell_{\infty}$ and $\ell_{2,\infty}$ bounds using leave-one-out techniques tailored to various noise models and matrix structures.

Abbe, Fan, Wang, and Zhong \cite{AFWZ20} establish $\ell_{2,\infty}$ perturbation bounds for symmetric low-rank plus noise matrices with sub-Gaussian noise. This is further extended by Abbe, Fan, and Wang \cite{AFW22}, who provide a comprehensive analysis of $\ell_{2,p}$ norms in the context of hollowed PCA. Chen, Fan, Ma, and Wang \cite{CFMW19} generalize the analysis to asymmetric transition matrices, while Lei \cite{Lei19} considers more intricate dependence structures in the noise. Zhong and Boumal \cite{ZB18} obtain $\ell_{\infty}$ bounds for the leading eigenvector in the phase synchronization model, using tools from semidefinite programming. Bhardwaj and Vu \cite{BV23} propose a stochastic analogue of the Davis--Kahan theorem that yields $\ell_\infty$ bounds for both eigenvectors and singular vectors under general noise. More recently, Yan and Wainwright \cite{YW24} introduce a novel expansion technique that yields sharp $\ell_{2,\infty}$ bounds and distributional characterizations of the estimation error, refining earlier results by Yan, Chen, and Fan \cite{YCF21}. Agterberg, Lubberts, and Priebe \cite{ALP22} also contribute by analyzing heteroskedastic noise and deriving Berry--Esseen-type results for entrywise estimation of singular vectors.

Several works investigate perturbation of linear and bilinear forms involving singular vectors. Koltchinskii and Xia \cite{MR3565274} derive concentration bounds for such forms, later extended to tensor settings by Xia and Zhou \cite{MR3960915}. Cape, Tang, and Priebe \cite{CTP19} provide $\ell_{2,\infty}$ bounds for structured matrices, and Eldridge, Belkin, and Wang \cite{EBW18} use a Neumann series expansion to control entrywise deviations. Li, Cai, Poor, and Chen \cite{LCPC21} propose de-biased estimators for projected eigenvectors under mild eigengap assumptions. Recent works \cite{CWC21, Ag23} explore this regime further, especially when eigen-gaps are small or scale with noise.

Comprehensive insights into the $\ell_2$ and $\ell_{\infty}$ analyses of current perturbation results and their practical implications are available in the survey \cite{CCFM21}.

\paragraph*{Connections to Random Matrix Theory}
Random matrix theory offers asymptotic insights into the spectral behavior of low-rank plus noise models. The BBP phase transition \cite{BBP05} describes a critical threshold for the emergence of outlier eigenvalues and eigenvector localization. Subsequent studies \cite{DP07, BGN, BGN1, BS06, El07, BY08, BGM11, BEKYY2016, CM2017, CD18, BDWW20, BDW21, FFHL19, BW21, Ben20} explore the fine behavior of extreme eigenvalues and eigenvectors across different regimes and ensembles. While these results are typically asymptotic and ensemble-specific, they inform the design and understanding of non-asymptotic perturbation techniques used in this work.

The selection of references cited herein represents a snapshot of a rapidly advancing field and is not intended to be exhaustive.

\section{Discussions}\label{sec:dis}
This section summarizes the key contributions of this work, and examines the lower bounds and the role of rank $r$ in our perturbation bounds.
\subsection{Optimality of our results}\label{sec:optimality}
Our operator norm bound in Theorem~\ref{thm:subspace-simple} improves upon prior results, notably \cite{OVW2}, in its dependence on the signal rank $r$. In particular, when $k=1$, we show that with probability $1-(N+n)^{-C}$,
\begin{equation}\label{eq:better-r}
\sin \angle ( u_1, \widetilde{u}_1) \lesssim \frac{\sqrt{r +\log(N+n)}}{\delta_1} + \frac{\|E\|}{\sigma_1}.
\end{equation}
This bound is conjectured to be near-optimal, up to constant factors.  First, the necessity of the second term in \eqref{eq:better-r} is supported by a high-probability lower bound established in \cite[Theorem~3]{OVW22}, which shows that $\max\{\sin \angle ( u_1, \widetilde{u}_1), \sin \angle ( v_1, \widetilde{v}_1)\} \gtrsim \frac{\|E\|}{\sigma_1}$. Second, the first term of order ${\sqrt{r +\log(N+n)}}/{\delta_1}$ merits further attention. A line of work including \cite{CZ18, CFMW19, CLCPC21, LCPC21, CCFM21, CWC21} establishes minimax lower bounds for singular subspace estimation.  In particular, \cite[Theorem~3]{CWC21} shows that, under Gaussian noise, any estimator must incur error at least of order $1/\delta_1$. This confirms the necessity of the $1/\delta_1$ scaling in our bound.

To better understand the role of $r$, we conduct simulations varying $r$ and plot the empirical distribution of $\sin \angle ( u_1, \widetilde{u}_1).$ The left panel of Figure~\ref{fig:varynr} shows that the distribution depends on the rank $r$, while the right panel demonstrates its sensitivity to the matrix dimensions, which controls the noise level $\|E\|\asymp \sqrt{n}$. Although a matching lower bound capturing the full ${\sqrt{r}}$ dependence is not currently known, we denote the unknown scaling factor as $f(r)$ and empirically investigate its behavior in the supplementary material. These simulations suggest a sublinear but non-negligible scaling with $r$, with $\sqrt{r}$ emerging as a plausible candidate.

For the $\ell_{\infty}$ bound, we observe that 
\[ \min_{\mathsf{s}\in\{\pm 1\}}\| u_1- \mathsf{s}\widetilde{u}_1\|_{\infty}\ge \min_{\mathsf{s}\in\{\pm 1\}}\frac{1}{\sqrt{N}} \| u_1- \mathsf{s}\widetilde{u}_1\| \ge \frac{1}{\sqrt{N}} \sin \angle ( u_1, \widetilde{u}_1).\]
Combining the above discussion, we expect the following lower bound:
\[ \min_{\mathsf{s}\in\{\pm 1\}}\| u_1- \mathsf{s}\widetilde{u}_1\|_{\infty}\gtrsim \frac{1}{\sqrt N} \frac{f(r)}{\delta_1} + \frac{1}{\sigma_1}.\]
On the other hand, under the incoherence assumption $\|U\|_{\max} \lesssim \frac{1}{\sqrt N}$, our bound in Corollary~\ref{cor:matrix} (with $\tau=1$) gives with high probability
\[  \min_{\mathsf{s}\in\{\pm 1\}}\| u_1- \mathsf{s}\widetilde{u}_1\|_{\infty} \le  \frac{1}{\sqrt N}\frac{\sqrt{r}\sqrt{r +\log(N+n)}}{\delta_1} + \frac{\sqrt{r\log(N+n)}}{\sigma_1}.\]
The gap in $r$-dependence between the upper and lower bounds is at most a factor of $r$ (up to logs), suggesting near-optimal rank scaling in our entrywise bound.

\begin{figure}[!ht]
\centering
   \includegraphics[width=11cm,height=5.5cm]{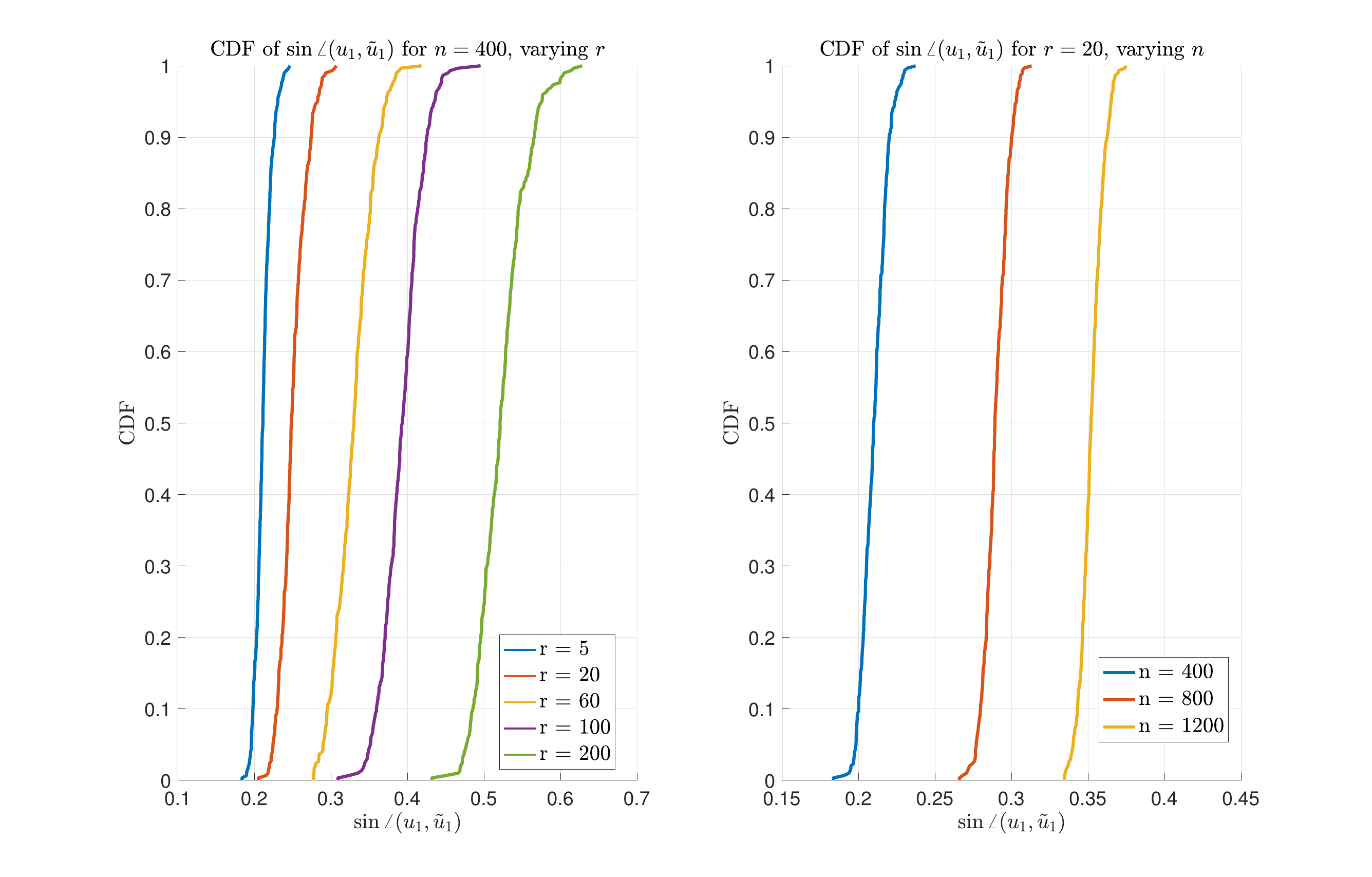}
   \caption{CDF plots of $\sin \angle ( u_1, \widetilde{u}_1)$ across 300 trials. The signal matrix $A\in \mathbb{R}^{n\times n}$ has rank $r$. The noise matrix $E$ has i.i.d. standard Gaussian entries. We set the largest singular value of $A$ as $\sigma_1=100$. \textbf{Left:} We set $n=400$, $\delta_1=20$ and $\sigma_2=\ldots=\sigma_r=80$, varying $r=5, 20, 60, 100, 200$. \textbf{Right:} We fix $r=20$, $\delta_1=40$, and $\sigma_2=\ldots=\sigma_r=60$, and vary the dimension $n=400, 800, 1200.$}
\label{fig:varynr}
\end{figure}

The supplementary material \cite{WangSupp24} provides extensive numerical validation of our bounds. While these experiments confirm the tightness of rank dependence in angular and Frobenius norm bounds, they also reveal opportunities for improvement. Specifically, the $\ell_{\infty}$ and $\ell_{2,\infty}$ bounds might be refined under incoherence assumptions, and the operator norm bounds could be slightly improved through refined analysis.

\subsection{Contribution of This Work}\label{sec:ourcontribution}

This paper establishes near-optimal, non-asymptotic perturbation bounds for singular vectors in the low-rank matrix denoising model, under the assumption that the noise matrix has i.i.d.\ Gaussian entries. Building on the framework of \cite{OVW22}, which introduced the isotropic local law into singular vector analysis, we extend and refine the results in several important directions.

First, we obtain tighter perturbation bounds under milder assumptions by sharpening the decomposition of error terms and enhancing the control of resolvent expansions. 
Most notably, we eliminate a restrictive condition in \cite{OVW22}, which required distinct singular values among $\sigma_k,\cdots,\sigma_s$ to be separated by a distance of order $r^2\sqrt{\log(N+n)}$. This restrictive spectral gap condition, often challenging to verify in practice, is no longer necessary in our theorem through our refined analysis of perturbed singular value locations.

Second, we provide a more detailed investigation of the role of signal rank $r$, establishing bounds with improved $r$-dependence. In several key regimes, this dependence is shown to be near-optimal, with our numerical experiments supporting its necessity. Our analysis addresses key limitations of previous works: while \cite{OVW2} provided bounds for general noise without utilizing the isotropic local law, it included an extra additive term that we eliminate here under Gaussian noise. Similarly, \cite{OVW22}, though introducing the isotropic local law for Gaussian noise, focused primarily on operator norm bounds without optimizing rank scaling. The current framework thus achieves sharper results by employing a more delicate decomposition of the perturbation terms and sharper control over resolvent expansions. 

Our bounds are formulated under a wide class of norms, including unitarily invariant norms, the $\ell_{2,\infty}$ norm, and a weighted $\ell_{2,\infty}$ norm that captures heterogeneity in the signal structure. This weighted bound is especially powerful in statistical applications, where standard norm-based bounds often fail to achieve optimal practical performance. We illustrate its utility through two well-studied problems--Gaussian mixture model and submatrix localization--in Section~\ref{sec:app}. These examples demonstrate how our general perturbation bounds yield concise and effective guarantees for important statistical tasks.

Finally, the isotropic local law framework we adopt proves to be both flexible and robust. It accommodates various matrix norms and opens the door to further extensions beyond Gaussian noise.

\section{Applications}\label{sec:app}

\subsection{Gaussian mixture model}\label{sec:gmm}

The Gaussian mixture model (GMM) is a type of probabilistic model often used for clustering and density estimation. It assumes that the observed data are generated from a mixture of several Gaussian distributions, each characterized by a mean vector and a covariance matrix. 

Consider observed data $X=(X_1,\cdots,X_n) \in \mathbb R^{p\times n}$, where each $X_i$ is a $p$-dimensional vector. We assume there are $k$ distinct clusters represented by the centers $\theta_{1},\cdots, \theta_k \in \mathbb R^p$. Denote $[n]:=\{1,\cdots,n\}$. Let $\mathbf z=(z_1,\cdots,z_n)^\T \in [k]^n$ be the latent variable that represents the true cluster labels for each observation $X_i$. The model assumes that each $X_i$ is generated as a result of adding a Gaussian noise term $\epsilon_i$ to its corresponding center $\theta_{z_i}$, with $\epsilon_i$'s being i.i.d. $\mathcal N(0,I_p)$. In particular,  
$X_i = \theta_{z_i} + \epsilon_i$ and we denote 
\begin{align}\label{eq:gmm}
X=\E(X) + E.
\end{align}

The goal of the GMM is to classify the observed data $X$ into $k$ clusters, and recover the latent variable $\mathbf z$. Let $\widetilde{\mathbf z}$ be the output of a clustering algorithm for the GMM and the accuracy of this algorithm can be evaluated using the misclassification rate, defined as:
$$\mathcal M (\mathbf z, \widetilde{\mathbf z}):=\frac{1}{n} \min_{\pi \in \mathcal S_k} \left| \{ i\in [n]: z_i \neq \pi(\widetilde z_i)\} \right|,$$
where $\mathcal S_k$ is the set of all permutations of $[k]$.

To solve the clustering problem, typically, more satisfying outcomes can be obtained by beginning with an initial estimate and then refining it with other tools like iteration or semidefinite programming (SDP). However, our discussion will focus exclusively on the application of simple spectral methods to illustrate perturbation results. Such methods have recently received considerable attention in the literature, as seen in \cite{CZ18, AFW22, LZZ21, ZZ22}, among others. Notably, the case of a two-cluster GMM with centers $\pm \mu$ for a fixed vector $\mu$ has been extensively studied in \cite{CZ18, AFW22}.

In the context of a general $k$-cluster framework, it is important to recognize insights from \cite{LZZ21} that establish spectral clustering as optimal for GMM. Our main goal is to show that the application of our perturbation results provides a succinct and effective proof for examining the theoretical performance of spectral algorithms.


Denote the minimum distance among centers as
$$\Delta:=\min_{j,l\in [k]: j\neq l} \|\theta_j -\theta_l\|.$$
When the separation between cluster centers, denoted by 
$\Delta$, is sufficiently large, distance-based clustering methods become particularly commendable.  

The principle of spectral clustering is elegantly simple.  Consider the SVD of $\E(X)=U \Sigma V^\T$, where $\Sigma$ is a $k \times k$ diagonal matrix. If the rank $r$ of $\E(X)$ is less than $k$, then $\Sigma$ will have $k-r$ zero diagonal entries. The matrices $U$ and $V$ respectively consist of $k$ orthonormal vectors that contain the left and right singular vectors of $\E(X)$. Let us denote $(U^\T \E (X))_j$ the columns of $U^\T \E (X) \in \mathbb{R}^{k\times n}$. We can demonstrate, as elaborated in  \cite{WangSupp24}, that for any columns $\theta_i$ and $\theta_j$ of $\E(X) = (\theta_{z_1},\cdots,\theta_{z_n})$, 
$$\|\theta_ i - \theta_j \| = \|(U^\T \E (X))_i-(U^\T \E (X))_j\|.$$
This indicates that the columns of $U^\T \E (X) =  \Sigma V^\T$ preserve the geometric relationship among the centers. 

Consider the SVD of $X = \widetilde U \widetilde \Lambda \widetilde V^\T$ and we use the previously defined notations $\widetilde U_{s}, \widetilde \Lambda_{s}, \widetilde V_{s}$. The crux of the analysis lies in proving that, with high probability, the following  holds:
\begin{align}\label{eq:gmmineq}
\max_{1\le j \le n}\|(\widetilde U_{k}^T X)_j - (U^\T \E (X))_j\| < \frac{1}{5} \Delta.
\end{align}
If this is the case, then performing clustering based on the distances among the columns of $\widetilde U_{r}^T X$ will, with high probability, successfully recover the correct cluster labels. In light of the preceding analysis, we hereby present the following algorithm:

\begin{algorithm}
    \caption{Spectral algorithm for GMM}
    \label{alg:gmm}
    \begin{algorithmic}
   \State \textbf{Input:} data matrix $X \in \mathbb R^{n\times p}$ and cluster number $k$.
     \State \textbf{Output:} cluster labels $\widetilde{\mathbf z} \in [k]^n$.
    \State  \emph{Step 1.} Perform SVD on $X$ and denote $\widetilde U_k \in \mathbb R^{p\times k}$ the singular vector matrix composed of the leading $k$ left singular vectors of $X$.
    \State \emph{Step 2.} Perform $k$-means clustering on the columns of $\widetilde U_k^\T X$.
    \end{algorithmic}
\end{algorithm}
 Algorithm \ref{alg:gmm} is identical to the algorithm proposed in \cite{LZZ21} and \cite{ZZ22}. This SVD-based algorithm has been widely adopted to address a variety of well-known problems in computer science and statistics, including the hidden clique, hidden bisection, hidden coloring, and matrix completion, among others (see for instance \cite{Vu18,LZZ21} and references therein for more discussion). 

The use of $k$-means clustering in \emph{Step 2} of Algorithm \ref{alg:gmm} is not a crucial component. The key requirement is to establish the inequality in \eqref{eq:gmmineq}; once this is achieved, alternative distance-based clustering algorithms may be employed in place of $k$-means.

For the output $\widetilde{\mathbf z}$ of Algorithm \ref{alg:gmm}, we could show the following result:
\begin{theorem}\label{thm:gmm} Consider the GMM \eqref{eq:gmm} with cluster number $k$. Let $\sigma_{\min}>0$ be the smallest singular value of $\E(X)$. Denote the smallest cluster size by $c_{\min}$.  Let $L>0$ and assume $(\sqrt n + \sqrt p)^2 \ge 32(L+7) \log(n+p) + 64(\log 9) k$.
If 
\begin{align}\label{eq:Deltabd}
&\Delta \ge \max\left\{ \frac{40(\sqrt n + \sqrt p)}{\sqrt{c_{\min}}}, 1800k\sqrt{(L+7)\log(n+p)}  \right\},\\
&\sigma_{\min} \ge 40(\sqrt n + \sqrt p)+ 3.8\times 10^4  k \sqrt{2(\log 9)  k + (L+7)\log(n+p)}\nonumber,
\end{align}
then $\E \mathcal M (\mathbf z, \widetilde{\mathbf z}) \le 40(n+p)^{-L}$.
\end{theorem}
The proof of Theorem \ref{thm:gmm} is a direct application of Corollary \ref{cor:weighted} and is detailed in \cite{WangSupp24}. By setting $L=(n+p)/\log(n+p)$, for instance, we achieve an exponential rate of misclassification. 

L{\"o}ffler, Zhang and Zhou \cite{LZZ21} have demonstrated that for the output $\widetilde{\mathbf z}$ of Algorithm \ref{alg:gmm}, provided that $\Delta \gg \frac{k^{10} (\sqrt n + \sqrt p)}{\sqrt{c_{\min}}},$ the following bound holds:  
 \begin{align}\label{bd:gmm-lzz}
 \E \mathcal M (\mathbf z, \widetilde{\mathbf z}) \le \exp \left( - (1-o(1)) \Delta^2/8 \right) + \exp(-0.08n).
 \end{align}
More recently, Zhang and Zhou \cite{ZZ22} have developed another innovative approach to analyze the output $\widetilde{\mathbf z}$ and obtained the same asymptotic exponential error rate \eqref{bd:gmm-lzz} for the GMM, assuming 
\begin{align*}
c_{\min} \ge 100 k^3 \quad\text{and} \quad \Delta\gg \frac{k^{3} (n + p)/\sqrt{n}}{\sqrt{c_{\min}}}.
\end{align*} 
Additionally, \cite[Theorem 3.1]{ZZ22} analyzes the estimator $\widetilde{\mathbf z}$ for the sub-Gaussian mixture model. For the output $\widetilde{\mathbf z}$ of Algorithm \ref{alg:gmm}, where in \emph{Step 2} the selection is made for $\widetilde{U}_r$ with $r = \rank(\E(X))$ (implying the use of exactly all $r$ singular vectors of $\E(X)$), an exponential error rate is attainable when
$$c_{\min}  \ge 10 k, \quad {\Delta}> C\frac{\sqrt{k}({\sqrt n + \sqrt p})}{\sqrt{c_{\min}}}
\quad \text{and}\quad
 {\sigma_r}> C({\sqrt n +\sqrt p})$$
 for some $C>0$. Abbe, Fan, and Wang \cite{AFW22} also explored the sub-Gaussian mixture model, employing the eigenvectors of the hollowed Gram matrix $\mathcal{H}(X^\top X)$ for clustering. Their approach leverages the $\ell_p$ perturbation results formulated in their paper but necessitates stricter conditions on the number of clusters, their sizes, and the collinearity of the cluster centers.

It is noteworthy that in the context of the GMM, results in \cite{LZZ21} and \cite{ZZ22} do not require any assumptions regarding the smallest singular value $\sigma_{\min}$, due to the exploitation of the Gaussian nature of the noise matrix $E$. Our Theorem \ref{thm:gmm} aligns with the findings for the sub-Gaussian mixture model in \cite{ZZ22}. Since our proof does not fully utilize the Gaussianality, we only employ the rotation invariance property to simplify the proof of isotropic local law, as given in Lemma \ref{lem:iso}. Our findings can be extended to scenarios where the entries of $E$ are sub-Gaussian random variables.  The extension to sub-Gaussian noise can be achieved through standard random matrix techniques, as discussed in Remark \ref{rem:gaussian-var}.

\subsection{Submatrix localization}\label{sec:submatrix} The general formulation of the submatrix localization or recovery problem involves locating or recovering a $k \times s$ submatrix with entries sampled from a distribution $\mathcal{P}$ within a larger $m \times n$ matrix populated with samples from a different distribution $\mathcal{Q}$.  In particular, when $\mathcal P$ and $\mathcal Q$ are both Bernoulli or Gaussian random matrices, the detection and recovery of the submatrix have been extensively studied. These investigations span various domains, including hidden clique, community detection, bi-clustering, and stochastic block models (see \cite{AKS98, FR10, Mc01, BKRSW11, BI13, BIS15,ACV14, DGGR14, KBR11, MRZ15, MW15,CX16, CLR17, HWX16, BBH18, Vu18, BV23, DHB23} and references therein).   

The task of recovering a single submatrix has been intensively explored (see for instance, \cite{CLR17,HWX16,BIS15,CX16,Vu18,Mc01} and references therein), but research on locating a growing number of submatrices is comparatively limited \cite{CX16, CLR17, DHB23}. In this section, we focus on the recovery of multiple (non-overlapping) submatrices within the model of size $m\times n$: 
\begin{align}\label{def:sub}
X=M + E,
\end{align}
where the entries of the noise matrix $E$ are i.i.d. standard Gaussian random variable. The signal matrix is given by
\begin{align*}
M=\sum_{i=1}^k \lambda_i \mathds{1}_{R_i} \mathds{1}_{C_i}^\T , 
\end{align*}
where $\{ R_i\}_{i=1}^k$ are disjoint subsets in $[m]$ and $\{ C_i\}_{i=1}^k$ are non-overlapping subsets in $[n]$. We denote $\mathds{1}_{R_i}$ as a vector in $\mathbb R^m$ with entries equal to 1 for indices in the set $R_i$ and 0 elsewhere, and $\mathds{1}_{C_i}$ is defined analogously.  Denote $|R_i| =r_i$ and $|C_i| = c_i$. Assume $\lambda_i\neq 0$ for all $1\le i \le k$. The goal is to discover the pairs $\{ (R_i, C_i) \}_{i=1}^k$ from the matrix $X$. 

Observe 
that the SVD of $M$ is given by 
$M= \sum_{i=1}^k \sigma_ i u_i v_i^\T:=U D V^\T,$ where $$\sigma_i := |\lambda_i | \sqrt{r_i c_i}, \, u_i := \mathrm{sgn}(\lambda_i) \frac{\mathds{1}_{R_i}}{\sqrt{r_i}},\, v_i:=\frac{\mathds{1}_{C_i}}{\sqrt{c_i}}.$$The columns of $U$ and $V$ are composed of $u_i$'s and $v_i$'s respectively and $D=\diag(\sigma_1,\cdots,\sigma_k)$. 

Note that $|X_{ij} - M_{ij}| = |E_{ij}|$ and with high probability, $\max_{i,j}|E_{ij}| \lesssim \sqrt{\log n}$. If $\min_{i,j} |M_{ij}| = \min_{1\le l \le k} |\lambda_l| \gtrsim \sqrt{\log n}$ and is greater than $\max_{i,j}|E_{ij}| $, a simple element-wise thresholding proves effective for identifying the submatrices. 

In general, as in Section \ref{sec:gmm}, we apply the same spectral clustering method to locate the submatrices. Denote $C_0:=[n] \setminus \cup_{i=1}^k C_i$ the set of isolated column indices with size $|C_0|=c_0$; define $R_0$ and its size $r_0$ analagouly. Let $(U^\T M)_j$ represent the columns of $U^\T M$. From $U^\T M = D V^\T$ and the definitions of $D$ and $V$, it follows that $(U^\T M)_j$ has only 1 non-zero entry $\lambda_l \sqrt{r_l}$ if $j \in C_l$ for some $l\in [k]$ and it is a zero vector if $j \in C_0$. In particular, if $i,j \in [n]$ belong to the same $C_l$ for $0\le l \le k$, it holds that $(U^\T M)_i = (U^\T M)_j$. For $i,j \in [n]$ from different submatrices, we have that $$\min_{\substack{i \in C_{l}, j \in C_{s}, \\0\le l \neq s\le k}}\|(U^\T M)_i - (U^\T M)_j\| = \min_{1\le i \le k}|\lambda_i| \sqrt{r_i} :=\Delta_R.$$ 
In particular, if $\Delta_R$ is sufficiently large, distance-based clustering can effectively be adapted to identify the column index sets of the submatrices.

Let $\widetilde U \widetilde D \widetilde V^\T$ be the SVD of $X=M+E$ and consider $\widetilde U_k^\T X$. The main objective is to show that, with high probability, 
$$\max_{1\le j \le n}\| (\widetilde U_k^\T X)_j - (U^\T M)_j \| < \frac{1}{5} \Delta_R.$$
Achieving this allows us to employ a standard clustering approach, such as 
$k$-means, based on distance to classify the columns of $\widetilde U_k^\T X$ and thus recover the column index subsets $\{ C_i\}_{i=0}^k$. Similarly, to identify the row index subsets  $\{ R_i\}_{i=0}^k$, we utilize the parameter 
$$
\Delta_C:= \min_{1\le i \le k}|\lambda_i| \sqrt{c_i}
$$
and apply $k$-means clustering to the rows of $X \widetilde V_r$. We propose the following algorithm: 

\begin{algorithm}
    \caption{Spectral algorithm for submatrix localization}
    \label{alg:sub}
    \begin{algorithmic}
   \State \textbf{Input:} data matrix $X \in \mathbb R^{m\times n}$ and submatrix number $k$.
     \State \textbf{Output:} column index subsets $\{ \widetilde C_i\}_{i=0}^k$ and row index subsets $\{\widetilde R_i\}_{i=0}^k$.
    \State  \emph{Step 1.} Perform SVD on $X$ and denote $\widetilde U_k \in \mathbb R^{m\times k}$ and $\widetilde V_k \in \mathbb R^{n\times k}$ the singular vector matrices composed of the leading $k$ left and right singular vectors of $X$ respectively.
    \State \emph{Step 2.} Perform $(k+1)$-means clustering on the columns of $\widetilde U_k^\T X$. Output the column index subsets $\{ \widetilde C_i\}_{i=0}^k$.
     \State \emph{Step 3.} Perform $(k+1)$-means clustering on the rows of $X \widetilde V_r$. Output the row index subsets $\{ \widetilde R_i\}_{i=0}^k$.
    \end{algorithmic}
\end{algorithm}

Define 
\begin{align*}
&\sigma_{\min}:= \min_{1\le i \le k} |\lambda_i| \sqrt{r_i c_i}, \quad r_{\min} := \min_{0\le i \le k} r_i, \quad c_{\min}:=\min_{0\le i \le k} c_i. 
\end{align*}

\begin{theorem}\label{thm:sub}Consider the submatrix localization model \eqref{def:sub} with $k$ submatrices. Let $L>0$ and assume $(\sqrt n + \sqrt p)^2 \ge 32(L+7) \log(n+p) + 64(\log 9) k$. Given that
\begin{align*}
&\Delta_R \ge \max\left\{ \frac{40(\sqrt m + \sqrt n)}{\sqrt{r_{\min}}}, 1800k\sqrt{(L+7)\log(m+n)} \right\},\\
&\Delta_C \ge \max\left\{ \frac{40(\sqrt m + \sqrt n)}{\sqrt{c_{\min}}}, 1800k\sqrt{(L+7)\log(m+n)}  \right\},\\
&\sigma_{\min} \ge 40(\sqrt m + \sqrt n)+ 3.8\times 10^4 k \sqrt{2(\log 9)  k + (L+7)\log(m+n)},
\end{align*}
Algorithm \ref{alg:sub} succeeds in finding $\widetilde R_i = R_i$ and $\widetilde C_i = C_{\pi(i)}$, $0\le i \le k$ for a bijection $\pi:[k+1]\to[k+1]$ with probability at least $1-40(m+n)^{-L}$.
\end{theorem}
The proof of Theorem \ref{thm:sub} parallels that of Theorem \ref{thm:gmm}, and therefore we omit the details.

Previous research on the model \eqref{def:sub} of multiple submatrix localization includes notable contributions such as those found in \cite{CX16, CLR17, DHB23}.  Chen and Xu \cite{CX16} examine this problem across different regimes, each corresponding to unique statistical and computational complexities. They focus on scenarios where all $k$ submatrices are identically sized at $K_R \times K_C$ and share a common positive value $\lambda_i = \lambda$. Their analysis of the Maximum Likelihood Estimator (MLE), a convexified version of MLE, and a simple thresholding algorithm address the challenges specific to hard, easy, and simple regimes, respectively. In the work of Dadon, Huleihel and Bendory \cite{DHB23}, the primary objective is to explore the computational and statistical limits associated with the detection and reconstruction of hidden submatrices. Under the same setting as \cite{CX16} in the context of the multiple submatrix recovery problem, the authors introduce a MLE alongside an alternative peeling estimator and investigate the performance of these estimators.

Our Algorithm \ref{alg:sub} is identical to Algorithm 3 presented in Cai, Liang and Rakhlin's paper \cite{CLR17}. The assumptions laid out in \cite{CLR17} include $r_i \asymp K_R, c_i \asymp K_C, \lambda_i \asymp \lambda$ for all  $1 \le i \le k$ and $\min\{ K_R, K_C \}\gtrsim \max\{\sqrt m, \sqrt n \}$. Given that 
\begin{align}\label{assump:clr}
\lambda \gtrsim \frac{\sqrt k}{\min\{\sqrt{K_R}, \sqrt{K_C}\}} + \max\left\{\sqrt{\frac{\log m}{K_C}}, \sqrt{\frac{\log n}{K_R} }\right\} + \frac{\sqrt m + \sqrt n}{\sqrt{K_R K_C}},
\end{align}
 the authors of \cite{CLR17} demonstrate that Algorithm \ref{alg:sub} successfully recovers the true submatrix row and column index sets with probability at least $1-m^{-c} - n^{-c} -2 \exp(-c(m+n)).$ The entries of the noise matrix $E$ in \cite{CLR17} are assumed to be i.i.d zero-mean sub-Gaussian random variables. 

While our method does not require that all row or column index sets have the same order of sizes, in the special case where $r_i \asymp K_R$, $c_i \asymp K_C$, and $\lambda_i \asymp \lambda$ for all $1 \le i \le k$, and furthermore $r_0\gtrsim K_R$ and $c_0\gtrsim K_C$, our analysis indicates that if
\begin{align}\label{assump:ours}
\lambda \gtrsim   \frac{k \sqrt{\log(m+n)}}{\min\{\sqrt{K_R}, \sqrt{K_C} \}}+\frac{\sqrt m + \sqrt n}{\min\{K_R, K_C \}},
\end{align}
then Algorithm \ref{alg:sub} successfully recovers the submatrix index subsets with probability at least $1-(m+n)^{-c}$. It should be emphasized that the condition in \eqref{assump:ours} is more stringent than that in \eqref{assump:clr}, a difference that becomes particularly pronounced in cases where $K_R$ and $K_C$ are highly unbalanced. An interesting direction for future research would be to improve our perturbation bounds to accommodate cases with unbalanced matrix dimensions.

\section{Basic tools and proof overview}\label{sec:basic}

\subsection{Basic tools}\label{sec:LA} This section presents the basic tools necessary for the proofs of our main results, many of which build upon the previous work by O'Rourke, Vu and the author \cite{OVW22}.

We start with the standard linearization of the perturbation model \eqref{def:tildeA}. Consider the $(N+n)\times (N+n)$ matrices 
\begin{align*}
\mathcal A := \begin{pmatrix}
0& A\\
A^\T& 0
\end{pmatrix}
\quad\text{and} \quad \mathcal E:=\begin{pmatrix}
0& E\\
E^\T& 0
\end{pmatrix}
\end{align*}
in block form.  
Define
\[ \widetilde{\mathcal A}:=\mathcal A + \mathcal E. \]
The non-zero eigenvalues of $\mathcal A$ are given by $\lambda_j= \sigma_j$ and $\lambda_{j+r} = -\sigma_j$ for $1\le j \le r$. Then $\mathbf u_j:=\frac{1}{\sqrt 2}(u_j^\T, v_j^\T)^\T$  and $\mathbf u_{j+r}:=\frac{1}{\sqrt 2}(u_j^\T, -v_j^\T)^\T$ for $1\le j \le r$ are their corresponding orthonormal eigenvectors. The spectral decomposition of $\mathcal A$ is 
\begin{equation}\label{eq:A}
 \mathcal A = \mathcal{U} \mathcal{D} {\mathcal U}^\T, 
 \end{equation}
where $\mathcal U:=(\mathbf{u}_1,\ldots,\mathbf{u}_{2r})$ and $\mathcal D := \diag(\lambda_1,\cdots, \lambda_{2r})$. It follows that $\mathcal U^\T \mathcal U = I_{2r}$. Similarly, the non-zero eigenvalues  
of $\widetilde{\mathcal A}$ are denoted by $\widetilde\lambda_j=\widetilde\sigma_j$ and  $\widetilde\lambda_{j+\min\{N, n\}}=-\widetilde\sigma_j$ for $1\le j \le \min\{N, n\}$.  The eigenvector corresponding to $\widetilde\lambda_j$ is denoted by $\widetilde{\mathbf{u}}_j$ and is formed by the right and left singular vectors of $\widetilde{A}$.  

For $z \in \mathbb{C}$ with $|z| > \|\mathcal E\|$, we define the resolvent of $\mathcal E$ as
\[ G(z) := (zI - \mathcal E)^{-1}. \] 
Often we will drop the identity matrix and simply write $(z - \mathcal E)^{-1}$ for this matrix. We use $G_{ij}(z)$ to denote the $(i,j)$-entry of $G(z)$.   

The key observation is that $G(z)$ can be approximated by a diagonal matrix. 
Consider a random diagonal matrix 
\begin{equation}\label{eq:Phi}
 \Phi(z):=\begin{pmatrix}
\frac{1}{\phi_1(z)}I_N & 0 \\
0 &\frac{1}{\phi_2(z)}I_n
\end{pmatrix},
\end{equation}
where 
\begin{equation}\label{eq:phi12}
\phi_1(z):= z - \sum_{t \in \llbracket N+1,N+n \rrbracket} G_{tt}(z),\qquad\phi_2(z):= z - \sum_{s \in \llbracket 1,N \rrbracket} G_{ss}(z).
\end{equation}
Using the Schur complement, the Green function $G(z)$ of $\mathcal E$ is composed of four blocks: the Green function $G_1(z)$ of $EE^T$, the Green function $G_2(z)$ of $E^TE$, and $E^TG_1(z)$ and $G_1(z)E$ (with minor modifications). The approximation $\Phi(z)$ is motivated by the quadratic equations satisfied by the Stieltjes transforms of the spectral distributions of $EE^\T$ and $E^\T E$, similar to how the Stieltjes transform of the semi-circle law satisfies a quadratic equation in the symmetric case.

By setting 
\[
\mathcal I^\mathrm{u} := \begin{pmatrix}
I_N & 0\\
0& 0
\end{pmatrix} 
\quad\text{and}\quad
\mathcal I^\mathrm{d} := \begin{pmatrix}
0 & 0\\
0& I_n
\end{pmatrix},
 \]
one can rewrite \eqref{eq:phi12} as 
\begin{align}\label{eq:defphi}
\phi_1(z) = z - \tr \mathcal I^\mathrm{d} G(z), \qquad \phi_2(z) = z- \tr \mathcal I^\mathrm{u} G(z).
\end{align}
By elementary linear algebra (the proof deferred to \cite{WangSupp24}), it can be verified that
\begin{align}\label{eq:difference}
\phi_1(z) = \phi_2(z) - \frac{1}{z}(n-N).
\end{align}
From the definition of $ \mathcal U$ in \eqref{eq:A}, it is easy to verify that  
\begin{align}\label{eq:productPhi}
 \mathcal U^\T \Phi(z) \mathcal U = \begin{pmatrix}
\alpha(z) I_r & \beta(z) I_r\\
\beta(z) I_r & \alpha(z) I_r
\end{pmatrix}, 
\end{align}
where we denote $$\alpha(z):=\frac{1}{2}\left(\frac{1}{\phi_1(z)} + \frac{1}{\phi_2(z)} \right)\quad \text{and}\quad \beta(z):= \frac{1}{2}\left(\frac{1}{\phi_1(z)} - \frac{1}{\phi_2(z)} \right)$$ for notational brevity. 
It follows that 
\begin{align}\label{eq:Phinorm}
\|\mathcal U^\T \Phi(z) \mathcal U\| = \max \left\{\frac{1}{|\phi_1(z)|}, \frac{1}{|\phi_2(z)|} \right\}.
\end{align}
The next lemma offers bounds for the resolvent and the functions $\phi_i(z)$'s. The proof follows similarly to that of Lemma 16 in \cite{OVW22} and is omitted for brevity. 
\begin{lemma} \label{lemma:resolventnorm}
On the event where $\|E \| \leq 2 (\sqrt{N}+\sqrt{n})$,
\[ \| G(z) \| \leq \frac{b}{b-1}\frac{1}{|z|}\]
and
\begin{equation}\label{eq:phibd}
\left(1-\frac{1}{4b(b-1)}\right) |z| \le |\phi_i(z)| \le \left(1+\frac{1}{4b(b-1)}\right) |z| \quad\text{for } i=1,2
\end{equation}
for any $z \in \mathbb{C}$ with $|z| \geq 2b (\sqrt{N}+\sqrt{n})$ and for any $k \in \llbracket 1,N+n \rrbracket$.  
\end{lemma}
Consequently, by Lemma \ref{lemma:resolventnorm}, we obtain
\begin{align}\label{eq:bdontrace}
\max\{|\tr G(z )|, |\tr \mathcal I^u G(z)  |, |\tr \mathcal I^d G(z)  |\}  \le (N+n)\|G(z)\| \le \frac{b}{b-1}\frac{N+n}{|z|}.
\end{align}

The subsequent isotropic local law is derived using a proof similar to that of \cite[Lemma 27]{OVW22}. For completeness, we briefly describe the proof in  the supplementary material \cite{WangSupp24}.
\begin{lemma}\label{lem:iso}Let $K>0$ and assume $(\sqrt N + \sqrt n)^2 \ge 32 (K+1)\log(N+n)$. For any unit vectors $\mathbf x, \mathbf y\in \mathbb{R}^{N+n}$ and for any $z \in \mathbb{C}$ with $|z| \ge 2b(\sqrt{N}+ \sqrt{n})$, 
\begin{align}\label{eq:quadlocal}
\left| \mathbf x^\T \left(G(z)-\Phi(z) \right) \mathbf y\right| \le \frac{5b^2}{(b-1)^2}\frac{  \sqrt{(K+1)\log (N+n)}}{ |z|^2}
\end{align}
with probability at least $1 - 9(N+n)^{-(K+1)}$.  
\end{lemma}

Recall $$\eta=\frac{11 b^2}{(b-1)^2} \sqrt{(K+7)\log (N+n)+(\log 9) r}.$$
Denote $$\mathsf D:= \{ z\in \mathbb C : 2b(\sqrt N + \sqrt n)\le |z| \le 2n^3\}.$$ Using the previous lemma and a standard $\varepsilon$-net argument, we obtain the following result that is analogous to \cite[Lemma 9]{OVW22}:
 \begin{lemma} \label{lem:betterbd}
Under the assumptions of Theorem \ref{thm:subspace}, one has
\[ \max_{ z \in \mathsf D} |z|^2 \left\|  \mathcal U^\T \left(G(z)-\Phi(z) \right) \mathcal U \right\| \leq \eta \] 
with probability at least $1 - 9(N+n)^{-K}$.
\end{lemma}
Lemma \ref{lem:betterbd} improves the rank $r$-dependence in the bound of \cite[Lemma 9]{OVW22}. The proof of Lemma \ref{lem:betterbd} is included in \cite{WangSupp24}. For the case $2b(\sqrt N + \sqrt n)> 2n^3$ where $\mathsf D$ is empty, $G(z)$ can be approximately be even simpler matrices (see Lemma \ref{lem:simpleapprox} below). 

The following result on the location of perturbed singular values is obtained using Lemma \ref{lem:betterbd}. Consider the random function
\begin{equation} \label{eq:def:varphi}
	\varphi(z):=\phi_1(z)\phi_2(z),
\end{equation} 
where $\phi_1(z)$ and $\phi_2(z)$ are defined in \eqref{eq:phi12}. Define the auxiliary functions for $b\ge 2$:
\begin{align*}
\xi(b):= 1+ \frac{1}{2(b-1)^2}\quad \text{and} \quad \chi(b):= 1+ \frac{1}{4b(b-1)}.
\end{align*}
Define a set in the complex plane in the neighborhood of any $\sigma\in \mathbb R$ by
\begin{equation} \label{eq:resolventset}
	S_{\sigma} : = \left\{w \in \mathbb{C} : | \Im(w) | \leq 20 \chi(b) \eta r, \sigma - 20 \chi(b) \eta r \leq \Re(w) \leq \chi(b)\sigma +  20 \chi(b) \eta r \right\}.
\end{equation}

\begin{theorem}[Singular value locations] \label{thm:singularlocation}
Let $A$ and $E$ be $N \times n$ real matrices, where $A$ is deterministic and the entries of $E$ are i.i.d. standard Gaussian random variables. Assume $A$ has rank $r \geq 1$. Let $K>0$ and $b\ge 2$. Denote $\eta:= \frac{11b^2}{(b-1)^2} \sqrt{2(\log 9) r+(K+7) \log (N+n)}$. Assume $(\sqrt N + \sqrt n)^2 \ge 32(K+7)\log(N+n) + 64(\log 9) r$. Let $1\le r_0 \le r$ such that $ \sigma_{r_0} \ge 2b(\sqrt N+ \sqrt{n}) + 80b\eta r$ and $\delta_{r_0}\geq 75\chi(b)\eta r$. Consider any $1 \le k\le s \le r_0$ satisfying $\min\{\delta_{k-1},\delta_s \} \ge 75 \chi(b) \eta r$. For any $j \in \llbracket k, s \rrbracket$, there exists $j_0 \in \llbracket k, s \rrbracket$ such that $\widetilde \sigma_j \in S_{\sigma_{j_0}}$, and
\begin{align}\label{eq:location}
|\varphi(\widetilde{\sigma}_{j}) - \sigma_{j_0}^2| \leq 20\xi(b) \chi(b)\eta r \left(\widetilde{\sigma}_{j}+\chi(b)\sigma_{j_0} \right)
\end{align}
with probability at least $1-10(N+n)^{-K}$.
\end{theorem}
Our Theorem \ref{thm:singularlocation} extends the analysis of Theorem 12 in \cite{OVW22} in an important direction by removing the singular value separation requirement. This improvement stems from a careful refinement of the theoretical framework, detailed in the supplementary material \cite{WangSupp24}.

The next result suggests that when $|z|$ is large, the resolvent $G(z)$ can be approximated by simpler matrices. The proof is analogous to that of Lemma 17 in \cite{OVW22}, and is omitted. 
\begin{lemma} \label{lem:simpleapprox}On the event where $\|E \| \leq 2 (\sqrt{N}+\sqrt{n})$,
\begin{equation*}
\left\| G(z) - \frac{1}{z} I_{N+n} \right\| \le \frac{b}{b-1}\frac{\| E\|}{|z|^2}
\quad\mbox{and}\quad
\left\| G(z) - \frac{1}{z} I_{N+n} - \frac{\mathcal E}{z^2}\right\| \le \frac{b}{b-1}\frac{\| E\|^2}{|z|^3}
\end{equation*}
for any $z \in \mathbb{C}$ with $|z| \geq 2b (\sqrt{N}+\sqrt{n})$.
\end{lemma}
\begin{lemma}[Lemma 13 from \cite{OVW22}]\label{lem:svloc-large}Under the assumptions of Theorem \ref{thm:subspace}, 
$$\max_{l \in \llbracket 1, r_0\rrbracket : \sigma_l >\frac{1}{2}n^2} |\widetilde \sigma_l - \sigma_l |\le \eta r$$
with probability at least $1- (N+n)^{-1.5 r^2(K+4)}$.
\end{lemma}

The next result provides a non-asymptotic bound on the operator bound of $\|E\|$. 
\begin{lemma}[Spectral norm bound; see (2.3) from \cite{MR2827856}] \label{lemma:norm}
Let $E$ be an $N \times n$ matrix whose entries are independent standard Gaussian random variables.  Then
\[ \|E \| \leq 2 (\sqrt{N}+\sqrt{n}) \]
with probability at least $1 - 2 e^{-(\sqrt{N}+\sqrt{n})^2/2}$.  
\end{lemma}

\subsection{Proof overview}\label{sec:overview}
In this section, we outline our proof strategy, which leverages techniques from random matrix theory, particularly the resolvent method, to analyze the eigenvalues and eigenvectors of the symmetric matrices $\mathcal A$ and $\widetilde{\mathcal A}=\mathcal A + \mathcal E$ in Section \ref{sec:LA}.

At the heart of our analysis is the isotropic local law (Lemma \ref{lem:iso}), which asserts that the resolvent $G(z)=(zI - \mathcal E)^{-1}$ can be approximated by a simpler matrix $\Phi(z)$. This approximation streamlines complex calculations involving $G(z)$ and is a technique commonly used to study extreme eigenvalues and eigenvectors in random matrix theory, as seen in, for instance, \cite{BEKYY14, KY, BDW21,BDWW20}. Our work diverges from these prior approaches by selecting $\Phi(z)$ as a random matrix derived from $G(z)$ itself, which better suits the finite sample context, compared to the deterministic approximations used in previous studies that rely on Stieltjes transforms in the asymptotic regime.

Building upon the isotropic local law, we determine the singular value locations of $\widetilde A$ in Theorem \ref{thm:singularlocation} and achieve the control of $\left\|  \mathcal U^\T \left(G(z)-\Phi(z) \right) \mathcal U \right\|$ as given in Lemma \ref{lem:betterbd}. These instruments have been previously explored in the previous work by O'Rourke, Vu and the author \cite{OVW22}. In this paper, we refine these estimations and ease the conditions in \cite{OVW22}. Furthermore, we deploy these refined tools to derive a variety of new perturbation bounds.

To illustrate the key ideas behind our perturbation bounds, we simply focus on the largest eigenvector $\widetilde{\mathbf u}_1$ of $\widetilde{\mathcal A}$. While a complete analysis requires considering both $\widetilde{\mathbf u}_1$ and $\widetilde{\mathbf u}_{r+1}$ (as they jointly involve the largest singular vectors $u_1$ and $v_1$), we temporarily consider only $\widetilde{\mathbf u}_1$ for clarity. We emphasize that the calculation presented below is simplified to convey the main idea. In the actual proof, we work with all eigenvectors $\mathbf{u}_j$ with $j\neq 1, r+1$ and $j\in [2r]$ simultaneously, leading to a more comprehensive analysis involving the block structure of $\Phi$.

We start with the decomposition 
\begin{align}\label{eq:simdec}
\widetilde{\mathbf u}_1 = ({\mathbf u}_1 {\mathbf u}_1^\T )\widetilde{\mathbf u}_1 + \mathcal P_1 \widetilde{\mathbf u}_1 + \mathcal Q \widetilde{\mathbf u}_1,
\end{align}
where $ \mathcal P_1 = \mathcal U_1 \mathcal U_1^\T$ and $\mathcal U_1$ is the matrix of eigenvectors of $\mathcal A$ excluding ${\mathbf u}_1$.  Meanwhile, $\mathcal Q$ is the orthogonal projection matrix onto the null space of $\mathcal A$. The challenge in establishing perturbation bounds for $\widetilde{\mathbf u}_1$ lies in quantifying the latter two terms on the right-hand side of \eqref{eq:simdec}. 

First, for the $\ell_2$ analysis, we aim to bound $\sin \angle ({\mathbf u}_1, \widetilde{\mathbf u}_1)$. By taking the Frobenius norm on both sides of \eqref{eq:simdec}, we obtain
$$1 = \cos^2  \angle ({\mathbf u}_1, \widetilde{\mathbf u}_1) + \| \mathcal P_1 \widetilde{\mathbf u}_1\|^2 + \|\mathcal Q \widetilde{\mathbf u}_1\|^2.$$
Rearranging the terms yields
$$\sin^2 \angle ({\mathbf u}_1, \widetilde{\mathbf u}_1) = \| \mathcal P_1 \widetilde{\mathbf u}_1\|^2 + \|\mathcal Q \widetilde{\mathbf u}_1\|^2.$$
A straightforward linear algebra argument allows us to bound  $\|\mathcal Q \widetilde{\mathbf u}_1\|$ by the noise-to-signal ratio $\|E\|/\sigma_1$. The main task is then to establish a bound for $\| \mathcal P_1 \widetilde{\mathbf u}_1\| \le \|\mathcal U_1^\T\widetilde{\mathbf u}_1\|$, which effectively comes down to bounding $|{\mathbf u}_j^\T \widetilde{\mathbf u}_1|$ for $j\neq 1$. We explain how to achieve this bound below.

From the equation $\widetilde{\mathcal A} \widetilde{\mathbf u}_1 = (\mathcal A + \mathcal E) \widetilde{\mathbf u}_1 = \widetilde{\lambda}_1 \widetilde{\mathbf u}_1$, we can express $\widetilde{\mathbf u}_1$ as
$$ \widetilde{\mathbf u}_1 =(\widetilde \lambda_1 I - \mathcal E)^{-1} \mathcal A \widetilde{\mathbf u}_1=G(\widetilde \lambda_1) \mathcal A \widetilde{\mathbf u}_1$$ and further rewrite it as 
$$ \widetilde{\mathbf u}_1 = \Phi(\widetilde \lambda_1) \mathcal A \widetilde{\mathbf u}_1 + \left(G(\widetilde \lambda_1)-\Phi(\widetilde \lambda_1) \right) \mathcal A \widetilde{\mathbf u}_1.$$
Hence, for $j\neq 1$, we have
\begin{equation}\label{eq:idea1}
{\mathbf u}_j^\T \widetilde{\mathbf u}_1={\mathbf u}_j^\T \Phi(\widetilde \lambda_1) \mathcal A \widetilde{\mathbf u}_1 + {\mathbf u}_j^\T\left(G(\widetilde \lambda_1)-\Phi(\widetilde \lambda_1) \right) \mathcal A \widetilde{\mathbf u}_1
\end{equation}
Calculations similar to those in \eqref{eq:productPhi} indicate that the first term on the right-hand side of \eqref{eq:idea1}, ${\mathbf u}_j^\T \Phi(\widetilde \lambda_1) \mathcal A \widetilde{\mathbf u}_1$, is exactly $\lambda_j \alpha(\widetilde \lambda_1){\mathbf u}_j^\T \widetilde{\mathbf u}_1$ (omitting the term containing ${\mathbf u}_{r+j}$). We continue from \eqref{eq:idea1} to get 
$$\left(1-\lambda_j \alpha(\widetilde \lambda_1) \right) {\mathbf u}_j^\T \widetilde{\mathbf u}_1 \approx {\mathbf u}_j^\T \left(G(\widetilde \lambda_1)-\Phi(\widetilde \lambda_1) \right) \mathcal A \widetilde{\mathbf u}_1.$$
To control $|{\mathbf u}_j^\T \widetilde{\mathbf u}_1|$, we apply Theorem \ref{thm:singularlocation} to analyze the coefficient $1-\lambda_j \alpha(\widetilde \lambda_1)$ that precedes it. Lemma \ref{lem:betterbd} is applied to manage the term on the right-hand side. 

Next, for the $\ell_\infty$ analysis, from \eqref{eq:simdec}, we obtain
$$\| \widetilde{\mathbf u}_1 - ({\mathbf u}_1 {\mathbf u}_1^\T )\widetilde{\mathbf u}_1\|_{\infty} \le \|\mathcal P_1 \widetilde{\mathbf u}_1\|_{\infty} + \|\mathcal Q \widetilde{\mathbf u}_1\|_{\infty}\le \|\mathcal U\|_{2,\infty} \|\mathcal U_1^\T \widetilde{\mathbf u}_1\| + \|\mathcal Q \widetilde{\mathbf u}_1\|_{\infty}.$$
The bound for $\|\mathcal U_1^\T \widetilde{\mathbf u}_1\|$ has already been established in the preceding $\ell_2$ analysis. The second term, $\|\mathcal Q \widetilde{\mathbf u}_1\|_{\infty}$, can be bounded by considering the fact
$$\mathcal Q \widetilde{\mathbf u}_1 = \mathcal Q \left(G(\widetilde \lambda_1)-\Phi(\widetilde \lambda_1) \right)\mathcal A \widetilde{\mathbf u}_1$$ and then applying Lemma \ref{lem:betterbd}. 

These are the main ideas that we have incorporated in our proofs. Before concluding this section, we would like to highlight that the results presented in this paper can be extended to scenarios where the noise matrix $E$ contains independent sub-Gaussian entries. This extension would rely on a lemma similar to Lemma \ref{lem:iso}, which can be demonstrated using the tools provided by random matrix theory. However, due to the technical complexities involved, we have chosen to reserve the discussion of this extension to sub-Gaussian cases for a forthcoming paper. 
It remains a highly interesting direction to further establish these perturbation bounds when the noise matrix $E$ comprises heteroskedastic random variables. We believe that new tools and insights, extending beyond the scope of the methods presented in this paper, will be required to rigorously establish such extensions.






\medskip

{\bf Acknowledgement:} The author would like to thank Sean O'Rourke and Van Vu for their insightful discussions and invaluable suggestions.  The author is also grateful to Yuling Yan for bringing the relevant results \cite{YCF21,YW24} to her attention. The author would also like to thank the anonymous referees, Associate Editor, and Editor for their constructive comments that improved the quality of this paper.


\newpage
\appendix

\begin{center}
    {\huge \textbf{Supplementary material to ``Analysis of singular subspaces under random perturbations"}} \\[0.5cm]
    {\large Ke Wang}
\end{center}

\section*{Abstract}
\begin{quotation}
\noindent This file contains the detailed proofs of the theorems in the paper \cite{Wang24}, as well as those of basic tools related to them and numerical simulations. Specifically, Section \ref{sec:simu} presents simulation results that empirically explore the sharpness of our perturbation bounds and examine their dependence on the signal rank \(r\). Section \ref{sec:basic1} provides preliminary information on matrix norms (in Section \ref{sec:matrixnorms}) and the distance between subspaces (in Section \ref{sec:distance}). The proofs of Theorems 2.3, 2.10, 2.11 and 2.13 can be found in Section \ref{sec:proofmain}. In Section \ref{sec:general}, the proofs of Theorems 2.6 and 2.7 are presented. The proof of Theorem 5.1 is given in Section \ref{sec:app:gmm}. Proofs of \eqref{eq:equiv}, \eqref{eq:equivup}, Proposition \ref{prop:2inf} and \eqref{eq:svdUV} are compiled in Section \ref{sec:otherproofs}. Section \ref{app:twolem} collects the proofs of Lemma 6.2 and Lemma 6.3. Finally, Section \ref{app:bdeach}, Section \ref{app:singularlocation} and Section \ref{app:proof2parts} include the proofs of Lemma \ref{lem:projsv}, Theorem 6.4, and (28) and Proposition \ref{prop:blockphi}, respectively. 
\end{quotation}

\vspace{1cm}
\hrule
\vspace{1cm}

\setcounter{section}{0}
\setcounter{equation}{0}
\renewcommand{\thesection}{S.\arabic{section}}
\renewcommand{\theequation}{S.\arabic{equation}}

\section{Discussions and Simulations}\label{sec:simu}
This section presents numerical experiments that evaluate the sharpness of our perturbation bounds, focusing specifically on how error scales with signal rank $r$. Our simulations aim to assess the tightness of our bounds and identify potential areas for improvement. We also compare our bounds with related results in the literature.

\subsection{Leading singular vector perturbations}
In Section 4.1 of \cite{Wang24}, we explored the optimality of the leading singular vector bound. In particular, we established that with probability $1-(N+n)^{-C}$,
\begin{equation}\label{eq:better-r-sup}
\sin \angle ( u_1, \widetilde{u}_1) \lesssim \frac{\sqrt{r +\log(N+n)}}{\delta_1} + \frac{\|E\|}{\sigma_1}.
\end{equation}
We argued that the a high-probability lower bound should be in terms of 
\begin{equation*}
\sin \angle ( u_1, \widetilde{u}_1) \gtrsim \frac{f(r)}{\delta_1} + \frac{\|E\|}{\sigma_1}.
\end{equation*}
Here, we further investigate the precise form of $f(r)$.

To examine whether a lower bound on $\sin \angle(u_1, \widetilde{u}_1)$ should scale as $\sqrt{r}$ or follow a different power law in $r$, we simulate the model $\widetilde{A}=A + E$ with various ranks $r \in \{50, 55, 60,\ldots, 600\}$ and compute the 10th percentile of $\sin \angle(u_1, \widetilde{u}_1)$ over 100 trials for each $r$. The noise $E$ has i.i.d. standard Gaussian noise. To assess potential scaling laws, we compare three candidates ($\sqrt{r}$, $r^{1/3}$, and $\log(r)$) by fitting the empirical data via linear regression against each scaling term. Figure~\ref{fig:powerfit-top-n800} shows that the $\sqrt{r}$ scaling provides the strongest fit (coefficient of determination $R^2 = 0.9986$, where $R^2=1$ indicates perfect fit), outperforming both $r^{1/3}$ ($R^2=0.9939$) and $\log(r)$ ($R^2=0.9675$).

\begin{figure}[!ht]
 \begin{center}
   \includegraphics[width=13cm,height=5.5cm]{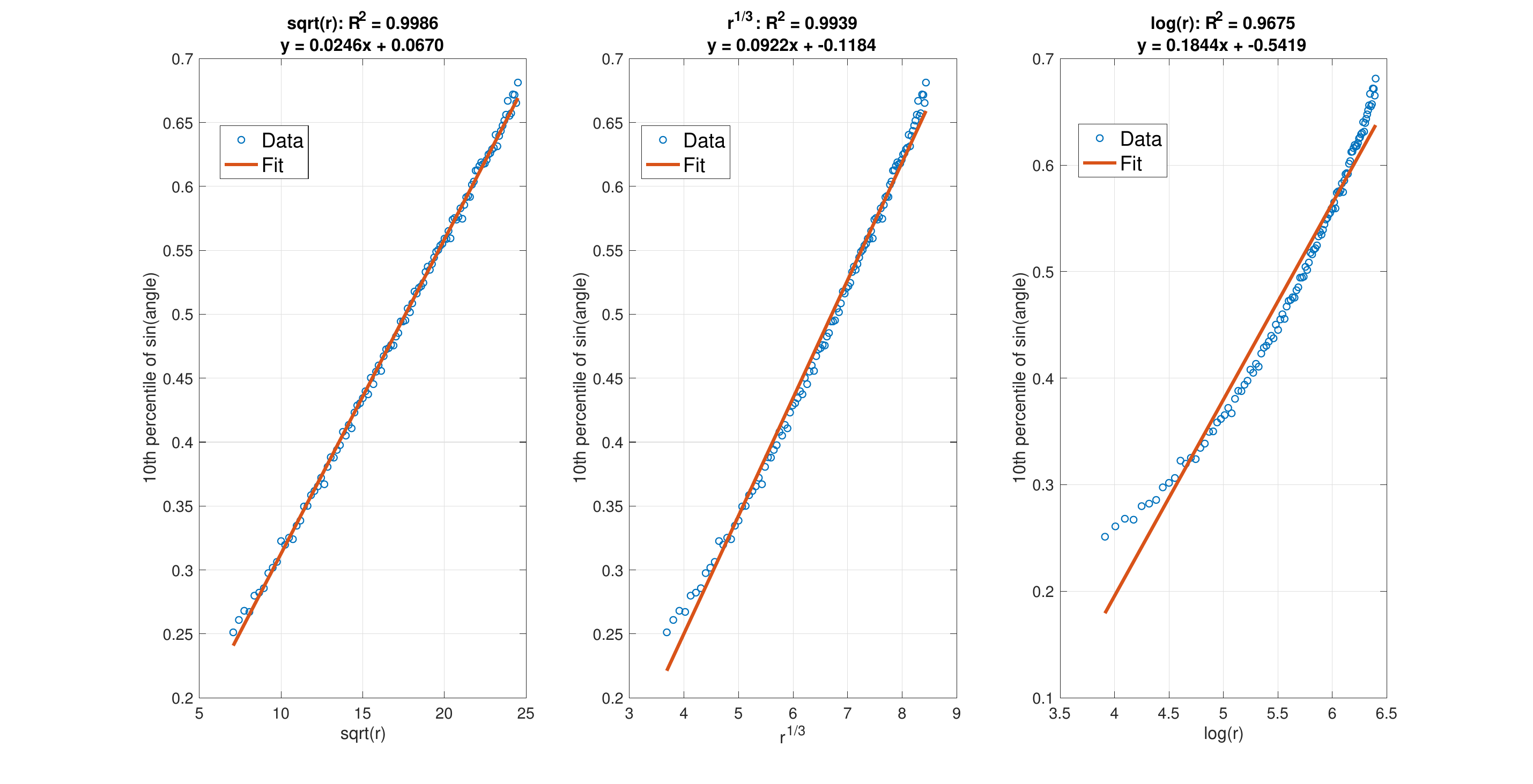}
   \caption{\textbf{Empirical exploration of rank dependence in singular vector perturbation.} For each rank $r \in \{50, 55, 60, \ldots, 600\}$, we simulate 100 independent trials of a rank-$r$ signal matrix $A$ corrupted by Gaussian noise and compute the \emph{10th percentile} of $\sin \angle(u_1, \widetilde{u}_1)$. We set $n=800$. The singular values of $A$ are given as follows:  $\sigma_1 = 150$, $\sigma_2 = \sigma_1 - 20= 130$ with $\delta_1=20$, and the remaining $r-2$ nonzero singular values decay linearly down to 120. The figure compares three scaling laws, e.g., $\sqrt{r}$, $r^{1/3}$, and $\log(r)$, via linear fits and $R^2$ values, showing that $\sqrt{r}$ provides the best empirical fit ($R^2 = 0.9986$). }
 \end{center}\label{fig:powerfit-top-n800}
\end{figure}

We now examine the \(\ell_\infty\) norm bound, following the discussion Section 4.1 of \cite{Wang24}. Under the incoherence assumption $\|U\|_{\max} \lesssim \frac{1}{\sqrt N}$,  Corollary~2.9 implies that with high probability
\begin{align}\label{eq:cor29bd}
\min_{\mathsf{s}\in\{\pm 1\}}\| u_1- \mathsf{s}\widetilde{u}_1\|_{\infty} \lesssim  \frac{1}{\sqrt N}\frac{\sqrt{r}\sqrt{r +\log(N+n)}}{\delta_1} + \frac{\sqrt{r\log(N+n)}}{\sigma_1}.
\end{align}
 Based on the lower bound for $\sin \angle(u_1, \widetilde{u}_1)$, we expect 
\[ \min_{\mathsf{s}\in\{\pm 1\}}\| u_1- \mathsf{s}\widetilde{u}_1\|_{\infty}\gtrsim \frac{1}{\sqrt N} \frac{f(r)}{\delta_1} + \frac{1}{\sigma_1},\]
with $f(r)$ consistent with the angular lower bound.

We numerically explore this lower bound using the model $\widetilde{A}=A + E$ with parameters $n = 400$, $\sigma_1 = 1000$, $\sigma_r = 950$, and $\sigma_2 = \sigma_1 - \delta_1 = 980$ ($\delta_1 = 20$). Assume $E$ has i.i.d. standard Gaussian entries.  For ranks $r = 10, 30, \ldots, 350$, we compute the 10th percentile of \(\min_{\mathsf{s}\in\{\pm 1\}} \| u_1 - \mathsf{s} \widetilde{u}_1 \|_\infty\) over 100 trials. The singular vectors of $A$ are drawn uniformly from the Haar measure, ensuring incoherence. As shown in Figure~\ref{fig:linf}, comparing different scaling laws reveals that the $\sqrt{r}$-fit provides the best fit ($R^2 = 0.9965$), outperforming both $r$ ($R^2 = 0.9591$) and $r^{1/3}$ ($R^2 = 0.9909$). Additionally, a log-log regression estimates the scaling as $r^{0.567}$, close to the expected $r^{0.5}$. These results suggest that $\ell_{\infty}$-norm perturbations exhibit similar rank behavior as the $\ell_2$ case and imply that we could potentially improve the rank dependence in the first term of \eqref{eq:cor29bd} by an extra $\sqrt{r}$ factor.



Finally, we compare our results with existing bounds in related settings. For symmetric matrix models with sub-gaussian noise, Zhong \cite{Zhong17} analyzes the top eigenvector perturbation $\sin\angle(u_1,\widetilde{u}_1)$. For low-rank signal matrices, they obtain a similar two-term upper bound as \eqref{eq:better-r-sup}, but require a stronger signal condition $\lambda_1 \gtrsim n\log n$. Other works, such as \cite[Theorem 1]{CWC21}, establish individual singular vector perturbation bounds under more general noise assumptions, but their $\ell_2$ and $\ell_{\infty}$ bounds exhibit less favorable dependence on the signal rank $r$, even with additional constraints. Our numerical experiments suggest the $r$ dependence in our bound \eqref{eq:better-r-sup} for $\sin\angle(u_1,\widetilde{u}_1)$ is tight, while for the $\ell_{\infty}$ bound  \eqref{eq:cor29bd}, there is potential to improve the factor $\sqrt{r}\sqrt{r +\log(N+n)}$ to $\sqrt{r +\log(N+n)}$.

\begin{figure}[ht]
\centering
\includegraphics[width=0.9\textwidth]{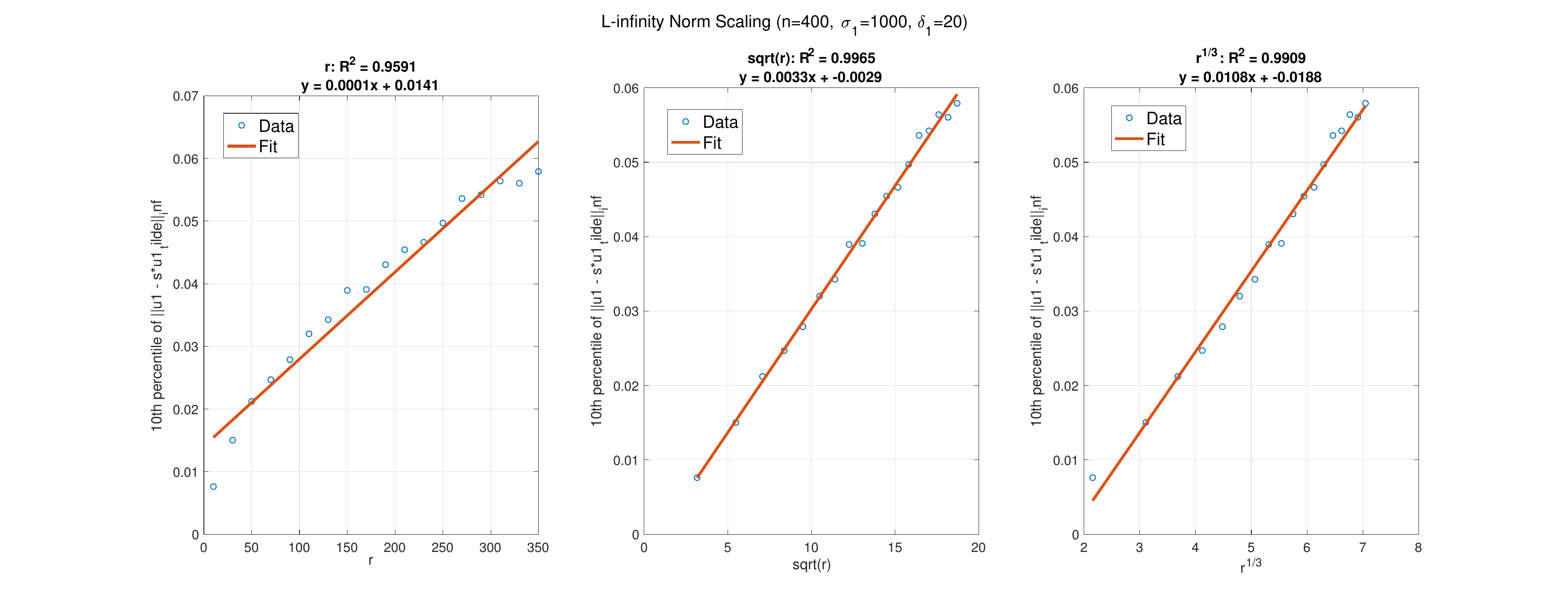}
\caption{\textbf{Empirical exploration of rank dependence in singular vector perturbation: $\ell_{\infty}$ norm.} For each rank $r \in \{10, 30, \ldots, 350\}$, we simulate 100 independent trials of a rank-$r$ signal matrix $A$ corrupted by Gaussian noise and compute the \emph{10th percentile} of $\min_{\mathsf{s}\in\{\pm 1\}} \|u_1 - \mathsf{s} \widetilde{u}_1\|_\infty$. We set $n = 400$. The singular values of $A$ are given as follows: $\sigma_1 = 1000$, $\sigma_2 = 980$ with $\delta_1=20$, and the remaining $r-2$ nonzero singular values decay linearly down to 950. The figure compares three scaling laws, $r$, $\sqrt{r}$, and $r^{1/3}$, via linear fits and $R^2$ values, showing that $\sqrt{r}$ provides the best empirical fit ($R^2 = 0.9965$).}
\label{fig:linf}
\end{figure}

\subsection{Singular subspace perturbations} We now analyze the optimality of our singular subspace perturbation bounds, with particular attention to their dependence on the signal rank $r$.

\medskip

\paragraph*{\textbf{The Frobenius and operator norm bounds}} We begin by examining bounds for unitarily invariant norms, specifically the Frobenius and operator norms. Under the assumptions of Theorem~2.1, we establish the following high-probability bounds for any $1 \leq k \le r$ (note that the Frobenius norm bound below follows directly from our proof and improves upon Theorem~2.1 by removing the $\sqrt{k_0}$ factor):
\begin{align}
\|\sin \angle (U_k, \widetilde{U}_k)\|_F &\lesssim \sqrt{k} \cdot \frac{\sqrt{r + \log(N+n)}}{\delta_k} + \frac{\|E\|_F}{\sigma_k}, \label{eq:Frobnorm} \\
\|\sin \angle (U_k, \widetilde{U}_k)\| &\lesssim \sqrt{k} \cdot \frac{\sqrt{r + \log(N+n)}}{\delta_k} + \frac{\|E\|}{\sigma_k}, \label{eq:opnorm}
\end{align}
which reduce to the classical Wedin's bounds when $k = r$. 

We first focus on the necessity of the second terms involving $\|E\|_F$ and $\|E\|$. The following lemma shows that these terms are not merely technical artifacts but reflect intrinsic lower bounds in the presence of random noise:
\begin{lemma}\label{lem:lowerbdFO}
Let $A$ and $E$ be $N\times n$ real matrices, where $A$ is deterministic with rank $r\ge 1$ and the entries of $E$ are i.i.d. sub-gaussian random variables with mean 0 and variance 1. Assume the $k$-th ($1\le k \le r$) largest singular value of $A$ satisfies $\sigma_k \ge 4(\sqrt N+ \sqrt n)$. Then, with probability at least $1 - C\exp(-cN) - C\exp(-cn)$, we have  
\begin{align}
&\max\{ \|\sin \angle ( U_{k}, \widetilde{ U}_{k})\|_F, \|\sin \angle ( V_{k}, \widetilde{ V}_{k})\|_F\} \ge \frac{ \sqrt{N+n-2r}}{3(1+\sqrt 2)} \sqrt{\sum_{i=1}^k \frac{1}{\sigma_i^2}},\label{eq:lowerbd-F}\\
&\max\{\|\sin \angle ( U_{k}, \widetilde{ U}_{k})\|, \|\sin \angle ( V_{k}, \widetilde{ V}_{k})\| \} \ge \frac{ \sqrt{N+n-2r}}{3(1+\sqrt 2)\sqrt{k}} \sqrt{\sum_{i=1}^k \frac{1}{\sigma_i^2}}\label{eq:lowerbd-O}
\end{align}
where $C,c>0$ are constants depending on the sub-gaussian norm of the entries of $E$.
\end{lemma}
Lemma~\ref{lem:lowerbdFO} shows the necessity of the second terms in \eqref{eq:Frobnorm} and \eqref{eq:opnorm}. For an \(n \times n\) matrix \(A\) of rank \(r \leq 0.8n\) (say), the lemma yields:
\[
\|\sin \angle (U_k, \widetilde{U}_k)\|_F \gtrsim \sqrt{n} \sqrt{\sum_{i=1}^k \frac{1}{\sigma_i^2}}, \quad \|\sin \angle (U_k, \widetilde{U}_k)\| \gtrsim \frac{\sqrt{n}}{\sqrt{k}} \sqrt{\sum_{i=1}^k \frac{1}{\sigma_i^2}}.
\]
When the singular values are relatively flat (\(\sigma_1 \asymp \sigma_k\)) and \(k \asymp n\), the Frobenius norm lower bound reduces to \(\frac{ \sqrt{k n}}{\sigma_k} \asymp \frac{\|E\|_F}{\sigma_k}\), since \(\|E\|_F\asymp n\). This matches \(\frac{\|E\|_F}{\sigma_k}\) up to a constant. Similarly, the operator norm lower bound simplifies to \(\frac{\sqrt{n}}{\sigma_k}$, matching the term $\frac{\|E\|}{\sigma_k}\). These lower bounds verify that the noise terms \(\frac{\|E\|_F}{\sigma_k}\) and \(\frac{\|E\|}{\sigma_k}\) are unavoidable, as they are achieved (up to constants) in this regime. The proof of Lemma~\ref{lem:lowerbdFO} is deferred to Section~\ref{sec:prooflower}.

The leading term $\sqrt{k} \cdot \frac{\sqrt{r + \log(N+n)}}{\delta_k}$ appears in both bounds, though we believe it is sharp only for the Frobenius norm. The use of the inequality $\|B\| \leq \|B\|_F$ in the proof for \eqref{eq:opnorm} suggests potential looseness in the operator norm bound.

To assess the $\sqrt{r}$-dependence in the Frobenius case, we conduct numerical experiments for $k = 10$ and $k = 50$. For each $k$, we construct a rank-$r$ matrix $A \in \mathbb{R}^{400 \times 400}$ with top $k$ singular values equal to $1000$, decaying linearly to $\sigma_r = 950$ to ensure $\delta_k \approx 50$. The noise matrix $E$ has i.i.d. standard Gaussian entries. Over 100 trials, we compute $\|\sin \angle(U_k, \widetilde{U}_k)\|_F$ and report the 10th percentile to capture typical behavior.

Figures~\ref{fig:fro-k10} and \ref{fig:fro-k50} support the plausibility of the $\sqrt{r}$-dependence. For $k = 10$, the $\sqrt{r}$-fit achieves $R^2 = 0.9528$, slightly better than $\sqrt{r + \log(2n)}$ ($R^2 = 0.9475$) and significantly better than $r$ ($R^2 =0.8610$). A log-log fit yields $r^{0.366}$ ($R^2 = 0.9825$). For $k = 50$, the corresponding values are $R^2 = 0.9596$ for $\sqrt{r}$, $R^2=0.9576$ for $\sqrt{r + \log(2n)}$, $R^2=0.9062$ for $r$, and a log-log scaling of $r^{0.663}$ ($R^2 = 0.9266$). For reference, $k = 30$ produces $r^{0.520}$ ($R^2 = 0.9447$). The consistently high $R^2$ values for $\sqrt{r}$ validate its role in the bound. 

In contrast, for the operator norm, the rank dependence in \eqref{eq:opnorm} appears to be loose. Numerical evidence suggests that the growth of $\|\sin \angle(U_k, \widetilde{U}_k)\|$ with $r$ and $k$ is much slower than the $\sqrt{k r}$ factor in our bound. This discrepancy stems from the use of the inequality $\|B\| \leq \|B\|_F$ in the proof. A sharper analysis, potentially directly bounding the largest singular value of the perturbation matrix, may yield improved bounds. We leave this improvement for future exploration.

\begin{figure}[ht]
\centering
\includegraphics[width=0.9\textwidth]{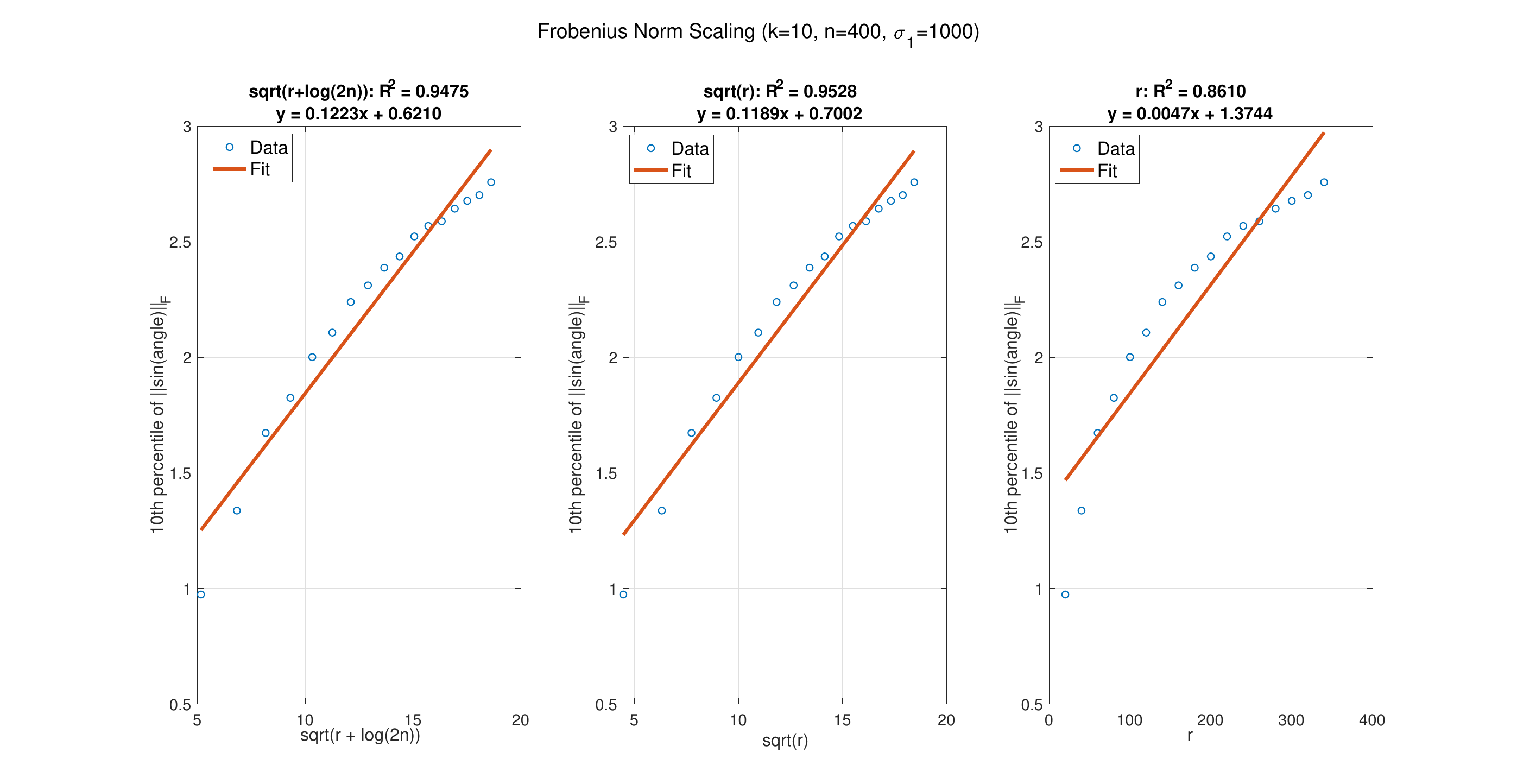}
\caption{\textbf{Empirical exploration of rank dependence in singular subspace perturbation: Frobenius norm.} For each rank $r \in \{20, 40, 60,\ldots, 350\}$, we simulate 100 independent trials of a rank-$r$ signal matrix $A$ corrupted by Gaussian noise and compute the \emph{10th percentile} of $\|\sin \angle(U_{10}, \widetilde{U}_{10})\|_F$. We set $n = 400$. The singular values of $A$ are given as follows: $\sigma_1 = \cdots = \sigma_{10} = 1000$, and the remaining $r-10$ nonzero singular values decay linearly down to 950. The figure compares three scaling laws, $\sqrt{r + \log(2n)}$, $\sqrt{r}$, and $r$, via linear fits and $R^2$ values, showing that $\sqrt{r}$ provides the best empirical fit ($R^2 = 0.9528$).}
\label{fig:fro-k10}
\end{figure}

\begin{figure}[ht]
\centering
\includegraphics[width=0.9\textwidth]{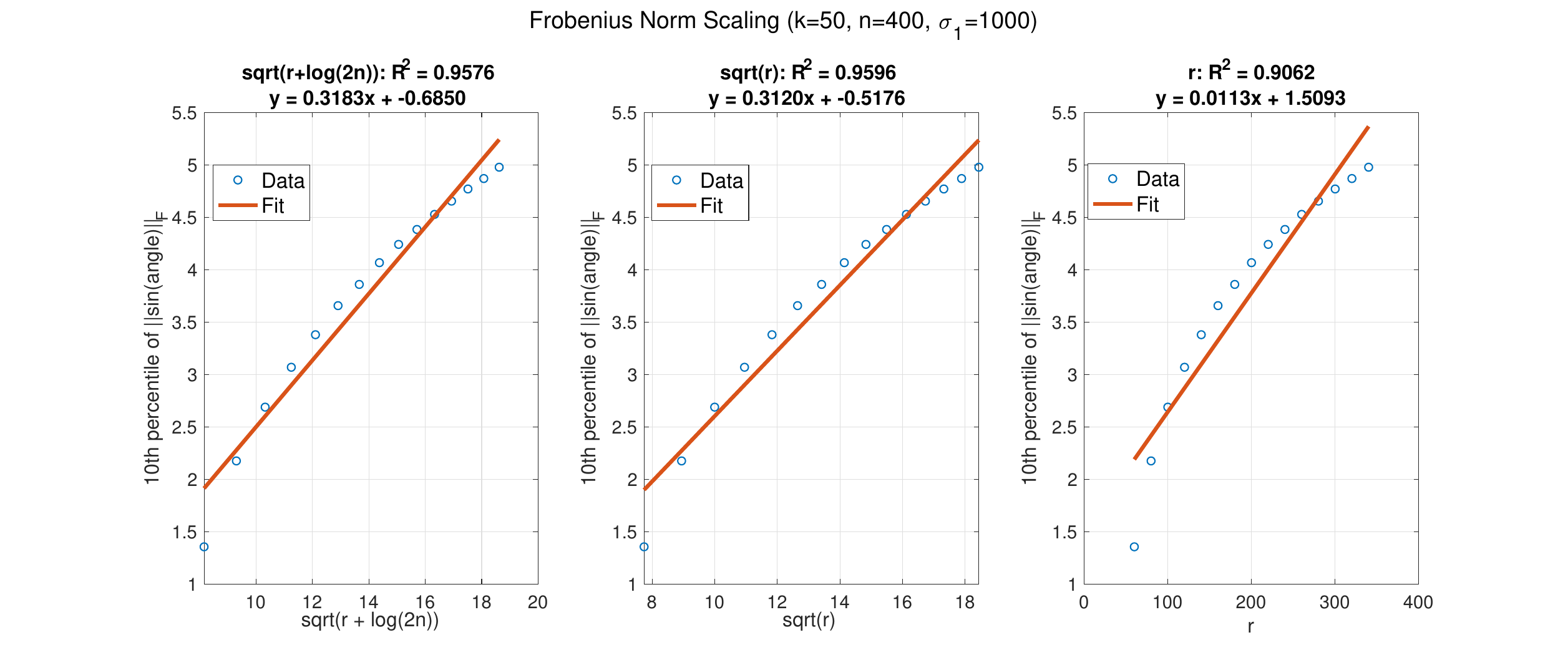}
\caption{\textbf{Empirical exploration of rank dependence in singular subspace perturbation: Frobenius norm.} For each rank $r \in \{20, 40, 60,\ldots, 350\}$, we simulate 100 independent trials of a rank-$r$ signal matrix $A$ corrupted by Gaussian noise and compute the \emph{10th percentile} of $\|\sin \angle(U_{50}, \widetilde{U}_{50})\|_F$. We set $n = 400$. The singular values of $A$ are given as follows: $\sigma_1 = \cdots = \sigma_{50} = 1000$, and the remaining $r-50$ nonzero singular values decay linearly down to 950. The figure compares three scaling laws, $\sqrt{r + \log(2n)}$, $\sqrt{r}$, and $r$, via linear fits and $R^2$ values, showing that $\sqrt{r}$ provides the best empirical fit ($R^2 = 0.9696$).}
\label{fig:fro-k50}
\end{figure}

\medskip

\paragraph*{\textbf{The \(\ell_{2,\infty}\) bound}} Corollary 2.9 establishes that for some $O\in \mathbb O^{k\times k}$,
\begin{align}\label{eq:ell-k}
 \| \widetilde{U}_{k} - U_k O \|_{2,\infty} \lesssim & \sqrt{k}\frac{\sqrt{r+\log(N+n)}}{\delta_{k}}  \|U\|_{2,\infty}+ \sqrt{k}\frac{\sqrt{r\log(N+n)}}{\sigma_k}  +  \frac{\|E\|^2}{\sigma_k^2} \|U_k\|_{2,\infty}.
\end{align}
Assume $U$ is incoherent, e.g., $\|U\|_{\max} \lesssim \frac{1}{\sqrt N}$, \eqref{eq:ell-k} simplifies to 
\begin{align}\label{eq:ell-inco}
\| \widetilde{U}_{k} - U_k O \|_{2,\infty} \lesssim \sqrt{k}\frac{\sqrt{r+\log(N+n)}}{\delta_k} \frac{\sqrt{r}}{\sqrt N} + \sqrt{k}\frac{\sqrt{r\log(N+n)}}{\sigma_k}.
\end{align}
In comparison, a naive bound
\[\| \widetilde{U}_{k} - U_k O \|_{2,\infty} \lesssim \sqrt{k}\frac{\sqrt{r+\log(N+n)}}{\delta_k} \sqrt{r} + \frac{\sqrt{r}\sqrt{N}}{\sigma_k}\]
follows from the inequality $\|B\|_{2,\infty} \le \sqrt{r} \|B\|$ for $B\in\mathbb R^{n\times r}$ and \eqref{eq:opnorm}. Compared to this, the refined bound in \eqref{eq:ell-k} gains a factor of \(\frac{1}{\sqrt{N}}\), assuming \(k = O(1)\) and ignoring log terms. This improvement is crucial in statistical applications--such as clustering--where sharper \(\ell_{2,\infty}\) control enables simple spectral methods to achieve optimal or near-optimal performance. 

We also explored the \(\ell_{2,\infty}\) norm bound in \eqref{eq:ell-inco} for \( k = 10 \). With \( n = 400 \), \(\sigma_1 = \cdots = \sigma_{10} = 1000\), \(\sigma_{11} = 980\), \(\sigma_r = 950\), and incoherent \( U \) (generated using Haar-measure), we computed the 10th percentile of \(\| \widetilde{U}_{10} - U_{10} O \|_{2,\infty}\) over 100 trials for \( r = 20, 40, \ldots, 350 \). Figure~\ref{fig:l2inf-k10} shows linear fits against \( r \), \(\sqrt{r}\), and \( r^{1/3} \), with \(\sqrt{r}\) providing the best fit (\( R^2 = 0.9993 \)). A log-log fit yields \( r^{0.540} \) (\( R^2 = 0.9987 \)), supporting the \(\sqrt{r}\)-dependence. These results indicate that the current $\sqrt{k}r$ factor in the first term of \eqref{eq:ell-inco} may be improved.

\begin{figure}[ht]
\centering
\includegraphics[width=0.9\textwidth]{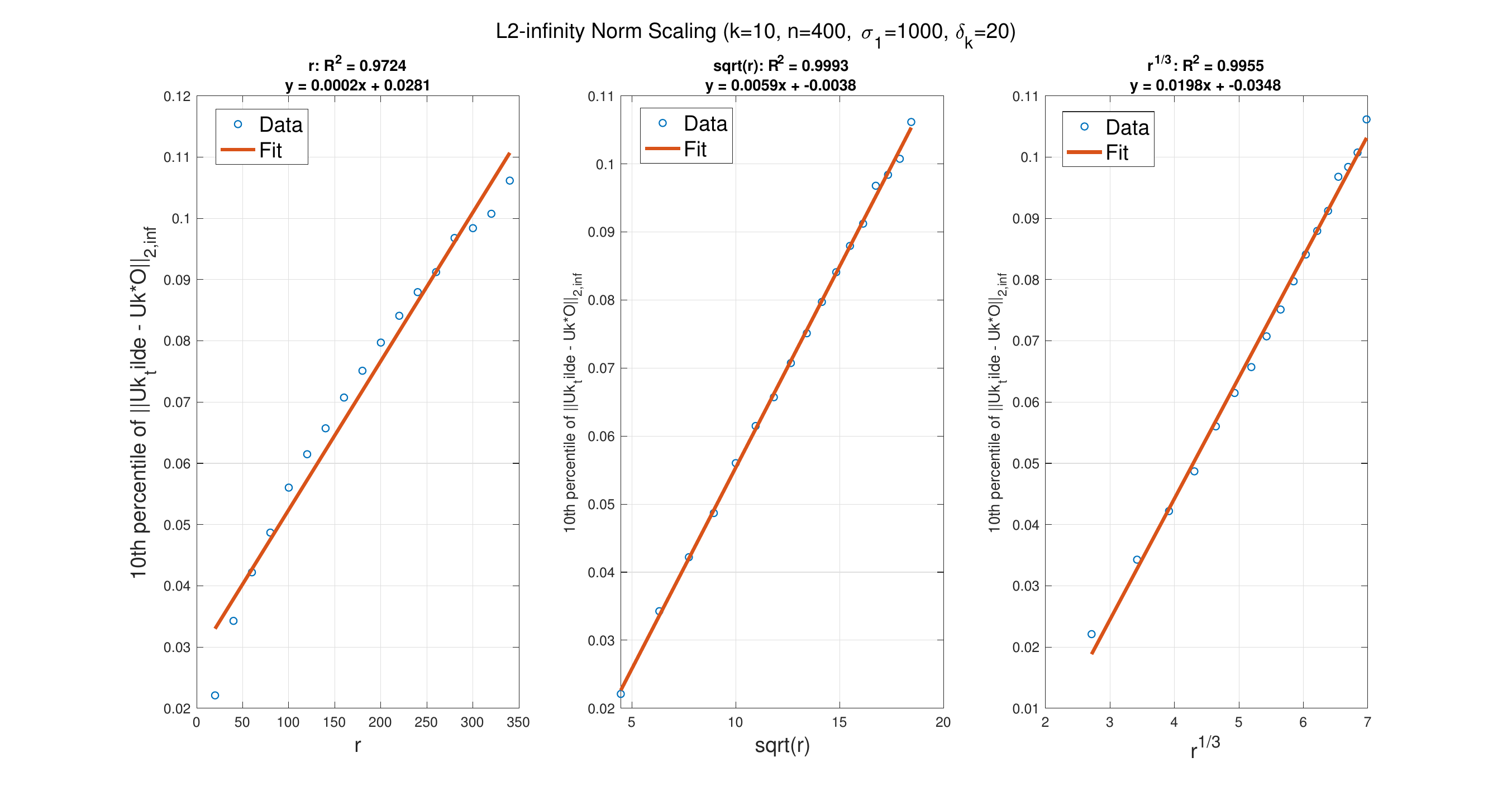}
\caption{\textbf{Empirical exploration of rank dependence in singular subspace perturbation: $\ell_{2,\infty}$ norm.} For each rank $r \in \{20, 40, \ldots, 350\}$, we simulate 100 independent trials of a rank-$r$ signal matrix $A$ corrupted by Gaussian noise and compute the \emph{10th percentile} of \(\| \widetilde{U}_{10} - U_{10} O \|_{2,\infty}\). We set $n = 400$. The singular values of $A$ are given as follows: $\sigma_1 = \cdots = \sigma_{10} = 1000$, $\sigma_{11} = 980$, and the remaining $r-11$ nonzero singular values decay linearly down to 950. The figure compares three scaling laws, $r$, $\sqrt{r}$, and $r^{1/3}$, via linear fits and $R^2$ values, showing that $\sqrt{r}$ provides the best empirical fit ($R^2 = 0.9993$).}
\label{fig:l2inf-k10}
\end{figure}

When \(k = r\), a case extensively studied in several recent works \cite{YCF21,YW24,CCFM21}, we have
\begin{align*}
\min_{O\in \mathbb O^{r\times r}} \| \widetilde{U}_{r} - U O \|_{2,\infty} \lesssim  \frac{r\sqrt{\log(N+n)}}{\sigma_r}  +  \frac{\|E\|^2}{\sigma_r^2} \|U\|_{2,\infty}.
\end{align*}
For context, Theorem 4.4 from \cite{CWC21} establishes $\ell_{2,\infty}$ bounds for the perturbation of the entire singular subspace under general noise matrices $E$.  Assuming square matrices of size $n \times n$, their result  yields:
\[
\min_{O \in \mathbb{O}^{r \times r}} \|\widetilde{U}_r - U O\|_{2,\infty} \lesssim \frac{\sqrt{r \log n}}{\sigma_r} + \frac{\sigma_1}{\sqrt{r} \sigma_r} \cdot \frac{\|E\|}{\sigma_r} \|U\|_{2,\infty},
\]
requiring $\sigma_r \ge c \sqrt{n \log n}$.
A sharper bound was recently established in \cite[Proposition 2]{YW24}:
\begin{align*}
\min_{O\in \mathbb O^{r\times r}} \| \widetilde{U}_{r} - U O \|_{2,\infty} \lesssim  \frac{\sqrt{r+\log(N+n)}}{\sigma_r}  +  \frac{\|E\|^2}{\sigma_r^2} \|U\|_{2,\infty}.
\end{align*}
This suggests that the \(\sqrt{k(r + \log(N+n))} \|U\|_{\infty}\) factor in \eqref{eq:ell-k} may be improvable. Additionally, the minimax lower bound from \cite[Theorem 2]{CLCPC21} shows that any estimator for $U$ under the $\ell_{2,\infty}$ norm incurs error at least $\min\{\frac{\sqrt{n}}{\sigma_r^2} +  \frac{1}{\sigma_r}, 1\}.$ If \(U\) is incoherent, e.g., \(\|U\|_{2,\infty} \lesssim \frac{\sqrt{r}}{\sqrt{n}}\), then the upper and lower bounds match up to a factor in terms of \(r\), indicating the near-optimality of our result.

\subsection{Summary}

Our numerical experiments offer compelling support for the rank dependence in several perturbation bounds derived in this paper.

\begin{itemize}
    \item \textit{Leading singular vector.} Our simulations confirm that $\sin \angle(u_1, \widetilde{u}_1)$ scales as $\sqrt{r}/\delta_1$ up to logarithmic factors. This validates the $\sqrt{r}$ dependence in the upper bound \eqref{eq:better-r-sup} and suggests that this bound is tight with high probability.

    For the $\ell_\infty$ norm bound, our empirical findings indicate that the leading-order factor \(\sqrt{r}\) in the upper bound may be unnecessary under incoherence. The $\ell_{\infty}$ bound contains a factor of $\|U\|_{2,\infty}$, which scales as $\sqrt{r/N}$ for incoherent matrices. This introduces an extra $\sqrt{r}$ factor that is not reflected in the observed lower bound behavior. A refined analysis of the role of $\|U\|_{2,\infty}$ could potentially improve the bound from $r$ to $\sqrt{r}$.
   \smallskip 
   
    \item \textit{Singular subspaces.} For general singular subspaces, the Frobenius norm bound \eqref{eq:Frobnorm} shows the expected $\sqrt{k}\sqrt{r + \log(N+n)}$ scaling, aligning with simulations. However, for the operator norm in \eqref{eq:opnorm}, the $\sqrt{k r}$ dependence appears loose, as numerical evidence suggests slower growth in $k\cdot r$. The current proof uses $\|B\| \leq \|B\|_F$, introducing a loose reduction that could be improved by a direct analysis of the operator norm.

For the $\ell_{2,\infty}$ norm bound, simulations support the \(\sqrt{r}\) scaling, but like the $\ell_\infty$ case, the leading term contains $\|U\|_{2,\infty}$, contributing an extra $\sqrt{r}$ under incoherence. Refining the role of $\|U\|_{2,\infty}$ under incoherence could improve the rank dependence in \eqref{eq:ell-inco}'s first term. 
\end{itemize}

Our numerical analysis suggests that the dependence on signal rank $r$ is empirically tight for Frobenius norm and angular perturbation bounds. However, the experiments identify potential room for improvement in both operator norm and $\ell_{2,\infty}$ norm bounds. We leave these for future exploration.

\section{Preliminary}\label{sec:basic1}
\subsection{Matrix norms}\label{sec:matrixnorms}
Consider an $N \times n$ matrix $A=(a_{ij})$ with singular values $\sigma_1\ge \cdots \ge \sigma_{\min\{N,n\}}\ge 0$. Let $\vvvert A\vvvert $ be a norm of $A$ of certain interest. 

The first type of matrix norms are the unitarily invariant norms. The norm $\vvvert \cdot \vvvert $ on $\mathbb R^{N\times n}$ is said to be \emph{unitarily invariant} if $\vvvert A\vvvert  = \vvvert UAV\vvvert $ for all orthogonal matrices $U \in \mathbb R^{N\times N}$ and $V \in \mathbb R^{n\times n}$. There is an intimate connection between the unitarily invariant norms and the singular values of matrices via the symmetric gauge functions (see \cite[Section IV]{Bhatia}). 
\begin{definition}[Symmetric gauge function]
A function $f: \mathbb R^n \to \mathbb R$ is a \emph{symmetric gauge function} if 
\begin{enumerate}[(i)]
\item $f$ is a norm,
\item $f(P x) = f(x)$ for all $x\in \mathbb R^n$ and $P\in S_n$ (the set of permutation matrices),
\item $f(\epsilon_1 x_1, \cdots, \epsilon_n x_n) = f( x_1, \cdots, x_n)$ if $\epsilon_j = \pm 1$.
\end{enumerate}
\end{definition}
We say the symmetric gauge function $f$ is \emph{normalized} if
$f(1,0,\cdots,0)=1.$
\begin{theorem}[Theorem IV.2.2 from \cite{Bhatia}]\label{thm:char}A norm $\vvvert \cdot \vvvert $ on $\mathbb{R}^{N \times n}$ is unitarily invariant if and only if $\vvvert A \vvvert  = f(\sigma_1, \dots, \sigma_{\min\{N,n\}})$ for all $A \in \mathbb{R}^{N \times n}$ for some symmetric gauge function $f$, where $\sigma_1, \dots, \sigma_{\min\{N,n\}}$ are the singular values of $A$.
\end{theorem}
If $\vvvert \cdot \vvvert $ is a unitarily invariant norm on $\mathbb R^{N\times n}$ and $f$ is its associated symmetric gauge function, then for $k,s\le \min\{N, n\}$, a unitarily invariant norm on $\mathbb R^{k\times s}$ can be defined by $\vvvert A \vvvert = f(\sigma_1, \cdots, \sigma_{\min\{k,s\}},0,\cdots,0)$, where $\sigma_i$'s are the singular values of $A\in \mathbb R^{k\times s}$. As a result, a family of matrix norms can be defined based on $f$ that can be applied to matrices of varying dimensions. As such, we will not explicitly mention the dimensions of the unitarily invariant norm $\vvvert \cdot \vvvert $.

Moreover, a unitarily invariant norm $\vvvert \cdot \vvvert $ is said to be \emph{normalized} if its associated symmetric gauge function $f$ is {normalized}. Consequently, a normalized unitarily invariant norm always satisfies $\vvvert \diag(1,0,\cdots,0) \vvvert =1$.

Another characterization of the unitarily invariant norm is given by the symmetric property. 
\begin{theorem}[Proposition IV.2.4 from \cite{Bhatia}]\label{thm:char1}
A norm $\vvvert \cdot \vvvert $ on $\mathbb{R}^{N \times n}$ is unitarily invariant if and only if the norm is symmetric, that is, 
$$\vvvert ABC \vvvert \le \|A\| \cdot \vvvert B \vvvert \cdot \|C\|$$ 
for every $A \in \mathbb{R}^{N \times N}, B \in \mathbb{R}^{N \times n}$ and $C \in \mathbb{R}^{n \times n}$.
\end{theorem}

A wide range of matrix norms that are commonly used are part of the class of unitarily invariant norms. For instance, for $p\in [1, \infty]$, the Schatten $p$-norm of $A$ is defined by $$\|A\|_p = \left( \sum_{i=1}^{\min\{N,n\}} \sigma_i^p \right)^{1/p}.$$ In particular, the case $p = 2$ yields the Frobenius norm $\|A\|_F=\sqrt{\sum_{i,j} a_{ij}^2}$. The case $p=\infty$ yields the operator norm $\|A\|=\sigma_1$. The case $p=1$ yields the nuclear (or trace) norm $$\|A\|_* =\|A\|_1= \sum_{i=1}^{\min\{N,n\}}  \sigma_i = \tr\left(\sqrt{A A^\T} \right).$$ Note that
\begin{equation}\label{eq:sp}
\|A\|_p^2 = \|A^\T A \|_{p/2} \quad\text{for } p \ge 2.
\end{equation}

Another class of unitarily invariant norms is the Ky Fan $k$-norm $$\|A\|_{(k)} = \sum_{i=1}^k \sigma_i,\quad 1\le k \le \min\{N,n\}.$$  Hence, $\|A\|_{(1)} = \|A\|$ and $\|A\|_{(\min\{N,n\})} = \|A\|_*$. A highly significant result known as the Fan dominance theorem is connected to the Ky Fan norm: 
\begin{theorem}[Theorem IV.2.2 from \cite{Bhatia}]\label{thm:fan}
Let $A,B$ be two $n\times n$ matrices. If $$\|A\|_{(k)} \le \|B\|_{(k)}\quad \text{for } k=1,2,\cdots, n,$$ then $\vvvert A\vvvert \le \vvvert B\vvvert$ for all unitarily invariant norms. 
\end{theorem}
If $\vvvert \cdot \vvvert$ is also normalized, then a direct implication of Theorem \ref{thm:fan} is that
\begin{align}
&\|A\| \le \vvvert A\vvvert \le \|A\|_*, \label{eq:equivnorms}\\
&\sigma_{\min}(A) \vvvert B\vvvert \le \vvvert AB\vvvert \le \|A\| \vvvert B\vvvert,\nonumber\\
&\sigma_{\min}(A) \vvvert B\vvvert \le \vvvert BA\vvvert \le \|A\| \vvvert B\vvvert.\nonumber
\end{align}
It also follows from Theorem \ref{thm:char1} and \eqref{eq:equivnorms} that $\vvvert AB\vvvert \le \vvvert A\vvvert \vvvert B\vvvert.$

We also consider the following norms, which do not belong to the class of unitarily invariant norms. Denote $A_{i,\cdot}$'s the rows of $A \in \mathbb{R}^{N \times n}$. 
The $\ell_{2,\infty}$ norm of $A$ is $$\|A\|_{2,\infty} =\max_{i} \|A_{i,\cdot}\|_2=\max_{1\le i\le N} \|e_i^\T A\|.$$ Finally, denote $\|A\|_{\max} = \max_{i,j} |a_{ij}|$. Note that $\|\cdot \|_{\max}$ is not 
sub-multiplicative.

\subsection{Distance between subspaces}\label{sec:distance} Using the angles between subspaces and using the orthogonal projections to describe their separation are two popular approaches for measuring the distance between subspaces. These two methods are essentially equivalent when it comes to any unitarily invariant norm $\vvvert \cdot \vvvert$. We start with some basic notions. 

If $U$ and $V$ are two subspaces of the same dimension $r$, then one could define the principal angles $0\le \theta_1\le \cdots \le \theta_r\le {\pi}/{2}$  between them recursively: 
\begin{equation*}
\cos (\theta_i) = \max_{u\in U, v\in V} u^T v =u_i^T v_i, \qquad \|u\|=\|v\|=1
\end{equation*}
subject to the constraint
$$ u_i^T u_l=0, \quad v_i^T v_l =0\quad \text{for } l=1,\ldots,i-1.$$
Denote  $\angle (U,V) := \diag(\theta_1, \cdots , \theta_r)$. Further, let 
\begin{align*}
&\sin\angle (U,V)  := \diag(\sin\theta_1, \cdots , \sin\theta_r),\\
&\cos\angle (U,V)  := \diag(\cos\theta_1, \cdots , \cos\theta_r).
\end{align*}
With abuse of notation, we also let $U=(u_1,\cdots,u_r)$ and $V=(v_1,\cdots,v_r)$ be matrices of size $n\times r$ whose columns are orthonormal bases of subspaces $U$ and $V$ respectively. Then $P_U=U U^T$ (resp. $P_V = V V^T$) is the orthogonal projection matrix onto the subspace $U$ (resp. $V$). For a subspace $W$, denote its complement by $W^\perp$.

The following facts are collected from \cite[Exercises VII. 1. 9 -- 1.11]{Bhatia}.
\begin{proposition} \label{prop:spaces}Let $U, V, P_U, P_V, \sin\angle (U,V) , \cos\angle (U,V) $ be as above.
\begin{enumerate}[(i)]
\item The nonzero singular values of $P_U P_V$ are the same as the nonzero singular values of $U^T V$.
\item The singular values of $P_U P_V$ are $\cos\theta_1, \cdots , \cos\theta_r$.
The nonzero singular values of $P_{U^\perp} P_V$ are the nonzero values of $\sin\theta_1, \cdots , \sin\theta_r$. 
\item The nonzero singular values of $P_U-P_V$ are the nonzero singular values $P_{U^\perp} P_V$, each counted twice; i.e., these are the nonzero numbers in $$\sin\theta_1, \sin\theta_1, \sin\theta_2, \sin\theta_2,\cdots, \sin\theta_r, \sin\theta_r.$$ 
\end{enumerate}
\end{proposition}

For any unitarily invariant norm $\vvvert \cdot \vvvert$, by Proposition \ref{prop:spaces},  we observe 
\begin{align}\label{eq:sineq}
\vvvert \sin\angle (U,V)  \vvvert=\vvvert P_{U^\perp} P_V\vvvert=\vvvert P_{V^\perp} P_U\vvvert
\end{align}
and $$\vvvert P_U- P_V\vvvert = \vvvert P_{U^\perp} P_V \oplus P_U P_{V^\perp} \vvvert.$$ This suggests the (near) equivalence of $\vvvert \sin\angle (U,V)  \vvvert$ and $\vvvert P_U- P_V\vvvert$.  For instance, for the Schatten $p$-norm, we have
\begin{equation*}
\|P_U- P_V\|_{p} = 2^{\frac{1}{p}}\|P_{U^\perp} P_V\|_p= 2^{\frac{1}{p}} \| \sin\angle (U,V) \|_{p}.
\end{equation*}
For the Ky Fan $k$-norm, denote $\|A\|_{(0)}=0$. Then
\begin{equation*}
\|P_U- P_V\|_{(k)} = 
\begin{cases}
\|P_{U^\perp} P_V\|_{(\frac{k-1}{2})} + \|P_{U^\perp} P_V\|_{(\frac{k+1}{2})}, &\quad \text{if $k$ is odd};\\
 2 \|P_{U^\perp} P_V\|_{(\frac{k}{2})}, &\quad \text{if $k$ is even}.
\end{cases}
\end{equation*}

Another method that is commonly used to quantify the distance between $U$ and $V$ is to use  
$$\min_{O\in \mathbb{O}^{r\times r}} \vvvert UO- V\vvvert.$$
It is shown in \cite[Lemma 2.6]{CCFM21} that the above distance is (near) equivalent to the $\vvvert P_U- P_V\vvvert$ for the Frobenius norm and the operator norm. In fact, we can demonstrate that these distances are (near) equivalent when considering the Schatten-$p$ norm for any $p\in [2,\infty]:$ 
\begin{equation}\label{eq:equiv}
 \|\sin\angle (U,V) \|_p\le \min_{O\in \mathbb{O}^{r\times r}} \| UO- V\|_p \le \sqrt{2}\|\sin\angle (U,V) \|_p.
\end{equation}
The proof of \eqref{eq:equiv} is given in Appendix \ref{app:proofofnorm}.

More generally, for any unitarily invariant norm $\vvvert \cdot \vvvert $ on $\mathbb R^{N\times n}$, 
we have 
\begin{equation}\label{eq:equivup}
\min_{O\in \mathbb{O}^{r\times r}} \vvvert UO- V\vvvert \le \sqrt{2}\vvvert \sin\angle (U,V)  \vvvert.
\end{equation}
The proof of \eqref{eq:equivup} is given in Appendix \ref{app:proofequi}. 

In certain applications, the primary focus is to compare the matrices $U=(u_1,\cdots,u_r)$ and $V=(v_1,\cdots,v_r)$ of size $n\times r$ with respect to specific directions. 
According to Proposition \ref{prop:spaces},  the SVD of $U^\T V$ is given by 
\begin{align}\label{eq:svdUV}
U^\T V = O_1 \cos\angle(U,V) O_2^\T,
\end{align}
where $O_1, O_2 \in  \mathbb{O}^{r\times r}$. Denote $O:=O_1 O_2^\T\in  \mathbb{O}^{r\times r}$. We highlight the following deterministic result, the proof of which can be found in Appendix \ref{app:2inf}.
\begin{proposition}\label{prop:2inf}Let $x$ be any unit vector in $\mathbb R^{n}$ and $y$ be any unit vector in $\mathbb R^{r}$. We have
\begin{align*}
\|x^\T (V- UO) \| \le \|x^\T (V - P_U V)\| + \|x^\T U\| \|\sin\angle(U, V)\|^2
\end{align*}
and
\begin{align*}
\left|x^\T (V- UO) y \right| \le \left|x^\T (V- P_U V) y \right| + \| x^\T U\| \|\sin\angle(U, V)\|^2.
\end{align*}
In particular,
\begin{align*}
\|V- UO\|_{2,\infty} \le \|V - P_U V\|_{2,\infty} + \|U\|_{2,\infty} \|\sin\angle(U, V)\|^2.
\end{align*}
\end{proposition}
Finally, it can be verified from the definition that for any orthogonal matrix $O$,
$$\|V- UO\|_{2,\infty}= \|VO^\T- U\|_{2,\infty}.$$

\section{Proofs of Theorems 2.3, 2.10, 2.11 and 2.13  and Lemma \ref{lem:lowerbdFO}} \label{sec:proofmain}

In the proofs below, we always work on the event where $\|E\|\le 2(\sqrt{N}+\sqrt n)$; Lemma 6.7 shows this event holds with probability at least $1 - 2 e^{-(\sqrt{N}+\sqrt{n})^2/2}\ge 1 - 2(N+n)^{-16(K+7)}$ since $(\sqrt N+\sqrt n)^2\ge 32(K+7)\log(N+n)$ by assumption.

Denote
$$I:=\llbracket k, s\rrbracket \cup \llbracket r+k, r+s\rrbracket$$
and
 \begin{align*} 
 J:=\llbracket 1, 2r \rrbracket \setminus I = \llbracket 1, k-1\rrbracket\cup \llbracket s+1, r\rrbracket  \cup \llbracket r+1, r+k-1\rrbracket \cup \llbracket r+s+1, 2r\rrbracket .
 \end{align*}
We first obtain an identity for the eigenvector $\widetilde{\mathbf u}_i$. For each $i\in I$, by Weyl's inequality, $|\widetilde \lambda_i| \ge \widetilde \sigma_{r_0}\ge \sigma_{r_0}-\|E\|>\|E\|=\|\mathcal E\|$ by supposition on $\sigma_{r_0}$ and thus $G(\widetilde \lambda_i)$ and $\Phi(\widetilde \lambda_i)$ are well-defined. As $(\mathcal A + \mathcal E)\widetilde{\mathbf u}_i = \widetilde \lambda_i \widetilde{\mathbf u}_i$, we solve for $\widetilde{\mathbf u}_i$  to obtain 
\begin{align*}
	\widetilde{\mathbf u}_i =(\widetilde \lambda_i I - \mathcal E)^{-1} \mathcal A \widetilde{\mathbf u}_i=G(\widetilde \lambda_i) \mathcal A \widetilde{\mathbf u}_i.
\end{align*}  

In fact, we can approximate $G(\widetilde \lambda_i)$ with a simpler random matrix, denoted as $\Pi(z)$ below. Depending on the magnitude of $\widetilde \lambda_i$ under consideration (the specifics of which will become clear in subsequent context), we choose $\Pi(\widetilde \lambda_i)$ to be either $\Phi(\widetilde \lambda_i)$ if $|\widetilde \lambda_i|$ is relatively small and $\frac{1}{\widetilde \lambda_i} I_{N+n} + \frac{1}{\widetilde \lambda_i^2} \mathcal E$ or even simply $\frac{1}{\widetilde \lambda_i} I_{N+n}$ if $|\widetilde \lambda_i|$ is sufficiently large. 
Furthermore, denote 
\begin{align}\label{eq:Xi}
\Xi(\widetilde \lambda_i):=G(\widetilde \lambda_i)-\Pi(\widetilde \lambda_i).
\end{align}
Hence, we rewrite
\begin{align}\label{eq:u}
\widetilde{\mathbf u}_i=\Pi(\widetilde \lambda_i) \mathcal A \widetilde{\mathbf u}_i + \Xi(\widetilde \lambda_i) \mathcal A \widetilde{\mathbf u}_i
\end{align}
This decomposition of $\widetilde{\mathbf u}_i$ is critical in facilitating the extraction of its desired property information. Note that Lemma 6.3 and Lemma 6.5 provides precise control on the size of $\Xi(\widetilde \lambda_i)$.

For $J \subset \llbracket 1, 2r\rrbracket$, we introduce the notation $\mathcal U_J$ to denote the $(N+n) \times |J|$ matrix formed from $\mathcal U$ by removing the columns containing ${\mathbf u}_i$ for $i \not\in J$.  Similarly, $\mathcal{D}_J$ will denote the $|J| \times |J|$ matrix formed from $\mathcal D$ by removing the rows and columns containing ${\lambda}_i$ for $i \not\in J$. Let $I:= \llbracket 1, 2r\rrbracket \setminus J$. In this way, we can decompose $A$ as
\begin{equation} \label{eq:Adecomp}
	\mathcal A = \mathcal U \mathcal D \mathcal U^T = {\mathcal U}_J {\mathcal D}_J {\mathcal U}_J^\T + {\mathcal U}_{I} {\mathcal D}_{I} {\mathcal U}_{I}^\T.   
\end{equation} 
Let $\mathcal P_{J}$ be the orthogonal projection onto the subspace $\Span\{\mathbf u_{k}: {k\in J}\}$. Clearly, $\mathcal P_J = \mathcal U_J \mathcal U_J^\T.$ If $J={\llbracket 1, k\rrbracket}$, we sometimes simply write $\mathcal U_{k}$ for $\mathcal U_J$ and $\mathcal P_k$ for $\mathcal P_J$.
Analogous notations $ \widetilde{\mathcal U}_J, \widetilde{P}_J, \widetilde{\mathcal D}_J$ are also defined for $\widetilde{\mathcal A}$. We also use ${\mathcal P}_{\llbracket 2r+1, N+n\rrbracket}={\mathcal U}_{\llbracket 2r+1, N+n\rrbracket} {\mathcal U}_{\llbracket 2r+1, N+n\rrbracket}^\T$ to denote the orthogonal projection onto the null space of $\mathcal A$.

Now we proceed to the proofs of the main results. 


\subsection{Proof of Theorem 2.3}

From \eqref{eq:sineq}, we start by observing 
\begin{align*}
\vvvert \sin \angle (\mathcal U_{I}, \widetilde{\mathcal U}_{I}) \vvvert =\vvvert {\mathcal P}_{I^c}  \widetilde{\mathcal P}_I \vvvert \le   \vvvert {\mathcal P}_{\llbracket 2r+1, N+n\rrbracket}  \widetilde{\mathcal P}_I \vvvert +  \vvvert {\mathcal P}_J  \widetilde{\mathcal P}_I \vvvert.
\end{align*}
We bound the two terms $ \vvvert {\mathcal P}_{\llbracket 2r+1, N+n\rrbracket}  \widetilde{\mathcal P}_I \vvvert$ and  $\vvvert {\mathcal P}_J  \widetilde{\mathcal P}_I \vvvert$ respectively. 

\begin{lemma}With probability 1, 
\begin{equation}\label{eq:nullbd}
 \vvvert {\mathcal P}_{\llbracket 2r+1, N+n\rrbracket}  \widetilde{\mathcal P}_I \vvvert \le 2\frac{\vvvert  \mathcal P_{\llbracket 2r+1, N+n\rrbracket} \mathcal E \widetilde{\mathcal P}_I \vvvert}{\sigma_{s}}.
\end{equation}
\end{lemma}
\begin{proof}By Proposition  \ref{prop:spaces} (i) and Theorem \ref{thm:char},
$$ \vvvert {\mathcal P}_{\llbracket 2r+1, N+n\rrbracket}  \widetilde{\mathcal P}_I \vvvert = \vvvert \mathcal U_{\llbracket 2r+1, N+n\rrbracket}^\T  \widetilde{\mathcal U}_I \vvvert.$$
From the spectral decomposition of $\widetilde{\mathcal A}$, we have
\begin{align*}
(\mathcal A + \mathcal E) \widetilde{\mathcal U}_I = \widetilde{\mathcal U}_I \widetilde{\mathcal D}_I.
\end{align*}
Multiplying by $\mathcal U_{\llbracket 2r+1, N+n\rrbracket}^\T$ on the left of the equation above, we further have
\begin{align}\label{eq:061125}
\mathcal U_{\llbracket 2r+1, N+n\rrbracket}^\T \mathcal E \widetilde{\mathcal U}_{I} = \mathcal U_{\llbracket 2r+1, N+n\rrbracket}^T  \widetilde{\mathcal U}_I \widetilde{\mathcal D}_I.
\end{align}
As $\sigma_{r_0} \ge b\|E\| \ge 2 \| E \|$ by supposition, Weyl's inequality 
implies that 
\begin{equation} \label{eq:hatlambdabnd}
	\widetilde{\sigma}_{i} \geq \sigma_{i} - \|E \| \geq \frac{1}{2} \sigma_{i} 
\end{equation}  
for $k\le i \le s$. 
Hence, $\widetilde{\mathcal D}_I$ is invertible since $|\widetilde\lambda_i| =\widetilde \sigma_i >0$ for  $i\in \llbracket k, s\rrbracket$ and $|\widetilde\lambda_i| = \widetilde \sigma_{i-r}>0$ for  $i\in \llbracket r+k, r+s\rrbracket$. It follows from Theorem \ref{thm:char1} that 
\begin{align*}
\vvvert \mathcal U_{\llbracket 2r+1, N+n\rrbracket}^\T  \widetilde{\mathcal U}_I \vvvert &= \vvvert \mathcal U_{\llbracket 2r+1, N+n\rrbracket}^\T \mathcal E \widetilde{\mathcal U}_I \widetilde{\mathcal D}_I^{-1}\vvvert\\
 &\le \vvvert \mathcal U_{\llbracket 2r+1, N+n\rrbracket}^\T \mathcal E \widetilde{\mathcal U}_I\vvvert  \|\widetilde{\mathcal D}_I^{-1}\| \\
&= \frac{\vvvert \mathcal P_{\llbracket 2r+1, N+n\rrbracket} \mathcal E \widetilde{\mathcal P}_I \vvvert}{\widetilde{\sigma}_{s}}.
\end{align*}
The last equation above follows from the fact that for $U$ with orthonormal columns, $U^\T B$ and $UU^\T B$ share the same singular values. 

Thus by another application of \eqref{eq:hatlambdabnd}, we get
\begin{align}\label{eq:conclusionbd1}
\vvvert {\mathcal P}_{\llbracket 2r+1, N+n\rrbracket}  \widetilde{\mathcal P}_I \vvvert \le \frac{\vvvert  \mathcal P_{\llbracket 2r+1, N+n\rrbracket} \mathcal E \widetilde{\mathcal P}_I \vvvert}{\widetilde{\sigma}_{s}} \le 2 \frac{\vvvert  \mathcal P_{\llbracket 2r+1, N+n\rrbracket} \mathcal E \widetilde{\mathcal P}_I \vvvert}{{\sigma}_{s}}.
\end{align}
as desired.  
\end{proof}

It remains to bound $ \vvvert {\mathcal P}_J  \widetilde{\mathcal P}_I \vvvert =  \vvvert \mathcal{U}_J^\T  \widetilde{\mathcal U}_I \vvvert$. We apply \eqref{eq:equivnorms} to obtain
\begin{align}\label{eq:2ndtermub}
\vvvert {\mathcal P}_J  \widetilde{\mathcal P}_I \vvvert &=  \vvvert \mathcal{U}_J^\T  \widetilde{\mathcal U}_I \vvvert \le \|\mathcal{U}_J^\T  \widetilde{\mathcal U}_I\|_* \nonumber\\
&\le \sqrt{\rank(\mathcal{U}_J^\T  \widetilde{\mathcal U}_I)}\cdot \|\mathcal{U}_J^\T  \widetilde{\mathcal U}_I\|_F \nonumber\\
&\le 2\sqrt{\min\{s-k+1,r-s+k-1\}} \|\mathcal{U}_J^\T  \widetilde{\mathcal U}_I\|_F\nonumber \\
&= 2\sqrt{\min\{s-k+1,r-s+k-1\}} \sqrt{\sum_{i\in I}  \| \mathcal U_J^\T \widetilde{\mathbf u}_i \|^2}.
\end{align}
In particular, for the operator norm, when $|J|\neq 0$, we simply have
\begin{align}\label{eq:2ndterm-operator}
\|{\mathcal P}_J  \widetilde{\mathcal P}_I\| = \|\mathcal{U}_J^\T  \widetilde{\mathcal U}_I\|\le \|\mathcal{U}_J^\T  \widetilde{\mathcal U}_I\|_F = \sqrt{\sum_{i\in I}  \| \mathcal U_J^\T \widetilde{\mathbf u}_i \|^2}.
\end{align}

It remains to bound $\| \mathcal U_J^\T \widetilde{\mathbf u}_i \|$ for each $i\in I$. 
We have the following estimates
\begin{lemma}\label{lem:projsv}For every $i\in I$, 
\begin{align}\label{lem:bdeach-new}
 \| \mathcal U_J^\T \widetilde{\mathbf u}_i \| \le 3 \frac{(b+2)^2}{(b-1)^2} \frac{\eta}{\min\{\delta_{k-1}, \delta_s\}}
\end{align}
with probability at least $1-20(N+n)^{-K}$.
\end{lemma}
The proof of Lemma \ref{lem:projsv} closely mirrors the strategy employed in Lemma 20 from \cite{OVW22}. For the sake of completeness, we have included the proof of Lemma \ref{lem:projsv} in Appendix \ref{app:bdeach}.
 
It follows from \eqref{eq:2ndtermub} and Lemma \ref{lem:projsv} that
\begin{align*}
&\vvvert {\mathcal P}_J  \widetilde{\mathcal P}_I \vvvert \le 6\sqrt{2} \frac{(b+2)^2}{(b-1)^2}\sqrt{\min\{s-k+1,r-s+k-1\}} \frac{ \eta\sqrt{s-k+1}}{\min\{\delta_{k-1}, \delta_s\}}
\end{align*}
with probability at least $1 - 20(N+n)^{-K}$.

Consequently, we arrive at 
\begin{align}\label{eq:conclude}
&\vvvert \sin \angle (\mathcal U_{I}, \widetilde{\mathcal U}_{I}) \vvvert =\vvvert {\mathcal P}_{I^c}  \widetilde{\mathcal P}_I \vvvert \nonumber\\
& \leq  6\sqrt{2} \frac{(b+2)^2}{(b-1)^2}\sqrt{\min\{s-k+1,r-s+k-1\}} \frac{ \eta\sqrt{s-k+1}}{\min\{\delta_{k-1}, \delta_s\}} \nonumber\\
&\quad + 2 \frac{ \vvvert \mathcal P_{\llbracket 2r+1, N+n\rrbracket} \mathcal E \widetilde{\mathcal P}_I\vvvert}{\sigma_{s} }.
\end{align}
In particular, the following bound holds for the operator norm:
\begin{align}\label{eq:conclude-op}
\| \sin \angle (\mathcal U_{I}, \widetilde{\mathcal U}_{I}) \| \le 3\sqrt{2} \frac{(b+2)^2}{(b-1)^2} \mathbf{1}_{\{|J| \neq 0 \}} \frac{ \eta\sqrt{s-k+1}}{\min\{\delta_{k-1}, \delta_s\}} + 2\frac{\|E\|}{\sigma_{s}}.
\end{align}

Let $\gamma_1,\cdots,\gamma_{2(s-k+1)}$ be the principal angles of  the subspaces $\mathcal U_{I}$ and $\widetilde{\mathcal U}_{I}$. Denote $\alpha_1,\cdots,\alpha_{s-k+1}$ (resp. $\beta_1,\cdots,\beta_{s-k+1}$) the principal angles of $ U_{k,s}, \widetilde{ U}_{k,s}$ (resp. $ V_{k,s}, \widetilde{ V}_{k,s}$). From the proof of \cite[Proposition 8]{OVW22}, we see that the singular values of $\mathcal U_{I}^\T \widetilde{\mathcal U}_{I}$, given by $\cos\gamma_1,\cdots,\cos\gamma_{2(s-k+1)}$, are exactly
$$\cos\alpha_1,\cdots,\cos\alpha_{s-k+1}, \cos\beta_1,\cdots,\cos\beta_{s-k+1}.$$
Hence, $$\vvvert \sin \angle (\mathcal U_{I}, \widetilde{\mathcal U}_{I}) \vvvert = \vvvert \sin \angle ( U_{k,s}, \widetilde{ U}_{k,s}) \oplus \sin \angle ( V_{k,s}, \widetilde{ V}_{k,s}) \vvvert.$$

Note that by the definitions of $\mathcal E, \widetilde{\mathcal P}_I$ and $\mathcal P_{\llbracket 2r+1, N+n\rrbracket}$,
\begin{align*}
\mathcal P_{\llbracket 2r+1, N+n\rrbracket} \mathcal E \widetilde{\mathcal P}_I =\begin{pmatrix}
0 & P_{U^\perp} E P_{\widetilde V_{k,s}}\\
P_{V^\perp} E^\T P_{\widetilde U_{k,s}} & 0
\end{pmatrix}.
\end{align*}
Using the unitary equivalence, we find 
$$\vvvert  \mathcal P_{\llbracket 2r+1, N+n\rrbracket} \mathcal E \widetilde{\mathcal P}_I \vvvert = \vvvert P_{U^\perp} E P_{\widetilde V_{k,s}} \oplus P_{V^\perp} E^\T P_{\widetilde U_{k,s}} \vvvert.$$
Hence, from \eqref{eq:conclude}, we conclude that
\begin{align}\label{eq:conclusionbd2}
&\vvvert \sin \angle ( U_{k,s}, \widetilde{ U}_{k,s}) \oplus \sin \angle ( V_{k,s}, \widetilde{ V}_{k,s}) \vvvert \nonumber\\
&\le 6\sqrt{2} \frac{(b+2)^2}{(b-1)^2}\sqrt{\min\{s-k+1,r-s+k-1\}} \frac{ \eta\sqrt{s-k+1}}{\min\{\delta_{k-1}, \delta_s\}} \nonumber\\
&\quad + 2 \frac{\vvvert P_{U^\perp} E P_{\widetilde V_{k,s}} \oplus P_{V^\perp} E^\T P_{\widetilde U_{k,s}} \vvvert}{{\sigma}_{s}}.
\end{align}
The conclusion of Theorem 2.3 follows immediately from the fact that
$$\max\{ \vvvert \sin \angle ( U_{k,s}, \widetilde{ U}_{k,s})\vvvert, \vvvert\sin \angle ( V_{k,s}, \widetilde{ V}_{k,s}) \vvvert  \} \le \vvvert \sin \angle ( U_{k,s}, \widetilde{ U}_{k,s}) \oplus \sin \angle ( V_{k,s}, \widetilde{ V}_{k,s}) \vvvert$$
by Theorem \ref{thm:fan}. 

Specifically, for the operator norm, from \eqref{eq:conclude-op}, we see that
\begin{align*}
&\max\{ \|\sin \angle ( U_{k,s}, \widetilde{ U}_{k,s})\|, \|\sin \angle ( V_{k,s}, \widetilde{ V}_{k,s})\| \}\\
&\le \|\sin \angle ( U_{k,s}, \widetilde{ U}_{k,s}) \oplus \sin \angle ( V_{k,s}, \widetilde{ V}_{k,s})\|=\|\sin \angle (\mathcal U_{I}, \widetilde{\mathcal U}_{I})\|\\
&\le 3\sqrt{2} \frac{(b+1)^2}{(b-1)^2} \mathbf{1}_{\{s-k+1\neq r \}}\frac{ \eta\sqrt{s-k+1}}{\min\{\delta_{k-1}, \delta_s\}} + 2\frac{\|E\|}{{\sigma}_{s}}.
\end{align*}
This completes the proof. 

\subsection{Proof of Theorem 2.10}\label{sec:proofofthminfinitybd}
 We start with the decomposition \eqref{eq:u}\footnote{ This decomposition may appear structurally similar to the Neumann expansion technique used in Eldridge et al.~\cite{EBW18}, where the perturbed eigenvector is expanded algebraically and then decomposed into components parallel and orthogonal to the true eigenvector. However, our decomposition is functional: we approximate the resolvent operator \( G(z) \) with a diagonal matrix \( \Pi(z) \), and the remainder is rigorously controlled using the isotropic local law. This enables us to identify the leading component of \( \widetilde{\mathbf u}_i \) and bound the remainder probabilistically, rather than through algebraic expansion. }:
$$\widetilde{\mathbf u}_i=\Pi(\widetilde \lambda_i) \mathcal A \widetilde{\mathbf u}_i + \Xi(\widetilde \lambda_i) \mathcal A \widetilde{\mathbf u}_i.$$

In the proof, we set
\begin{align}\label{eq:Pi-new}
\Pi(\widetilde \lambda_i):=\begin{cases}
\Phi(\widetilde \lambda_i), &\text{if } |\lambda_i|\le n^2; \\
\frac{1}{\widetilde \lambda_i} I_{N+n}, &\text{if } |\lambda_i| > n^2,
\end{cases}
\end{align} 
and recall that
\begin{align*}
\Xi(\widetilde \lambda_i)=G(\widetilde \lambda_i)-\Pi(\widetilde \lambda_i).
\end{align*}
Let $\mathcal Q=I-\mathcal P_r$ be the orthogonal projection matrix onto the null space of $\mathcal A$. It is elementary to verify that $\mathcal P_r \Pi(\widetilde \lambda_i) \mathcal A = \Pi(\widetilde \lambda_i) \mathcal A$ using the definitions of $\mathcal U$ and $\mathcal P_r = \mathcal U \mathcal U^\T$.  When  $\Pi(\widetilde \lambda_i )$ is a scalar matrix, the result follows immediately since $\mathcal P_r\mathcal A= \mathcal A.$ For the case where $\Pi(\widetilde \lambda_i) = \Phi(\widetilde \lambda_i)$, from \cite[Eq. (29)]{Wang24}, we have
\[\mathcal P_r \Pi(\widetilde \lambda_i) \mathcal A = \mathcal U \cdot \mathcal U^\T \Phi(\widetilde \lambda_i)\mathcal U \cdot \mathcal D \mathcal U^\T=\mathcal U \begin{pmatrix} \alpha I_r & \beta I_r \\ \beta I_r & \alpha I_r\end{pmatrix} \cdot \mathcal D \mathcal U^\T\]
and from \cite[Eq. (25)]{Wang24}, \[ \Pi(\widetilde \lambda_i) \mathcal A = \begin{pmatrix}
\frac{1}{\phi_1}I_N & 0 \\
0 &\frac{1}{\phi_2}I_n
\end{pmatrix} \mathcal U \cdot \mathcal D \mathcal U^\T.\]
Using the definitions of $\alpha,\beta$ in \cite[Eq. (29)]{Wang24} and the structure of $\mathcal U$, one can verify that \[ \mathcal U \begin{pmatrix} \alpha I_r & \beta I_r \\ \beta I_r & \alpha I_r\end{pmatrix} = \frac{1}{\sqrt 2} \begin{pmatrix} (\alpha+\beta)U & (\alpha+\beta)U\\ (\alpha-\beta)V  & -(\alpha-\beta)V\end{pmatrix} = \begin{pmatrix}
\frac{1}{\phi_1}I_N & 0 \\
0 &\frac{1}{\phi_2}I_n
\end{pmatrix} \mathcal U, \]
which establishes $\mathcal P_r \Pi(\widetilde \lambda_i) \mathcal A = \Pi(\widetilde \lambda_i) \mathcal A$. 

Hence, continuing from \eqref{eq:u}, we can derive the following expression
\begin{align*}
\mathcal Q \widetilde{\mathbf u}_i = \mathcal Q \Xi(\widetilde \lambda_i) \mathcal A \widetilde{\mathbf u}_i.
\end{align*}
Furthermore, we obtain the decomposition
\begin{align*}
\widetilde{\mathbf u}_i &= \mathcal P_I \widetilde{\mathbf u}_i + \mathcal P_J \widetilde{\mathbf u}_i + \mathcal Q \widetilde{\mathbf u}_i\\
&=\mathcal P_I \widetilde{\mathbf u}_i + \mathcal P_J \widetilde{\mathbf u}_i +  \mathcal Q \Xi(\widetilde \lambda_i) \mathcal A \widetilde{\mathbf u}_i.
\end{align*}
It follows that
\begin{align}\label{eq:decompU}
\widetilde{\mathcal U}_I - \mathcal P_I \widetilde{\mathcal U}_I = \mathcal P_J \widetilde{\mathcal U}_I + (\mathcal Q \Xi(\widetilde \lambda_i) \mathcal A \widetilde{\mathbf u}_i)_{i\in I}.
\end{align}

We aim to bound $$\|\widetilde U_{k,s}-P_{U_{k,s}}\widetilde U_{k,s}\|_{2,\infty} =\max_{1\le l \le N} \| \mathsf{e}_l^\T (U_{k,s}-P_{U_{k,s}}\widetilde U_{k,s})\|.$$ where $\mathsf{e}_l$'s are the canonical vectors in $\mathbb R^N$. From the definition of 
$$\mathcal U=\frac{1}{\sqrt 2} \begin{pmatrix}
U & U\\
V & -V
\end{pmatrix}
$$ from Section 6.1, it is elementary to check that 
$$\widetilde{\mathcal U}_I - P_I \widetilde{\mathcal U}_I  = \frac{1}{\sqrt 2} \begin{pmatrix}
\widetilde U_{k,s}-P_{U_{k,s}}\widetilde U_{k,s} &\widetilde U_{k,s}-P_{U_{k,s}}\widetilde U_{k,s}\\
\widetilde V_{k,s}-P_{V_{k,s}}\widetilde V_{k,s} & -(\widetilde V_{k,s}-P_{V_{k,s}}\widetilde V_{k,s})
\end{pmatrix}.$$ Hence,
\begin{align*}
\|\widetilde U_{k,s}-P_{U_{k,s}}\widetilde U_{k,s}\|_{2,\infty}= \max_{1\le l \le N}\left\|e_l^\T \left(\widetilde{\mathcal U}_I - \mathcal P_I \widetilde{\mathcal U}_I \right) \right\|, 
\end{align*}
where $e_l$'s are the canonical vectors in $\mathbb R^{N+n}$. Continuing from \eqref{eq:decompU}, we see
\begin{align}\label{eq:infbd}
\|\widetilde U_{k,s}-P_{U_{k,s}}\widetilde U_{k,s}\|_{2,\infty} &= \max_{1\le l \le N} \|e_l^\T \left(\widetilde{\mathcal U}_I - \mathcal P_I \widetilde{\mathcal U}_I \right)\| \nonumber\\
&\le \max_{1\le l \le N} \|e_l^\T \mathcal P_J \widetilde{\mathcal U}_I\|\cdot \mathbf{1}_{\{ |J|\neq 0 \}}+ \max_{1\le l \le N} \|e_l^\T (\mathcal Q \Xi(\widetilde \lambda_i) \mathcal A \widetilde{\mathbf u}_i)_{i\in I}\|.
\end{align}

Provided that $|J|\neq 0$ or equivalently, $s-k+1\neq r $, the first term on the right-hand side of \eqref{eq:infbd} can be bounded by
\begin{align*}
\max_{1\le l \le N} \|e_l^\T \mathcal P_J \widetilde{\mathcal U}_I\| &= \max_{1\le l \le N} \|e_l^\T \mathcal U_J \cdot \mathcal U_J^\T \widetilde{\mathcal U}_I\|\\
&\le \max_{1\le l \le N} \|e_l^\T \mathcal U_J\| \cdot \| \mathcal U_J^\T \widetilde{\mathcal U}_I\|  \\
&\le \| U\|_{2,\infty}  \| \mathcal U_J^\T \widetilde{\mathcal U}_I\|_F = \| U\|_{2,\infty}  \sqrt{\sum_{i\in I} \| \mathcal U_J^\T \widetilde{\mathbf u}_i \|^2}.
\end{align*}
By Lemma \ref{lem:projsv}, we further obtain 
\begin{align}\label{eq:1stbd}
\max_{1\le l \le N} \|e_l^\T \mathcal P_J \widetilde{\mathcal U}_I\| &\le  3\sqrt{2}\frac{(b+1)^2}{(b-1)^2}  \| U\|_{2,\infty} \frac{\eta\sqrt{s-k+1}}{\min\{\delta_{k-1}, \delta_s\}}
\end{align}
with probability at least $1-20(N+n)^{-K}$.

Next, we bound the second term on the right-hand side of \eqref{eq:infbd}: 
\begin{align}\label{eq:2ndbd}
\max_{1\le l \le N} \|e_l^\T (\mathcal Q \Xi(\widetilde \lambda_i) \mathcal A \widetilde{\mathbf u}_i)_{i\in I}\|=\max_{1\le l \le N} \sqrt{\sum_{i\in I}  (e_l^\T \mathcal Q \Xi(\widetilde \lambda_i) \mathcal A \widetilde{\mathbf u}_i)^2}.
\end{align}
For each $i \in I$, 
\begin{align}\label{eq:est1}
\left|e_l^\T \mathcal Q \Xi(\widetilde \lambda_i) \mathcal A \widetilde{\mathbf u}_i \right| &=  \left| e_l^\T (I-\mathcal U \mathcal U^\T) \Xi(\widetilde \lambda_i) \mathcal U \cdot \mathcal D \mathcal U^\T \widetilde{\mathbf u}_i \right|\nonumber\\
&\le  \left\| e_l^\T (I-\mathcal U \mathcal U^\T) \Xi(\widetilde \lambda_i) \mathcal U \right\| \cdot \left\| \mathcal D \mathcal U^\T \widetilde{\mathbf u}_i \right\|.
\end{align} 
Observe from $(\mathcal A + \mathcal E)  \widetilde{\mathbf u}_i =  \widetilde{\lambda}_i \widetilde{\mathbf u}_i$ that $\mathcal U \mathcal D \mathcal U^\T \widetilde{\mathbf u}_i=( \widetilde{\lambda}_i I -\mathcal E)\widetilde{\mathbf u}_i$. Multiplying $\mathcal U^\T$ on both sides, we get the bound 
\begin{align}\label{eq:bdonev01}
\left\|{\mathcal D} {\mathcal U}^\T \widetilde{\mathbf u}_i \right\| \le \|\mathcal E\|+|\widetilde{\lambda}_i| \le \left(1+\frac{1}{b-1} \right)|\widetilde{\lambda}_i| = \frac{b}{b-1}|\widetilde{\lambda}_i|
\end{align}
using the assumption $\|\mathcal E\|=\|E\| \le \frac{1}{b}|{\lambda}_i|$ and the Weyl's inequality $|\widetilde{\lambda}_i|\ge |{\lambda}_i| - \|E\| \ge (b-1)\|E\|$. 

To estimate 
\begin{align}\label{eq:con01}
\left\| e_l^\T (I-\mathcal U \mathcal U^\T) \Xi(\widetilde \lambda_i) \mathcal U \right\| &\le \left\|e_l^\T  \Xi(\widetilde \lambda_i) \mathcal U \right\| + \left\| e_l^\T \mathcal U \mathcal U^\T \Xi(\widetilde \lambda_i) \mathcal U \right\|\nonumber\\
& \le \left\|e_l^\T  \Xi(\widetilde \lambda_i) \mathcal U \right\| + \left\| U \|_{2,\infty} \|\mathcal U^\T \Xi(\widetilde \lambda_i) \mathcal U \right\|,
\end{align}
we split the index set $I$ into two disjoint sets: 
$$\mathcal I_s := \left\{ i \in I: | \lambda_i| \le n^2 \right\} \quad\text{and}\quad \mathcal I_b := \left\{ i \in I: | \lambda_i| > n^2 \right\}.$$
Note that $\mathcal I_s$ or  $\mathcal I_b$ could be the empty set.

\medskip

\noindent\emph{Case (1):} $i \in I \cap \mathcal I_s$. In this case,  
 $$\Xi(\widetilde \lambda_i)=G(\widetilde \lambda_i)-\Phi(\widetilde \lambda_i).$$  Note that if $z\in S_{\sigma_i}$ specified in \cite[Eq. (35)]{Wang24} for any $1\le i \le r_0$, then $|z| \ge 2b(\sqrt N + \sqrt n)$ by the supposition of $\sigma_i$. Recall $$\eta=\frac{11 b^2}{(b-1)^2} \sqrt{(K+7)\log (N+n)+2(\log 9) r}.$$ 

We work on the event $\mathsf E:=\cap_{i\in \llbracket k,s \rrbracket\cap \mathcal I_s} \mathsf E_i$ where 
\begin{align}\label{event}
\mathsf E_i&:=\left\{ \widetilde \sigma_i \in S_{\sigma_{l_i}} \text{ for some } l_i \in \llbracket 1,r_0\rrbracket\right\} \cap \left\{ \left\|  \mathcal U^\T  \Xi(\widetilde \sigma_i) \mathcal U \right\| \leq \frac{\eta}{\widetilde \sigma_i^2}\right\} \nonumber\\
&\cap \left\{\left| e_l^\T  \Xi(\widetilde \sigma_i) {\mathbf u}_s  \right| \leq  \frac{5b^2}{(b-1)^2} \frac{\sqrt{(K+7)\log (N+n)}}{\widetilde \sigma_i^2}\text{ for } 1\le l \le N+n, 1\le s \le r \right\}.
\end{align}
By Theorem 6.4, Lemma 6.3 and Lemma 6.2, the event $\mathsf E$ holds with probability at least $1-20(N+n)^{-K}$. For $i\in \llbracket k,s \rrbracket\cap \mathcal I_s$, $\widetilde \lambda_i=\widetilde \sigma_i $. It follows immediately that
\begin{align}\label{eq:bd0208}
&\left\|e_l^\T  \Xi(\widetilde \lambda_i) \mathcal U \right\| + \| U \|_{2,\infty} \left\|\mathcal U^\T \Xi(\widetilde \lambda_i) \mathcal U \right\| \nonumber\\
&= \sqrt{\sum_{s=1}^{2r} \left(e_l^\T  \Xi(\widetilde \lambda_i) \mathbf u_s \right)^2}+ \| U \|_{2,\infty} \left\|\mathcal U^\T \Xi(\widetilde \lambda_i) \mathcal U \right\| \nonumber\\ 
&< \frac{18b^2}{(b-1)^2} \frac{ \sqrt{r(K+7)\log(N+n)}  }{\widetilde \lambda_i^2}(1+ \|  U \|_{2,\infty}) .
\end{align}
Continuing from \eqref{eq:est1}, \eqref{eq:bdonev01} and \eqref{eq:con01}, we further have for any $i\in I\cap \mathcal I_s$,
\begin{equation}\label{eq:est11}
\left| e_l^\T \mathcal Q \Xi(\widetilde \lambda_i) \mathcal A \widetilde{\mathbf u}_i \right| \le  \frac{2b}{b-1}\frac{\gamma }{|\widetilde \lambda_i|}(1+ \| U \|_{2,\infty})
\end{equation}
where we define $$\gamma:= \frac{9 b^2}{(b-1)^2}\sqrt{r(K+7)\log(N+n)} $$ for the sake of brevity.
For $i\in \llbracket r+k, r+s\rrbracket\cap \mathcal I_s$, $\widetilde \lambda_i=-\widetilde \sigma_{i-r} $. Note that $\Xi(\widetilde \lambda_i) \sim -\Xi(\widetilde \sigma_{i-r})$ since the distribution of $E$ is symmetric. The bound \eqref{eq:est11} still holds. 

\medskip

\noindent\emph{Case (2):} $i \in I \cap \mathcal I_b$. In this case,  $$\Xi(\widetilde \lambda_i)=G(\widetilde \lambda_i)-\frac{1}{\widetilde \lambda_i} I_{N+n}.$$
By Weyl's inequality, $|\widetilde \lambda_i| \ge n^2 - \|\mathcal E\| \ge 4(\sqrt N + \sqrt n)$ for every $i \in \mathcal I_b$, we apply Lemma 6.5 to get
\begin{equation*}
\|\Xi(\widetilde{\lambda}_i) \| \le \frac{2\|\mathcal E\|}{\widetilde{\lambda}_i^2}. 
\end{equation*}
As a result,
\begin{align*}
\left\|e_l^\T  \Xi(\widetilde \lambda_i) \mathcal U \right\| + \|  U \|_{2,\infty} \left\|\mathcal U^\T \Xi(\widetilde \lambda_i) \mathcal U \right\| 
&\le (1+ \|  U \|_{2,\infty}) \|\Xi(\widetilde{\lambda}_i) \| \le \frac{2\|\mathcal E\| }{\widetilde{\lambda}_i^2}(1+ \|  U \|_{2,\infty}).
\end{align*}
Continuing from \eqref{eq:est1} and \eqref{eq:con01}, we further have
\begin{align}\label{eq:est21}
\left| e_l^\T \mathcal Q \Xi(\widetilde \lambda_i) \mathcal A \widetilde{\mathbf u}_i \right| \le \frac{2b}{b-1}\frac{\|\mathcal E\| }{|\widetilde{\lambda}_i|}(1+ \|  U \|_{2,\infty}).
\end{align}

Note by Weyl's inequality, for $i\in \llbracket k, s \rrbracket$, $|\widetilde{\lambda}_i| \ge \frac{b-1}{b} \sigma_i$ and for $i\in \llbracket r+k, r+s\rrbracket$, $|\widetilde{\lambda}_i| \ge \frac{b-1}{b} \sigma_{i-r}$. Continuing from \eqref{eq:2ndbd} with \eqref{eq:est11} and \eqref{eq:est21}, we conclude that
\begin{align*}
&\max_{1\le l \le N} \left\|e_l^\T (Q \Xi(\widetilde \lambda_i) \mathcal A \widetilde{\mathbf u}_i)_{i\in I}\right\|\\
&\le 2\sqrt{2}\frac{b^2}{(b-1)^2} (1+ \|  U \|_{2,\infty}) \sqrt{\sum_{i\in \llbracket k, s\rrbracket, \sigma_i\le n^2} \frac{\gamma^2}{\sigma_i^2} + \sum_{i\in \llbracket k, s\rrbracket, \sigma_i> n^2} \frac{\| E\|^2}{\sigma_i^2}}.
\end{align*}
Note that $\|E\|\le 2(\sqrt N + \sqrt n) \le 4\sqrt n$. Inserting the above estimate and \eqref{eq:1stbd} into \eqref{eq:infbd} yields that 
\begin{align*}
\left\|\widetilde U_{k,s}-P_{U_{k,s}}\widetilde U_{k,s} \right\|_{2,\infty}  & \le 3\sqrt{2}\frac{(b+1)^2}{(b-1)^2}  \| U \|_{2,\infty} \frac{\eta\sqrt{s-k+1}}{\min\{\delta_{k-1}, \delta_s\}} \mathbf{1}_{\{s-k+1\neq r\}}\\
& + \frac{2\sqrt{2} b^2}{(b-1)^2} (1+ \| U \|_{2,\infty}) \sqrt{\sum_{i\in \llbracket k, s\rrbracket, \sigma_i\le n^2} \frac{\gamma^2}{\sigma_i^2} + \sum_{i\in \llbracket k, s\rrbracket, \sigma_i> n^2} \frac{16 n}{\sigma_i^2}}.
\end{align*}
This concludes the proof by noting that $\|U\|_{2,\infty} \le 1$. 

\subsection{Proof of Theorem 2.11}\label{sec:prooflinear}
The proof strategy for Theorem 2.11 mirrors that of Theorem 2.10. We provide a brief outline below.

First, we estimate $\left\|x^\T (\widetilde U_{k,s}-P_{U_{k,s}}\widetilde U_{k,s}) \right\|$ for a unit vector $x\in \mathbb R^N$.  Let $\mathbf a=(x^\T, 0)^\T$ be a unit vector in $\mathbb{R}^{N+n}$.    Following the same line of the above proof, we first observe that
$$\left\|x^\T (\widetilde U_{k,s}-P_{U_{k,s}}\widetilde U_{k,s}) \right\| = \left\|\mathbf a^\T (\widetilde{\mathcal U}_I - P_I \widetilde{\mathcal U}_I) \right\|.$$
Using the same proof as that of \eqref{eq:infbd}, one gets
\begin{align}\label{eq:linearbd1}
\left\|\mathbf a^\T (\widetilde{\mathcal U}_I - P_I \widetilde{\mathcal U}_I) \right\| \le \|\mathbf a^\T \mathcal P_J \widetilde{\mathcal U}_I\| \mathbf{1}_{\{|J| \neq 0 \}} +  \|\mathbf a^\T (\mathcal Q \Xi(\widetilde \lambda_i) \mathcal A \widetilde{\mathbf u}_i)_{i\in I}\|.
\end{align}
For the first term on the right-hand side of \eqref{eq:linearbd1}, following the same line as \eqref{eq:1stbd}, we have with probability at least $1-20(N+n)^{-K}$ that
\begin{align*}
\|\mathbf a^\T \mathcal P_J \widetilde{\mathcal U}_I\|\le  3\sqrt{2}\frac{(b+1)^2}{(b-1)^2}  \| \mathbf a^\T U\| \frac{\eta\sqrt{s-k+1}}{\min\{\delta_{k-1}, \delta_s\}} \mathbf{1}_{\{|J| \neq 0 \}}.
\end{align*}
For the second term on the right-hand side of \eqref{eq:linearbd1}, using the similar arguments as \eqref{eq:2ndbd}, we also get with probability at least $1-20(N+n)^{-K}$ that
\begin{align*}
\|\mathbf a^\T (\mathcal Q \Xi(\widetilde \lambda_i) \mathcal A \widetilde{\mathbf u}_i)_{i\in I}\| \le \frac{2\sqrt{2} b^2}{(b-1)^2} (1+ \| \mathbf a^\T \mathcal U \|) \sqrt{\sum_{i\in \llbracket k, s\rrbracket, \sigma_i\le n^2} \frac{\gamma^2}{\sigma_i^2} +  \sum_{i\in \llbracket k, s\rrbracket, \sigma_i> n^2} \frac{16 n}{\sigma_i^2} }.
\end{align*}
Combining the above estimates and noting that $\| \mathbf a^\T \mathcal U \|=\| x^\T  U \|$, we conclude 
\begin{align*}
\left\|x^\T (\widetilde U_{k,s}-P_{U_{k,s}}\widetilde U_{k,s}) \right\| \le & 3\sqrt{2}\frac{(b+1)^2}{(b-1)^2}  \| x^\T  U \| \frac{\eta\sqrt{s-k+1}}{\min\{\delta_{k-1}, \delta_s\}}  \mathbf{1}_{\{s-k+1\neq r\}}\\
 &+ \frac{2\sqrt{2} b^2}{(b-1)^2} (1+\| x^\T  U \|) \sqrt{\sum_{i\in \llbracket k, s\rrbracket, \sigma_i\le n^2} \frac{\gamma^2}{\sigma_i^2} + \sum_{i\in \llbracket k, s\rrbracket, \sigma_i> n^2} \frac{16 n}{\sigma_i^2}}
\end{align*}
with probability at least $1-40(N+n)^{-K}$. 

Next, we turn to the estimation of $\left| x^\T (\widetilde U_{k,s}-P_{U_{k,s}}\widetilde U_{k,s}) y \right|$. Set $\mathbf b=(y^\T, 0)^\T \in \mathbb R^{2(k-s+1)}$. It is elementary to check that 
\begin{align*}
\left| x^\T (\widetilde U_{k,s}-P_{U_{k,s}}\widetilde U_{k,s}) y \right| =\sqrt{2} \left|\mathbf a^\T (\widetilde{\mathcal U}_I - P_I \widetilde{\mathcal U}_I)\mathbf b \right|.
\end{align*}
Using the decomposition in \eqref{eq:decompU}, we get 
\begin{align}\label{eq:bilinearbd1}
\left|\mathbf a^\T (\widetilde{\mathcal U}_I - P_I \widetilde{\mathcal U}_I)\mathbf b \right| \le \left| \mathbf a^\T \mathcal P_J \widetilde{\mathcal U}_I \mathbf b \right| \mathbf{1}_{\{|J| \neq 0 \}} +  \left| \mathbf a^\T (\mathcal Q \Xi(\widetilde \lambda_i) \mathcal A \widetilde{\mathbf u}_i)_{i\in I} \mathbf b \right|.
\end{align}
When $|J| \neq 0$, taking the definitions of $\mathbf a, \mathbf b$ into consideration, we can derive an upper bound for the first term on the right-hand side of equation \eqref{eq:bilinearbd1}:
\begin{align*}
 \left| \mathbf a^\T \mathcal P_J \widetilde{\mathcal U}_I \mathbf b \right| &= \left| \mathbf a^\T \mathcal U_J \cdot \mathcal U_J^\T \widetilde{\mathcal U}_I \mathbf b \right|\le \|\mathbf a^\T \mathcal U_J\| \cdot \| \mathcal U_J^\T \widetilde{\mathcal U}_I \mathbf b\|= \| x^\T U_{J_0}\| \left\| \sum_{i\in I_0} y_i \mathcal U_J^\T \widetilde{\mathbf u}_i \right\|
\end{align*}
where $J_0:=\llbracket 1, k-1\rrbracket \cup \llbracket s+1, r\rrbracket$ and $I_0:=\llbracket k, s\rrbracket$. By Lemma \ref{lem:projsv}, we further obtain for each $i\in I_0$,
\begin{align*}
\| \mathcal U_J^\T \widetilde{\mathbf u}_i \| \le 3 \frac{(b+1)^2}{(b-1)^2}  \frac{\eta}{\min\{\delta_{k-1}, \delta_s\}}
\end{align*}
with probability at least $1-20(N+n)^{-K}$. By Cauchy-Schwarz inequality, we further obtain
\begin{align}\label{eq:bilinear1}
 \left| \mathbf a^\T \mathcal P_J \widetilde{\mathcal U}_I \mathbf b \right| \le 3 \frac{(b+1)^2}{(b-1)^2}  \| x^\T U\| \frac{\eta \sqrt{\|y\|_0}}{\min\{\delta_{k-1}, \delta_s\}}
\end{align}
with probability at least $1-20(N+n)^{-K}$.

For the second term on the right-hand side of equation \eqref{eq:bilinearbd1},  we start with
\begin{align}\label{eq:bili01}
 \left| \mathbf a^\T (\mathcal Q \Xi(\widetilde \lambda_i) \mathcal A \widetilde{\mathbf u}_i)_{i\in I} \mathbf b \right| = \left| \sum_{i\in I_0} y_i \mathbf a^\T \mathcal Q \Xi(\widetilde \lambda_i) \mathcal A \widetilde{\mathbf u}_i \right|.
\end{align}
For each $i\in I_0$, similar to \eqref{eq:est1}, we have
\begin{align*}
\left|\mathbf a^\T \mathcal Q \Xi(\widetilde \lambda_i) \mathcal A \widetilde{\mathbf u}_i \right| &=  \left| \mathbf a^\T (I-\mathcal U \mathcal U^\T) \Xi(\widetilde \lambda_i) \mathcal U \cdot \mathcal D \mathcal U^\T \widetilde{\mathbf u}_i \right|\nonumber\\
&\le  \left\| \mathbf a^\T (I-\mathcal U \mathcal U^\T) \Xi(\widetilde \lambda_i) \mathcal U \right\| \cdot \left\| \mathcal D \mathcal U^\T \widetilde{\mathbf u}_i \right\|.
\end{align*}
Note that, by \eqref{eq:bdonev01}, 
$$\left\| \mathcal D \mathcal U^\T \widetilde{\mathbf u}_i \right\| \le \frac{b}{b-1} |\widetilde{\lambda}_i|.$$
With the same discussion after \eqref{eq:con01}, for $i\in I_0 \cap \mathcal I_s$ where $\mathcal I_s := \left\{ i \in I: | \lambda_i| \le n^2 \right\}$, the estimation of 
\begingroup
\allowdisplaybreaks
\begin{align*}
\left\| \mathbf a^\T (I-\mathcal U \mathcal U^\T) \Xi(\widetilde \lambda_i) \mathcal U \right\| &\le \left\|\mathbf a^\T  \Xi(\widetilde \lambda_i) \mathcal U \right\| + \left\| \mathbf a^\T \mathcal U \mathcal U^\T \Xi(\widetilde \lambda_i) \mathcal U \right\|\nonumber\\
& \le \left\|\mathbf a^\T  \Xi(\widetilde \lambda_i) \mathcal U \right\| + \|x^\T U \| \|\mathcal U^\T \Xi(\widetilde \lambda_i) \mathcal U \|\\
&= \sqrt{\sum_{s=1}^{2r} \left(\mathbf a^\T  \Xi(\widetilde \lambda_i) \mathbf u_s \right)^2}+ \| x^\T U \| \|\mathcal U^\T \Xi(\widetilde \lambda_i) \mathcal U \|\\
&< \frac{18b^2}{(b-1)^2} \frac{ \sqrt{r(K+7)\log(N+n)}  }{\widetilde \lambda_i^2}(1+ \| x^\T U \|)\\
&=\frac{2\gamma}{\widetilde \lambda_i^2}(1+ \| x^\T U \|)
\end{align*}
\endgroup
follows the same line as those of \eqref{eq:con01} and \eqref{eq:bd0208} with $e_l$ replaced by the unit vector $\mathbf a$. 
In particular, combined with $|\widetilde{\lambda}_i| \ge \frac{b-1}{b} \sigma_i$ for each $i\in I_0$ from the Weyl's inequality, we obtain
\begin{align*}
\left|\mathbf a^\T \mathcal Q \Xi(\widetilde \lambda_i) \mathcal A \widetilde{\mathbf u}_i \right| &\le  \left\| \mathbf a^\T (I-\mathcal U \mathcal U^\T) \Xi(\widetilde \lambda_i) \mathcal U \right\| \cdot \left\| \mathcal D \mathcal U^\T \widetilde{\mathbf u}_i \right\|\\
& \le \frac{2b}{b-1} \frac{\gamma }{|\widetilde \lambda_i|}(1+ \|x^\T U \|)\le \frac{2b^2}{(b-1)^2} \frac{\gamma }{\sigma_i}(1+ \|x^\T U \|)
\end{align*}
holds with probability at least $1-20(N+n)^{-K}$.  Next, for $i\in I_0 \cap \mathcal I_b$ where $\mathcal I_b := \left\{ i \in I: | \lambda_i| > n^2 \right\}$, similar to \eqref{eq:est21}, we obtain
\[ \left|\mathbf a^\T \mathcal Q \Xi(\widetilde \lambda_i) \mathcal A \widetilde{\mathbf u}_i \right| \le \frac{2b^2}{(b-1)^2} \frac{\|E\|}{\sigma_i} (1+\|x^\T U\|).\] 

Continuing from \eqref{eq:bili01}, we have
\begin{align}\label{eq:bilinear2} 
 \left| \mathbf a^\T (\mathcal Q \Xi(\widetilde \lambda_i) \mathcal A \widetilde{\mathbf u}_i)_{i\in I} \mathbf b \right| \le \frac{2b^2}{(b-1)^2}  (1+ \|x^\T U \|)  \left[  \sum_{i\in \llbracket k, s\rrbracket, \sigma_i\le n^2} \frac{{\gamma } |y_i|}{\sigma_i} + \sum_{i\in \llbracket k, s\rrbracket, \sigma_i> n^2} \frac{\|E\|}{\sigma_i} \right].
\end{align}
Finally, inserting \eqref{eq:bilinear1} and \eqref{eq:bilinear2} back into \eqref{eq:bilinearbd1}, we obtain that
\begin{align*}
\left| x^\T (\widetilde U_{k,s}-P_{U_{k,s}}\widetilde U_{k,s}) y \right| &=\sqrt{2} \left|\mathbf a^\T (\widetilde{\mathcal U}_I - P_I \widetilde{\mathcal U}_I)\mathbf b \right|\\
&\le 3\sqrt{2} \frac{(b+1)^2}{(b-1)^2}  \| x^\T U\| \frac{\eta \sqrt{\|y\|_0}}{\min\{\delta_{k-1}, \delta_s\}}\mathbf{1}_{\{s-k+1\neq r\}}\\
&\  \ + \frac{2\sqrt{2}b^2}{(b-1)^2}(1+ \|x^\T U \|) \left[  \sum_{i\in \llbracket k, s\rrbracket, \sigma_i\le n^2} \frac{{\gamma } |y_i|}{\sigma_i} +  \sum_{i\in \llbracket k, s\rrbracket, \sigma_i> n^2} \frac{\|E\|}{\sigma_i} \right]
\end{align*}
holds with probability at least $1-40(N+n)^{-K}$. The proof is completing by noting that $ \| x^\T U\| \le \|x\| \cdot \|U\| \le 1$.

\subsection{Proof of Theorem 2.13}\label{sec:proofweighted}The proof of Theorem 2.13 follows largely the proof of Theorem  2.10. We sketch the proof and focus on the difference. 
We start with the decomposition \eqref{eq:u}:
$$\widetilde{\mathbf u}_i=\Pi(\widetilde \lambda_i) \mathcal A \widetilde{\mathbf u}_i + \Xi(\widetilde \lambda_i) \mathcal A \widetilde{\mathbf u}_i$$
and set
\begin{align}\label{eq:Pi-new01}
\Pi(\widetilde \lambda_i):=\begin{cases}
\Phi(\widetilde \lambda_i), &\text{if } |\lambda_i|\le n^2; \\
\frac{1}{\widetilde \lambda_i} I_{N+n} + \frac{\mathcal E}{\widetilde \lambda_i^2}, &\text{if } |\lambda_i| > n^2.
\end{cases}
\end{align} 
The definition of $\Pi(\widetilde \lambda_i)$ differs from the one given in \eqref{eq:Pi-new} when $|\lambda_i| > n^2$. We adopt this definition to achieve more precise control over the error term by considering the weights $|\widetilde \lambda_i|$.

Let $\mathcal Q=I-\mathcal P_r$ be the orthogonal projection matrix onto the null space of $\mathcal A$. Using the same derivation as in the beginning of the proof of Theorem  2.10, we obtain the decomposition
\begin{align*}
\widetilde \lambda_i \widetilde{\mathbf u}_i =\mathcal P_I \widetilde \lambda_i \widetilde{\mathbf u}_i + \mathcal P_J \widetilde \lambda_i \widetilde{\mathbf u}_i +  \mathcal Q \Xi(\widetilde \lambda_i) \mathcal A \widetilde \lambda_i \widetilde{\mathbf u}_i
\end{align*}
and hence
\begin{align}\label{eq:decompU-new}
\widetilde{\mathcal U}_I \widetilde{\mathcal D}_I- \mathcal P_I \widetilde{\mathcal U}_I \widetilde{\mathcal D}_I= \mathcal P_J \widetilde{\mathcal U}_I \widetilde{\mathcal D}_I + (\mathcal Q \Xi(\widetilde \lambda_i) \mathcal A \widetilde \lambda_i \widetilde{\mathbf u}_i)_{i\in I}.
\end{align}
It can be verified using the definitions of $\widetilde{\mathcal U}_I, \widetilde{\mathcal D}_I$ and $\mathcal P_I = \widetilde{\mathcal U}_I \widetilde{\mathcal U}_I^\T$ that
\begin{align*}
\widetilde{\mathcal U}_I \widetilde{\mathcal D}_I- \mathcal P_I \widetilde{\mathcal U}_I \widetilde{\mathcal D}_I=\frac{1}{\sqrt 2}
\begin{pmatrix}
\widetilde U_{k,s} \widetilde D_{k,s} - P_{U_{k,s}} \widetilde U_{k,s} \widetilde D_{k,s} & -(\widetilde U_{k,s} \widetilde D_{k,s} - P_{U_{k,s}} \widetilde U_{k,s} \widetilde D_{k,s})\\
\widetilde V_{k,s} \widetilde D_{k,s} - P_{V_{k,s}} \widetilde V_{k,s} \widetilde D_{k,s} & \widetilde V_{k,s} \widetilde D_{k,s} - P_{V_{k,s}} \widetilde V_{k,s} \widetilde D_{k,s}.
\end{pmatrix}
\end{align*}
Therefore, combining \eqref{eq:decompU-new}, we observe that
\begin{align}\label{eq:wtbd}
&\|\widetilde U_{k,s} \widetilde D_{k,s} - P_{U_{k,s}} \widetilde U_{k,s} \widetilde D_{k,s}\|_{2,\infty} \nonumber\\&= \max_{1\le i \le N} \| e_l^\T (\widetilde{\mathcal U}_I \widetilde{\mathcal D}_I- \mathcal P_I \widetilde{\mathcal U}_I \widetilde{\mathcal D}_I)\| \nonumber\\
&= \max_{1\le i \le N} \| e_l^\T (\mathcal P_J \widetilde{\mathcal U}_I \widetilde{\mathcal D}_I + (\mathcal Q \Xi(\widetilde \lambda_i) \mathcal A \widetilde \lambda_i \widetilde{\mathbf u}_i)_{i\in I})\| \nonumber\\
&\le \max_{1\le i \le N} \| e_l^\T \mathcal P_J \widetilde{\mathcal U}_I \widetilde{\mathcal D}_I \| \mathbf{1}_{|J| \neq 0}+ \max_{1\le i \le N} \| e_l^\T (\mathcal Q \Xi(\widetilde \lambda_i) \mathcal A \widetilde \lambda_i \widetilde{\mathbf u}_i)_{i\in I})\|.
\end{align}
The first term on the right-hand side of \eqref{eq:wtbd} can be bounded similarly as that of \eqref{eq:1stbd} using Lemma \ref{lem:projsv}: if $|J| \neq 0$, then
\begin{align*}
\max_{1\le l \le N} \|e_l^\T \mathcal P_J \widetilde{\mathcal U}_I  \widetilde{\mathcal D}_I\| &= \max_{1\le l \le N} \|e_l^\T \mathcal U_J \cdot \mathcal U_J^\T \widetilde{\mathcal U}_I  \widetilde{\mathcal D}_I\|\\
&\le \max_{1\le l \le N} \|e_l^\T \mathcal U_J\| \cdot \| \mathcal U_J^\T \widetilde{\mathcal U}_I  \widetilde{\mathcal D}_I\|  \\
&\le \| U\|_{2,\infty}  \| \mathcal U_J^\T \widetilde{\mathcal U}_I  \widetilde{\mathcal D}_I\|_F = \| U\|_{2,\infty}  \sqrt{\sum_{i\in I}\widetilde \lambda_i^2 \| \mathcal U_J^\T \widetilde{\mathbf u}_i \|^2}\\
&\le  3\sqrt{2}\frac{(b+1)^2}{(b-1)^2}  \| U\|_{2,\infty} \frac{\eta \sigma_k \sqrt{s-k+1}}{\min\{\delta_{k-1}, \delta_s\}}
\end{align*}
with probability at least $1-20(N+n)^{-K}$.

The bound of the second term on the right-hand side of \eqref{eq:wtbd} proceeds in a similar manner to that of \eqref{eq:2ndbd}:
\begin{align}\label{eq:wtbd00}
\max_{1\le i \le N} \| e_l^\T (\mathcal Q \Xi(\widetilde \lambda_i) \mathcal A \widetilde \lambda_i \widetilde{\mathbf u}_i)_{i\in I})\|=\max_{1\le l \le N} \sqrt{\sum_{i\in I} \widetilde \lambda_i^2  (e_l^\T \mathcal Q \Xi(\widetilde \lambda_i) \mathcal A \widetilde{\mathbf u}_i)^2}.
\end{align}
For each $i\in I$, we first establish the following bound using \eqref{eq:est1}, \eqref{eq:bdonev01} and \eqref{eq:con01}:
\begin{align}\label{eq:wtbd01}
|\widetilde \lambda_i | \cdot \left| e_l^\T \mathcal Q \Xi(\widetilde \lambda_i) \mathcal A \widetilde{\mathbf u}_i \right| \le \frac{b}{b-1} \widetilde \lambda_i^2 \left( \|e_l^\T  \Xi(\widetilde \lambda_i) \mathcal U \| + \| U \|_{2,\infty} \|\mathcal U^\T \Xi(\widetilde \lambda_i) \mathcal U \|\right).
\end{align}
We then differentiate the cases by splitting the discussion according to whether $i\in I \cap \mathcal I_s$ or $i\in I \cap \mathcal I_b$, where 
$$\mathcal I_s := \left\{ i \in I: | \lambda_i| \le n^2 \right\} \quad\text{and}\quad \mathcal I_b := \left\{ i \in I: | \lambda_i| > n^2 \right\}.$$
If $i\in I \cap \mathcal I_s$, then by \eqref{eq:est11}, we immediately obtain
\begin{align}\label{eq:wtbd02}
|\widetilde \lambda_i | \cdot \left| e_l^\T \mathcal Q \Xi(\widetilde \lambda_i) \mathcal A \widetilde{\mathbf u}_i \right| \le \frac{2b}{b-1}\gamma (1+ \|U\|_{2,\infty})
\end{align}
with probability at least $1-20(N+n)^{-K}$.

If $i\in I \cap \mathcal I_b$, then the only difference from the proof of \eqref{eq:est21} is that $$\Xi(\widetilde \lambda_i)=G(\widetilde \lambda_i)-\frac{1}{\widetilde \lambda_i} I_{N+n}-\frac{\mathcal E}{\widetilde \lambda_i^2}$$ and by Lemma 6.5, we have 
$$\|\Xi(\widetilde \lambda_i)\| \le \frac{2\|\mathcal E\|^2 }{|\widetilde \lambda_i|^3}. $$
It follows that
\begin{align*}
\left\|e_l^\T  \Xi(\widetilde \lambda_i) \mathcal U \right\| + \|  U \|_{2,\infty} \left\|\mathcal U^\T \Xi(\widetilde \lambda_i) \mathcal U \right\| 
&\le (1+ \|  U \|_{2,\infty}) \|\Xi(\widetilde{\lambda}_i) \| \le \frac{2\|\mathcal E\|^2 }{|\widetilde{\lambda}_i|^3}(1+ \|  U \|_{2,\infty}).
\end{align*}
Continuing from \eqref{eq:wtbd01}, we further get
\begin{align}\label{eq:wtbd03}
|\widetilde \lambda_i | \cdot \left| e_l^\T \mathcal Q \Xi(\widetilde \lambda_i) \mathcal A \widetilde{\mathbf u}_i \right| \le \frac{2b}{b-1} \frac{\|E\|^2}{|\widetilde{\lambda}_i|}(1+ \|  U \|_{2,\infty}).
\end{align}

Note by Weyl's inequality, for $i\in \llbracket k, s \rrbracket$, $|\widetilde{\lambda}_i| \ge \frac{b-1}{b} \sigma_i$ and for $i\in \llbracket r+k, r+s\rrbracket$, $|\widetilde{\lambda}_i| \ge \frac{b-1}{b} \sigma_{i-r}$. Inserting \eqref{eq:wtbd02} and \eqref{eq:wtbd03} back into \eqref{eq:wtbd00}, we obtain that with probability at least $1-20(N+n)^{-K}$,
\begin{align*}
&\max_{1\le i \le N} \| e_l^\T (\mathcal Q \Xi(\widetilde \lambda_i) \mathcal A \widetilde \lambda_i \widetilde{\mathbf u}_i)_{i\in I})\| \\
&\le \frac{2\sqrt 2 b^2}{(b-1)^2} (1+ \|  U \|_{2,\infty}) \sqrt{\gamma^2(s-k+1) + \sum_{i\in \llbracket k, s\rrbracket,\sigma_i>n^2} \frac{\|E\|^2}{\sigma_i}}\\
&\le  \frac{4\sqrt 2 b^2}{(b-1)^2}  \sqrt{\gamma^2(s-k+1) + 16},
\end{align*}
where we used $ \|  U \|_{2,\infty}\le 1$ and the crude estimate
$$\sum_{i\in \llbracket k, s\rrbracket,\sigma_i>n^2} \frac{\|E\|^2}{\sigma_i} \le n \frac{16n}{n^2} =16.$$
This concludes the proof.

\subsection{Proof of Lemma \ref{lem:lowerbdFO}}\label{sec:prooflower}
Throughout the proof, we work on the event $\|E\|\le 2(\sqrt{N}+\sqrt{n})$ which holds with probability at least $1-C\exp(-c(\sqrt N + \sqrt n)^2)$ (see \cite[Proposition 2.4]{Vbook}). Under this event, we have $\sigma_k \ge 2\|E\|$. We adopt the linearization notation introduced in Section 6.1.

For simplicity, denote
$I:=\llbracket 1, k\rrbracket \cup \llbracket r+1, r+k\rrbracket$
and
$J:=\llbracket 1, 2r \rrbracket \setminus I =  \llbracket k+1, r\rrbracket  \cup  \llbracket k+r+1, 2r\rrbracket.$ Also, let $I_0:=\llbracket 1, N+n \rrbracket \setminus (I\cup J) = \llbracket 2r+1, N+n\rrbracket$.

We begin with these key identities:
\begin{align}
&\| \sin \angle (\mathcal U_{I}, \widetilde{\mathcal U}_{I}) \|_F^2 = \|\sin \angle ( U_{k}, \widetilde{ U}_{k})\|_F^2 + \|\sin \angle ( V_{k}, \widetilde{ V}_{k})\|_F^2,\label{eq:F-lowbd}\\
&\| \sin \angle (\mathcal U_{I}, \widetilde{\mathcal U}_{I}) \| = \max\{\|\sin \angle ( U_{k}, \widetilde{ U}_{k})\|, \|\sin \angle ( V_{k}, \widetilde{ V}_{k})\| \}.\label{eq:O-lowbd}
\end{align}
These follow from the observation that the principal angles between the subspaces $\mathcal U_{I}$ and $\widetilde{\mathcal U}_{I}$ are composed of the principal angles between $U_{k}$ and $\widetilde{ U}_{k}$ and those  between $V_{k}$ and $\widetilde{V}_{k}$ (see Proposition 8 from \cite{OVW22} for a detailed proof). Therefore, it suffices to work with the subspaces $\mathcal U_{I}$ and $\widetilde{\mathcal U}_{I}$. 

Let $\vvvert \cdot \vvvert$ denote either the Frobenius or operator norm. By applying \eqref{eq:sineq}, we have
\begin{align*}
\vvvert \sin \angle (\mathcal U_{I}, \widetilde{\mathcal U}_{I}) \vvvert =\vvvert {\mathcal P}_{I^c}  \widetilde{\mathcal P}_I \vvvert =\vvvert \mathcal P_{J} \widetilde{\mathcal P}_{I} + \mathcal P_{I_0} \widetilde{\mathcal P}_{I} \vvvert \ge \max\{\vvvert \mathcal P_{J} \widetilde{\mathcal P}_{I} \vvvert, \vvvert \mathcal P_{I_0} \widetilde{\mathcal P}_{I} \vvvert\},
\end{align*}
where the last inequality follows from the orthogonality $\mathcal P_{J} \mathcal P_{I_0}=0$. To establish a lower bound on $\vvvert \sin \angle (\mathcal U_{I}, \widetilde{\mathcal U}_{I}) \vvvert $, we focus on bounding $\vvvert \mathcal P_{I_0} \widetilde{\mathcal P}_{I} \vvvert$ from below.

Let $\mathcal P_{I_0} = \mathcal U_{I_0} \mathcal U_{I_0}^\T$. From \eqref{eq:061125}, we obtain $\mathcal U_{I_0}^\T \mathcal E \widetilde{\mathcal U}_{I} \widetilde{\mathcal D}_I^{-1}= \mathcal U_{I_0}^T  \widetilde{\mathcal U}_I,$ which is equivalent to $\mathcal P_{I_0} \widetilde{\mathcal U}_{I} = \mathcal P_{I_0} \mathcal E \widetilde{\mathcal U}_{I} \widetilde{\mathcal D}_I^{-1}$.  By the triangle inequality and \eqref{eq:equivnorms}, we derive
\begin{align}\label{eq:begineq}
\vvvert \mathcal P_{I_0} \widetilde{\mathcal P}_{I} \vvvert = \vvvert \mathcal P_{I_0} \widetilde{\mathcal U}_{I} \vvvert = \vvvert \mathcal P_{I_0} \mathcal E \widetilde{\mathcal U}_{I} \widetilde{\mathcal D}_I^{-1} \vvvert \ge \vvvert \mathcal P_{I_0} \mathcal E {\mathcal U}_{I}\mathcal O \widetilde{\mathcal D}_I^{-1} \vvvert - \vvvert \mathcal P_{I_0} \mathcal E ({\mathcal U}_{I}\mathcal O- \widetilde{\mathcal U}_{I}) \widetilde{\mathcal D}_I^{-1} \vvvert.
\end{align}
Here, $\mathcal O$ is an orthogonal matrix defined as follows. Let $$O_1 := \arg\min_{O\in\mathcal O^{k\times k}}\|U_k O-\widetilde{U}_k\|_F \quad\text{and}\quad O_2 := \arg\min_{O\in\mathcal O^{k\times k}}\|V_k O-\widetilde{V}_k\|_F.$$ Set $$\mathcal O:=\frac{1}{2}\begin{pmatrix}
O_1 + O_2& O_1-O_2\\
O_1 - O_2 & O_1 + O_2
\end{pmatrix}.
 $$
Since $O_1, O_2$ are orthogonal matrices, it is straightforward to verify that $\mathcal O$ is orthogonal.
 
Using the block structure of $\mathcal P_{I_0}$, we can write $\mathcal P_{I_0} = \begin{pmatrix}
Q_1 & 0\\
0&Q_2
\end{pmatrix},
$
where $Q_1$ and $Q_2$ are the orthogonal projections onto $U^{\perp}$ and $V^{\perp} $, respectively. Using the definition of $\mathcal U_I$, we compute
\begin{align*}
\mathcal U_I \mathcal O = \frac{1}{\sqrt 2}\begin{pmatrix}
U_k O_1 & U_k O_1\\
V_k O_2 & -V_k O_2
\end{pmatrix}.
\end{align*}
Consequently, with the definitions of $\mathcal E, \widetilde{\mathcal U}_{I}$ and $\widetilde{\mathcal D}_I$, we further obtain
 \begin{align}\label{eq:asmallcom-1}
 \mathcal P_{I_0} \mathcal E {\mathcal U}_{I}\mathcal O \widetilde{\mathcal D}_I^{-1} = \frac{1}{\sqrt 2}\begin{pmatrix}
Q_2 E^\T U_k O_1\widetilde{D}_k^{-1} & -Q_2 E^\T U_k O_1 \widetilde{D}_k^{-1}\\
Q_1 E V_k O_2\widetilde{D}_k^{-1} & Q_1 E V_k O_2  \widetilde{D}_k^{-1}
\end{pmatrix}
\end{align}
and
\begin{align}\label{eq:asmallcom-2}
({\mathcal U}_{I}\mathcal O- \widetilde{\mathcal U}_{I}) \widetilde{\mathcal D}_I^{-1} =\frac{1}{\sqrt 2}\begin{pmatrix}
(U_k O_1 - \widetilde{U}_k)\widetilde{D}_k^{-1} & -(U_k O_1 - \widetilde{U}_k)\widetilde{D}_k^{-1} \\
(V_k O_2 - \widetilde{V}_k)\widetilde{D}_k^{-1} & (V_k O_2 - \widetilde{V}_k)\widetilde{D}_k^{-1}
\end{pmatrix}.
\end{align}

For the Frobenius norm, we continue from \eqref{eq:begineq} to get
\begin{align}\label{eq:fnorm-lb}
\|\sin \angle (\mathcal U_{I}, \widetilde{\mathcal U}_{I})\|_F \ge \|\mathcal P_{I_0} \widetilde{\mathcal P}_{I}\|_F \ge \|\mathcal P_{I_0} \mathcal E {\mathcal U}_{I}\mathcal O \widetilde{\mathcal D}_I^{-1}\|_F - \|E\| \|({\mathcal U}_{I}\mathcal O- \widetilde{\mathcal U}_{I}) \widetilde{\mathcal D}_I^{-1}\|_F.
\end{align}
For the first term on the right-hand side of \eqref{eq:fnorm-lb}, \eqref{eq:asmallcom-1} yields
\begin{align}\label{eq:061625-1}
\|\mathcal P_{I_0} \mathcal E {\mathcal U}_{I}\mathcal O \widetilde{\mathcal D}_I^{-1}\|_F &= \sqrt{\|Q_2 E^\T U_k O_1\widetilde{D}_k^{-1}\|_F^2 + \|Q_1 E V_k O_2\widetilde{D}_k^{-1} \|_F^2}\nonumber\\
&=\sqrt{\sum_{i=1}^k \frac{1}{\widetilde{\sigma}_i^2}\left(  \|Q_2 E^\T \hat{u}_i\|^2 +  \|Q_1 E \hat{v}_i\|^2 \right)}, 
\end{align}
where we denote the columns of $U_k O_1, V_k O_2$ as $ \hat{u}_i,  \hat{v}_i$ ($1\le i \le k$) respectively. Note that $ \hat{u}_i,  \hat{v}_i$ are deterministic unit vectors. 

By the projection lemma from \cite[Lemma 11.8]{OVW}, for each $i$ and any $t>0$,
\begin{align*}
\Prob\left(|\|Q_1 E \hat{v_i}\|^2 - (N-r)| \ge t \right) \le C\exp\left(-\min\Big\{\frac{t^2}{N-r}, t \Big\}\right)
\end{align*}
and similarly, 
\begin{align*}
\Prob\left(|\|Q_2 E^\T \hat{u_i}\|^2 - (n-r)| \ge t \right) \le C\exp\left(-\min\Big\{\frac{t^2}{n-r}, t \Big\}\right).
\end{align*}
Here, the constant $C>0$ depends on the sub-gaussian norm of the entries of $E$. 

Therefore,  with probability at least $1-Ck\exp(-\frac{1}{4}(n-r))-Ck\exp(-\frac{1}{4}(N-r))$, we have for every $1\le i \le k$,
\begin{align*}
\|Q_1 E \hat{v_i}\|^2 \ge \frac{1}{2}(N-r) \quad\text{and}\quad \|Q_2 E^\T \hat{u_i}\|^2 \ge \frac{1}{2}(n-r).
\end{align*}
Using $\widetilde{\sigma}_i \le \frac{3}{2}\sigma_i$ by Weyl's inequality and our assumption that $\|E\| \le \frac{1}{2}\sigma_i$,  continuing from \eqref{eq:061625-1} gives
\begin{align*}
\|\mathcal P_{I_0} \mathcal E {\mathcal U}_{I}\mathcal O \widetilde{\mathcal D}_I^{-1}\|_F  \ge \frac{\sqrt{2}}{3} \sqrt{N+n-2r} \sqrt{\sum_{i=1}^k \frac{1}{\sigma_i^2}}.
\end{align*}

For the second term on the right-hand side of \eqref{eq:fnorm-lb}, by \eqref{eq:asmallcom-2}, we have
\begin{align*}
\|({\mathcal U}_{I}\mathcal O- \widetilde{\mathcal U}_{I}) \widetilde{\mathcal D}_I^{-1} \|_F &= \sqrt{\|(U_k O_1 - \widetilde{U}_k)\widetilde{D}_k^{-1} \|_F^2 + \|(V_k O_2 - \widetilde{V}_k)\widetilde{D}_k^{-1}\|_F^2 }\\
&\le \frac{1}{\widetilde{\sigma}_k} \sqrt{\|(U_k O_1 - \widetilde{U}_k) \|_F^2 + \|(V_k O_2 - \widetilde{V}_k)\|_F^2 }\\
&\le 2\sqrt{2} \frac{1}{\sigma_k}\sqrt{\|\sin\angle(U_k, \widetilde{U}_k) \|_F^2 + \|\sin\angle(V_k, \widetilde{V}_k)\|_F^2 } \\
&= 2\sqrt{2} \frac{1}{\sigma_k}\| \sin \angle (\mathcal U_{I}, \widetilde{\mathcal U}_{I}) \|_F,
\end{align*}
where in the last inequality, we used \eqref{eq:equiv} and $\widetilde{\sigma}_k \ge \frac{1}{2}\sigma_k$ by Weyl's inequality.

Combining the above discussion with \eqref{eq:fnorm-lb}, we obtain
\begin{align*}
\| \sin \angle (\mathcal U_{I}, \widetilde{\mathcal U}_{I}) \|_F &\ge \frac{\sqrt{2}}{3} \sqrt{N+n-2r} \sqrt{\sum_{i=1}^k \frac{1}{\sigma_i^2}} - 2\sqrt{2} \frac{\|E\|}{\sigma_k}\| \sin \angle (\mathcal U_{I}, \widetilde{\mathcal U}_{I}) \|_F\\
&\ge \frac{\sqrt{2}}{3} \sqrt{N+n-2r} \sqrt{\sum_{i=1}^k \frac{1}{\sigma_i^2}} - \sqrt{2} \| \sin \angle (\mathcal U_{I}, \widetilde{\mathcal U}_{I}) \|_F
\end{align*}
due to $\|E\|/\sigma_k \le 1/2$. Rearranging the terms, we arrive at 
\begin{align*}
\| \sin \angle (\mathcal U_{I}, \widetilde{\mathcal U}_{I}) \|_F  \ge \frac{\sqrt 2}{3(1+\sqrt 2)} \sqrt{N+n-2r} \sqrt{\sum_{i=1}^k \frac{1}{\sigma_i^2}}
\end{align*}
with probability at least $1-Ck\exp(-\frac{1}{4}(n-r))-Ck\exp(-\frac{1}{4}(N-r))$.
The final result follows from \eqref{eq:F-lowbd}:
$$\max\{ \|\sin \angle ( U_{k}, \widetilde{ U}_{k})\|_F, \|\sin \angle ( V_{k}, \widetilde{ V}_{k})\|_F\} \ge \frac{ \sqrt{N+n-2r}}{3(1+\sqrt 2)} \sqrt{\sum_{i=1}^k \frac{1}{\sigma_i^2}}.$$

Finally, the operator norm bound follows  from $\|\sin \angle (\mathcal U_{I}, \widetilde{\mathcal U}_{I})\| \ge \frac{1}{\sqrt{2k}} \|\sin \angle (\mathcal U_{I}, \widetilde{\mathcal U}_{I})\|_F$ and \eqref{eq:O-lowbd}, yielding
\begin{align*}
\max\{\|\sin \angle ( U_{k}, \widetilde{ U}_{k})\|, \|\sin \angle ( V_{k}, \widetilde{ V}_{k})\| \} &= \|\sin \angle (\mathcal U_{I}, \widetilde{\mathcal U}_{I})\| \ge\frac{ \sqrt{N+n-2r}}{3(1+\sqrt 2)\sqrt{k}} \sqrt{\sum_{i=1}^k \frac{1}{\sigma_i^2}}.
\end{align*}
This concludes the proof.

\section{Proofs of Theorems 2.6 and 2.7}\label{sec:general}
\subsection{Proof of Theorem 2.6}
By the min-max theorem, for an $N\times n$ matrix $S$, the $j$th largest singular value of $S$ is 
\begin{align}\label{eq:minmax}
\sigma_j(S) = \max_{ \substack{ W\in \mathbb R^N, \dim(W)=j \\ K \in \mathbb R^n, \dim(K)=j } } \min_{\substack{(x,y) \in W\times K \\ \|x\|=\|y\|=1} } x^T S y.
\end{align}
For the lower bound (8) of $\widetilde \sigma_k$, by \eqref{eq:minmax},
\begin{align*}
\widetilde \sigma_k \ge \min_{\substack{(x,y) \in U_k \times V_k \\ \|x\|=\|y\|=1} } x^T (A+E) y \ge \sigma_k - \max_{\substack{(x,y) \in U_k \times V_k \\ \|x\|=\|y\|=1} } |x^T E y|.
\end{align*}
Note that
\begin{align*}
\|U_k^T E V_k \|= \max_{\substack{(x,y) \in U_k \times V_k \\ \|x\|=\|y\|=1} } |x^T E y|.
\end{align*}
Our assumption immediately yields that $\widetilde \sigma_k \ge \sigma_k - t$ with probability at least $1-\varepsilon$.

For the upper bound (9) of $\widetilde \sigma_k$, by \eqref{eq:minmax},
\begin{align}\label{eq:lbd1}
\sigma_k \ge \widetilde \sigma_k - \max_{\substack{(x,y) \in \widetilde U_k \times \widetilde V_k \\ \|x\|=\|y\|=1} } |x^T E y|.
\end{align}

It is enough to bound the second term on the right side. For any unit vectors $x\in \widetilde U_k$ and $y \in \widetilde V_k$, we decompose
$x = P_{U^{\perp} } x + U U^T x$ and  $y = P_{V^{\perp} } y + V V^T \widetilde y. $ It follows from triangle inequality that
$$ |x^T E y| \le  \| P_{U^{\perp} } x\| \cdot \| P_{V^{\perp} } y \|  \cdot \| E\| + (\| P_{U^{\perp} }x \| + \| P_{V^{\perp} } y \| ) \|E\| + \| U^T E V \|. $$
To bound the term $\| P_{U^{\perp} } x\|$, first notice by Cauchy-Schwarz inequality, 
$$\max_{x\in \widetilde U_k, \|x\|=1 } \| P_{U^{\perp} } x\| \le \sqrt{k} \max_{1\le s \le k} \| P_{U^{\perp} } \widetilde u_s \|.$$
Next, multiplying $P_{U^{\perp} }$ on the left side of the equation $(A+E) \widetilde v_s = \widetilde \sigma_s \widetilde u_s$, we get $P_{U^{\perp} } E \widetilde v_s = \widetilde \sigma_s P_{U^{\perp} } \widetilde u_s.$ It follows that
\begin{align}\label{eq:proj}
\max_{1\le s \le k} \| P_{U^{\perp} } \widetilde u_s \| \le {\| E\| }/{\widetilde \sigma_k}
\end{align}
and thus
$$\max_{x\in \widetilde U_k, \|x\|=1 } \| P_{U^{\perp} } x\| \le \sqrt{k} {\| E\| }/{\widetilde \sigma_k}.$$
The same calculation leads to $$\max_{y\in \widetilde V_k, \|y\|=1 } \| P_{V^{\perp} } y\| \le \sqrt{k} {\| E\| }/{\widetilde \sigma_k}.$$  Therefore,
\begin{align}\label{eq:lbd2}
\max_{\substack{(x,y) \in \widetilde U_k \times \widetilde V_k \\ \|x\|=\|y\|=1} } |x^T E y| \le 2 \sqrt{k}\frac{\| E\|^2 }{\widetilde \sigma_k} + k \frac{\| E\|^3 }{\widetilde \sigma_k^2 } + \| U^T E V \|.
\end{align}
Since $\|U^T E V \| \le L$ and $\|E\| \le B$ with probability at least $1-\varepsilon$, by \eqref{eq:lbd1} and \eqref{eq:lbd2}, we obtain, with probability at least $1-\varepsilon$
$$\widetilde \sigma_k \le \sigma_k + 2\sqrt{k} \frac{B^2}{\widetilde \sigma_k} + k \frac{B^3}{\widetilde \sigma_k^2} + L.$$

\subsection{Proof of Theorem 2.7}
In the proof, we work on the event $$\Omega:=\{ \|E\| \le B \text{ and } \|U^T E V\| \le L \}.$$ By the supposition, $\Omega$ holds with probability at least $1-\varepsilon$.
Observe that $\|U_k^T E V_k\| \le \|U^T E V\| \le L.$
Using (8), the lower bound for $\widetilde \sigma_k$, together with $\delta_k \ge 2 L$, we have 
\begin{align*}
\widetilde \sigma_k \ge \sigma_k - L \ge {\sigma_k}/{2}>0
\end{align*}
and 
\begin{align*}
\widetilde \sigma_k - \sigma_{k+1}= {\widetilde \sigma}_k-\sigma_k + \delta_k \ge \delta_k-L \ge {\delta_k}/{2}.
\end{align*}


From \eqref{eq:sineq} and triangle inequality, we see
\begin{align}\label{eq:gensin}
\vvvert \sin\angle (U_k, \widetilde U_k) \vvvert = \vvvert P_{U_k^{\perp} } P_{\widetilde U_k} \vvvert \le  \vvvert P_{U^{\perp} } P_{\widetilde U_k} \vvvert +  \vvvert P_{U_{k+1,r} } P_{\widetilde U_k} \vvvert \mathbf{1}_{\{ k< r\}}.
\end{align}
We first bound the first term on the right-hand side of \eqref{eq:gensin}. Suppose $P_{U^{\perp}} = U_0 {U_0}^\T$ where the columns of $U_0$ are an orthonormal basis of the subspace $U^{\perp}$. Then
\begin{align*}
 \vvvert P_{U^{\perp} } P_{\widetilde U_k} \vvvert  = \vvvert U_0^\T \widetilde U_k \vvvert. 
\end{align*}
Multiplying $U_0^\T$ on the both sides of $(A+E) \widetilde V_k = \widetilde U_k \widetilde D_k$, we see $U_0^\T E \widetilde V_k = U_0^\T \widetilde U_k \widetilde D_k$ and hence, $ U_0^\T \widetilde U_k = U_0^\T E \widetilde V_k \widetilde D_k^{-1}.$ It follows that
\begin{align}\label{eq:gen1bd}
 \vvvert P_{U^{\perp} } P_{\widetilde U_k} \vvvert = \vvvert U_0^\T E \widetilde V_k \widetilde D_k^{-1} \vvvert \le \vvvert U_0^\T E \widetilde V_k  \vvvert \| \widetilde D_k^{-1} \| = \frac{\vvvert P_{U^{\perp} } E P_{\widetilde V_k}  \vvvert }{\widetilde \sigma_k}. 
\end{align}
In particular, for the operator norm, we have
\begin{align}\label{eq:gen1bd-op}
\|P_{U^{\perp} } P_{\widetilde U_k}\| \le \frac{\|E\|}{\widetilde \sigma_k}.
\end{align}

We proceed to bound the second term on the right-hand side of \eqref{eq:gensin}. It suffices to consider $k<r$. Observe that
\begin{align*}
\vvvert P_{U_{k+1,r} } P_{\widetilde U_k} \vvvert=\vvvert {U_{k+1,r}^\T } {\widetilde U_k} \vvvert &\le \|{U_{k+1,r}^\T } {\widetilde U_k}\|_{*}\\
&\le \sqrt{\rank({U_{k+1,r} }^\T {\widetilde U_k})} \| {U_{k+1,r}^\T } {\widetilde U_k}\|_F\\
&\le \sqrt{\min\{k,r-k \}} \sqrt{\sum_{i=1}^k \|{U_{k+1,r}^\T } {\widetilde u_i}\|^2}.
\end{align*}
In particular, for the operator norm, we have
\begin{align*}
\| P_{U_{k+1,r} } P_{\widetilde U_k} \| \le \| {U_{k+1,r}^\T } {\widetilde U_k}\|_F = \sqrt{\sum_{i=1}^k \|{U_{k+1,r}^\T } {\widetilde u_i}\|^2}.
\end{align*}

It remains to bound $\| U_{k+1,r}^\T {\widetilde u_i}\|$ for $1\le i \le k$. Multiplying $U_{k+1,r}^\T$ on both sides of $(A+E)\widetilde v_i =\widetilde \sigma_i \widetilde u_i$, we get
$$
D_{k+1,r} V_{k+1,r}^\T \widetilde v_i + U_{k+1,r}^\T E \widetilde v_i =\widetilde \sigma_i U_{k+1,r}^\T \widetilde u_i,
$$
which yields 
\begin{align}\label{eq:gensv1}
\widetilde \sigma_i D_{k+1,r} V_{k+1,r}^\T \widetilde v_i = \widetilde \sigma_i^2 U_{k+1,r}^\T \widetilde u_i - \widetilde \sigma_i U_{k+1,r}^\T E \widetilde v_i.
\end{align}
Similarly, multiplying $V_{k+1,r}^\T$ on both sides of $(A^\T+E^\T)\widetilde u_i =\widetilde \sigma_i \widetilde v_i$, we also get
$$
D_{k+1,r} U_{k+1,r}^\T \widetilde u_i + V_{k+1,r}^\T E^\T \widetilde U_i =\widetilde \sigma_i V_{k+1,r}^\T \widetilde v_i,
$$
which implies
\begin{align}\label{eq:gensv2}
\widetilde \sigma_i D_{k+1,r} V_{k+1,r}^\T \widetilde v_i =D_{k+1,r}^2 U_{k+1,r}^\T \widetilde u_i + D_{k+1,r}V_{k+1,r}^\T E^\T \widetilde u_i.
\end{align}
Combining \eqref{eq:gensv1} and \eqref{eq:gensv2}, one has
$$(\widetilde \sigma_i^2 I - D_{k+1,r}^2) U_{k+1,r}^\T \widetilde u_i  = \widetilde \sigma_i U_{k+1,r}^\T E \widetilde v_i + D_{k+1,r}V_{k+1,r}^\T E^\T \widetilde u_i.$$
As a result, by noting $\widetilde \sigma_i^2 I - D_{k+1,r}^2=\diag(\widetilde \sigma_i^2-\sigma_{k+1}^2,\cdots, \widetilde \sigma_i^2-\sigma_{r}^2)$, we obtain the following bound
\begin{align}\label{eq:genui}
\|U_{k+1,r}^\T \widetilde u_i \| &\le \frac{\widetilde \sigma_i \| U_{k+1,r}^\T E \widetilde v_i\| + \sigma_{k+1}\|V_{k+1,r}^\T E^\T \widetilde u_i\| }{\widetilde \sigma_i^2-\sigma_{k+1}^2}\nonumber\\
&\le \frac{ \max\{ \| U_{k+1,r}^\T E \widetilde v_i\|,\|V_{k+1,r}^\T E^\T \widetilde u_i\| \} }{\widetilde \sigma_k-\sigma_{k+1}}.
\end{align}
Now we turn to bound the numerator of the above expression. Decompose 
$\widetilde u_i = P_U \widetilde u_i +  P_{U^\perp} \widetilde u_i$. Then $$V_{k+1,r}^\T E^\T \widetilde u_i = V_{k+1,r}^\T E^\T U U^\T\widetilde u_i +  V_{k+1,r}^\T E^\T  P_{U^\perp} \widetilde u_i$$
and 
\begin{align*}
\| V_{k+1,r}^\T E^\T \widetilde u_i \| &\le \| V_{k+1,r}^\T E^\T U \|  +  \| E^\T \| \cdot\| P_{U^\perp} \widetilde u_i\|\\
&\le \|U^\T E V\| + \frac{\|E\|^2}{\widetilde \sigma_k}.
\end{align*}
The last inequality above follows from $\| P_{U^\perp} \widetilde u_i\| \le \|P_{U^{\perp} } P_{\widetilde U_k}\|$ and by applying \eqref{eq:gen1bd-op}. Likewise, we also have
\begin{align*}
\| U_{k+1,r}^\T E \widetilde v_i\| \le \|U^\T E V\| + \frac{\|E\|^2}{\widetilde \sigma_k}.
\end{align*}
Continuing from \eqref{eq:genui}, we see
\begin{align*}
\|U_{k+1,r}^\T \widetilde u_i \| &\le \frac{\|U^\T E V\| + {\|E\|^2}/{\widetilde \sigma_k}}{\widetilde \sigma_k-\sigma_{k+1}}\\
&\le 2\frac{\|U^\T E V\| }{\delta_k} + 4\frac{\|E\|^2}{\delta_k \sigma_k}
\end{align*}
by plugging in $ \widetilde \sigma_k-\sigma_{k+1}\ge \delta_k/2$ and $\widetilde \sigma_k \ge \sigma_k/2$.
Consequently, 
\begin{align*}
\vvvert P_{U_{k+1,r} } P_{\widetilde U_k} \vvvert &\le \sqrt{\min\{k,r-k \}} \sqrt{\sum_{i=1}^k \|{U_{k+1,r} }^\T {\widetilde u_i}\|^2}\\
&\le 2\sqrt{k \min\{k,r-k \}}\left(\frac{\|U^\T E V\| }{\delta_k} + 2\frac{\|E\|^2}{\delta_k \sigma_k} \right).
\end{align*}
In particular, for the operator norm,
\begin{align*}
\|P_{U_{k+1,r} } P_{\widetilde U_k}\|\le \sqrt{\sum_{i=1}^k \|{U_{k+1,r} }^\T {\widetilde u_i}\|^2}\le 2 \sqrt{k} \left(\frac{\|U^\T E V\| }{\delta_k} + 2\frac{\|E\|^2}{\delta_k \sigma_k} \right).
\end{align*}

Combining the above estimates with \eqref{eq:gen1bd}, and considering \eqref{eq:gensin}, we ultimately arrive at
\begin{align*}
\vvvert \sin\angle (U_k, \widetilde U_k) \vvvert &\le 2 \sqrt{k \min\{k,r-k \}} \left(\frac{\|U^\T E V\| }{\delta_k} + 2\frac{\|E\|^2}{\delta_k \sigma_k} \right) + \frac{\vvvert P_{U^{\perp} } E P_{\widetilde V_k}  \vvvert }{\widetilde \sigma_k}\\
&\le 2 \sqrt{k \min\{k,r-k \}} \left(\frac{\|U^\T E V\| }{\delta_k} + 2\frac{\|E\|^2}{\delta_k \sigma_k} \right) + 2\frac{k \|E\| }{\sigma_k}.
\end{align*}
More specifically, for the operator norm, we have
\begin{align*}
\| \sin\angle (U_k, \widetilde U_k) \| \le 2 \sqrt{k} \left(\frac{\|U^\T E V\| }{\delta_k} + 2\frac{\|E\|^2}{\delta_k \sigma_k} \right)\mathbf{1}_{\{ k< r\}} + 2 \frac{\|E\|}{\sigma_k}.
\end{align*}
By applying the result to $A^\T$ and $A^\T + E^\T$, we observe that the same bounds also hold for $\sin\angle (V_k, \widetilde V_k)$. This concludes the proof.

\section{Proof of Theorem 5.1}\label{sec:app:gmm}

We consider the model (17) and rewrite
$$\E(X) = (\theta_{z_1},\cdots,\theta_{z_n}) = (\theta_{1},\cdots,\theta_{k}) Z^\T,$$
where $Z \in \{0,1\}^{n\times k}$ with entry $Z_{ij}=1$ if $z_i = j$ and $Z_{ij}=0$ otherwise. It is clear that the information regarding the cluster labels $\mathbf z$ is entirely encoded within $Z$. Additionally, let $D=\diag(d_1,\cdots,d_k)$ where $d_i$ represents the cluster size associated with center $\theta_i$. Consequently, the matrix $Z D^{-1/2}$ has orthonormal columns. 

Given that $\theta_i$'s could be colinear, the rank of $\E(X)$ or $(\theta_{1},\cdots,\theta_{k})$, denoted as $r$, could be smaller than the number of clusters $k$. Consider the SVD of $$(\theta_{1},\cdots,\theta_{k}) D^{1/2} = U \Lambda {W}^\T,$$ where $\Lambda$ is a $k\times k$ diagonal matrix with rank $r$ and ${W}$ is a $k\times k$ orthogonal matrix. Observe that if we denote the SVD of $\E(X)$ as $\E(X)=U \Sigma V^\T$ with $U \in \mathbb R^{p\times k}$ and $V \in \mathbb R^{n\times k}$, then the following relationship emerges:
$$\E (X)  =(\theta_{1},\cdots,\theta_{k}) D^{1/2} (ZD^{-1/2})^\T= U \Lambda (Z D^{-1/2}  W)^\T= U \Sigma V^\T.$$
Therefore, $\Sigma = \Lambda=\diag(\sigma_1,\cdots,\sigma_r,0,\cdots,0)$ and 
$V= Z D^{-1/2} W$. Note that the choice of $U,V$ is not unique and we can only decide $U$ or $V$ up to an orthogonal transformation. 

Next, we show that the geometric relationship among the columns of $\E(X)$ is preserved among the columns of $U^\T \E(X)$. Consider the SVD of $\E(X)=U \Sigma V^\T$, where each column $\theta_j$ of $\E(X)$ can be expressed as $\theta_j =U \Sigma (V^\T)_j $ by denoting $(V^\T)_j$ as the column of $V^\T$. Let $(U^\T \E (X))_j$ represent the columns of $U^\T \E (X) = (U^\T \theta_{z_1},\cdots,U^\T \theta_{z_n}) $. For any two columns $\theta_i$ and $\theta_j$ of $\E(X)$, we have
\begin{align*}
\|\theta_i- \theta_j\|^2=(\theta_i- \theta_j)^\T (\theta_i- \theta_j)= \left((V^\T)_i- (V^\T)_j\right)^\T \Sigma^2 \left((V^\T)_i- (V^\T)_j\right) 
\end{align*}
Moreover, their corresponding columns $(U^\T \E (X))_i$ and $(U^\T \E (X))_j$ of $U^\T \E (X)$  satisfy
\begin{align*}
\|(U^\T \E (X))_i- (U^\T \E (X))_j\|^2 &= \| U^\T \theta_i - U^\T \theta_j \|^2 = (\theta_i- \theta_j)^\T  U U^\T(\theta_i- \theta_j)\\
&=\left((V^\T)_i- (V^\T)_j\right)^\T \Sigma U^\T  U U^\T  U\Sigma \left((V^\T)_i- (V^\T)_j\right) \\
&=\|\theta_i- \theta_j\|^2.
\end{align*}
It follows that $$\|(U^\T \E (X))_i- (U^\T \E (X))_j\| = \|\theta_i- \theta_j\|.$$ Therefore, if $i, j \in [n]$ belong to the same cluster, then $\|(U^\T \E (X))_i- (U^\T \E (X))_j\|=0$. On the other hand, if $i, j \in [n]$ belong to the distinct clusters, then $\|(U^\T \E (X))_i- (U^\T \E (X))_j\|\ge \Delta.$

The main step of the proof, as explained in Section 5.1, is to prove (18). Specifically, we aim to show that 
\begin{align}\label{eq:gmmineq1}
\max_{1\le j \le n}\|(\widetilde U_{k}^T X)_j - (U^\T \E (X))_j\| < \frac{1}{5} \Delta
\end{align}
holds with high probability. 

Recall that throughout the paper, we always assume $\|E\| \le 2(\sqrt n + \sqrt p)$. We start with the decompositions 
$$U^\T \E(X)=\Lambda V^\T = \begin{pmatrix} \Lambda_r V_r^\T \\
0
\end{pmatrix}
$$
and 
$$\widetilde U_{r}^T X = {\widetilde \Lambda_k} {\widetilde V_k^\T} = \begin{pmatrix}
{\widetilde \Lambda_r} {\widetilde V_r^\T}\\
{\widetilde \Lambda_{\llbracket r+1, k\rrbracket}} {\widetilde V_{\llbracket r+1, k\rrbracket}^\T}
\end{pmatrix}.
$$
Observe that $$\|{\widetilde \Lambda_{\llbracket r+1, k\rrbracket}} {\widetilde V_{\llbracket r+1, k\rrbracket}^\T}\| \le \widetilde \sigma_{r+1} \le \|E\| \le \frac{1}{20}\Delta$$
by Weyl's inequality and $\Delta \ge \sigma_r \ge 20 \|E\|$ from the supposition. Hence,
\begingroup
\allowdisplaybreaks
\begin{align}\label{eq:gmmbd}
&\max_{1\le j \le n}\|(\widetilde U_{k}^T X)_j - (U^\T \E (X))_j\|\nonumber \\
&\le \max_{1\le j \le n}\left(\|({\widetilde \Lambda_r} {\widetilde V_r^\T})_j - (\Lambda_r V_r^\T)_j\| +\|({\widetilde \Lambda_{\llbracket r+1, k\rrbracket}} {\widetilde V_{\llbracket r+1, k\rrbracket}^\T})_j\|\right) \nonumber\\
&\le \max_{1\le j \le n}\|({\widetilde \Lambda_r} {\widetilde V_r^\T})_j - (\Lambda_r V_r^\T)_j\| +\frac{1}{20}\Delta\nonumber\\
&=\max_{1\le j \le n }\| e_j^\T{\widetilde V_r} {\widetilde \Lambda_r} - e_j^\T V_r \Lambda_r \|+\frac{1}{20}\Delta =\| {\widetilde V_r} {\widetilde \Lambda_r} - V_r \Lambda_r\|_{2,\infty}+\frac{1}{20}\Delta\nonumber \\
&\le \| {\widetilde V_r} {\widetilde \Lambda_r} - V_r \widetilde{\Lambda}_r\|_{2,\infty} + \| V_r (\widetilde{\Lambda}_r -  \Lambda_r)\|_{2,\infty} + \frac{1}{20}\Delta.
\end{align}
\endgroup
Due to non-uniqueness of the choice of $V$ in the SVD of $\E(X)$, we choose a specified $V_r$ such that the conclusion of Corollary 2.14 holds (with $b=20$): that is, with probability at least $1 - 40(N+n)^{-L}$,
$$\| \widetilde{V}_{r}\widetilde{\Lambda}_{r} - V_r \widetilde{\Lambda}_{r}\|_{2,\infty} \le 45 r \sqrt{(L+7)\log (n+p)} (1+\|V_r\|_{2,\infty})+ 8 \|V_r\|_{2,\infty} \frac{(\sqrt n + \sqrt p)^2}{\sigma_r}.$$
Note that 
\begin{align*}
\|V_r\|_{2,\infty}\le \|V\|_{2,\infty}=\max_{i} \|e_i^\T V\| =\max_{i} \sqrt{\frac{1}{d_{z_i}} \sum_{j=1}^k W_{z_i, j}^2}\le \frac{1}{\sqrt{c_{\min}}}\le 1.
\end{align*}
Continuing from \eqref{eq:gmmbd}, we further obtain
\begin{align*}
&\max_{1\le j \le n}\|(\widetilde U_{k}^T X)_j - (U^\T \E (X))_j\|\\
&\le \| {\widetilde V_r} {\widetilde \Lambda_r} - V_r \widetilde{\Lambda}_r\|_{2,\infty} + \| V_r\|_{2,\infty} \|\widetilde{\Lambda}_r -  \Lambda_r\| + \frac{1}{20}\Delta\\
&\le 90 k \sqrt{(L+7)\log (n+p)} + \frac{8(\sqrt n + \sqrt p)^2}{\sqrt{c_{\min}} \sigma_r} + \frac{2(\sqrt n + \sqrt p)}{\sqrt{c_{\min}}} + \frac{1}{20}\Delta\\
&\le \frac{1}{20}\Delta + \frac{1}{200}\Delta + \frac{1}{20}\Delta + \frac{1}{20}\Delta<\frac{1}{5}\Delta
\end{align*}
by Weyl's inequality $\|\widetilde{\Lambda}_r -  \Lambda_r\|\le \|E\|$ and the suppositions that $\sigma_r\ge 40(\sqrt n + \sqrt p)$ and 
$$\Delta \ge \max\left\{ \frac{40(\sqrt n + \sqrt p)}{\sqrt{c_{\min}}}, 1800k\sqrt{(L+7)\log(n+p)}  \right\}.$$This concludes the proof.

\section{Proofs of \eqref{eq:equiv}, \eqref{eq:equivup}, Proposition \ref{prop:2inf} and \eqref{eq:svdUV}}\label{sec:otherproofs}
\subsection{Proof of \eqref{eq:equiv}}\label{app:proofofnorm}
It follows from \eqref{eq:sp} that for the Schatten $p$-norm with $p\ge 2$, 
\begin{align}\label{eq:upsp}
\min_{O\in \mathbb{O}^{r\times r}} \| UO- V\|_p^2 &= \min_{O\in \mathbb{O}^{r\times r}} \| (UO- V)^\T (UO- V)\|_{p/2} \nonumber\\
&=\min_{O\in \mathbb{O}^{r\times r}} \| 2I_r - O^\T U^\T V - V^\T U O\|_{p/2}.
\end{align}
Note that the SVD of $U^\T V$ can be written as $O_1^\T \cos\Theta O_2$ where $O_1,O_2$ are orthogonal matrices and $\cos\Theta:=\cos\angle(U,V)=\diag(\cos \theta_1,\cdots, \cos \theta_r)$. Continuing from \eqref{eq:upsp}, by the definition of unitarily invariant norms, we further have
\begin{align*}
\eqref{eq:upsp} &= \min_{O\in \mathbb{O}^{r\times r}} \| 2 I_r - \cos\Theta \cdot O_2O^\T O_1^\T - O_1 O O_2^\T \cos\Theta\|_{p/2} \nonumber\\
&\le \| 2I_r - 2 \cos\Theta\|_{p/2} = 2 \left( \sum_{i=1}^r (1-\cos\theta_i)^{p/2}\right)^{2/p} \\
&\le 2 \left( \sum_{i=1}^r \sin^{p}\theta_i\right)^{2/p}=2\|\sin\Theta\|_p^2,
\end{align*}
where we denote $\sin\Theta:=\sin\angle(U,V)$. In the first inequality, we choose $O=O_1^\T O_2$. The second inequality follows from $1-\cos\theta_i \le 1-\cos^2\theta_i=\sin^2\theta_i.$

For the lower bound, we still consider
\begin{align*}
\eqref{eq:upsp}&= \min_{O\in \mathbb{O}^{r\times r}} \| 2 I_r - \cos\Theta \cdot O_2O^\T O_1^\T - O_1 O O_2^\T \cos\Theta\|_{p/2} \nonumber\\
&=\min_{Y\in \mathbb{O}^{r\times r}} \| 2 I_r - \cos\Theta \cdot Y - Y^\T \cos\Theta\|_{p/2}
:=\min_{Y\in \mathbb{O}^{r\times r}} \| B_Y \|_{p/2}.
\end{align*}
Observe that $B_Y$ is positive semidefinite. To see this, let $x$ be an arbitrary unit vector in $\mathbb R^r$ and we have
\begin{align*}
x^\T B_Yx = 2 - 2 x^\T \cos\Theta\cdot Y x &\ge 2(1- |x^\T \cos\Theta\cdot Y x|) \ge 2 (1-\|\cos\Theta\|)\ge 0.
\end{align*}
Denote $p'=p/2$ for brevity. We use the following variational formula for the Schatten norms of positive semidefinite matrices:
\begin{align}\label{eq:varsch}
\|B_Y\|_{p'} =\max_{\|X\|_{q'}\le 1} \tr (B_Y X),
\end{align}
where $\|X\|_{q'}$ is the Schatten-$q'$ norm of $X\in \mathbb R^{r\times r}$ and $\frac{1}{p'} + \frac{1}{q'}=1$. To prove \eqref{eq:varsch}, by the H\"older's inequality for Schatten norms, 
$$\max_{\|X\|_{q'}\le 1} \tr (B_Y X) \le \|B_Y\|_{p'}\max_{\|X\|_{q'}\le 1} \|X\|_{q'} \le \|B_Y\|_{p'}.$$
On the other hand, taking $X=B_Y^{p'-1}/\|B_Y^{p'-1}\|_{q'}$,
$$\max_{\|X\|_{q'}\le 1} \tr (B_Y X) \ge \tr(B_Y^{p'})/\|B_Y^{p'-1}\|_{q'}= \|B_Y\|_{p'}^{p'}/\|B_Y\|_{p'}^{p'-1}=\|B_Y\|_{p'},$$
where we used $\tr(B_Y^{p'}) = \|B_Y\|_{p'}^{p'}$ since $B_Y$ is positive semidefinite and $\|B_Y^{p'-1}\|_{q'} = \|B_Y\|_{p'}^{p'-1}$ due to $\frac{1}{p'} + \frac{1}{q'}=1$. This proves \eqref{eq:varsch}.

Set $S=\sin^2\Theta$ for simplicity and let $$X=\frac{S^{p'-1}}{\|S^{p'-1}\|_{q'} }=\frac{S^{p'-1}}{\|S\|_{p'}^{p'-1}}.$$ We continue from \eqref{eq:varsch}:
\begin{align*}
\|B_Y\|_{p'} &\ge \tr\left(2 X - \cos\Theta \cdot Y X - Y^\T \cos\Theta X  \right)=2\tr(X-X\cos\Theta \cdot Y)\\
&\ge 2\left( \tr(X) - \|X\cos\Theta\|_{*}\right),
\end{align*}
where we applied H\"older's inequality $\left|\tr(X\cos\Theta \cdot Y ) \right| \le \|Y\| \cdot \|X\cos\Theta \|_{*}=\|X\cos\Theta \|_{*}.$ Plugging in $X$ and $S$, we get
\begin{align*}
\|B_Y\|_{p'} &\ge \frac{2}{\|S\|_{p'}^{p'-1}}\left( \tr(S^{p'-1}) - \|S^{p'-1} \cos\Theta\|_* \right)\\
&= \frac{2}{\|\sin^2\Theta\|_{p'}^{p'-1}}\left( \sum_{i=1}^r (\sin\theta_i)^{2(p'-1)} (1-\cos\theta_i) \right).
\end{align*}
Note that $\|\sin^2\Theta\|_{p'}^{p'-1} = \|\sin\Theta\|_{2p'}^{2(p'-1)}$ and $$1-\cos\theta_i = 2\sin^2(\theta_i/2) = \frac{\sin^2\theta_i}{2\cos^2(\theta_i/2)}\ge \frac{1}{2}\sin^2\theta_i.$$We further obtain
\begin{align*}
\|B_Y\|_{p'} &\ge \frac{1}{\|\sin\Theta\|_{2p'}^{2(p'-1)}}  \sum_{i=1}^r (\sin\theta_i)^{2 p'} = \frac{\|\sin\Theta\|_{2p'}^{2p'}}{\|\sin\Theta\|_{2p'}^{2(p'-1)}} = \|\sin\Theta\|_{2p'}^{2}=\|\sin\Theta\|_{p}^{2}.
\end{align*}
Therefore,
\begin{align*}
\eqref{eq:upsp}=\min_{Y\in \mathbb{O}^{r\times r}} \| B_Y \|_{p/2}\ge \|\sin\Theta\|_{p}^{2}.
\end{align*}
This completes the proof of \eqref{eq:equiv}.

\subsection{Proof of \eqref{eq:equivup}}\label{app:proofequi}
For simplicity, denote $\cos\Theta = \cos\angle(U,V)$ and $\sin\Theta = \sin\angle(U,V)$. Note that for any orthogonal matrices $Y, Z \in \mathbb R^{r\times r}$,
$$\vvvert UYZ^\T- V\vvvert =\vvvert UY- VZ\vvvert.$$
We use \cite[Theorem VII.1.8]{Bhatia}: there exist $r\times r$ orthogonal matrices $Y, Z$ and $n\times n$ orthogonal matrix $Q$ such that if $2r\le n$, then
\begin{equation*}
QUY= \begin{pmatrix}
I_r\\
0\\
0
\end{pmatrix} \quad \text{and}\quad
QUZ= \begin{pmatrix}
\cos\Theta\\
\sin\Theta\\
0
\end{pmatrix}.
\end{equation*}
Hence, 
\begin{equation*}
\min_{O\in \mathbb{O}^{r\times r}} \vvvert UO- V\vvvert \le \vvvert UYZ^\T- V\vvvert =\vvvert QUY- QVZ\vvvert= \vvvert 
\begin{pmatrix}
I_r - \cos\Theta \\
-\sin\Theta
\end{pmatrix}
\vvvert.
\end{equation*}
If $2r>n$, then 
\begin{equation*}
QUY= \begin{pmatrix}
I_{n-r} & 0\\
0 & I_{2r-n}\\
0&0
\end{pmatrix} \quad \text{and}\quad
QUZ= \begin{pmatrix}
\cos\Theta_1& 0 \\
0 & I_{2r-n}\\
\sin\Theta_1 & 0
\end{pmatrix},
\end{equation*}
where $\Theta_1$ is a diagonal matrix composed of the largest $n-r$ diagonal entries of $\Theta$ (note that the remaining diagonal entries of $\Theta$ are all zero). Therefore, by unitary equivalent, we still have 
\begin{equation}\label{eq:mindis}
\min_{O\in \mathbb{O}^{r\times r}} \vvvert UO- V\vvvert \le \vvvert UYZ^\T- V\vvvert =\vvvert QUY- QVZ\vvvert= \vvvert 
\begin{pmatrix}
I_r - \cos\Theta \\
-\sin\Theta
\end{pmatrix}
\vvvert.
\end{equation}
Note that the matrix on the right-hand side of \eqref{eq:mindis} has singular values $$\sqrt{(1-\cos\theta_i)^2 + \sin^2\theta_i} = 2\sin\left(\frac{\theta_i}{2} \right)$$ for $i=1,\cdots, r$. Then by Theorem \ref{thm:char},
\begin{equation*}
\min_{O\in \mathbb{O}^{r\times r}} \vvvert UO- V\vvvert \le f(2 \sin(\theta_1/2), \cdots,2 \sin\left(\theta_r/2)\right)
\end{equation*}  for the symmetric gauge function $f$ associated with the norm. Combining the above fact with the inequality $\sin (\theta/2) =\frac{1}{2} \frac{\sin \theta}{\cos(\theta/2)} \le \frac{\sin\theta}{\sqrt 2}$ for $\theta \in [0,\pi/2]$ and Theorem \ref{thm:fan}, we get the bound 
\begin{equation*}
\min_{O\in \mathbb{O}^{r\times r}} \vvvert UO- V\vvvert \le f(\sqrt{2} \sin(\theta_1), \cdots,\sqrt{2} \sin(\theta_r)) =\vvvert \sqrt{2}\sin\angle (U,V)  \vvvert.
\end{equation*}


\subsection{Proof of Proposition \ref{prop:2inf}}\label{app:2inf}
For any orthogonal matrix $O$, we first have
\begin{align*}
\|x^\T (V- UO) \| &\le \|x^\T (V - P_U V)\| + \| x^\T (U U^\T V - UO)\|\\
&=\|x^\T (V - P_U V)\| + \| x^\T U ( U^\T V - O)\|\\
&\le \|x^\T (V - P_U V)\| + \| x^\T U\| \| U^\T V - O\|.
\end{align*}
It remains to estimate $\| U^\T V - O\|$. Now consider a specific orthogonal matrix $O=O_1 O_2^\T$, where as per \eqref{eq:svdUV}, we have $U^\T V = O_1 \cos\angle(U,V) O_2^\T$.  Hence,
\begin{align*}
\| U^\T V - O\| &= \| O_1 \cos\angle(U,V) O_2^\T - O_1 O_2^\T\| = \|\cos\angle(U,V)-I_r\|\\
&=1-\cos\theta_r\le 1-\cos^2\theta_r = \sin^2\theta_r = \|\sin\angle(U, V)\|^2.
\end{align*}
Putting these estimates together, we arrive at 
\begin{align*}
\|x^\T (V- UO) \| &\le \|x^\T (V - P_U V)\| + \| x^\T U\| \|\sin\angle(U, V)\|^2.
\end{align*}

The other inequalities can be proved immediately by noting that
\begin{align*}
\left|x^\T (V- UO) y \right| &\le \left|x^\T (V- P_U V) y \right| + \left|x^\T (UU^\T V- UO) y \right|\\
&\le \left|x^\T (V- P_U V) y \right| + \|x^\T U\| \|U^\T V- O\|
\end{align*}
and
\begin{align*}
\|V- UO\|_{2,\infty} &\le \|V - P_U V\|_{2,\infty} +\|UU^\T V- UO\|_{2,\infty}\\
&\le \|V - P_U V\|_{2,\infty} +\|U\|_{2,\infty} \|U^\T V- O\|.
\end{align*}

\subsection{Proof of \eqref{eq:svdUV}}\label{app:eq:weightedbdnew}
As in \eqref{eq:svdUV}, from the SVD of ${U}_k^\T \widetilde{U}_{k} = O_1 \cos\angle(U_k, \widetilde{U}_{k}) O_2^\T$, we choose the orthogonal matrix $O=O_1 O_2^\T$. For notational simplicity, let us denote $\cos\angle(U_k, \widetilde{U}_{k}) =\diag(\cos\theta_1,\cdots,\cos\theta_k):= \cos \Theta$. 

Using a similar argument as in the proof of Proposition \ref{prop:2inf}, we obtain
\begin{align}\label{eq:240319}
\| \widetilde{U}_{k}\widetilde{D}_{k} - {U}_kO\widetilde{D}_{k}\|_{2,\infty}\le &\| \widetilde{U}_{k}\widetilde{D}_{k} - P_{U_k} \widetilde{U}_{k} \widetilde{D}_{k}\|_{2,\infty} + \|{U}_k\|_{2,\infty} \|({U}_k^\T \widetilde{U}_{k} - O)\widetilde{D}_{k}\|.
\end{align}
It suffices to establish a bound for $\|({U}_k^\T \widetilde{U}_{k} - O)\widetilde{D}_{k}\|$ in the second term on the right-hand side of \eqref{eq:240319}. A similar bound for the case $k=r$ has  been  previously established in \cite[Lemma 15]{YW24}. We generalize the proof from \cite[Lemma 15]{YW24} here with explicit constants.

Let $U_0$ denote  the matrix whose columns are orthonormal and span the complement of the subspace $U_k$. We first show that 
\begin{align}\label{eq:rewrite0319}
O - {U}_k^\T \widetilde{U}_{k} = 2 O_1 \cdot\sin(\Theta/2)\cdot (\sin\Theta)^{-1}\sin(\Theta/2) \cdot O_3^\T U_0^\T \widetilde{U}_{k}
\end{align}
for some orthogonal matrix $O_3$. Here, $\sin(\Theta/2) =\diag(\sin(\theta_1/2),\cdots, \sin(\theta_k/2))$ and $$(\sin\Theta)^{-1}\sin(\Theta/2) = \diag\left( \frac{\sin(\theta_1/2)}{\sin(\theta_1)}, \cdots, \frac{\sin(\theta_k/2)}{\sin(\theta_k)}\right).$$
To see \eqref{eq:rewrite0319}, since
\begin{align*}
(\widetilde{U}_{k}^\T U_0)(\widetilde{U}_{k}^\T U_0)^\T = \widetilde{U}_{k}^\T U_0 U_0^\T \widetilde{U}_{k} &=  \widetilde{U}_{k}^\T (I-U_k U_k^\T) \widetilde{U}_{k}\\
& = I - O_2 \cos^2(\Theta) O_2^\T=O_2\sin^2(\Theta)O_2^\T,
\end{align*}
the SVD of $\widetilde{U}_{k}^\T U_0$ is given by 
\begin{align}\label{eq:svd0319}
\widetilde{U}_{k}^\T U_0 = O_2 (\sin \Theta) O_3^\T
\end{align}
for some orthogonal matrix $O_3$. Combining \eqref{eq:svd0319} with 
$$O-{U}_k^\T \widetilde{U}_{k} = O_1(I - \cos\Theta) O_2^\T =2 O_1 \sin^2(\Theta/2) O_2^\T,$$
we prove \eqref{eq:rewrite0319} and consequently
\begin{align}\label{eq:rewrite031901}
(O - {U}_k^\T \widetilde{U}_{k}) \widetilde{D}_{k} = 2 O_1 \cdot\sin(\Theta/2)\cdot (\sin\Theta)^{-1}\sin(\Theta/2) \cdot O_3^\T U_0^\T \widetilde{U}_{k}\widetilde{D}_{k}.
\end{align}

To bound $ \|(O-{U}_k^\T \widetilde{U}_{k} )\widetilde{D}_{k}\|$, first observe 
\begin{align*}
\|\sin(\Theta/2)\| \le \frac{\sqrt 2}{2}\|\sin\Theta\| \quad \text{and} \quad \|(\sin\Theta)^{-1}\sin(\Theta/2)\| \le \frac{\sqrt 2}{2}
\end{align*}
by the facts that $\cos(\theta/2) \ge 1/\sqrt{2}$ and $\sin(\theta/2)=\frac{\sin\theta}{2\cos(\theta/2)}\le \frac{\sqrt 2}{2}\sin\theta$ for $\theta\in [0,\pi/2]$. Then continuing from \eqref{eq:rewrite031901}, we have
\begin{align*}
\|(O - {U}_k^\T \widetilde{U}_{k}) \widetilde{D}_{k}\| \le \|\sin\Theta\| \cdot \|U_0^\T \widetilde{U}_{k}\widetilde{D}_{k}\|. 
\end{align*}
It remains to bound $\|U_0^\T \widetilde{U}_{k}\widetilde{D}_{k}\|$. Let us denote $D_{k+1,r} = \diag(\sigma_{k+1},\cdots,\sigma_r)$ and $U_{k+1,r} = (u_{k+1},\cdots,u_r)$. Define $V_{k+1,r}$ analogously. Since $ \widetilde{U}_{k}\widetilde{D}_{k} = \widetilde{A} \widetilde{V}_{k} = (A+E)\widetilde{V}_{k}$, we have
\begin{align*}
\|U_0^\T \widetilde{U}_{k}\widetilde{D}_{k}\| = \| U_0^\T(A+E)\widetilde{V}_{k}\| &\le \| U_0^\T A \widetilde{V}_{k}\| +\| U_0^\T E \widetilde{V}_{k}\|\\
&\le \| D_{k+1,r} V_{k+1,r}^\T \widetilde{V}_{k}\| + \|E\| \\
&\le \sigma_{k+1} \| \sin\angle ( V_{k}, \widetilde{ V}_{k})\| + \|E\|.
\end{align*}
To obtain the last inequality above, let $V_0$ be the matrix whose columns are orthonormal and span the complement of the subspace $V_k$. Then
\begin{align*}
\| D_{k+1,r} V_{k+1,r}^\T \widetilde{V}_{k}\| \le \sigma_{k+1} \|V_{k+1,r}^\T \widetilde{V}_{k}\|&\le \sigma_{k+1} \|V_{0}^\T \widetilde{V}_{k}\| = \sigma_{k+1}\|P_{V_k^\perp}P_{\widetilde{V}_{k}} \|\\
&\le \sigma_{k+1} \|P_{V_k} -P_{\widetilde{V}_{k}}\| =\sigma_{k+1}\| \sin\angle ( V_{k}, \widetilde{ V}_{k})\|.
\end{align*}
We obtain that 
\begin{align*}
\|(O - {U}_k^\T \widetilde{U}_{k}) \widetilde{D}_{k}\| \le \|\sin\angle(U_k, \widetilde{U}_{k})\| \left(\sigma_{k+1} \| \sin\angle ( V_{k}, \widetilde{ V}_{k})\| + \|E\| \right). 
\end{align*}
This completes the proof. 

\subsubsection{Comparing $\| \widetilde{U}_{r}\widetilde{D}_{r} - {U}O\widetilde{D}_{r}\|_{2,\infty}$ and $\|\widetilde{U}_r \widetilde{D}_r O^\T- U D \|_{2,\infty}$}
This section examines the connections between the quantity $\| \widetilde{U}_{r}\widetilde{D}_{r} - {U}O\widetilde{D}_{r}\|_{2,\infty}$ studied in this paper and $\|\tilde{U}_r \tilde{D}_r O^\T- U D \|_{2,\infty}$ found in prior works (e.g., Proposition 3 in \cite{YW24}). The orthogonal matrix $O:=O_1 O_2^\T$ is derived from the SVD of ${U}^\T \widetilde{U}_{r} = O_1 \cos\angle(U, \widetilde{U}_{r}) O_2^\T$. By rewriting 
\[ \widetilde{U}_r \widetilde{D}_r O^\T- U D =(\widetilde{U}_r \widetilde{D}_r - UO \widetilde{D}_r)O^\T +  U (O \widetilde{D}_rO^\T - {D} ), \]
we observe that (since $\|MO\|_{2,\infty} \le \|M\|_{2,\infty}$ for any orthogonal matrix $O$)
\[ \left| \| \widetilde{U}_r \widetilde{D}_r O^\T- U D \|_{2,\infty} - \|\widetilde{U}_r \widetilde{D}_r - UO \widetilde{D}_r\|_{2,\infty} \right| \le \|U\|_{2,\infty} \|O \widetilde{D}_rO^\T - {D}\|. \]
While establishing a tight analytical bound on $\|\widetilde{D}_r O^\T - O^\T D\| = \|O\widetilde{D}_r O^\T - D\|$ remains challenging, our numerical experiments in Figure \ref{fig:difference} demonstrate that $\|\widetilde{D}_r O^\T - O^\T D\|$ is comparable in magnitude to $\|R\widetilde{D}_r O^\T - D\|$, where $R:= X Y^\T$ is an orthogonal matrix derived from the SVD of ${V}^\T \widetilde{V}_{r} = X \cos\angle(V, \widetilde{V}_{r}) Y^\T$. As shown in \cite[Lemma 16]{YW24},  with high probability  
\[\|R\widetilde{D}_r O^\T - D\| \lesssim \frac{\|E\|^2}{\sigma_r} + \sqrt{r+\log(N+n)}. \]
Our numerical results suggest that a similar bound may hold for $\|O\widetilde{D}_r O^\T -  D\|$. 

Figure \ref{fig:difference2} presents a comparison between $\| \widetilde{U}_{r}\widetilde{D}_{r} - {U}O\widetilde{D}_{r}\|_{2,\infty}$ and $\|\widetilde{U}_r \widetilde{D}_r O^\T- U D \|_{2,\infty}$. The numerical evidence indicates that  $\| \widetilde{U}_{r}\widetilde{D}_{r} - {U}O\widetilde{D}_{r}\|_{2,\infty}$ is typically smaller than $\|\widetilde{U}_r \widetilde{D}_r O^\T- U D \|_{2,\infty}$, though both quantities increase with the rank $r$. Given that \cite[Proposition 3]{YW24} establishes that with high probability that \[ \|\widetilde{U}_r \widetilde{D}_r O^\T- U D \|_{2,\infty} \lesssim \sqrt{r+\log(N+n)} + \|U\|_{2,\infty} \frac{\|E\|^2}{\sigma_r},\] this observation suggests potential for improving our Corollary 2.14 by replacing the term $r\sqrt{\log(N+n)}$ with $\sqrt{r+\log(N+n)}$. The theoretical verification of this observation remains an open question for future research.

\begin{figure}[!ht]
\centering
\includegraphics[width=0.9\textwidth]{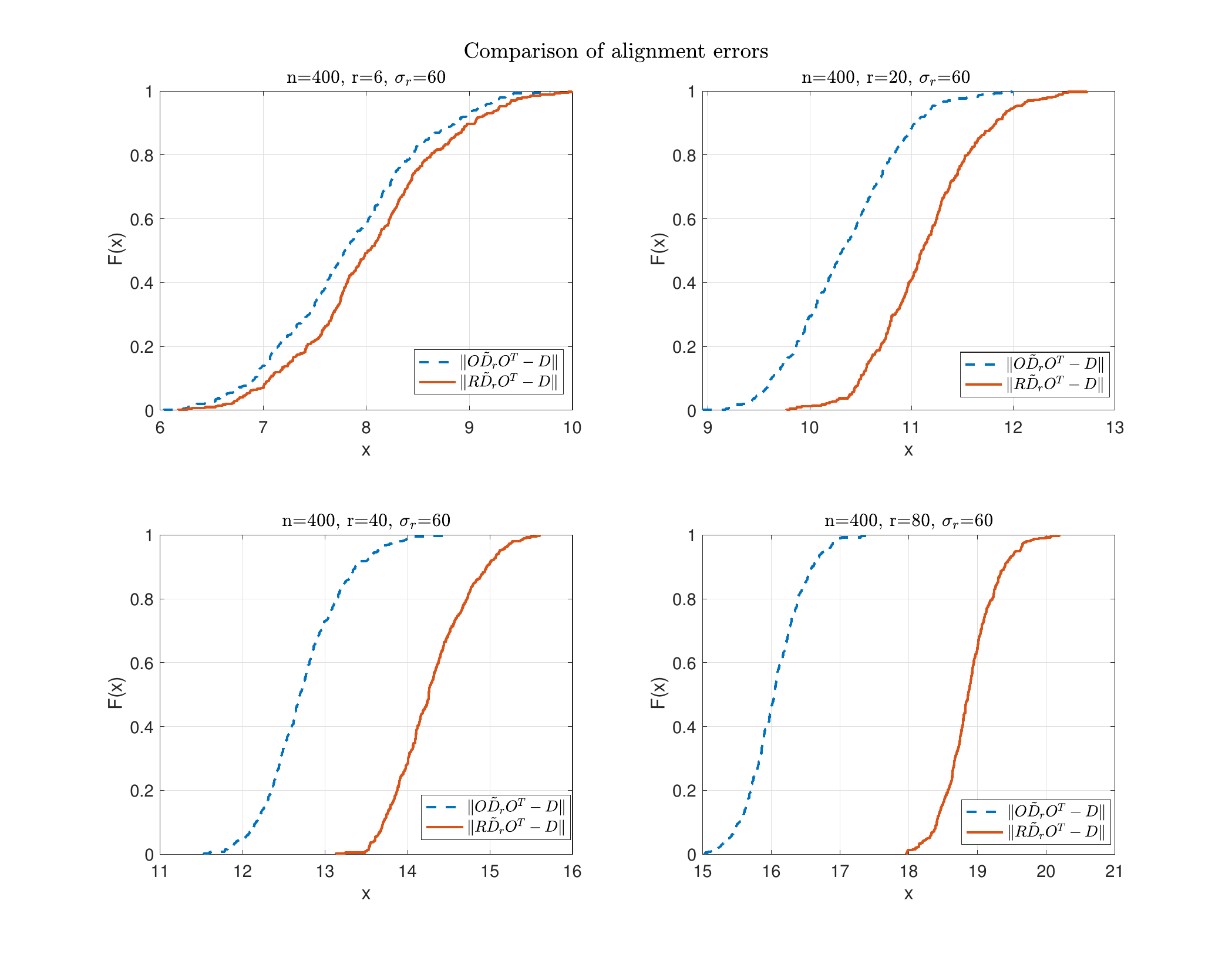}
   \caption{CDF plots comparing the errors $\|O\widetilde{D}_r O^\T - D\|$ (blue dotted curve) and $\|R\widetilde{D}_r O^\T - D\|$ (red curve) across 400 trials. The signal matrix $A = UDV^\T$ has rank $r$, with its right and left singular vectors $U$ and $V$ generated independently from Haar-distributed orthonormal matrices. The noise matrix $E$ has i.i.d. standard Gaussian entries. We set the smallest singular value of $A$ to $\sigma_r=60$. The four figures correspond to $r=6, 20, 40, 80$ respectively.}
\label{fig:difference}
\end{figure}

\begin{figure}[!ht]
\centering
\includegraphics[width=0.9\textwidth]{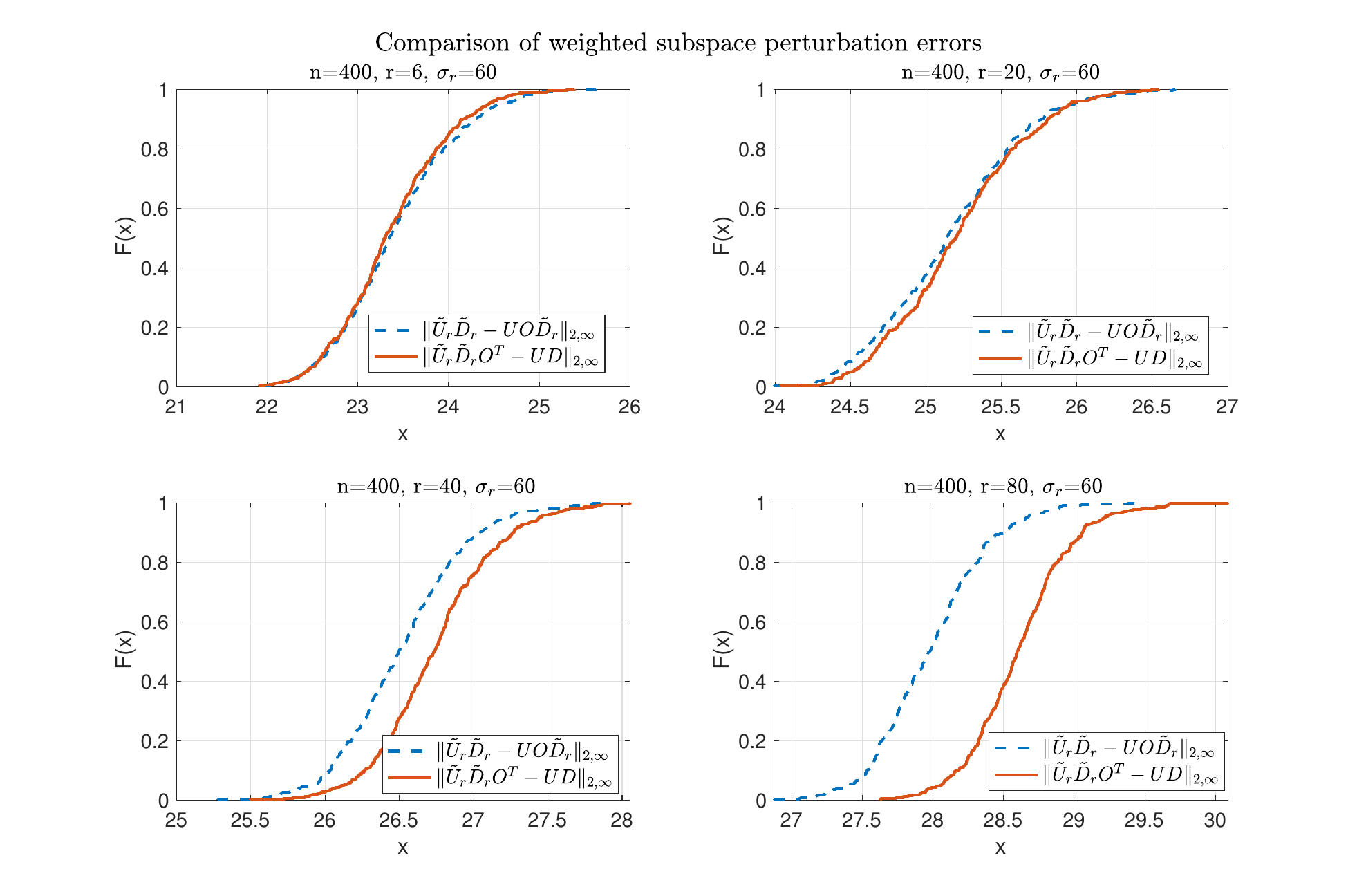}
   \caption{CDF plots comparing the weighted subspace perturbations $\| \widetilde{U}_{r}\widetilde{D}_{r} - {U}O\widetilde{D}_{r}\|_{2,\infty}$ (blue dotted curve) and $\|\widetilde{U}_r \widetilde{D}_r O^\T- U D \|_{2,\infty}$ (red curve) across 400 trials. The signal matrix $A = UDV^\T$ has rank $r$, with its right and left singular vectors $U$ and $V$ generated independently from Haar-distributed orthonormal matrices. The noise matrix $E$ has i.i.d. standard Gaussian entries. We set the smallest singular value of $A$ to $\sigma_r=60$. The four figures correspond to $r=6, 20, 40, 80$ respectively.}
\label{fig:difference2}
\end{figure}

\section{Proofs of Lemma 6.2 and Lemma 6.3}\label{app:twolem}
\subsection{Proof of Lemma 6.2}\label{app:lem:iso}
By the rotational invariance of $E$ and definition of $\mathcal E$, we observe that for any orthogonal matrices $$\mathcal O_1 =\begin{pmatrix}
O_1 & 0\\
0& O_2
\end{pmatrix}, \,
\mathcal O_2 =\begin{pmatrix}
\widehat O_1 & 0\\
0& \widehat O_2
\end{pmatrix}
$$
where $O_1,\widehat O_1 \in \mathbb R^{N\times N}$ and $O_2,\widehat O_2 \in \mathbb R^{n \times n}$ are orthogonal matrices, $$\mathcal O_1 \mathcal E \mathcal O_2 \sim \mathcal E.$$ Consequently,
$$\mathbf x^\T \left(G(z)-\Phi(z) \right) \mathbf y \sim ( \mathcal O_1 \mathbf x)^\T \left(G(z)-\Phi(z) \right) (\mathcal O_2\mathbf y).$$ Hence, it suffices to assume $\mathbf x = (x_1, 0\cdots,0,x_{N+1},0,\cdots,0)^\T$ with $x_1^2 + x_{N+1}^2=1$ and $\mathbf y = (y_1, 0\cdots,0,y_{N+1},0,\cdots,0)^\T$ with $y_1^2 + y_{N+1}^2=1$. Furthermore, 
\begin{align}\label{eq:reduceG}
&\mathbf x^\T \left(G(z)-\Phi(z) \right) \mathbf y \nonumber\\
&= x_1 y_1 (G_{11}(z)-\Phi_{11}(z)) + x_{N+1} y_{N+1}  (G_{N+1,N+1}(z)-\Phi_{N+1,N+1}(z))\nonumber \\ 
&\quad+x_1 y_{N+1} G_{1,N+1}(z) +  x_{N+1} y_1 G_{N+1,1}(z).
\end{align}
In order to prove Lemma 6.2,  it suffices to show that for each fixed $k \in \llbracket 1, N+n\rrbracket$,
\begin{align}\label{eq:Gkk}
\left| G_{kk}(z) - \Phi_{kk}(z) \right| \leq \frac{2b^2}{(b-1)^2}\frac{ \sqrt{(K+1)\log (N+n)}}{|z|^2} 
\end{align}
with probability at least $1 - 4(N+n)^{-(K+1)}$ and for fixed $i\neq j\in \llbracket 1, N+n\rrbracket$,
\begin{align}\label{eq:Gij}
  |G_{ij}(z)| \leq 2\sqrt{2}\left(\frac{b}{b-1} \right)^2\frac{  \sqrt{(K+1)\log (N+n)}}{ |z|^2}
  \end{align}
with probability at least $1 - 0.5 (N+n)^{-(K+1)}$. 


If so, continuing from \eqref{eq:reduceG}, we find that 
\begin{align*}
&\left| \mathbf x^\T \left(G(z)-\Phi(z) \right) \mathbf y \right|\\
 &\le \left(2\left(\frac{b}{b-1} \right)^2 + 2\sqrt{2}\left(\frac{b}{b-1} \right)^2 \right) \frac{  \sqrt{(K+1)\log (N+n)}}{ |z|^2}\\
&\le 5\left(\frac{b}{b-1} \right)^2 \frac{  \sqrt{(K+1)\log (N+n)}}{ |z|^2}
\end{align*}
with probability at least $1-9(N+n)^{-(K+1)}$. Here, we use the fact that $|x_1 y_1|+|x_{N+1} y_{N+1}|\le 1$ by the Cauchy-Schwarz inequality. 

The proofs of Equation \eqref{eq:Gkk} and Equation \eqref{eq:Gij} closely resemble the proof presented in \cite[Lemma 28]{OVW22}, with only minor cosmetic modifications. For completeness, we outline the main steps. As noted at the beginning of Section \ref{sec:proofmain}, we work on the event that $\|\mathcal E\|=\|E \| \leq 2 (\sqrt N +\sqrt{n})$.

We first prove \eqref{eq:Gkk}. From Lemma 27 in \cite{OVW22}, which follows from the Schur complement, for any $1\le k \le N+n$, \[ G_{kk}(z)  = \frac{1}{z - \mathcal E_{kk} - \sum_{ s,t}^{(k)} \mathcal E_{sk} G_{st}^{(k)}(z) \mathcal E_{tk} }. \]
For $1\le k \le N$, using the block structure of $\mathcal E$ and the expression for $\phi_1(z)$ in \cite[Eq. (26)]{Wang24} and \cite[Eq. (27)]{Wang24}, we have
\begin{align} \label{eq:diagonal}
	\left| G_{kk}(z) - \frac{1}{\phi_1(z)}  \right| &= \left| \frac{1}{z - \sum_{1\le i,j\le n}^{(k)} E_{ki} G_{N+i,N+j}^{(k)}(z) E_{kj}}- \frac{1}{\phi_1(z)} \right| \nonumber\\
	&\le \frac{ |\sum_{1\le i,j\le n}^{(k)} E_{ki} G_{N+i,N+j}^{(k)}(z) E_{kj}  - \sum_{t \in \llbracket N+1,N+n \rrbracket} G_{tt}(z)|}{ |z - \sum_{1\le i,j\le n}^{(k)} E_{ki} G_{N+i,N+j}^{(k)}(z) E_{kj}| |z - \tr \mathcal I^\mathrm{d} G(z)| }. 
\end{align} 
To bound the right-hand side of \eqref{eq:diagonal}, we need to establish an upper bound for its numerator and lower bounds for the two terms in the denominator. 

Let us first address the numerator using concentration inequalities. Define $$\mathsf X:=\left(G_{N+i,N+j}^{(k)}(z)\right)_{i,j\neq k}$$ and let $E_{(k)}$ denote the $k$-th row of $E$. Since $E_{(k)}$ is independent of $G^{(k)}$, we can write $\sum_{1\le i,j\le n}^{(k)} E_{ki} G_{N+i,N+j}^{(k)}(z) E_{kj} := E_{(k)} \mathsf{X} E_{(k)}^\T$. Conditioning on $X$ or $G^{(k)}$, we have 
\[ E_{(k)} \mathsf{X} E_{(k)}^\T - \E (E_{(k)} \mathsf{X} E_{(k)}^\T) = \sum_{1\le i,j\le n}^{(k)} E_{ki} G_{N+i,N+j}^{(k)}(z) E_{kj}  - \sum_{t \in \llbracket N+1,N+n \rrbracket} G_{tt}^{(k)}(z).\]
This quadratic form can be bounded using either the Hanson-Wright inequality \cite{HW} or by exploiting the rotation invariance of Gaussian vectors to reduce it to a sum of independent random variables and applying Bernstein's inequality (as in \cite[Lemma 28]{OVW22}). These approaches yield, with probability at least $1-4(N+n)^{-(K+1)}$ that 
\begin{equation} \label{eq:HW1}
\left|\sum_{1\le i,j\le n}^{(k)} E_{ki} G_{N+i,N+j}^{(k)}(z) E_{kj}  - \sum_{t \in \llbracket N+1,N+n \rrbracket} G_{tt}^{(k)}(z) \right| \le 2\sqrt{2} \sqrt{(K+1)\log(N+n)}. 
\end{equation} 
This bound follows from combining the estimates
\[ \| \mathsf X\| \le \left\|\left(G_{N+i,N+j}^{(k)}(z)\right)_{1\le i,j\le n} \right\| \leq \| G^{(k)}(z) \|\leq \frac{b}{b-1}\frac{1}{|z|}  \]
and
\[ \|\mathsf X\|_F^2 \le \left\|\left(G_{N+i,N+j}^{(k)}(z)\right)_{1\le i,j\le n} \right\|_F^2 \leq n \left\|\left(G_{N+i,N+j}^{(k)}(z)\right)_{1\le i,j\le n} \right\|^2\le \left(\frac{b}{b-1}\right)^2\frac{n}{|z|^2} \]
due to Lemma 6.1.

To complete the upper bound for the term in the numerator of the right-hand side of \eqref{eq:diagonal}, it remains to show that $\sum_{t \in \llbracket N+1,N+n \rrbracket} G_{tt}^{(k)}(z)$ and $\sum_{t \in \llbracket N+1,N+n \rrbracket} G_{tt}(z)$ are close.   Their difference can be expressed as
\[ \left| \sum_{t \in \llbracket N+1,N+n \rrbracket} G_{tt}^{(k)}(z) -\sum_{t \in \llbracket N+1,N+n \rrbracket} G_{tt}(z)\right| = \left|\tr \mathcal I^\mathrm{d}(G^{(k)}(z)-G(z)) \right|. \]
Using the resolvent identity $B^{-1} - C^{-1} = B^{-1} (C - B) C^{-1} $, this becomes
\[ \left|\tr \mathcal I^\mathrm{d}(G^{(k)}(z)-G(z)) \right| = \left| \tr \mathcal I^\mathrm{d} G(z) (\mathcal E - {\mathcal E}^{(k)}) {G}^{(k)}(z)  \right|. \]
Since $\mathcal E - {\mathcal E}^{(k)}$ has at most rank $2$, we can bound this term:
\begin{align*}
 \left| \tr \mathcal I^\mathrm{d} G(z) (\mathcal E - {\mathcal E}^{(k)}) {G}^{(k)}(z)  \right| &\leq 2 \| G(z) (\mathcal E - {\mathcal E}^{(k)}) {G}^{(k)}(z) \| \\
 &\leq 2 \| G(z)\| (\|\mathcal E\| + \|{\mathcal E}^{(k)}\|) \|{G}^{(k)}(z) \| \\
 &\leq 4\left(\frac{b}{b-1}\right)^2\frac{\|E\|}{|z|^2}  \leq  \frac{4b}{(b-1)^2}\frac{1}{|z|}. 
\end{align*}
The last inequality follows from $|z| \ge 2b(\sqrt N+ \sqrt{n})$ and $\|E\| \leq 2(\sqrt N+ \sqrt{n})$.  Therefore,
\[ \left| \sum_{t \in \llbracket N+1,N+n \rrbracket} G_{tt}^{(k)}(z) -\sum_{t \in \llbracket N+1,N+n \rrbracket} G_{tt}(z)\right| \leq \frac{4b}{(b-1)^2}\frac{1}{|z|} \le \frac{2}{(b-1)^2}\frac{1}{\sqrt N + \sqrt n}. \]

Next, we establish lower bounds for the denominator terms in \eqref{eq:diagonal}, assuming \eqref{eq:HW1} holds. From \cite[Eq. (32)]{Wang24}, we have
\begin{align*}
 \left|\sum_{t \in \llbracket N+1,N+n \rrbracket} G_{tt}(z)\right|= |\tr \mathcal I^d G(z)  | \le \frac{|z|}{4b(b-1)}.
\end{align*}
Note that $2\sqrt{2}\sqrt{(K+1)\log (N+n)} \le \frac{1}{2}(\sqrt N + \sqrt n) \le \frac{|z|}{4b}$ since $(\sqrt N + \sqrt n)^2 \ge 32 (K+1)\log(N+n)$ by assumption and $|z|\ge 2b(\sqrt N + \sqrt n)$. Combining this with \eqref{eq:HW1} yields
$$\left|\sum_{1\le i,j\le n}^{(k)} E_{ki} G_{N+i,N+j}^{(k)}(z) E_{kj} \right| \le \frac{b-2}{4b(b-1)}|z|.$$Therefore, by the triangle inequality,
\begin{align*}
 |z - \sum_{1\le i,j\le n}^{(k)} E_{ki} G_{N+i,N+j}^{(k)}(z) E_{kj}| |z - \tr \mathcal I^\mathrm{d} G(z)| \ge \frac{2(b-1)^2}{b^2}|z|^2.
\end{align*}
The case when $N+1\le k \le N+n$ can be computed similarly and we omit the details. The combination of these bounds for the numerator and denominator terms completes the proof of \eqref{eq:Gkk}.

The proof of the off-diagonal bound \eqref{eq:Gij} is more straightforward. Using Lemma 27 from \cite{OVW22}, which follows from the Schur complement, we have 
\[ G_{ij}(z) = - G_{ii}(z) \sum_k^{(i)} \mathcal E_{ik} G^{(i)}_{kj}(z). \]
By Lemma 6.1, we first have 
\[ |G_{ij}(z)| \le \frac{b}{b-1}\frac{1}{|z|}\left|\sum_k^{(i)} \mathcal E_{ik} G^{(i)}_{kj}(z) \right|.\]
The $i$-th row of $\mathcal E$ is independent of $G^{(i)}$, so conditioning on $G^{(i)}$, the sum $\sum_k^{(i)} \mathcal E_{ik} G^{(i)}{kj}(z)$ follows a Gaussian distribution with mean 0. Its variance is bounded by $((G^{(i)})^\ast G^{(i)} )_{jj} \leq \|G^{(i)} \|^2 \le (b/(b-1)|z|)^2$, from which the desired bound follows.

\subsection{Proof of Lemma 6.3}\label{app:lem:betterbd}
In this section, we prove Lemma 6.3 using Lemma 6.2 and a standard $\varepsilon$-net argument. 

For convenience, denote $\Delta(z):=G(z)-\Phi(z)$.  
We first show that for any fixed $z\in \mathbb C$ with $|z| \ge 2b(\sqrt N + \sqrt n)$, 
\[ \left\| \mathcal U^T \Delta(z) \mathcal U  \right\| \leq 10\left(\frac{b}{b-1} \right)^2 \frac{\sqrt{(K+1)\log (N+n)+2(\log 9) r}}{|z|^2}\]
with probability at least $1-9(N+n)^{-(K+1)}$.

 Let $\mathcal N$ be the ${1}/{4}$-net of the unit sphere $\mathbb{S}^{2r-1}$. A simple volume argument (see for instance \cite[Corollary 4.2.13]{Vbook}) shows $\mathcal{N}$ can be chosen such that $|\mathcal N| \le 9^{2r}$. Furthermore, since for any $\mathbf y \in \mathbb{S}^{2r-1}$, there exists a $\mathbf x \in \mathcal N$ such that $\|\mathbf y- \mathbf x\|\le 1/4$, we have
\begin{align*}
\left|\mathbf y^\T \mathcal U^\T \Delta(z) \mathcal U \mathbf y \right| &\le \left|\mathbf x^\T \mathcal U^\T \Delta(z) \mathcal U \mathbf x \right| + \left|(\mathbf y-\mathbf x)^\T \mathcal U^\T \Delta(z) \mathcal U \mathbf x \right| + \left|\mathbf y^\T \mathcal U^\T \Delta(z) \mathcal U (\mathbf y -\mathbf x)\right|\\
&\le \left|\mathbf x^\T \mathcal U^\T \Delta(z) \mathcal U \mathbf x \right| + \frac{1}{2} \| \mathcal U^\T \Delta(z) \mathcal U \|.
\end{align*}
Therefore, 
$\| \mathcal U^\T \Delta(z) \mathcal U \| \le 2 \max_{\mathbf x \in \mathcal N} \left|\mathbf x^\T \mathcal U^\T \Delta(z) \mathcal U \mathbf x \right|$
and for any $K_1>0$ and $(\sqrt N + \sqrt n)^2 \ge 32 (K_1+1)\log(N+n)$,  by the union bound,
\begin{align*}
&\Prob\left(\| \mathcal U^\T \Delta(z) \mathcal U \| \ge 10\left(\frac{b}{b-1} \right)^2 \frac{  \sqrt{(K_1+1)\log (N+n)}}{ |z|^2} \right)\\
 &\le  \Prob \left(\max_{\mathbf x \in \mathcal N} \left|\mathbf x^\T \mathcal U^\T \Delta(z) \mathcal U \mathbf x \right| \ge 5\left(\frac{b}{b-1} \right)^2 \frac{  \sqrt{(K_1+1)\log (N+n)}}{ |z|^2} \right)\\
&\le  9^{2r+1} (N+n)^{-(K_1+1)},
\end{align*}
where in the last inequality, we applied Lemma 6.2. Now choose $K_1=K +\frac{2\log 9}{\log(N+n)} r$ and assume $$(\sqrt N + \sqrt n)^2 \ge 32 (K_1+1)\log(N+n)=32(K+1)\log(N+n) + 64(\log 9) r.$$ The conclusion becomes that 
$$\| \mathcal U^\T \Delta(z) \mathcal U \| \le 10\left(\frac{b}{b-1} \right)^2 \frac{\sqrt{(K+1)\log (N+n)+2(\log 9) r}}{|z|^2} $$
with probability at least $1-10(N+n)^{-(K+1)}$.

In particular, for any $z\in \mathsf D= \{ z\in \mathbb C : 2b(\sqrt N + \sqrt n)\le |z| \le 2n^3\}$,
\[ \left\| \mathcal U^T \Delta(z) \mathcal U  \right\| \leq 10\left(\frac{b}{b-1} \right)^2 \frac{\sqrt{(K+7)\log (N+n)+2(\log 9) r}}{|z|^2}  \] 
with probability at least $1 - 9(N+n)^{-(K+7)}$, as long as
\begin{align}\label{eq:dimension}
(\sqrt N + \sqrt n)^2 \ge 32(K+7)\log(N+n) + 64(\log 9) r.
\end{align}

Let $\mathcal{N}$ be a $1$-net of $\mathsf D$.  A simple volume argument (see for instance \cite[Lemma 3.3]{OW}) shows $\mathcal{N}$ can be chosen so that $|\mathcal{N}| \leq  (1+8n^3)^2 <n^7$. 
By the union bound, 
\begin{equation} \label{eq:net}
	\max_{ z \in \mathcal{N}} |z|^2 \left\| \mathcal U^\T \Delta(z) \mathcal U \right\| \leq 10 \left(\frac{b}{b-1} \right)^2 \sqrt{(K+7)\log (N+n)+2(\log 9) r}
\end{equation} 
with probability at least $1 - 9(N+n)^{-K}$.  We now wish to extend this bound to all $z \in \mathsf D$.  

Define the functions
\[ f(z) := z^2 \mathcal U^\T G(z) \mathcal U, \qquad g(z) := z^2\mathcal U^\T \Phi(z) \mathcal U. \]
In order to complete the proof, it suffices to show that $f$ and $g$ are $\frac{3b^2}{(b-1)^2}$-Lipschitz in $\mathsf D$.  In other words, we want to show that $\| f(z) - f(w)\| \leq \frac{3b^2}{(b-1)^2} |z-w|$ and $\| g(z) - g(w)\| \leq \frac{3b^2}{(b-1)^2} |z-w|$ for all $z,w \in \mathsf D$.  Indeed, in view of \eqref{eq:net}, if $z \in \mathsf D$, then there exists $w \in \mathcal{N}$ so that $|z - w| \leq 1$, and hence \begin{align*}
	&|z|^2 \left\| \mathcal U^\T G(z) \mathcal U - \mathcal U^T \Phi(z) \mathcal U \right\| \\
	&
	\le \|f(z) - f(w)\| +\|f(w) - g(w)\| + \|g(w) - g(z)\| \\
	&\leq \frac{6 b^2}{(b-1)^2} + |w|^2 \left\| \mathcal U^\T G(w) \mathcal U - \mathcal U^T \Phi(w) \mathcal U \right\| \\
	&\leq \frac{6 b^2}{(b-1)^2} + 10\left(\frac{b}{b-1} \right)^2 \sqrt{(K+7)\log (N+n)+2(\log 9) r}  \\
	&<  11\left(\frac{b}{b-1} \right)^2 \sqrt{(K+7)\log (N+n)+2(\log 9) r}=\eta,  
\end{align*}
 where we used the Lipschitz continuity of $f$ and $g$ in the second inequality. In the last inequality, we use \eqref{eq:dimension} to obtain a crude bound $N+n\ge 32\cdot 7 \log(N+n)$ and hence $N+n\ge 1600$. This implies $\sqrt{(K+7)\log (N+n)+2(\log 9) r} \ge \sqrt{7\log (1600)+2(\log 9) }\approx 7.5$. 
 
It remains to show that $f$ and $g$ are $\frac{3b^2}{(b-1)^2}$-Lipschitz in $\mathsf D$.  Recall that we work on the event where $\|E \| \leq 2(\sqrt N+ \sqrt{n})$ through the proofs.  Let $z,w \in \mathsf D$, and assume without loss of generality that $|z| \geq |w| \ge 2b(\sqrt N+ \sqrt{n})$.  Then
\begin{align*}
	\| f(z) - f(w) \| &\leq \|z^2 \mathcal U^\T G(z) \mathcal U - zw \mathcal U^\T G (z) \mathcal U \| + \| zw \mathcal U^\T G(z) \mathcal U - w^2 \mathcal U^\T G(z) \mathcal U \| \\
	&\qquad + \| w^2 \mathcal U^\T G(z) \mathcal U - w^2 \mathcal U^\T G(w) \mathcal U \| \\
	&\leq |z| \|G(z)\| |z - w| + |w| |z-w| \|G(z) \| + |w|^2 |z-w| \|G(z) \| \| G(w) \| \\
	&\leq \frac{2b}{b-1} |z-w| + \frac{b^2}{(b-1)^2}|z-w| = \frac{b(3b-2)}{(b-1)^2}|z-w|\\
	&\leq \frac{3b^2}{(b-1)^2} |z-w|,
\end{align*}
where we used the resolvent identity $B^{-1} - C^{-1} = B^{-1} (C - B) C^{-1} $, Lemma 6.1, and the fact that $\frac{|w|}{|z|} \leq 1$.  This shows that $f$ is $\frac{3b^2}{(b-1)^2}$-Lipschitz in $\mathsf D$.  

The proof for $g$ is similar. First, by the triangle inequality, we have 
\begin{align}\label{eq:bdforg}
	\| g(z)& - g(w) \| \nonumber\\
	&\leq \|z^2 \mathcal U^\T \Phi(z) \mathcal U - zw \mathcal U^\T \Phi (z) U \| + \| zw \mathcal U^\T \Phi(z) \mathcal U - w^2 \mathcal U^\T \Phi(z) \mathcal U \| \nonumber\\
	&\qquad + \| w^2 \mathcal U^\T \Phi(z) \mathcal U - w^2 \mathcal U^\T \Phi(w) \mathcal U \| \nonumber \\
	&\leq |z||z - w| \|\mathcal U^\T \Phi(z) \mathcal U\|  + |w| |z-w| \|\mathcal U^\T \Phi(z) \mathcal U\| + |w|^2 \|\mathcal U^\T (\Phi(z)-\Phi(w)) \mathcal U \|.
\end{align}
Using the explicit expression in \cite[Eq. (29)]{Wang24}, we find that
\begin{align*}
\|\mathcal U^\T (\Phi(z)-\Phi(w)) \mathcal U \| = \max\left\{ \frac{|\phi_1(z)-\phi_1(w)|}{|\phi_1(z)\phi_1(w)|}, \frac{|\phi_2(z)-\phi_2(w)|}{|\phi_2(z)\phi_2(w)|} \right\}.
\end{align*}
By \cite[Eq. (27)]{Wang24} and the resolvent identity $B^{-1} - C^{-1} = B^{-1} (C - B) C^{-1} $, 
\begin{align*}
|\phi_1(z)-\phi_1(w)| &= |z-w - \tr \mathcal I^\mathrm{d} (G(z)-G(w))|\\
&=|z-w - (z-w)\tr \mathcal I^\mathrm{d} G(z)G(w)| \\
&\le |z-w| \left( 1+ (N+n) \|G(z)\| \|G(w)\| \right)\\
&\le |z-w| \left( 1+ \frac{b^2}{(b-1)^2}\frac{(N+n) }{|z| |w|} \right) \\
&\le \left( 1+ \frac{1}{4(b-1)^2} \right)|z-w|,
\end{align*}
where we used Lemma 6.1 and the facts that $\frac{N+n}{|w|}\le \frac{|w|}{4b^2}$ and $\frac{|w|}{|z|} \leq 1$. The same upper bound also holds for $|\phi_2(z)-\phi_2(w)|$. Combining these estimates with \cite[Eq. (31)]{Wang24}, we have $$\|\mathcal U^\T (\Phi(z)-\Phi(w)) \mathcal U \| \le \frac{1+1/4(b-1)^2}{(1-1/4b(b-1))^2} \frac{|z-w|}{|z| |w|}.$$
Notice that 
$\|\mathcal U^\T \Phi(z) \mathcal U\| \le \frac{1}{1-1/4b(b-1)}\frac{1}{|z|}$ for any $|z| \ge 2b(\sqrt N + \sqrt n)$, which can be verified using \cite[Eq. (30)]{Wang24} and the bounds in \cite[Eq. (31)]{Wang24}. Inserting these bounds into \eqref{eq:bdforg} yields that 
\begin{align*}
\| g(z) - g(w) \| &\le \frac{2}{1-1/4b(b-1)}|z-w|+\frac{1+1/4(b-1)^2}{(1-1/4b(b-1))^2} |z-w|\\
&\le \frac{4b(12b^3 - 24 b^2 +11 b +2)}{(4b^2-4b-1)^2}|z-w|< \frac{3b^2}{(b-1)^2} |z-w|,
\end{align*}
where the last inequality is check via Mathematica. Hence, $g$ is $\frac{3b^2}{(b-1)^2}$-Lipschitz in $\mathsf D$.

\section{Proof of Lemma \ref{lem:projsv}}\label{app:bdeach}
In this section, we estimate $\| \mathcal U_J^\T \widetilde{\mathbf u}_i \|$ for each $i \in I=\llbracket k, s\rrbracket \cup \llbracket r+k, r+s\rrbracket$. Recall the decomposition of $\widetilde{\mathbf u}_i$ in \eqref{eq:u}:
\begin{align}\label{eq:u2}
\widetilde{\mathbf u}_i=\Pi(\widetilde \lambda_i) \mathcal A \widetilde{\mathbf u}_i + \Xi(\widetilde \lambda_i) \mathcal A \widetilde{\mathbf u}_i,
\end{align}
where $\Pi(z)$ is a function to be further specified during the proof, and $\Xi(z)=G(z)-\Pi(z)$.

We split the estimation of $\| \mathcal U_J^\T \widetilde{\mathbf u}_i \|$ for $i\in I$ into two cases: when $|\lambda_i|$ is large and when $|\lambda_i|$ is relatively small. 

We start with the simpler case when $|\lambda_i|>n^2/2$. We choose $$\Pi(\widetilde\lambda_i):=\frac{1}{\widetilde \lambda_i} I_{N+n} + \frac{1}{\widetilde \lambda_i^2} \mathcal E.$$ We work on the event $$\max_{l \in \llbracket 1, r_0\rrbracket : \sigma_l >\frac{1}{2}n^2} |\widetilde \sigma_l - \sigma_l |\le \eta r.$$ By Lemma 6.6, this event holds with probability at least $1-(N+n)^{-1.5r^2(K+4)}$. Hence, $|\widetilde \lambda_i| \ge |\lambda_i| -\eta r \ge 2b(\sqrt N + \sqrt n)$. We apply Lemma 6.5 to get
\begin{equation}\label{eq:case21}
\left\| \Xi(\widetilde \lambda_i) \right\|=\left\| G(\widetilde{\lambda}_i) - \Pi(\widetilde{\lambda}_i)\right\| \le \frac{b}{b-1}\frac{\|\mathcal E\|^2}{|\widetilde{\lambda}_i|^3}.
\end{equation}
Multiplying $\mathcal U_J^\T$ on both sides of \eqref{eq:u2}, we obtain the following equation: 
\begin{equation*}
\mathcal U_J^\T \widetilde{\mathbf u}_i = \mathcal U_J^\T \Pi(\widetilde \lambda_i) \mathcal A \widetilde{\mathbf u}_i + \mathcal U_J^\T \Xi(\widetilde \lambda_i)\mathcal A \widetilde{\mathbf u}_i.
\end{equation*}
Plugging in \eqref{eq:Adecomp} and using the facts $\mathcal U_J^\T \mathcal U_I =0$ and $\mathcal U_J^\T \mathcal U_J =I$, we further get
\begin{align*}
\mathcal U_J^\T \widetilde{\mathbf u}_i = \frac{1}{\widetilde \lambda_i} \mathcal D_J \mathcal U_J^\T \widetilde{\mathbf u}_i +  \frac{1}{\widetilde \lambda_i^2} \mathcal U_J^\T \mathcal E \mathcal U \mathcal D \mathcal U^\T \widetilde{\mathbf u}_i + \mathcal U_J^\T \Xi(\widetilde \lambda_i) \mathcal U \mathcal D \mathcal U^\T  \widetilde{\mathbf u}_i,
\end{align*}
which, by rearranging the terms, is reduced to
\begin{align*}
(\widetilde \lambda_i I- \mathcal D_J)\mathcal U_J^\T \widetilde{\mathbf u}_i = \frac{1}{\widetilde \lambda_i} \mathcal U_J^\T \mathcal E \mathcal U \mathcal D \mathcal U^\T \widetilde{\mathbf u}_i + \widetilde \lambda_i \mathcal U_J^\T \Xi(\widetilde \lambda_i) \mathcal U \mathcal D \mathcal U^\T  \widetilde{\mathbf u}_i.
\end{align*}
Hence, 
\begin{align*}
\min_{j\in J} |\widetilde \lambda_i - \lambda_j| \cdot \|\mathcal U_J^\T \widetilde{\mathbf u}_i\| \le \frac{1}{|\widetilde \lambda_i|} \|  \mathcal U^\T \mathcal E \mathcal U\| \cdot \|\mathcal D \mathcal U^\T \widetilde{\mathbf u}_i\| + |\widetilde \lambda_i| \|\Xi(\widetilde \lambda_i)\| \cdot \|\mathcal D \mathcal U^\T \widetilde{\mathbf u}_i\|.
\end{align*}

Note that $\|\mathcal D \mathcal U^\T \widetilde{\mathbf u}_i\| \le \|\mathcal E\| + |\widetilde \lambda_i|\le \frac{b}{b-1} |\widetilde \lambda_i|$ as in \eqref{eq:bdonev01}. Inserting \eqref{eq:case21} into the above inequality, we arrive at
\begin{align}\label{eq:bdtwosides}
\min_{j\in J} |\widetilde \lambda_i - \lambda_j| \cdot \|\mathcal U_J^\T \widetilde{\mathbf u}_i\| \le \frac{b}{b-1} \|  \mathcal U^\T \mathcal E \mathcal U\| + \frac{b^2}{(b-1)^2}\frac{\|E\|^2}{|\widetilde \lambda_i|}.
\end{align}
For the remaining arguments, we work on the event 
\begin{align*}\label{eventF0}
\mathsf F:=\big\{\|  \mathcal U^\T \mathcal E \mathcal U\| \le 2 \sqrt{r}+\sqrt{2(K+7)\log(N+n)}\big\}.
\end{align*}
The following lemma is proved in \cite[Lemma 18]{OVW2}. 
\begin{lemma}\label{lem:Unoise}
Let $K$ be an arbitrary positive constant. With probability at least $1-2(N+n)^{-K}$, we have 
$$\| \mathcal U^T \mathcal E \mathcal U \| \le 2\sqrt{r} + \sqrt{2K\log(N+n)}.$$
\end{lemma}
Therefore, the event $\mathsf F$ holds with probability at least $1-2(N+n)^{-(K+7)}.$ We continue the estimation of $\|\mathcal U_J^\T \widetilde{\mathbf u}_i\|$ from \eqref{eq:bdtwosides}.  Note  that
$$\|  \mathcal U^\T \mathcal E \mathcal U\| \le 2 \sqrt{r}+\sqrt{2(K+7)\log(N+n)}< \eta.$$  Also, ${\|E\|^2}/{|\widetilde \lambda_i|}\le 4(2\sqrt n)^2/n^2< \eta$ where we used the crude bound $|\widetilde \lambda_i| \ge \frac{1}{4}n^2$ by Weyl's inequality. It follows that
\begin{equation}\label{eq:2ndtolast}
\min_{j\in J} |\widetilde \lambda_i - \lambda_j| \cdot \|\mathcal U_J^\T \widetilde{\mathbf u}_i\| \le  \frac{b(2b-1)}{(b-1)^2}\eta.
\end{equation}
To bound the left-hand side of \eqref{eq:2ndtolast}, we first consider $i \in  \llbracket k, s\rrbracket$. Then $$\min_{j\in J} |\widetilde \lambda_i - \lambda_j| = \min_{j \in \llbracket 1, k-1\rrbracket\cup\llbracket s+1, r\rrbracket} |\widetilde \sigma_i - \sigma_j|=\min\{\sigma_{k-1}-\widetilde \sigma_i, \widetilde \sigma_i -\sigma_{s+1} \}$$ by $|\widetilde \sigma_i -  \sigma_i|\le \eta r$ and the supposition $\min\{\delta_{k-1},\delta_s\}\ge 75\chi(b) \eta r$. Next, applying the inequality $\min\{\delta_{k-1},\delta_s \}\ge 75\chi(b)  \eta r$ again, we get
$$\min_{j\in J} |\widetilde \lambda_i - \lambda_j|  \ge \left(1-\frac{1}{75\chi(b)}\right) \min\{ \delta_{k-1},\delta_s \}.$$ 
It follows from \eqref{eq:2ndtolast} that
\begin{align}\label{eq:UJbd3}
 \|\mathcal U_J^\T \widetilde{\mathbf u}_i\| &\le \frac{\eta}{ \min\{\delta_{k-1},\delta_s\} }\left(1-\frac{1}{75\chi(b)}\right)^{-1}\frac{b(2b-1)}{(b-1)^2}\nonumber\\
 &=\frac{75(2b-1)^3b}{(b-1)^2(296b^2-296 b+75)}\frac{\eta}{ \min\{\delta_{k-1},\delta_s\} } \nonumber\\
 &< 3 \frac{(b+1)^2}{(b-1)^2}\frac{\eta}{\min\{\delta_{k-1},\delta_s\}} 
\end{align} 
for every $i \in \llbracket k, s\rrbracket$ satisfying $\lambda_i=\sigma_i \ge n^2/2$. The last inequality was checked by Mathematica. Finally, for $i \in \llbracket r+k, r+s\rrbracket$ such that $|\lambda_i| \ge n^2/2$, analogous arguments yield that the same bound
\begin{align}\label{eq:UJbd4}
\| {\mathcal U}_J^\T \widetilde{\mathbf u}_i \| 
\le 3 \frac{(b+1)^2}{(b-1)^2}\frac{\eta}{\min\{\delta_{k-1},\delta_s\}}.
\end{align}

\medskip

The estimation of $\| \mathcal U_J^\T \widetilde{\mathbf u}_i \|$ when $| \lambda_i|$ is relatively small is more involved. From the previous discussion, it suffices to assume there is a certain $l_0 \in \llbracket 1, r_0\rrbracket$ for which $\sigma_{l_0} \le n^2/2$. We claim that there exists an index $i_0\in \llbracket 1, r_0\rrbracket$ such that $\sigma_j \le  n^3$ for $j\ge i_0$ and $\sigma_j> n^3$ for $j < i_0$, and 
$$\delta_{i_0-1} =\sigma_{i_0-1}- \sigma_{i_0}\ge 75 \chi(b) \eta r.$$
To determine $i_0$, we propose a simple iterative algorithm: start with $\sigma_1$. If $\sigma_1\le n^3$, set $i_0=1$ and terminate the algorithm, since $\sigma_0=\infty$ and $\delta_0=\infty$ by definition. Assume $\sigma_1> n^3$ and evaluate $\sigma_2$. If $\sigma_2 \le n^3 - 75 \chi(b) \eta r$,  set $i_0=2$ and exit. Assume $\sigma_2 > n^3 - 75 \chi(b) \eta r$ and evaluate $\sigma_3$. We continue this process and terminate the algorithm with $i_0=k $ unless 
\begin{align}\label{eq:i00}
\sigma_1> n^3, \sigma_2 > n^3 - 75 \chi(b) \eta r, \cdots, \sigma_k > n^3 - 75 \chi(b) \eta r.
\end{align}
Note that the condition \eqref{eq:i00} cannot hold for $k=l_0$ because $\sigma_{l_0} \le n^2/2 < n^3 - 75 \chi(b) \eta r$, based on the assumption that $(\sqrt N + \sqrt n)^2 \ge 32(K+7)\log(N+n) + 64(\log 9) r$. Therefore, $i_0$ must satisfy $i_0 \le l_0 - 1$.

We shall fix such an index $i_0$  throughout the rest of the proof. We now turn our attention to estimating $\| \mathcal U_J^\T \widetilde{\mathbf u}_i \|$  for $i\in \llbracket k_0, s\rrbracket\cup \llbracket r+k_0, r+s\rrbracket$, where we define $$k_0:=\max\{ k,i_0\}$$ for the sake of notational simplicity. Note that $\min\{\delta_{k_0-1},\delta_s\} \ge 75 \chi(b) \eta r$. Furthermore, in this scenario, $|\lambda_i | \le n^3$. We take $$\Pi(\widetilde{\mathbf u}_i) = \Phi(\widetilde{\mathbf u}_i).$$ 
Continuing from \eqref{eq:u}, we have
\begin{equation*}
\widetilde{\mathbf u}_i=\Phi(\widetilde \lambda_i) \mathcal A \widetilde{\mathbf u}_i + \Xi(\widetilde \lambda_i)  \mathcal A \widetilde{\mathbf u}_i
\end{equation*}
and multiply on the left by $\mathcal U_J^\T$ to get
\begin{equation}\label{eq:middle}
\mathcal U_J^\T \widetilde{\mathbf u}_i = \mathcal U_J^\T \Phi(\widetilde \lambda_i) \mathcal A \widetilde{\mathbf u}_i + \mathcal U_J^\T \Xi(\widetilde \lambda_i) \mathcal A \widetilde{\mathbf u}_i.
\end{equation}
Plugging in \eqref{eq:Adecomp}, we further have
\begin{equation*}
\mathcal U_J^\T \widetilde{\mathbf u}_i = \mathcal U_J^\T \Phi(\widetilde \lambda_i) {\mathcal U}_J {\mathcal D}_J {\mathcal U}_J^\T \widetilde{\mathbf u}_i + \mathcal U_J^\T \Xi(\widetilde \lambda_i) {\mathcal U} {\mathcal D} {\mathcal U}^\T \widetilde{\mathbf u}_i,
\end{equation*}
where we used $\mathcal U_J^\T \Phi(\widetilde \lambda_i) \mathcal U_{I}=0$.  Hence, 
\begin{equation} \label{eq:recursive}
\left( I_{2(r-s+k-1)} - \mathcal U_J^\T \Phi(\widetilde \lambda_i) {\mathcal U}_J {\mathcal D}_J \right) \mathcal U_J^\T \widetilde{\mathbf u}_i = \mathcal U_J^\T \Xi(\widetilde \lambda_i) {\mathcal U} {\mathcal D} {\mathcal U}^\T \widetilde{\mathbf u}_i.
\end{equation}

We are now in position to bound $\| \mathcal U_{J}^\T \widetilde{\mathbf u}_i \|$. This can be achieved by obtaining an upper bound for the right-hand side of \eqref{eq:recursive} and estimating the smallest singular value of the matrix 
\begin{equation}\label{eq:bigmatrix}
 I_{2(r-s+k-1)} - \mathcal U_J^\T \Phi(\widetilde \lambda_i) {\mathcal U}_J {\mathcal D}_J 
\end{equation}
on the left-hand side of \eqref{eq:recursive}. We establish these estimates in the following two steps. Recall that 
$$\xi(b)= 1+ \frac{1}{2(b-1)^2}\quad \text{and} \quad \chi(b)= 1+ \frac{1}{4b(b-1)}.$$
For each $k_0 \le i \le s$, by Theorem  6.4, there exists $k_0 \le l_i \le s$ such that $\widetilde \sigma_i \in S_{\sigma_{l_i}}$ specified in \cite[Eq. (35)]{Wang24}, and 
\begin{align}\label{eq:location01}
|\varphi(\widetilde{\sigma}_{i}) - \sigma_{l_i}^2| \leq 20\xi(b) \chi(b)\eta r \left(\widetilde{\sigma}_{i}+\chi(b)\sigma_{l_i} \right)
\end{align}
with probability at least $1-10(N+n)^{-K}$. Denote this event as $\mathsf E_1$. Furthermore, on the event $\mathsf E_1$, by Lemma 6.3, for every $k_0\le i \le s$,
$$\left\|  \mathcal U^\T \Xi(\widetilde \sigma_i) \mathcal U \right\| \leq \frac{\eta}{\widetilde \sigma_i^2}$$ holds with probability at least $1-9(N+n)^{-K}$. Let us denote this event as $\mathsf E_2$. 

In the remaining proof, we will work on the event $\mathsf E_1 \cap \mathsf E_2$ which holds with probability at least $1-19(N+n)^{-K}$.

\medskip
\noindent\emph{Step 1. Upper bound for the right-hand side of \eqref{eq:recursive}}.  We first consider the case when $i\in \llbracket k_0, s\rrbracket$ and $\widetilde \lambda_i = \widetilde \sigma_i$. Note that $\mathcal U_J^\T \Xi(\widetilde \sigma_i) {\mathcal U}$ is a sub-matrix of $\mathcal U^\T \Xi(\widetilde \sigma_i) {\mathcal U}$. Thus, using \eqref{event} and the fact that the spectral norm of any sub-matrix is bounded by the spectral norm of the full matrix, we deduce that 
\begin{align*}
\| \mathcal U_J^\T \Xi(\widetilde \sigma_i) {\mathcal U}\cdot {\mathcal D} {\mathcal U}^\T \widetilde{\mathbf u}_i \| &\le \frac{\eta}{\widetilde \sigma_i^2}\|{\mathcal D} {\mathcal U}^\T \widetilde{\mathbf u}_i \|.
\end{align*}
Recall the bound in \eqref{eq:bdonev01}:
\begin{align}\label{eq:bdonev}
\|{\mathcal D} {\mathcal U}^\T \widetilde{\mathbf u}_i \| \le  \frac{b}{b-1}\widetilde\sigma_i
\end{align}
Hence, 
\begin{align}\label{eq:upbd1}
\left\| \mathcal U_J^\T \Xi(\widetilde \sigma_i) {\mathcal U} {\mathcal D} {\mathcal U}^\T \widetilde{\mathbf u}_i \right\| &\le \frac{b}{b-1}\frac{\eta}{\widetilde\sigma_i}.
\end{align}
For the case when $i\in \llbracket r+k_0, r+s\rrbracket$, $\widetilde \lambda_i = -\widetilde \sigma_{i-r}$. Observe that $$ G(-\widetilde \sigma_{i-r}) = (-\widetilde \sigma_{i-r}-\mathcal E)^{-1} = - (\widetilde \sigma_{i-r}+\mathcal E)^{-1}\sim -(\widetilde \sigma_{i-r}-\mathcal E)^{-1}=-G(\widetilde \sigma_{i-r}) $$ because the distribution of $\mathcal E$ is symmetric. Hence $$\Phi(-\widetilde \sigma_{i-r})\sim -\Phi(\widetilde \sigma_{i-r})$$ by the definition \cite[Eq. (25)]{Wang24}. Repeating the arguments from the previous case, we see that
\begin{align}\label{eq:upbd2}
\left\| \mathcal U_J^\T \Xi(\widetilde \lambda_i) {\mathcal U} {\mathcal D} {\mathcal U}^\T \widetilde{\mathbf u}_i \right\| &\le \frac{b}{b-1}\frac{\eta}{\widetilde\sigma_{i-r}}.
\end{align}
\medskip

\noindent\emph{Step 2. Lower bound for the smallest singular value of the matrix \eqref{eq:bigmatrix}}. 
In fact, the singular values of the matrix \eqref{eq:bigmatrix} can be calculated explicitly via elementary linear algebra. The following proposition presents a subtle modification of the one found in \cite[Proposition 10]{OVW2}.  For completeness, the proof is provided in Section \ref{app:svaluecompute}.
\begin{proposition}\label{prop:blockphi}For $1\le r_0 < r$ and $1\le k \le s \leq r_0$, denote the index sets $I:=\llbracket k, s\rrbracket \cup \llbracket r+k, r+s\rrbracket$ and $J:=\llbracket 1, 2r\rrbracket\setminus I$. For any $x\in \mathbb R$ satisfying $|x|>\|\mathcal E\|$, the singular values of $I_{2(r-s+k-1)} -\mathcal U_J^\T \Phi(x) {\mathcal U}_J \mathcal D_J$ are given by
\begin{align*}
\left|\sqrt{1+\beta(x)^2 \sigma_t^2} \pm |\alpha(x)| \sigma_t \right|
\end{align*}
for $t\in \llbracket 1, k-1\rrbracket \cup \llbracket s+1, r\rrbracket$.
\end{proposition}
In order to bound the singular values, we first estimate $\phi_1(\widetilde\sigma_i)\phi_2(\widetilde\sigma_i)$, $\phi_1(\widetilde\sigma_i)$ and $\phi_2(\widetilde\sigma_i)$ for $ i\in \llbracket k_0, s\rrbracket$. Since $\widetilde \sigma_i \in S_{\sigma_{l_i}}$ for some $l_i\in \llbracket k_0, s\rrbracket$ where $S_{\sigma}$ is defined in \cite[Eq. (35)]{Wang24}, we have 
\begin{align}\label{eq:lowerbdsv}
 \widetilde \sigma_i \ge \sigma_{l_i} -20\chi(b)\eta r \ge \left( 1-\frac{\chi(b)}{4b} \right)\sigma_{l_i}
 \end{align}
and
\begin{align}\label{eq:upperbdsv}
 \widetilde \sigma_i \le \chi(b)\sigma_{l_i} + 20\chi(b) \eta r \le  \chi(b)\left( 1+\frac{1}{4b} \right)\sigma_{l_i}
 \end{align}
 by the supposition $\sigma_{l_i} \ge 2b(\sqrt N+\sqrt n) + 80b\eta r$. 


Observe from \cite[Eq. (31)]{Wang24} that 
\begin{equation}\label{eq:phiprodbd1}
\left(1-\frac{1}{4b(b-1)} \right) \widetilde\sigma_i \le \phi_s(\widetilde\sigma_i) \le \chi(b) \widetilde\sigma_i \quad\text{for } s=1,2.
\end{equation}
Using these estimates, we crudely bound $$0< \alpha(\widetilde \sigma_i)=\frac{1}{2}\left(\frac{1}{\phi_1(\widetilde \sigma_i)} + \frac{1}{\phi_2(\widetilde \sigma_i)} \right) \le \frac{\tau(b)}{\widetilde\sigma_i} \quad\text{with } \tau(b) :=\left(1-\frac{1}{4b(b-1)} \right)^{-1}$$  and by (28), 
\begin{align*}
\beta(\widetilde \sigma_i)&=\frac{1}{2}\left(\frac{1}{\phi_1(\widetilde \sigma_i)} - \frac{1}{\phi_2(\widetilde \sigma_i)} \right) = \frac{\phi_2(\widetilde \sigma_i) - \phi_1(\widetilde \sigma_i)}{2 \phi_1(\widetilde \sigma_i)  \phi_2(\widetilde \sigma_i) }\\
&= \frac{n-N}{\widetilde \sigma_i}\frac{1}{ 2\phi_1(\widetilde \sigma_i)  \phi_2(\widetilde \sigma_i)} \le \frac{\tau(b)^2}{8b^2\widetilde \sigma_i}
\end{align*}
by noting that $\widetilde \sigma_i^2 \ge (2b(\sqrt N + \sqrt n))^2> 4b^2(N+n)$. 

We are ready to bound the singular values of $I_{2(r-s+k-1)} - \mathcal U_J^\T \Phi(\widetilde \sigma_i) {\mathcal U}_J {\mathcal D}_J$. We start with the case when $i\in \llbracket k_0, s\rrbracket$ and $\widetilde \lambda_i = \widetilde \sigma_i$. In view of Proposition \ref{prop:blockphi}, the goal is to bound
\begingroup
\allowdisplaybreaks
\begin{align*}
&\min_{t\in \llbracket 1, k-1\rrbracket \cup \llbracket s+1, r\rrbracket} \left| \sqrt{1+\beta(\widetilde \sigma_i)^2 \sigma_t^2} \pm |\alpha(\widetilde \sigma_i)| \sigma_t \right|\\
&=\min_{t\in \llbracket 1, k-1\rrbracket \cup \llbracket s+1, r\rrbracket} \left| \sqrt{1+\beta(\widetilde \sigma_i)^2 \sigma_t^2} - \alpha(\widetilde \sigma_i) \sigma_t \right|\\
&=\min_{t\in \llbracket 1, k-1\rrbracket \cup \llbracket s+1, r\rrbracket} \left| \frac{1-(\alpha(\widetilde \sigma_i)^2 -\beta(\widetilde \sigma_i)^2) \sigma_t^2}{\sqrt{1+\beta(\widetilde \sigma_i)^2 \sigma_t^2} + \alpha(\widetilde \sigma_i) \sigma_t}\right| \\
&=\min_{t\in \llbracket 1, k-1\rrbracket \cup \llbracket s+1, r\rrbracket} \frac{\left| 1- \frac{\sigma_t^2}{\phi_1(\widetilde\sigma_i)\phi_2(\widetilde\sigma_i)}\right|}{\sqrt{1+\beta(\widetilde \sigma_i)^2 \sigma_t^2} + \alpha(\widetilde \sigma_i) \sigma_t}.
\end{align*}
\endgroup
The upper bounds of $\alpha(\widetilde \sigma_i)$ and $\beta(\widetilde \sigma_i) $ obtained above yield that
\begin{align*}
\sqrt{1+\beta(\widetilde \sigma_i)^2 \sigma_t^2} + \alpha(\widetilde \sigma_i) \sigma_t \le \sqrt{1+ \frac{\tau(b)^4}{64b^4} \frac{\sigma_t^2}{\widetilde\sigma_i^2}} + \tau(b)\frac{\sigma_t}{\widetilde\sigma_i} \le 1+\tau(b)\left(1 +\frac{\tau(b)}{8b^2} \right)\frac{\sigma_t}{\widetilde \sigma_i}
\end{align*}
for any $t\in \llbracket 1, k-1\rrbracket \cup \llbracket s+1, r\rrbracket$. Hence, 
\begin{align*}
&\min_{t\in \llbracket 1, k-1\rrbracket \cup \llbracket s+1, r\rrbracket} \left| \sqrt{1+\beta(\widetilde \sigma_i)^2 \sigma_t^2} \pm \alpha(\widetilde \sigma_i) \sigma_t \right|\\
& \ge\min_{t\in \llbracket 1, k-1\rrbracket \cup \llbracket s+1, r\rrbracket} \frac{1}{1+\tau(b)\left(1 +\frac{\tau(b)}{8b^2} \right)\frac{\sigma_t}{\widetilde \sigma_i}
}\left|  \frac{\phi_1(\widetilde\sigma_i)\phi_2(\widetilde\sigma_i)-\sigma_t^2}{\phi_1(\widetilde\sigma_i)\phi_2(\widetilde\sigma_i)}\right|\\
&\ge \min_{t\in \llbracket 1, k-1\rrbracket \cup \llbracket s+1, r\rrbracket}\frac{1}{1+\tau(b)\left(1 +\frac{\tau(b)}{8b^2} \right)\frac{\sigma_t}{\widetilde \sigma_i}
} \frac{|\phi_1(\widetilde\sigma_i)\phi_2(\widetilde\sigma_i)-\sigma_t^2|}{ \chi(b)^2\widetilde\sigma_i^2}.
\end{align*}
 To continue the estimates, we consider the cases $t\in \llbracket 1, k-1\rrbracket$ and $t\in \llbracket s+1, r\rrbracket$ separately. 

First, for any $t\in \llbracket 1, k-1\rrbracket$, $\sigma_t \ge \sigma_{i}$ and $\sigma_t \ge \sigma_{l_i}$ since $i,l_i \in \llbracket k_0, s\rrbracket$. By \eqref{eq:location01}, $$\phi_1(\widetilde\sigma_i)\phi_2(\widetilde\sigma_i) \le \sigma_{l_i}^2 +20\xi(b)\chi(b)\eta r(\widetilde \sigma_i +\chi(b)\sigma_{l_i}).$$Thus, we obtain
\begin{align}\label{eq:firstcase1}
\sigma_t^2 - \phi_1(\widetilde\sigma_i)\phi_2(\widetilde\sigma_i) &\ge \sigma_t^2 -\sigma_{l_i}^2 -20\xi(b)\chi(b)\eta r(\widetilde \sigma_i +\chi(b)\sigma_{l_i})\nonumber\\
&\ge (\sigma_t-\sigma_{l_i})(\sigma_t+\sigma_{l_i}) -20\xi(b)\chi(b)^2 \eta r \left(\big(1+\frac{1}{4b}\big) \sigma_t + \sigma_{l_i}\right).
\end{align}
The last inequality is due to
$$ \widetilde \sigma_i \le \chi(b)\left( 1+\frac{1}{4b} \right)\sigma_{l_i}\le \chi(b)\left( 1+\frac{1}{4b} \right)\sigma_{t}$$
from \eqref{eq:upperbdsv} and $\sigma_t \ge \sigma_{l_i}$. Since $\sigma_t- \sigma_{l_i} \ge \delta_{k-1} \ge 75\chi(b) \eta r$, we further get
\begin{align*}
\sigma_t^2 - \phi_1(\widetilde\sigma_i)\phi_2(\widetilde\sigma_i) &\ge \left( 1- \frac{4}{15}\xi(b)\chi(b)\big(1+\frac{1}{4b}\big) \right) \delta_{k-1} (\sigma_t +\sigma_{l_i})>0
\end{align*}
since $1- \frac{4}{15}\xi(b)\chi(b)\big(1+\frac{1}{4b}\big) \ge 1- \frac{4}{15}\xi(2)\chi(2)\big(1+\frac{1}{8}\big)\approx 0.49$.

Hence, we further have 
\begingroup
\allowdisplaybreaks
\begin{align*}
&\min_{t\in \llbracket 1, k-1\rrbracket}\frac{1}{1+\tau(b)\left(1 +\frac{\tau(b)}{8b^2} \right)\frac{\sigma_t}{\widetilde \sigma_i}
} \frac{|\phi_1(\widetilde\sigma_i)\phi_2(\widetilde\sigma_i)-\sigma_t^2|}{ \chi(b)^2\widetilde\sigma_i^2}\\
&=\frac{1}{ \chi(b)^2 \widetilde\sigma_i} \min_{t\in \llbracket 1, k-1\rrbracket} \frac{\sigma_t^2-\phi_1(\widetilde\sigma_i)\phi_2(\widetilde\sigma_i)}{\widetilde\sigma_i+\tau(b)\left(1 +\frac{\tau(b)}{8b^2} \right){\sigma_t}}\\
&\ge  \left( 1- \frac{4}{15}\xi(b)\chi(b)\big(1+\frac{1}{4b}\big) \right)\frac{\delta_{k-1}}{ \chi(b)^2 \widetilde\sigma_i}  \min_{t\in \llbracket 1, k-1\rrbracket} \frac{\sigma_t +\sigma_{l_i}}{\widetilde\sigma_i+\tau(b)\left(1 +\frac{\tau(b)}{8b^2} \right){\sigma_t}}\\
&\ge \left( 1- \frac{4}{15}\xi(b)\chi(b)\big(1+\frac{1}{4b}\big) \right)\frac{\delta_{k-1}}{ \chi(b)^2 \widetilde\sigma_i}  \min_{t\in \llbracket 1, k-1\rrbracket} \frac{\sigma_{l_i}+\sigma_t}{ \chi(b)\left( 1+\frac{1}{4b} \right)\sigma_{l_i}+\tau(b)\left(1 +\frac{\tau(b)}{8b^2} \right){\sigma_t}}.
\end{align*}
\endgroup
Note that $\chi(b)\left( 1+\frac{1}{4b} \right) \ge \tau(b)\left(1 +\frac{\tau(b)}{8b^2} \right)$ for $b\ge 2$ (checked via Mathematica). We conclude that
\begin{align}\label{eq:firstcase}
\min_{t\in \llbracket 1, k-1\rrbracket} \left| \sqrt{1+\beta(\widetilde \sigma_i)^2 \sigma_t^2} \pm \alpha(\widetilde \sigma_i) \sigma_t \right|\ge \frac{1- \frac{4}{15}\xi(b)\chi(b)\big(1+\frac{1}{4b}\big)}{ \chi(b)^3 \left( 1+\frac{1}{4b} \right)} \frac{\delta_{k-1}}{\widetilde\sigma_i}.
\end{align}

Next, for any $t\in \llbracket s+1, r\rrbracket$, $\sigma_t/\sigma_{l_i} \le 1$ and by \eqref{eq:lowerbdsv}, $\sigma_t/{\widetilde \sigma_i} \le \left( 1-\frac{\chi(b)}{4b} \right)^{-1}.$ Consequently, 
\begin{align}\label{eq:anothercase}
&\min_{t\in  \llbracket s+1, r\rrbracket}\frac{1}{1+\tau(b)\left(1 +\frac{\tau(b)}{8b^2} \right)\frac{\sigma_t}{\widetilde \sigma_i}
} \frac{|\phi_1(\widetilde\sigma_i)\phi_2(\widetilde\sigma_i)-\sigma_t^2|}{ \chi(b)^2\widetilde\sigma_i^2} \nonumber\\
&\ge \frac{1}{1+\tau(b)\left(1 +\frac{\tau(b)}{8b^2} \right)\left( 1-\frac{\chi(b)}{4b} \right)^{-1}}\frac{1}{\chi(b)^2}\min_{t\in \llbracket s+1, r\rrbracket} \frac{|\phi_1(\widetilde\sigma_i)\phi_2(\widetilde\sigma_i)-\sigma_t^2|}{\widetilde\sigma_i^2}.
\end{align}
By \eqref{eq:location01}, $$\phi_1(\widetilde\sigma_i)\phi_2(\widetilde\sigma_i) \ge \sigma_{l_i}^2 -20\xi(b)\chi(b)\eta r(\widetilde \sigma_i +\chi(b)\sigma_{l_i}).$$
Using a similar argument as \eqref{eq:firstcase1}, one has
\begin{align*}
\phi_1(\widetilde\sigma_i)\phi_2(\widetilde\sigma_i)-\sigma_t^2 &\ge  \sigma_{l_i}^2 - \sigma_t^2  -20\xi(b)\chi(b)^2 \eta r \left(\big(1+\frac{1}{4b}\big) \sigma_t + \sigma_{l_i}\right)\\
&\ge \left( 1- \frac{4}{15}\xi(b)\chi(b)\big(1+\frac{1}{4b}\big) \right) \delta_s (\sigma_t + \sigma_{l_i})>0
\end{align*}
since $\sigma_{l_i} - \sigma_t \ge \delta_{s} \ge 75\chi(b)\eta r$. 
Continuing from \eqref{eq:anothercase}, we further get
\begin{align*}
&\min_{t\in  \llbracket s+1, r\rrbracket} \left| \sqrt{1+\beta(\widetilde \sigma_i)^2 \sigma_t^2} \pm \alpha(\widetilde \sigma_i) \sigma_t \right| \nonumber\\
&\ge \frac{1- \frac{4}{15}\xi(b)\chi(b)\big(1+\frac{1}{4b}\big)}{1+\tau(b)\left(1 +\frac{\tau(b)}{8b^2} \right)\left( 1-\frac{\chi(b)}{4b} \right)^{-1}}\frac{1}{\chi(b)^2}\frac{\delta_s}{\widetilde \sigma_i}\min_{t\in \llbracket s+1, r\rrbracket} \frac{\sigma_t + \sigma_{l_i}}{\widetilde\sigma_i}\nonumber\\
&\ge \frac{1- \frac{4}{15}\xi(b)\chi(b)\big(1+\frac{1}{4b}\big)}{1+\tau(b)\left(1 +\frac{\tau(b)}{8b^2} \right)\left( 1-\frac{\chi(b)}{4b} \right)^{-1}}\frac{1}{\chi(b)^3 \big(1+\frac{1}{4b}\big)}\frac{\delta_s}{\widetilde \sigma_i}:=\nu(b) \frac{\delta_s}{\widetilde \sigma_i},
\end{align*}
where the last inequality follows from \eqref{eq:upperbdsv}. 

Comparing \eqref{eq:firstcase} and \eqref{eq:lwsv1}, together with the observation $$\frac{1- \frac{4}{15}\xi(b)\chi(b)\big(1+\frac{1}{4b}\big)}{ \chi(b)^3 \left( 1+\frac{1}{4b} \right)} \ge \nu(b)$$ for $b\ge 2$ (checked via Mathematica), we conclude that
\begin{align}\label{eq:lwsv1}
&\min_{t\in \llbracket 1, k-1\rrbracket \cup \llbracket s+1, r\rrbracket} \left| \sqrt{1+\beta(\widetilde \lambda_i)^2 \sigma_t^2} \pm |\alpha(\widetilde \lambda_i)| \sigma_t \right| \ge \nu(b)\frac{\min\{\delta_{k-1},\delta_s\}}{\widetilde \sigma_i}.
\end{align}

For the case when $ i\in \llbracket r+k_0, r+s\rrbracket$ and $\widetilde \lambda_i = -\widetilde \sigma_{i-r}$. Use the observation that $\alpha(\widetilde \lambda_i) \sim -\alpha(\widetilde \sigma_{i-r})$ and $\beta(\widetilde \lambda_i) \sim -\beta(\widetilde \sigma_{i-r})$ from the definitions \cite[Eq. (27)]{Wang24}. A simple modification of the previous proof shows that
\begin{align}\label{eq:lwsv2}
\min_{t\in \llbracket 1, k-1\rrbracket \cup \llbracket s+1, r\rrbracket} \left| \sqrt{1+\beta(\widetilde \lambda_i)^2 \sigma_t^2} \pm |\alpha(\widetilde \lambda_i)| \sigma_t \right| \ge \nu(b)\frac{\min\{\delta_{k-1},\delta_s\}}{\widetilde \sigma_{i-r}}.
\end{align}

\medskip
\noindent\emph{Step 3. Combining the bounds above}.
With the estimates deduced in the previous two steps, we are in a position to bound $\| {\mathcal U}_J^\T \widetilde{\mathbf u}_i \|$. For $i\in \llbracket k_0,s\rrbracket $,
plugging \eqref{eq:upbd1} and \eqref{eq:lwsv1} and  into \eqref{eq:recursive}, we find that 
\begin{align*}
\| {\mathcal U}_J^\T \widetilde{\mathbf u}_i \| &\le \frac{b}{(b-1)\nu(b)} \frac{\eta}{\min\{\delta_{k-1}, \delta_s\}}.
\end{align*}
Finally, for simplicity, we employ the following bound $$ \frac{b}{(b-1)\nu(b)} < 3 \frac{(b+1)^2}{(b-1)^2}$$ for $b\ge 2$ (checked via Mathematica). We arrive at
\begin{align*}
\| {\mathcal U}_J^\T \widetilde{\mathbf u}_i \| &\le 3 \frac{(b+1)^2}{(b-1)^2} \frac{\eta}{\min\{\delta_{k-1}, \delta_s\}}. 
\end{align*}
Likewise, for $i\in \llbracket r+k_0, r+s\rrbracket$, using  \eqref{eq:upbd2} and \eqref{eq:lwsv2}, we also get
\begin{align*}
\| {\mathcal U}_J^\T \widetilde{\mathbf u}_i \| \le 3 \frac{(b+1)^2}{(b-1)^2} \frac{\eta}{\min\{\delta_{k-1}, \delta_s\}}.
\end{align*}
This completes the proof.

\section{Proof of Theorem 6.4}\label{app:singularlocation}

This section is devoted to the proof of Theorem 6.4. 
For convenience, 
denote $$M:= 2b(\sqrt N+ \sqrt{n}).$$ 
Note that the assumptions of Theorem  6.4 guarantees that for any $z\in S_{\sigma_j}$ ($1\le j \le r_0$), $|z| \ge \Re(z) \ge \sigma_j - 20\chi(b) \eta r \ge \sigma_j - 20\chi(2) \eta r=\sigma_j - \frac{45}{2} \eta r> M$.

We start with some reduction of the proof. First, note that if $\sigma_j > n^2/2$ for $1\le j \le r_0$, then by Lemma 6.6, with probability at least $1- (N+n)^{-1.5 r^2(K+4)}$, $$|\widetilde \sigma_j - \sigma_j |\le \eta r.$$Since $\varphi(z)=(z - \tr \mathcal I^\mathrm{d} G(z))(z - \tr \mathcal I^\mathrm{u} G(z))$,
\begin{align*}
|\varphi(\widetilde \sigma_j) - \sigma_j^2| &=\left|\widetilde \sigma_j^2 - \sigma_j^2 -\widetilde \sigma_j \tr G(\widetilde \sigma_j) + (\tr \mathcal I^\mathrm{u} G(\widetilde \sigma_j))(\tr \mathcal I^\mathrm{d} G(\widetilde \sigma_j)) \right|\\
&\le \eta r(\widetilde \sigma_j  + \sigma_j) + \widetilde \sigma_j \left|\tr G(\widetilde \sigma_j) \right| + \left|\tr \mathcal I^\mathrm{u} G(\widetilde \sigma_j)\right| \left|\tr \mathcal I^\mathrm{d} G(\widetilde \sigma_j)\right|.
\end{align*}
Note that by Weyl's inequality, $\widetilde \sigma_j \ge \sigma_j - \|E\| \ge \max\{M + 80 b\eta r,n^2/2\} - 2(\sqrt N + \sqrt n) \ge M$ by the suppositions on $N,n$. Hence, by \cite[Eq. (32)]{Wang24}, 
$$\max\{|\tr G(\widetilde \sigma_j )|, |\tr \mathcal I^u G(\widetilde \sigma_j )  |, |\tr \mathcal I^d G(\widetilde \sigma_j )  |\}  \le \frac{b}{b-1}\frac{N+n}{\widetilde \sigma_j}\le 2\frac{N+n}{\widetilde \sigma_j}.$$
It follows that
\begin{align*}
|\varphi(\widetilde \sigma_j) - \sigma_j^2| \le \eta r(\widetilde \sigma_j  + \sigma_j) + 2 (N+n) + 4\frac{(N+n)^2}{\widetilde \sigma_j^2}\le 10\eta r(\widetilde \sigma_j  + \sigma_j)
\end{align*}
by the supposition that $\sigma_j > n^2/2$ and the Weyl's inequality. In particular, the conclusion of Theorem 6.4 holds. 

Consequently, it is enough to examine the scenario where there is a certain $l_0 \in \llbracket 1, r_0\rrbracket$ for which $\sigma_{l_0} \le n^2/2$. We claim that there exists an index $i_0\in \llbracket 1, r_0\rrbracket$ such that $\sigma_j \le  n^3$ for $j\ge i_0$ and $\sigma_j> n^3$ for $j < i_0$, and 
$$\delta_{i_0-1} =\sigma_{i_0-1}- \sigma_{i_0}\ge 75 \chi(b) \eta r.$$
To determine $i_0$, we propose a simple iterative algorithm: start with $\sigma_1$. If $\sigma_1\le n^3$, set $i_0=1$ and terminate the algorithm, since $\sigma_0=\infty$ and $\delta_0=\infty$ by definition. Assume $\sigma_1> n^3$ and evaluate $\sigma_2$. If $\sigma_2 \le n^3 - 75 \chi(b) \eta r$,  set $i_0=2$ and exit. Assume $\sigma_2 > n^3 - 75 \chi(b) \eta r$ and evaluate $\sigma_3$. We continue this process and terminate the algorithm with $i_0=k $ unless 
\begin{align}\label{eq:i0}
\sigma_1> n^3, \sigma_2 > n^3 - 75 \chi(b) \eta r, \cdots, \sigma_k > n^3 - 75 \chi(b) \eta r.
\end{align}
Note that the condition \eqref{eq:i0} cannot hold for $k=l_0$ because $\sigma_{l_0} \le n^2/2 < n^3 - 75 \chi(b) \eta r$, based on the assumption that $(\sqrt N + \sqrt n)^2 \ge 32(K+7)\log(N+n) + 64(\log 9) r$. Therefore, $i_0$ must satisfy $i_0 \le l_0 - 1$.

We shall fix such an index $i_0$  throughout the rest of the proof. The goal is to demonstrate that the following holds with a probability of at least $1-10(N+N)^{-K}$: assume any  $i_0 \le k\le s \le r_0$ that fulfills $\min\{\delta_{k-1},\delta_s \} \ge 75 \chi(b) \eta r$. For any $j \in \llbracket k, s \rrbracket$, there exists $j_0 \in \llbracket k, s \rrbracket$ such that $\widetilde \sigma_j \in S_{\sigma_{j_0}}$ and \cite[Eq. (36)]{Wang24} is satisfied. 
 
Before moving forward with the proof, we review several results and introduce necessary notations collected from \cite{OVW22}. The proofs of these results are identical to these in \cite{OVW22}, utilizing Lemma 6.1, and we will not repeat them here. 
\begin{lemma}[Eigenvalue location criterion, Lemma 21 from \cite{OVW22}]
Assume $\mathcal A$ has rank $2r$ with the spectral decomposition $\mathcal A = \mathcal U \mathcal D \mathcal U^\T$, where $\mathcal U$ is an $(N+n) \times 2r$ matrix satisfying $\mathcal U^\T \mathcal U = I_{2r}$ and $\mathcal D$ is a $2r \times 2r$ diagonal matrix with non-zero $\lambda_1, \ldots, \lambda_{2r}$ on the diagonal.  Then the eigenvalues of $\mathcal A+\mathcal E$ outside of $[ - \|\mathcal E \|, \|\mathcal E\| ]$ are the zeros of the function 
\[ z \mapsto \det (\mathcal D^{-1} - \mathcal U^\T G(z) \mathcal U ). \]
Moreover, the algebraic multiplicity of each eigenvalue matches the corresponding multiplicity of each zero.  
\end{lemma}

Define the functions
\[ f(z) := \det(\mathcal D^{-1} - \mathcal U^\T G(z) \mathcal U), \qquad g(z) := \det \left( \mathcal D^{-1} - \mathcal U^\T \Phi(z) \mathcal U \right), \]
where $\Phi(z)$ is given in \cite[Eq. (25)]{Wang24}.  Observe that, by Lemma 6.1,  $1/\phi_1(z)$, $1/\phi_2(z)$ and thus $\Phi(z)$ are well-defined for any $|z| > M$. Therefore, $f$ and $g$ are both complex analytic in the region $\{z \in \mathbb{C} : |z| > M\}$.  Furthermore, a direct computation using \cite[Eq. (29)]{Wang24} suggests that the zeros of $g(z)$ are the values $z \in \mathbb{C}$ which satisfy the equations $\phi_1(z)\phi_2(z) = \sigma_l^2.$
This can be verified through the following calculation:
\begin{align*}
g(z)&=\det \left[ \begin{pmatrix} D^{-1} & 0\\ 0&-D^{-1} \end{pmatrix}- \begin{pmatrix} \alpha I_r & \beta I_r \\ \beta I_r & \alpha I_r \end{pmatrix} \right] \\
&= \det \left[ (D^{-1} -\alpha I_r)(-D^{-1} - \alpha I_r) -\beta^2 I_r \right]\\
&=\prod_{l=1}^r \left( (\alpha - \sigma_l^{-1})(\alpha + \sigma_l^{-1}) -\beta^2\right).
\end{align*}
The zeros of $g(z)$ can then be determined by substituting the expressions for $\alpha=\alpha(z)$ and $\beta=\beta(z)$ from \cite[Eq. (29)]{Wang24}.

Recall from \cite[Eq. (34)]{Wang24} and \cite[Eq. (27)]{Wang24} that 
\[ \varphi(z)= \phi_1(z)\phi_2(z)=(z-\tr \mathcal I^{\mathrm d} G(z))(z-\tr \mathcal I^{\mathrm u} G(z)).\]We use the function 
\begin{equation}\label{eq:xib}
\xi(b)=1+\frac{1}{2(b-1)^2}.
\end{equation}
The subsequent lemma establishes a set of properties exhibited by $\varphi$ within the complex plane as well as on the real axis. 
\begin{lemma}[Lemma 22 from \cite{OVW22}] \label{lemma:varphi}
The function $\varphi$ satisfies the following properties.  
\begin{enumerate}[(i)]
\item  For $z, w \in \mathbb{C}$ with $|z|, |w|, |z+w| \ge M$, 
\begin{equation} \label{eq:lipschitz}
	\left(1-\frac{1}{2(b-1)^2} \right) |z^2-w^2| \leq | \varphi(z) - \varphi(w)| \leq \xi(b) |z^2- w^2|.
\end{equation}
\item (Monotone) $\varphi$ is real-valued and strictly increasing on $[M, \infty)$.  
\item \label{item:crude} (Crude bounds) $0 < \varphi(z) < z^2$ for any $z \in [M, \infty)$.  
\end{enumerate}
\end{lemma}

Fix an index $j\in \llbracket 1, r_0\rrbracket$. Since $\varphi(M) < M^2$ and $\lim_{z \to \infty} \varphi(z) = \infty$, it follows from the previous lemma that there exists a unique positive real number $z_j > M$ such that $\varphi(z_j) =\sigma_j^2$.  Similarly, if $\sigma_{l} > M$ for  $\sigma_l\neq \sigma_j$, then there exists a unique positive real number $z_{l}$ with $\varphi(z_{l}) = \sigma_{l}^2$ so that $z_j > z_{l}$ if $l>j$ and $z_j< z_l$ if $l<j$.   

For the next result, we define the half space 
\[ H_j := \{z \in \mathbb{C} : \Re(z) \geq z_j - 20\chi(b) \eta r\} \quad \text{with } \chi(b)=1+\frac{1}{4b(b-1)}.\] 
\begin{proposition}[Proposition 23 from \cite{OVW22}] \label{prop:H}
Under the assumptions of Theorem 2.3, for every $z \in H_j$, 
\[ |z| \geq \sigma_j \ge M. \]
In particular, 
\begin{align}\label{eq:zjbd}
\sigma_j \le z_j \le \chi(b) \sigma_j. 
\end{align}
\end{proposition} 
\begin{proposition}[Proposition 24 from \cite{OVW22}]\label{prop:bigbound}
If $\sigma_j> \frac{1}{2}n^2$, then $|z_j - \sigma_j | \le \frac{3b}{b-1}\frac{1}{n}$.
\end{proposition}


We now complete the proof of Theorem 6.4.  Let $j$ be a fixed index in $\llbracket i_0, r_0 \rrbracket$. We will work in the set $H_j \cap S_{\sigma_j}$, where $S_{\sigma_j}$ is specified in \cite[Eq. (35)]{Wang24}.  It follows from Corollary 2.14 in \cite{Ipsen} that 
\begin{equation} \label{eq:rouche}
	\frac{ |f(z) - g(z)|}{|g(z)|} \leq \left( 1 + \varepsilon(z) \right)^{2r} - 1, 
\end{equation} 
where
\[ \eps(z) := \left \| \left( \mathcal D^{-1} - \mathcal U^\T \Phi(z) \mathcal U \right)^{-1} \right \| \left\| \mathcal U^\T (G(z)-\Phi(z)) \mathcal U \right\|. \]

The next result facilitates the comparison of the numbers of zeros of $f$ and $g$ inside a region and will be used repeatedly in the later arguments. 
\begin{lemma}\label{lem:comparezero}
For any region $\mathcal K\subset \mathbb C$ with closed contour $\partial \mathcal K$, if $\eps(z) \le \frac{0.34}{r}$ for all $z\in \partial \mathcal K$, then the number of zeros of $f$ inside $\mathcal{K}$ is the same as the number of zeros of $g$ inside $\mathcal{K}$.
\end{lemma}
\begin{proof}
Continuing from \eqref{eq:rouche}, we find that
\begin{equation} \label{eq:rouchebnd}
	\frac{ |f(z) - g(z)|}{|g(z)|} \leq \left( 1 + \frac{0.34}{r} \right)^{2r} - 1 \leq e^{0.68} - 1 < 1 
\end{equation} 
for each $z \in \partial \mathcal K$.  Therefore, by Rouch\'{e}'s theorem, we conclude that the numbers of zeros of $f$ and $g$ inside $\mathcal{K}$ are the same.
\end{proof}

In the remaining of the proof, we work on the event
\begin{align}\label{eventF}
\mathcal F:=\left\{\max_{ z \in \mathsf D} |z|^2 \left\|  \mathcal U^\T \left(G(z)-\Phi(z) \right) \mathcal U \right\| \leq \eta \right\} \cap \left\{ \max_{l \in \llbracket 1, r_0\rrbracket : \sigma_l >\frac{1}{2}n^2} |\widetilde \sigma_l - \sigma_l |\le \eta r \right\},
\end{align}
where $\mathsf D= \{ z\in \mathbb C : 2b(\sqrt N + \sqrt n)\le |z| \le 2n^3\}.$ 
By Lemma 6.3 and Lemma 6.6, the event $\mathcal F$
 holds with probability at least $1-9(N+n)^{-K} - (N+n)^{-1.5r^2(K+4)} > 1- 10(N+n)^{-K}.$

We first bound $\eps(z)$ for $z\in \mathsf D$. By Proposition 11 from \cite{OVW22}, 
\begin{align*}
\left \| \left( \mathcal D^{-1} - \mathcal U^\T \Phi(z) \mathcal U \right)^{-1} \right \|= \max_{1\le l \le r} \frac{\sigma_l}{|\sigma_l^2 - \phi_1 \phi_2|} \mathcal Q^{1/2},
\end{align*}
where $$\mathcal Q:=|\phi_1 \phi_2|^2 + \frac{1}{2}\sigma_l^2 (|\phi_1|^2 + |\phi_2|^2) + \frac{1}{2}\sigma_l \left[ 4 |\phi_1 \phi_2|^2 |\phi_1 + \bar{\phi}_2|^2 + \sigma_l^2(|\phi_1|^2 - |\phi_2|^2)^2\right]^{1/2}.$$
Recall $\chi(b)=1+\frac{1}{4b(b-1)}.$ Using \cite[Eq. (31)]{Wang24} from Lemma 6.1, for $z\in \mathsf D$, we get $|\phi_i(z)|\le \chi(b) |z|$ for $i=1,2$, and 
\begin{align*}
\mathcal Q &\le \chi(b)^4 |z|^4 + \chi(b)^2 \sigma_l^2 |z|^2 + \chi(b)^2\sigma_l  |z|^2 \sqrt{\sigma_l^2 +4\chi(b)^2|z|^2 }\\
&\le \chi(b)^4 |z|^4 + \chi(b)^2 \sigma_l^2 |z|^2 + \chi(b)^2\sigma_l  |z|^2 (\sigma_l + 2\chi(b) |z|)\\
&\le \chi(b)^4 |z|^4 + 2 \chi(b)^2 \sigma_l^2 |z|^2 + 2 \chi(b)^3\sigma_l  |z|^3\\
&\le \chi(b)^2|z|^2\left(\chi(b) |z|+\sqrt{2}\sigma_l\right)^2, 
\end{align*}
and thus
\begin{align*}
\left \| \left( \mathcal D^{-1} - \mathcal U^\T \Phi(z) \mathcal U \right)^{-1} \right \| 
&\le \chi(b) |z| \max_{1\le l \le r}  \frac{\sigma_l (\chi(b) |z|+\sqrt{2}\sigma_l)}{|\sigma_l^2 - \varphi(z)|}.
\end{align*}
Hence, we obtain that on the event $\mathcal F$, 
\begin{align}\label{eq:epsbd1}
\eps(z) \le\max_{1\le l \le r} \chi(b)\frac{\eta}{|z|}\frac{\sigma_l (\chi(b) |z|+\sqrt{2}\sigma_l)}{|\sigma_l^2 - \varphi(z)|}
\end{align}
for all $z \in \mathsf D$. Note that $S_{\sigma_j} \subset \mathsf D$ for all $j \in \llbracket i_0, r_0 \rrbracket$.

For each $j \in \llbracket i_0, r_0 \rrbracket$, we take $\mathcal{C}_j$ to be the circle of radius $20\chi(b)\eta r$ centered at $z_j$ and contained in $H_j \cap S_{\sigma_j}$. Note that $\mathcal{C}_j$'s may intersect each other. For any $i_0 \le k\le s \le r_0$ satisfying $\min\{\delta_{k-1},\delta_s \} \ge 75 \chi(b) \eta r$. Let $$\mathcal K_{k,s}:=\cup_{l=k}^{s} \mathcal{C}_l.$$ 

We now restrict ourselves to values of $z$ contained on $\partial \mathcal{K}_{k,s}$. The goal is to show $\eps(z)$ is small for all $z\in \mathcal \partial \mathcal{K}_{k,s}$. Continuing from \eqref{eq:epsbd1}, it suffices to show 
\begin{align}\label{eq:bdeps}
\chi(b)\frac{\eta}{|z|}\frac{\sigma_l (\chi(b) |z|+\sqrt{2}\sigma_l)}{|\sigma_l^2 - \varphi(z)|}
\end{align}
is small for all $1\le l \le r.$ 

Fix $z\in \mathcal \partial \mathcal{K}_{k,s}$. Assume $z\in \mathcal{C}_{j_0}$ for some $j_0 \in \llbracket k, s\rrbracket$. Then $$|z-z_{j_0}| = 20 \chi(b) \eta r.$$ Note that $\sigma_{j_0} \ge 80 b \eta r$. Using \eqref{eq:zjbd}, we have
\begin{align}\label{eq:zulb}
&|z| \le z_{j_0} + 20\chi(b) \eta r \le  \chi(b)\left(1+ \frac{1}{4b} \right) \sigma_{j_0},\nonumber\\
&|z| \ge z_{j_0} - 20\chi(b) \eta r \ge \left( 1-\frac{\chi(b)}{4b}\right) \sigma_{j_0}.
\end{align}
We split the discussion into two cases: $|\sigma_l-\sigma_{j_0}| \le 120 \eta r$ and $|\sigma_l - \sigma_{j_0}| > 120 \eta r$.

\emph{Case 1.} For any $l \in \llbracket 1, r\rrbracket$ satisfying $|\sigma_l-\sigma_{j_0}| \le 120 \eta r$,  observe that $|z-z_l| \ge 20 \chi(b) \eta r$. In view of \eqref{eq:lipschitz}, , we have
\begin{align*}
|\sigma_l^2 - \varphi(z)| &= |\varphi(z_l) - \varphi(z)| \ge \left(1-\frac{1}{2(b-1)^2} \right) |z_l^2 - z^2| \ge \frac{1}{2}|z_l-z| |z_l+z|\\
&\ge 10\chi(b)\eta r |z_l + z|\ge \left(1-\frac{\chi(b)}{4b} \right) \sigma_{j_0} +\sigma_l,
\end{align*}
where, in the last inequality, we used 
\begin{align*}
|z_l+z| =|z_l + z_{j_0} +z-z_{j_0}| &\ge z_l + z_{j_0} - 20\chi(b) \eta r \\
&\ge  \sigma_l + \sigma_{j_0} - 20\chi(b) \eta r \ge \left(1-\frac{\chi(b)}{4b} \right) \sigma_{j_0} +\sigma_l
\end{align*}
 by \eqref{eq:zjbd} and the supposition $\eta r \le {\sigma_{j_0}}/{80b}$.
Combining with \eqref{eq:zulb}, we estimate \eqref{eq:bdeps} as follows:
\begin{align*}
\chi(b)\frac{\eta}{|z|}\frac{\sigma_l (\chi(b) |z|+\sqrt{2}\sigma_l)}{|\sigma_l^2 - \varphi(z)|}\le \frac{1}{r} \frac{1}{10(1-\frac{\chi(b)}{4b})} \frac{\sigma_l}{\sigma_{j_0}} \frac{\chi(b)^2 (1+\frac{1}{4b}) \sigma_{j_0} + \sqrt{2}\sigma_l}{(1-\frac{\chi(b)}{4b}) \sigma_{j_0} +\sigma_l}.
\end{align*}
Note that $\sigma_l \le \sigma_{j_0}+120\eta r \le (1 + 120/80b) \sigma_{j_0} \le (7/4) \sigma_{j_0}$ for $b\ge 2$. Also, $\chi(b) \le \chi(2) = 9/8$ for $b\ge 2$. We further obtain that
\begin{align*}
\chi(b)\frac{\eta}{|z|}\frac{\sigma_l (\chi(b) |z|+\sqrt{2}\sigma_l)}{|\sigma_l^2 - \varphi(z)|}&\le \frac{1}{r} \frac{1}{10(1- 9/64)} \frac{7}{4}  \frac{(9/8)^3 \sigma_{j_0} + \sqrt 2 \sigma_l}{ (1- 9/64) \sigma_{j_0} + \sigma_l}\\
&\le \frac{1}{r} \frac{7}{40(1- 9/64)} \frac{(9/8)^3}{ 1- 9/64} < \frac{0.34}{ r}.
\end{align*}

\emph{Case 2. For any $l \in \llbracket 1, r\rrbracket$ satisfying $|\sigma_l-\sigma_{j_0}| > 120 \eta r$, } we start with 
\begin{align*}
|\sigma_l^2 - \varphi(z)| &\ge |\sigma_l^2- \sigma_{j_0}^2| -| \sigma_{j_0}^2 - \varphi(z)|=|\sigma_l^2- \sigma_{j_0}^2| -| \varphi(z_{j_0}) - \varphi(z)|\\
&\ge |\sigma_l-\sigma_{j_0}| (\sigma_l+\sigma_{j_0}) -  20 \chi(b) \xi(b)\eta r|z_{j_0} + z|.
\end{align*}
by \eqref{eq:lipschitz} and $|z-z_{j_0}| = 20\chi(b) \eta r$. Since  
\begin{align*}
|z_{j_0} + z| &\le 2z_{j_0}+| z- z_{j_0}| = 2z_{j_0} +20\chi(b) \eta r\\
& \le 2\chi(b) \sigma_{j_0} +20\chi(b) \eta r \le 2\chi(b) \left(1+ {1}/{8b} \right)\sigma_{j_0} \le \frac{34}{16}\chi(b) \sigma_{j_0}
\end{align*}
 and $|\sigma_l-\sigma_{j_0}| > 120 \eta r$, we further get 
\begin{align*}
|\sigma_l^2 - \varphi(z)| &\ge |\sigma_l-\sigma_{j_0}| (\sigma_l+\sigma_{j_0}) -  20 \chi(b)^2\xi(b) \frac{34}{16}\frac{1}{120} |\sigma_l-\sigma_{j_0}| (\sigma_l+\sigma_{j_0})\\
&=\left( 1- \frac{17}{48} \chi(b)^2\xi(b) \right)|\sigma_l-\sigma_{j_0}| (\sigma_l+\sigma_{j_0}).
\end{align*}

Hence, using  \eqref{eq:zulb}, we get
\begin{align*}
&\chi(b)\frac{\eta}{|z|}\frac{\sigma_l (\chi(b) |z|+\sqrt{2}\sigma_l)}{|\sigma_l^2 - \varphi(z)|} \le \frac{\chi(b) \eta}{(1-\frac{\chi(b)}{4b}) (1-\frac{17}{48} \chi(b)^2\xi(b))} \frac{\sigma_l}{\sigma_{j_0}}\frac{ \chi(b)^2(1+\frac{1}{4b})\sigma_{j_0} + \sqrt{2}\sigma_l }{|\sigma_l-\sigma_{j_0}| (\sigma_l+\sigma_{j_0})}.
\end{align*}
To continue the estimates, we simply use the fact that $\chi(b),\xi(b)$ are decreasing. Thus $\chi(b)\le \chi(2)=9/8$ and $\xi(b)=1+1/(2(b-1)^2)\le 3/2$ for $b\ge 2$. Hence,
\begin{align*}
&\chi(b)\frac{\eta}{|z|}\frac{\sigma_l (\chi(b) |z|+\sqrt{2}\sigma_l)}{|\sigma_l^2 - \varphi(z)|}\le \frac{(9/8)^4}{\frac{55}{64} (1-\frac{17}{48}(\frac{9}{8})^2\frac{3}{2})}\frac{\sigma_l}{\sigma_{j_0}} \frac{\eta}{|\sigma_{j_0}-\sigma_l|} < 6 \frac{\sigma_l}{\sigma_{j_0}} \frac{\eta}{|\sigma_{j_0}-\sigma_l|}.
\end{align*}
If $\sigma_l \le 2 \sigma_{j_0}$, then 
\begin{align*}
&6 \frac{\sigma_l}{\sigma_{j_0}} \frac{\eta}{|\sigma_{j_0}-\sigma_l|} \le 12 \frac{\eta}{120 \eta r} = \frac{0.1}{r}.
\end{align*}
If $\sigma_l \ge 2 \sigma_{j_0}$, then $\sigma_l-\sigma_{j_0} \ge 0.5 \sigma_l $ and 
\begin{align*}
&6 \frac{\sigma_l}{\sigma_{j_0}} \frac{\eta}{|\sigma_{j_0}-\sigma_l|}\le 12 \frac{\eta}{\sigma_{j_0}} \le  12 \frac{\eta}{160\eta r}=\frac{0.075}{r}.
\end{align*}

Thus, we conclude that 
\begin{equation}\label{eq:epsbd}
 \eps(z) \leq \frac{0.34}{ r} 
\end{equation} 
for all $z \in \mathcal{K}_{k,s}$. By Lemma \ref{lem:comparezero}, the number of zeros of $f$ inside $\mathcal{K}_{k,s}$ is the same as the number of zeros of $g$ inside $\mathcal{K}_{k,s}$. 

Since $\min\{\delta_{i_0-1},\delta_{r_0} \} \ge 75\chi(b) \eta r$ by our supposition, we could take $k=i_0$ and $s=r_0$ and thus $\mathcal{K}_{i_0,r_0} = \cup_{l=i_0}^{r_0} \mathcal{C}_l$. 
Since $g$ has $r_0 - i_0 +1$ zeros inside $\mathcal{K}_{i_0,r_0}$, it follows that $\mathcal A + \mathcal E$ has exactly $r_0 - i_0 +1$ eigenvalues inside $\mathcal{K}_{i_0,r_0}$. 
More generally, for any $i_0 \le k\le s \le r_0$ satisfying $\min\{\delta_{k-1},\delta_s \} \ge 75 \chi(b) \eta r$, we conclude that the number of eigenvalues of $\mathcal A + \mathcal E$ inside $\mathcal{K}_{k,s}$ is $s-k+1$, the same as the number of zeros of $g$ inside $\mathcal{K}_{k,s}$.  
It remains to show that these eigenvalues are exactly $\widetilde \sigma_{k}, \cdots, \widetilde \sigma_{s}$. If this is the case, then for any $j \in \llbracket k, s \rrbracket$, there exists $j_0 \in \llbracket k, s \rrbracket$ such that $\widetilde \sigma_j \in \mathcal{C}_{j_0}$ and thus $$|\widetilde \sigma_j -z_{j_0}| \le 20 \chi(b) \eta r. $$ In particular, $\widetilde \sigma_j \in S_{\sigma_{j_0}}$. By $\varphi(z_{j_0})=\sigma_{j_0}^2$, \eqref{eq:lipschitz} and \eqref{eq:zjbd},
\begin{align*}
|\varphi(\widetilde{\sigma}_{j}) - \sigma_{j_0}^2| =|\varphi(\widetilde{\sigma}_{j}) - \varphi(z_{j_0})|\leq 20\xi(b) \chi(b)\eta r \left(\widetilde{\sigma}_{j}+\chi(b)\sigma_{j_0} \right).
\end{align*}
This will complete the proof. 

It remains to prove that for any $i_0 \le k\le s \le r_0$ satisfying $\min\{\delta_{k-1},\delta_s \} \ge 75 \chi(b) \eta r$, the eigenvalues of $\mathcal A + \mathcal E$ inside $\cup_{l=k}^s \mathcal C_l$ are exactly $\widetilde \sigma_{k}, \cdots, \widetilde \sigma_{s}$.  We will do so by proving the following claims hold on the event $\mathcal F$ (see Figure \ref{fig:cycles} for an illustration): 
\begin{figure}[!ht]
 \begin{center}
   \includegraphics[width=12cm]{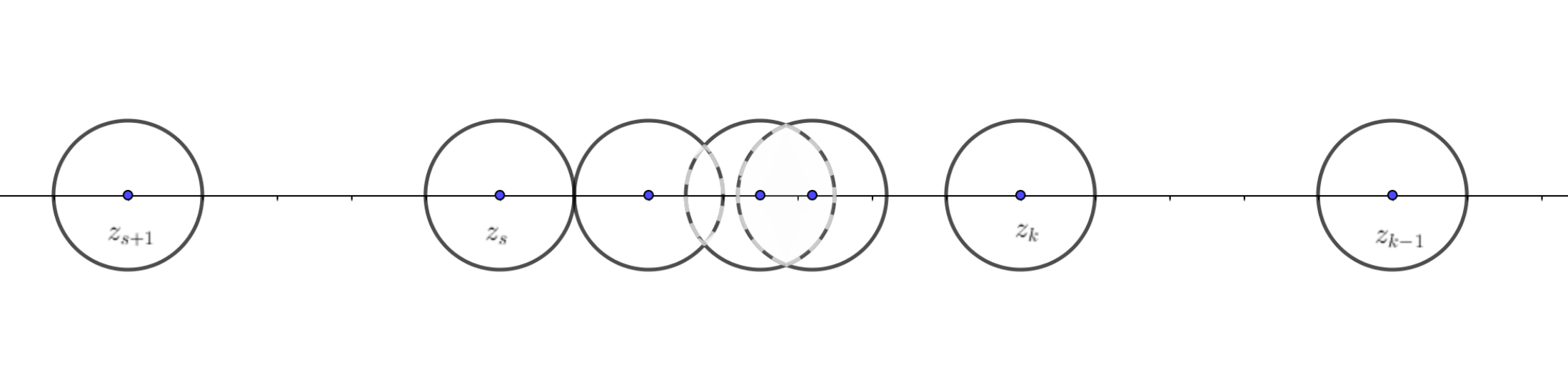}
   \caption{Distinct circles $\mathcal C_j$ with centers $z_j$ on the real line for $i_0 \le j \le r_0$.}
    \label{fig:cycles}
 \end{center}
\end{figure}

\noindent\emph{Claim 1.} For any $i_0 \le k\le s \le r_0$ satisfying $\min\{\delta_{k-1},\delta_s \} \ge 75 \chi(b) \eta r$, $\cup_{l=k}^s \mathcal C_l$ does not intersect other circles.\\
\noindent\emph{Claim 2.}  $\mathcal A+\mathcal E$ has exactly $i_0-1$ eigenvalues larger than $z_{i_0}+20\chi(b)\eta r$.\\
\noindent\emph{Claim 3.} No eigenvalues of $\mathcal A+\mathcal E$ lie between disjoint circles. 

For the moment, let us assume these claims are true. Note that $\widetilde \sigma_{i_0}$ has to lie inside one of the $\mathcal C_j$'s ($i_0 \le j \le r_0$) because it is the largest eigenvalue of $\mathcal A+\mathcal E$ that is no larger than $z_{i_0}+20\chi(b)\eta r$ (due to \emph{Claim 2}) and thus it satisfies $\widetilde \sigma_{i_0} > z_{r_0} - 20\chi(b)\eta r$. Since the number of zeros of $g(z)$ located inside $\mathcal K_{i_0,r_0}=\cup_{j=i_0}^{r_0}\mathcal C_j$, which is $r_0-i_0+1$, is the same as that of $f(z)$ inside $\mathcal K_{i_0,r_0}$, we have $\widetilde \sigma_{i_0},\ldots,\widetilde \sigma_{r_0}$ lie inside $\mathcal K_{i_0,r_0}$. The conclusion follows by \emph{Claim 1}, \emph{Claim 3} and the fact that the number of zeros of $g(z)$ in each $\mathcal K_{k,s}$ is the same as that of $f(z)$.

We start with the proof of the \emph{Claim 1}. It suffices to show that if $|\sigma_l -\sigma_j| \ge 75 \chi(b) \eta r$, then $\mathcal C_l$ and $\mathcal C_j$ do not intersect. By Lemma \ref{lemma:varphi},
$$|z_l^2 - z_j^2| \ge \frac{1}{\xi(b)}|\varphi(z_l) - \varphi(z_j)| =\frac{|\sigma_l^2 - \sigma_j^2|}{\xi(b)}  \ge \frac{75 \chi(b) \eta r}{\xi(b)}(\sigma_l + \sigma_j).$$
Since $|z_l^2 - z_j^2| = (z_l+z_j) |z_l-z_j| \le \chi(b)(\sigma_l + \sigma_j) |z_l-z_j|$ by Proposition \ref{prop:H}, we have
\begin{equation}\label{eq:zljdiff}
|z_l - z_j |\ge \frac{75}{\xi(b)}\eta r, 
\end{equation}
and thus
$$\dist(\mathcal C_j, \mathcal C_l) \ge |z_l - z_j| - 40\chi(b)  \eta r \ge \frac{75}{\xi(b)}\eta r-40\chi(b)\eta r\ge \left(\frac{75}{\xi(2)}-40\chi(2)\right)\eta r>0. $$

\medskip

Next, we prove \emph{Claim 2}. We split the proof into two cases: $i_0=1$ and $i_0>1$.

\noindent\emph{Case 1: $i_0=1$.} We prove that no eigenvalues of $\mathcal A+\mathcal E$ are larger than $z_1+20\chi(b) \eta r$. We now take $\mathcal{C}_0$ to be any circle with radius $20\chi(b)  \eta r$ centered at a point $z_0>z_1+20\chi(b)\eta r$ on the real line inside the region $H_1 \cap \hat{S}_{\sigma_1}$ such that $\dist(z_1,\mathcal{C}_0) \geq 20\chi(b)  \eta r$.  Here 
\begin{align}\label{eq:modifiedS}
\hat{S}_{\sigma_1} : = \{w \in \mathbb{C} : &| \Im(w) | \leq 20\chi(b) \eta r, \nonumber\\
&2b(\sqrt N + \sqrt n)+138 \eta r \leq \Re(w) \leq \frac{3}{2}\sigma_1 +  20\chi(b) \eta r \}
\end{align}
 is a slight modification of the set ${S}_{\sigma}$ in \cite[Eq. (35)]{Wang24}. Note that $\widetilde{\sigma}_1 \in \hat{S}_{\sigma_1}$: the upper bound $\widetilde{\sigma}_1\le \frac{3}{2}\sigma_1$ follows from the Weyl's inequality and the supposition $\|\mathcal E\|\le \frac{1}{b}\sigma_1 \le \frac{1}{2}\sigma_1$; the lower bound is because it is the largest eigenvalue and $\widetilde{\sigma}_1 \ge z_j -20\chi(b) \eta r \ge \sigma_j -20\chi(b)\eta r$ for any $j\in \llbracket i_0, r_0 \rrbracket$ due to fact that the number of eigenvalues of $\mathcal A+ \mathcal E$ inside $\cup_{l=i_0}^{r_0} \mathcal{C}_l$ is $r_0-i_0+1$. For $z\in \hat{S}_{\sigma_1}$, $$2b(\sqrt N + \sqrt n)\le |z| \le 40\chi(b)\eta r +\frac{3}{2}\sigma_2 \le \frac{40\chi(b)}{80b}\sigma_1 + \frac{3}{2}\sigma_1 \le \frac{57}{32}\sigma_1 < 2n^2,$$ hence $z\in \mathsf D$ and the conclusion of Lemma 6.3 holds. In particular, the bound \eqref{eq:epsbd1} also holds for $z\in \hat{S}_{\sigma_1}$. 
 We show
\[ \eps(z) < \frac{1}{3r} \]
for all $z \in \mathcal{C}_0$.  The proof is similar to the proof of \eqref{eq:epsbd} and we sketch it here. For any $z \in \mathcal{C}_0$, from $|z-z_0|=20\chi(b) \eta r$  and $z_0-z_1 > 40\chi(b)  \eta r$, we obtain 
$|z| \le z_0 + 20\chi(b)  \eta r$ and $$|z| \ge z_0 - 20\chi(b)  \eta r \ge z_1 + 20\chi(b)  \eta r > \sigma_1 + 20\chi(b) \eta r > \sigma_1.$$ Again, by Lemma \ref{lemma:varphi}, we see for any $1\le l \le r$, 
\begin{align}\label{eq:onelastineq}
|\sigma_l^2 - \varphi(z)| = |\varphi(z_l) -\varphi(z)| &\ge \frac{1}{2} |z_l^2 - z^2| \nonumber\\
&\ge  \frac{1}{2} (z_l + \Re(z))(\Re(z) -z_l)\nonumber\\
&\ge \frac{1}{2}(\sigma_l + z_0 -20\chi(b) \eta r)(z_0-z_l -20\chi(b)\eta r) \nonumber\\
&\ge 10\chi(b) \eta r (\sigma_l + z_0 -20\chi(b)\eta r).
\end{align}
Plugging these estimates back into \eqref{eq:epsbd1}, we see
\begin{align*}
\eps(z) \le  \max_{1\le l \le r}\chi(b) \frac{\eta\sigma_l}{\sigma_1} \frac{\chi(b)(z_0 + 20\chi(b) \eta r) + \sqrt{2} \sigma_l}{10\chi(b)\eta r( z_0+\sigma_l -20\chi(b)\eta r)} < \frac{1}{5r},
\end{align*}
 where we used the bound $\chi(b)(z_0 + 20\chi(b) \eta r) + \sqrt{2} \sigma_l \le 2(z_0+\sigma_l -20\chi(b)\eta r)$ in the last inequality. 

By Lemma \ref{lem:comparezero}, $f$ has the same number of zeros inside $\mathcal{C}_0$ as $g$.  As $g$ has no zeros inside $\mathcal{C}_0$\footnote{This follows from Lemma \ref{lemma:varphi} and the fact that $\Im(\varphi(z)) \neq 0$ whenever $\Im z \neq 0$ for all $|z| > M$.}, $\mathcal A+\mathcal E$ has no eigenvalues inside $\mathcal{C}_0$.  Since the circle $\mathcal{C}_0$ was arbitrarily chosen inside this region, we conclude that $\mathcal A+\mathcal E$ has no eigenvalues larger than $z_1+20\chi(b)\eta r$.

\medskip

\noindent\emph{Case 2: $i_0>1$.} 
On the event $\mathcal F$, we have
\begin{equation}\label{ineq:big}
\max_{l\in \llbracket 1,r_0 \rrbracket; \sigma_l>n^2/2}|\widetilde \sigma_l - \sigma_l |\le \eta r.
\end{equation}
Note that $\sigma_{i_0-1}>n^3>n^2$.  Combining \eqref{ineq:big}, Proposition \ref{prop:bigbound} and $$z_{i_0-1} - z_{i_0}\ge \frac{75}{\xi(b)}\eta r\ge \frac{75}{\xi(2)}\eta r = 50\eta r,$$ which follows from the supposition $\delta_{i_0-1}\ge 75\chi(b) \eta r$ and the same argument as \eqref{eq:zljdiff}, we get 
\begin{align*}
\widetilde \sigma_{i_0-1} \ge \sigma_{i_0-1} - \eta r \ge z_{i_0-1} - \frac{3b}{b-1}\frac{1}{n} - \eta r \ge z_{i_0} + 50\eta r- \frac{6}{n} > z_{i_0} + 20\chi(b) \eta r.
\end{align*}
Hence, $\mathcal A+\mathcal E$ has at least $i_0-1$ eigenvalues larger than $z_{i_0}+20\chi(b)\eta r$. 

We first consider $\sigma_{i_0} >\frac{1}{2}n^2$.  It follows from \eqref{ineq:big} and Proposition \ref{prop:bigbound} that$$\widetilde \sigma_{i_0} \le \sigma_{i_0}+\eta r \le z_{i_0} + \frac{3b}{b-1}\frac{1}{n} +\eta r \le z_{i_0} +\frac{6}{n} +\eta r < z_{i_0} + 20\chi(b)\eta r.$$ This shows that $\mathcal A+\mathcal E$ has exactly $i_0-1$ eigenvalues larger than $z_{i_0}+20\chi(b)\eta r$.

Now consider $\sigma_{i_0} \le \frac{1}{2}n^2$. By Weyl's inequality, $\widetilde \sigma_{i_0} \le \sigma_{i_0} + \|E\| \le (1+\frac{1}{b})\sigma_{i_0}$. If $(1+\frac{1}{b})\sigma_{i_0} \le z_{i_0} + 20\chi(b)\eta r$,   the proof is already done. Now we assume $(1+\frac{1}{b})\sigma_{i_0} > z_{i_0} + 20\chi(b)\eta r$. If $(1+\frac{1}{b})\sigma_{i_0} - (z_{i_0} + 20\chi(b)\eta r)< 20\eta r$, following \eqref{eq:zjbd},  we have $\chi(b)\sigma_{i_0} \ge z_{i_0} > (1+\frac{1}{b})\sigma_{i_0} - 20(\chi(b)+1)\eta r$ and thus $\sigma_{i_0} < 80b\eta r\frac{(b-1)(\chi(b)+1)}{4b-5} $. Note that $\frac{(b-1)(\chi(b)+1)}{4b-5}$ is decreasing for $b\ge 2$ and $\frac{(b-1)(\chi(b)+1)}{4b-5} \le 17/24$.  Hence,  $\sigma_{i_0} < 80b\eta r$ contradicts the supposition that $\sigma_{i_0} \ge 80b\eta r$. 

It suffices to assume $(1+\frac{1}{b})\sigma_{i_0} - (z_{i_0} + 20\chi(b)\eta r)\ge 20\eta r$, which implies that $z_{i_0} \le (1+\frac{1}{b})\sigma_{i_0} - 20(\chi(b)+1)\eta r$. To prove $\widetilde \sigma_{i_0} \le z_{i_0} + 20\chi(b)\eta r$, we show that $f$ has no zeros on the interval $(z_{i_0}+20\chi(b)\eta r, (1+\frac{1}{b})\sigma_{i_0})$. The proof is similar to the proof of \emph{Case 1} when $i_0=1$. We only mention the differences. Define $\hat{S}_{\sigma_{i_0}}$ as in \eqref{eq:modifiedS} and the bound \eqref{eq:epsbd1} also holds for $z\in \hat{S}_{\sigma_{i_0}}$. The goal is to show $\varepsilon(z)<1/3r$ for all $z\in\mathcal C_0$, where $\mathcal C_0$ is any circle with radius $10  \eta r$ centered at a point $z_0 \in (z_{i_0}+20\chi(b) \eta r, (1+\frac{1}{b})\sigma_{i_0})$  inside the region $H_{i_0} \cap \hat{S}_{\sigma_{i_0}}$ such that $\dist(z_{0},z_{i_0}+20\chi(b) \eta r) \geq 10  \eta r$ and $\dist(z_{0},(1+\frac{1}{b})\sigma_{i_0}) \geq 10  \eta r$. If so, by Lemma \ref{lem:comparezero}, $f$ has the same number of zeros inside $\mathcal C_0$ as $g$.  Note that $g$ has no zeros inside $\mathcal C_0$ since $\Im(\varphi(z)) \neq 0$ whenever $\Im z \neq 0$ for all $|z| > M$ and $z_{i_0-1} \ge \sigma_{i_0-1} - \frac{3b}{b-1}\frac{1}{n} > n^2 - \frac{3b}{b-1}\frac{1}{n} > \frac{3}{2}\sigma_{i_0}\ge(1+\frac{1}{b})\sigma_{i_0}$ by Proposition \ref{prop:bigbound}. Since $\mathcal{C}_0$ was arbitrarily chosen, $\mathcal A + \mathcal E$ has no eigenvalues on $(z_{i_0}+20\chi(b)\eta r, (1+\frac{1}{b})\sigma_{i_0})$. 

It remains to bound $\eps(z)$ from \eqref{eq:epsbd1}. Note that $z_0 - z_{i_0} \ge 10 \eta r + 20\chi(b) \eta r$ and $(1+\frac{1}{b})\sigma_{i_0} - z_0 \ge 10 \eta r$.  For $z\in \mathcal C_0$, from $|z-z_0| = 10 \eta r$, we get $|z| \ge z_0 - 10 \eta r \ge z_{i_0} + 20\chi(b) \eta r \ge \sigma_{i_0} + 20\chi(b) \eta r> \sigma_{i_0}$ and $|z| \le z_0 + 10 \eta r\le (1+\frac{1}{b})\sigma_{i_0}.$

The same arguments as those in \emph{Case 1} yield that 
$$\max_{i_0\le l \le r}  \chi(b)\frac{\eta}{|z|}\frac{\sigma_l (\chi(b) |z|+\sqrt{2}\sigma_l)}{|\sigma_l^2 - \varphi(z)|} < \frac{1}{3r}$$ for any $z\in \mathcal C_0$. 
We only need to control
\begin{align*}
\max_{1\le l \le i_0-1}  \chi(b)\frac{\eta}{|z|}\frac{\sigma_l (\chi(b) |z|+\sqrt{2}\sigma_l)}{|\sigma_l^2 - \varphi(z)|}.
\end{align*}
For any $1\le l \le i_0-1$,  using similar computation from \eqref{eq:onelastineq}, we get
\begin{align*}
|\sigma_l^2 - \varphi(z)| &\ge \frac{1}{2}(\sigma_l + z_0 -10\eta r)(z_l-z_0 -10\eta r).
\end{align*}
Plugging in $z_0 \ge z_{i_0} + 10 \eta r + 20\chi(b) \eta r\ge \sigma_{i_0}+10 \eta r + 20\chi(b) \eta r$, we obtain $$\sigma_l + z_0 -20\chi(b)\eta r \ge \sigma_l + \sigma_{i_0} +20\chi(b)\eta r.$$ From $\sigma_l  \geq n^3$, we see $\sigma_{i_0}\leq \frac{1}{2}n^2 < \frac{1}{2}\sigma_l$. This, together with \eqref{eq:zjbd} and $z_0\le (1+\frac{1}{b}) \sigma_{i_0} - 10\eta r$, implies that 
\begin{align*}
z_l-z_0 -10\eta r &\ge \sigma_l -(1+\frac{1}{b}) \sigma_{i_0} \ge \sigma_l - \frac{1}{2}(1+\frac{1}{b}) \sigma_l \ge \frac{1}{4}\sigma_l.
\end{align*}
Hence, $|\sigma_l^2 - \varphi(z)| \ge \frac{1}{8} \sigma_l (\sigma_l + \sigma_{i_0} +20\chi(b)\eta r)$ and 
\begin{align*}
\max_{1\le l \le i_0-1}  \chi(b)\frac{\eta}{|z|}\frac{\sigma_l (\chi(b) |z|+\sqrt{2}\sigma_l)}{|\sigma_l^2 - \varphi(z)|}&\le \max_{1\le l \le i_0-1}  \chi(b)\frac{\eta \sigma_l}{\sigma_{i_0}}\frac{\chi(b) (1+\frac{1}{b}) \sigma_{i_0} +\sqrt{2}\sigma_l}{\frac{1}{8}(\sigma_l + \sigma_{i_0} +20\chi(b)\eta r) \sigma_l }\\
&< 45 \frac{\eta}{\sigma_{i_0}} \le  \frac{45\eta}{160\eta r}< \frac{1}{3r}
\end{align*}using the assumption $\sigma_{i_0}\ge 80b \eta r\ge 160 \eta r$ and the bound $\chi(b)\le \chi(2)=9/8$.
Therefore, $\eps(z) < {1}/{3r}$ for all $z\in \mathcal C_0.$ 

\medskip

The proof of \emph{Claim 3} is similar to the previous argument. Let $\mathcal C_{j_1}, \mathcal C_{j_2}$ be two disjoint circles for $j_1, j_2 \in \llbracket i_0, r_0\rrbracket$. Then $|z_{j_1}-z_{j_2}| > 40 \chi(b) \eta r$. Let $\mathsf d:=\dist(\mathcal C_{j_1}, \mathcal C_{j_2})>0.$ We show that $\mathcal A + \mathcal E$ has no eigenvalues lying on the real line between $\mathcal C_{j_1}$ and $\mathcal C_{j_2}$. Take any point $x$ on the real line between the two circles so that $\mathcal C_x$, the circle centered at $x$ with radius $\mathsf r:=\frac{1}{10}\min\{d, 20\chi(b)\eta r\}$ (say), is inside the region $H_{j_1}\cap S_{\sigma_{j_1}}$ or $H_{j_2}\cap S_{\sigma_{j_2}}$, where $\dist(x, \mathcal C_{j_1}) > \mathsf r$ and $\dist(x, \mathcal C_{j_2}) > \mathsf r$. Then using similar calculations as in the proof of \emph{Claim 2}, it suffices to show that $\eps(z)<{1}/{3r}$. The remaining arguments are similar to those in the proof of \emph{Claim 2}; we omit the details.

\section{Proofs of (28) and Proposition \ref{prop:blockphi}} \label{app:proof2parts}

\subsection{Proof of (28)}
For the sake of completeness, we include the proof of (28) in this section. The proof was established in \cite[Appendix B.4]{OVW22}.

Assuming $n \geq N$, the matrix $E$ has rank $N$ almost surely, allowing us to express its singular value decomposition as $E = \sum_{i=1}^N \eta_i x_i y_i^T,$ 
with its null space spanned by orthonormal vectors $h_1,\ldots,h_{n-N}$. Given the block structure of $\mathcal{E}$:
\[
\mathcal{E} = \begin{pmatrix} 0 & E\\ E^T & 0\end{pmatrix},
\]
its spectral properties are determined by those of $E$. Specifically, the resolvent $G(z)=(z-\mathcal{E})^{-1}$ has the eigendecomposition:
\[
G(z) = \sum_{i=1}^N \frac{w_i w_i^T}{z-\eta_i} + \sum_{i=1}^N \frac{w_{N+i} w_{N+i}^T}{z+\eta_i} + \frac{1}{z}\sum_{j=1}^{n-N} w_{2N+j} w_{2N+j}^{T},
\]
where the eigenvectors are $w_i =\frac{1}{\sqrt 2} (x_i^T, y_i^T)^T$, $w_{N+i} =\frac{1}{\sqrt 2} (x_i^T, -y_i^T)^T$, and $w_{2N+j} = ( 0, h_j^T)^T$.

From this, we can compute the traces
\begin{align*}
\tr\mathcal{I}^d G(z) &= \frac{1}{2} \sum_{i=1}^N \left(\frac{1}{z-\eta_i} + \frac{1}{z+\eta_i}\right) + \frac{n-N}{z},\\
\tr\mathcal{I}^u G(z) &= \frac{1}{2} \sum_{i=1}^N \left(\frac{1}{z-\eta_i} + \frac{1}{z+\eta_i}\right).
\end{align*}
Equation (28) follows from the definitions that $\phi_1(z) = z - \tr \mathcal{I}^\mathrm{d} G(z)$ and $\phi_2(z) = z- \tr \mathcal{I}^\mathrm{u} G(z)$.

\subsection{Proof of Proposition \ref{prop:blockphi}}\label{app:svaluecompute}
Note that $\Phi(x)$ is well-defined when $|x|>\|\mathcal E\|$.  We use $\alpha\equiv\alpha(x)$ and $\beta \equiv \beta(x)$ for brevity. Denote $J=\llbracket 1,2r\rrbracket\setminus I :=J_1 \cup J_2$ where $J_1= \llbracket 1,r\rrbracket\setminus\llbracket k,s\rrbracket$ and $J_2= \llbracket r+1,2r\rrbracket\setminus\llbracket r+k,r+s\rrbracket$. Further, let $\mathsf t:=|J_1| =|J_2|.$

We analyze the eigenvalues of 
\begin{align}\label{eq:positivematrix}
\left( I_{2\mathsf t} - \mathcal U_J^\T \Phi(x) \mathcal U_J \mathcal D_J \right)^\T \left( I_{2l} - \mathcal U_J^\T \Phi(x) \mathcal U_J \mathcal D_J \right),
\end{align}
where, using the expressions of $\mathcal D$ and  $\mathcal U^\T \Phi(z) \mathcal U$ from Section 6, we have
\begin{align*}
I_{2\mathsf t} - \mathcal U_J^\T \Phi(x) \mathcal U_J \mathcal D_J = I_{2\mathsf t}- \begin{pmatrix} \alpha I_{\mathsf t} & \beta I_{\mathsf t}\\
\beta I_{\mathsf t} & \alpha I_{\mathsf t} 
\end{pmatrix}
\begin{pmatrix} D_{J_1} &0\\
0 & -D_{J_1}
\end{pmatrix}=
\begin{pmatrix} I_{\mathsf t}- \alpha D_{J_1} & \beta D_{J_1}\\
-\beta D_{J_1} & I_{\mathsf t}+ \alpha D_{J_1}
\end{pmatrix}.
\end{align*}

The matrix \eqref{eq:positivematrix} can be expressed as
$$
\begin{pmatrix}
I_{\mathsf t}-2\alpha D_{J_1} + (\alpha^2 +\beta^2) D_{J_1}^2 & -2\alpha \beta D_{J_1}^2\\
-2\alpha \beta D_{J_1}^2 & I_{\mathsf t}+2\alpha D_{J_1} + (\alpha^2 +\beta^2) D_{J_1}^2
\end{pmatrix},
$$
where each block is a diagonal matrix. Direct computation shows that the eigenvalues of \eqref{eq:positivematrix} are
\begin{align*}
1+(\alpha^2 +\beta^2)\sigma_l^2 \pm 2\alpha \sigma_l \sqrt{1+\beta^2 \sigma_l^2} = \left(\sqrt{1+\beta^2 \sigma_l^2} \pm \alpha \sigma_l \right)^2
\end{align*}
for $l \in J_1$. 
The conclusion of Propositions \ref{prop:blockphi} follows by taking the square root of these eigenvalues.


\end{document}